\documentclass{amsart}
\usepackage{latexsym}
\usepackage{amsmath}
\usepackage{amsthm}
\usepackage{amssymb}
\usepackage{graphicx}
\usepackage{pdfpages}
\usepackage[labelformat=empty]{caption}

\newtheorem{thm}{Theorem}

\newtheorem{lem}[thm]{Lemma}

\newtheorem{cor}[thm]{Corollary}
\newtheorem{defi}[thm]{Definition}

\newtheorem{prop}[thm]{Proposition}
\newtheorem{rk}[thm]{Remark}
\newtheorem{nota}[thm]{Notation}

\newcommand{\vip}{\vskip.2cm}
\newcommand{\e}{{\varepsilon}}

\newcommand{\rr}{{\mathbb{R}}}
\newcommand{\nn}{{\mathbb{N}}}
\newcommand{\E}{\mathbb{E}}
\newcommand{\Et}{{\E_\theta}}
\newcommand{\Prt}{{\Pr_\theta}}
\newcommand{\Var}{\mathbb{V}{\rm ar}\,}
\newcommand{\Vart}{\mathbb{V}{\rm ar}_\theta\,}

\newcommand{\Covt}{\mathbb{C}{\rm ov}_\theta\,}
\newcommand{\cF}{{\mathcal F}}
\newcommand{\cA}{{\mathcal A}}
\newcommand{\cB}{{\mathcal B}}
\newcommand{\cL}{{\mathcal L}}
\newcommand{\cM}{{\mathcal M}}
\newcommand{\cE}{{\mathcal E}}
\newcommand{\cV}{{\mathcal V}}
\newcommand{\cW}{{\mathcal W}}
\newcommand{\cH}{{\mathcal H}}
\newcommand{\cP}{{\mathcal P}}
\newcommand{\cD}{{\mathcal D}}
\newcommand{\cZ}{{\mathcal Z}}
\newcommand{\cX}{{\mathcal X}}
\newcommand{\cN}{{\mathcal N}}
\newcommand{\cU}{{\mathcal U}}
\newcommand{\baZ}{{\bar Z}}
\newcommand{\bZ}{{\mathbf Z}}
\newcommand{\bI}{{\mathbf I}}
\newcommand{\bJ}{{\mathbf J}}
\newcommand{\bM}{{\mathbf M}}

\newcommand{\bU}{{\mathbf U}}
\newcommand{\bun}{{\mathbf 1}_N}
\newcommand{\indiq}{{{\mathbf 1}}}
\newcommand{\intot}{{\int_0^t}}

\begin{document}

\title{Statistical inference versus mean field limit for Hawkes processes}

\author{Sylvain Delattre}

\author{Nicolas Fournier}

\address{Sylvain Delattre, Laboratoire de Probabilit\'es et Mod\`eles Al\'eatoires, UMR 7599, 
Universit\'e Paris Diderot, Case courrier 7012, avenue de France, 75205 Paris Cedex 13, France.}

\email{sylvain.delattre@univ-paris-diderot.fr}

\address{Nicolas Fournier, Laboratoire de Probabilit\'es et Mod\`eles Al\'eatoires, UMR 7599,
Universit\'e Pierre-et-Marie Curie, Case 188, 4 place Jussieu, F-75252 Paris Cedex 5, France.}

\email{nicolas.fournier@upmc.fr}

\thanks{We thank the anonymous referees for their remarks that allowed us to improve consequently the 
presentation of the paper.}

\begin{abstract}
We consider a population of $N$ individuals, of which we observe the number of {\it actions}
until time $t$.
For each couple of individuals $(i,j)$, $j$ may or not {\it influence} $i$, which we model
by i.i.d. Bernoulli$(p)$-random variables, for some unknown parameter $p\in (0,1]$.
Each individual acts {\it autonomously} at some unknown rate $\mu>0$ and acts {\it by mimetism} at 
some rate proportional to the sum of some function $\varphi$ of the 
ages of the actions of the individuals which influence him.
The function $\varphi$ is unknown but assumed, roughly, to be decreasing and with fast decay.
The goal of this paper is to estimate $p$, which is the main characteristic of the
{\it graph of interactions}, in the asymptotic $N\to\infty$, $t\to\infty$.
The main issue is that the mean field limit (as $N \to \infty$) of this model is unidentifiable,
in that it only depends on the parameters $\mu$ and $p\varphi$.
Fortunately, this mean field limit is not valid for large times.
We distinguish the subcritical case, where, roughly, the mean number $m_t$ of actions per individual 
increases linearly
and the supercritical case, where $m_t$ increases exponentially.
Although the nuisance parameter $\varphi$ is non-parametric,
we are able, in both cases, to estimate $p$ without estimating $\varphi$ in a nonparametric way,
with a precision of order $N^{-1/2}+N^{1/2}m_t^{-1}$, up to some arbitrarily small loss.
We explain, using a Gaussian toy model, the reason why this 
rate of convergence might be (almost) optimal.
\end{abstract}

\subjclass[2010]{62M09, 60J75, 60K35}

\keywords{Multivariate Hawkes processes, Point processes, 
Statistical inference, Interaction graph, Stochastic interacting particles, Propagation of chaos, 
Mean field limit.}

\maketitle

\section{Introduction and main results}

\subsection{Setting}
We consider some unknown parameters $p\in (0,1]$, $\mu > 0$ and $\varphi : [0,\infty)\mapsto [0,\infty)$.
For $N\geq 1$, we consider an i.i.d. family 
$(\pi^i(dt,dz))_{i=1,\dots,N}$ of Poisson measures on $[0,\infty)\times[0,\infty)$
with intensity measure $dtdz$, independent of an i.i.d. family $(\theta_{ij})_{i,j=1,\dots,N}$ of 
Bernoulli$(p)$-distributed random variables. We also consider the system of equations,
for $i=1,\dots,N$, 
\begin{align}\label{sN}
Z^{i,N}_t=\intot\int_0^\infty \indiq_{\{z \leq \lambda^{i,N}_{s}\}} \pi^i(ds,dz) \quad \hbox{where} \quad
\lambda^{i,N}_t= \mu + \frac1N \sum_{j=1}^N\theta_{ij}\int_0^{t-} \varphi(t-s)dZ^{j,N}_s.
\end{align}
Here and in the whole paper, $\intot$ means $\int_{[0,t]}$ and $\int_0^{t-}$ means $\int_{[0,t)}$.
The solution $((Z^{i,N}_t)_{t\geq 0})_{i=1,\dots,N}$ is a family of $N$ counting processes (that is, a.s.
integer-valued, c\`adl\`ag and non-decreasing). The following 
well-posedness result is more or less well-known, see e.g. Br\'emaud-Massouli\'e \cite{bm} and \cite{dfh}
(we will apply directly the latter reference).

\begin{prop}\label{wp}
Assume that $\varphi$ is locally integrable and fix $N\geq 1$. 
The system \eqref{sN} has a unique c\`adl\`ag $(\cF_t)_{t\geq 0}$-adapted solution $((Z^{i,N}_t)_{t\geq 0})_{i=1,\dots,N}$
such that $\sum_{i=1}^N\E[Z^{i,N}_t]<\infty$ for all $t\geq 0$,
where $\cF_t=\sigma(\pi^i(A)\,:\, A \in \cB([0,t]\times[0,\infty)),i=1,\dots,N)\lor
\sigma(\theta_{ij}\,:\, i,j=1,\dots,N)$.
\end{prop}

Let us provide a brief heuristic description of this process.
We have $N$ individuals and $Z^{i,N}_t$ stands for the number of actions of the $i$-th individual until $t$.
We say that $j$ {\it influences} $i$ if and only if $\theta_{ij}=1$ (with possibly $i=j$).
Each individual $i$ acts, at time $t$, with rate $\lambda^{i,N}_t$.
In other words, each individual has an {\it autonomous rate of action} $\mu$
as well as a {\it subordinate rate of action} $N^{-1} \sum_{j=1}^N\theta_{ij}\intot \varphi(t-s)dZ^{j,N}_s$,
which depends on the number of actions of the individuals that influence him,
with a weight $N^{-1}$ and taking into account the age of these actions through $\varphi$.
If for example $\varphi=a\indiq_{[0,K]}$, then the subordinate rate of action
of $i$ is simply $a/N$ times the total number of actions, during $[t-K,t]$, of all the individuals
that influence him.

\vip

As is well-known, a phase-transition occurs for such a model,
see Hawkes-Oakes \cite{ho} (or
\cite{dfh} for such considerations on large networks): setting $\Lambda=\int_0^\infty \varphi(t)dt$,

$\bullet$ in the subcritical case where $\Lambda p<1$, we will see that 
$Z^{1,N}_t$ increases linearly with time,
at least on the event where the family $(\theta_{ij})_{i,j=1,\dots,N}$ behaves reasonably;

$\bullet$ in the supercritical case where $\Lambda p>1$, 
we will see that $Z^{1,N}_t$ increases exponentially 
fast with time, at least on the event where the family $(\theta_{ij})_{i,j=1,\dots,N}$ behaves reasonably.

\vip

The limit theorems, and thus the statistical inference, completely differ in both cases,
so that the present paper contains essentially two independent parts.

\vip
 
We will not study the critical case where $\Lambda p=1$ because it is a very particular case.
However, it would be very interesting to understand what happens {\it near} the critical case.
Our results say nothing about this problem.

\subsection{Assumptions}
Recalling that $\Lambda=\int_0^\infty \varphi(s)ds$,
we will work under one of the two following conditions:
either for some $q\geq 1$,
\renewcommand\theequation{{$H(q)$}}
\begin{equation}
\mu \in (0,\infty), \quad \Lambda p \in (0,1) \quad \hbox{and} 
\quad \int_0^\infty s^q\varphi(s)ds <\infty
\end{equation}
or
\renewcommand\theequation{{$A$}}
\begin{equation}
\mu \in (0,\infty), \quad \Lambda p \in (1,\infty) \quad \hbox{and} \quad \int_0^t |d\varphi(s)|
\hbox{ increases at most polynomially.}
\end{equation}
\renewcommand\theequation{\arabic{equation}}\addtocounter{equation}{-2}
In many applications, $\varphi$ is smooth and has a fast decay, so that, except
in the critical case, 
either $H(q)$ is satisfied for all $q\geq 1$ or $A$ is satisfied.

\subsection{References and fields of application}

Hawkes processes have been introduced by Hawkes \cite{h}
and Oakes-Hawkes \cite{ho} have found a noticeable representation of such processes in terms of
Galton-Watson trees. Since then, there has been a huge literature on Hawkes processes,
see e.g.  Daley and Vere-Jones \cite{dvj} for an introduction, Massouli\'e \cite{m}, Br\'emaud-Massouli\'e 
\cite{bm} and \cite{dfh} for stability results, Br\'emaud-Nappo-Torrisi \cite{bnt}, Zhu \cite{z1,z2}
and \cite{bdhm2} for limit theorems, etc. Hawkes processes are used in various fields of applications: 

$\bullet$ earthquake replicas in seismology, see Helmstetter-Sornette \cite{hs}, Kagan \cite{k}, Ogata \cite{o2}, 
Bacry-Muzy \cite{bmu2}, 

$\bullet$ spike trains for brain activity in neuroscience, see Gr\"un et al. \cite{gda}, Okatan et al. \cite{owb},
Pillow et al. \cite{psp}, Reynaud et al. \cite{rrgt,rrt}, 

$\bullet$ genome analysis, see Reynaud-Schbath \cite{rs}, 

$\bullet$ various fields of mathematical finance, see Ait-Sahalia et al. \cite{acl}, Bauwens-Hautsch \cite{bh}, 
Hewlett \cite{he}, Bacry et al. \cite{bdhm1}, Bacry-Muzy \cite{bmu,bmu2},

$\bullet$ social networks interactions, see Blundell et al. \cite{bhb} and Zhou et al. \cite{zzs}.

\vip

Concerning the statistical inference for Hawkes processes, only the case of fixed finite dimension $N$ has been
studied, to our knowledge, in the asymptotic $t\to \infty$ (for possibly more general shapes of interaction).
Some parametric and nonparametric estimation procedures for $\mu$ and $\varphi$ have been proposed,
with or without rigorous proofs. Let us mention Ogata \cite{o1}, Bacry-Muzzy \cite{bmu2}, 
\cite{bdhm1}, the various recent results of Hansen et al. \cite{hrr} and Reynaud et al. \cite{rs,rrgt,rrt}, 
as well as the Bayesian study of Rasmussen \cite{r}.

\subsection{Goals and motivation}

In many applications, the number of individuals is very large (think of neurons, financial agents
or of social networks). Then we need some estimators in the asymptotic where $N$ and $t$ 
tend simultaneously
to infinity. This problem seems to be completely open.

\vip

We assume that we observe $(Z^{i,N}_s)_{i=1,\dots,N,s\in [0,t]}$
(or, for convenience, $(Z^{i,N}_s)_{i=1,\dots,N,s\in [0,2t]}$), that is all the actions
of the individuals on some (large) time interval.

\vip

In our point of view, we only observe the activity of the individuals, we do not know 
the {\it graph of interactions}. A very similar problem was studied in \cite{rrt}, although in fixed finite 
dimension
$N$. Our goal is to estimate $p$, which can be seen as the main characteristic of the graph of interactions,
since it represents the proportion of open edges. We consider $\mu$ and $\varphi$ as nuisance parameters,
although this is debatable.
In the supercritical case, we will be able to estimate $p$ without estimating $\mu$ nor $\varphi$.
In the subcritical case, we will be able to recover $p$ estimating only $\mu$ and the integral 
$\Lambda$ of $\varphi$.
In any case, we will {\it not} need to provide a nonparametric estimation of $\varphi$, and we believe 
it is a very good point: it would require regularity assumptions
and would complicate a lot the study.

\vip

The main goal of this paper is to provide the basic tools for the statistical estimation
of Hawkes processes when both the graph size and the observation time increase. 
Of course, this is only a toy model and we have no precise idea of real world applications, although
we can think e.g. of neurons spiking: 
they are clearly numerous (so $N$ is large), we can only observe their activities
(each time they spike), and we would like to have an idea of the graph of interactions.
See again \cite{rrt} for a more convincing biological background.
Think also of financial agents: they are also numerous, we can observe their actions (each time they buy or sell
a product), and we would like to recover the interaction graph.

\subsection{Mean field limit}

We quickly describe the expected {\it chaotic} behavior of $((Z^{i,N}_t)_{t\geq 0})_{i=1,\dots,N}$
as $N\to \infty$. We refer to Sznitman \cite{s} for an introduction to {\it propagation of chaos}.
Extending the method of \cite[Theorem 8]{dfh}, it is not hard to check, assuming that
$\int_0^\infty \varphi^2(s)ds <\infty$, that
for each given $k\geq 1$ and $T>0$, the sample $((Z^{i,N}_t)_{t\in [0,T]})_{i=1,\dots,k}$ 
goes in law, as $N\to\infty$, to a family $((Y^{i}_t)_{t\in [0,T]})_{i=1,\dots,k}$ of i.i.d. inhomogeneous
Poisson processes with intensity $(\lambda_t)_{t\geq 0}$, unique locally bounded nonnegative solution to
$\lambda_t=\mu + \intot p\varphi(t-s)\lambda_sds$. 

\vip

On the one hand, approximate independence is of course a good point for statistical inference.
On the other hand, the mean-field limit (i.e. the $(Y^i_t)_{t\geq 0}$'s) depends on $p$ and $\varphi$ only 
through $(\lambda_t)_{t\geq 0}$ and thus
through $p\varphi$, which is a negative point: the mean-field limit is {\it unidentifiable}.
The situation is however not hopeless
because roughly, the mean-field limit {\it does not} hold true for the whole sample 
$(Z^{i,N})_{i=1,\dots,N}$ and is less and less true as time becomes larger and larger.

\subsection{Main result in the subcritical case}

For $N\geq 1$ and for $((Z^{i,N}_t)_{t\geq 0})_{i=1,\dots,N}$ the solution to \eqref{sN}, we introduce
$\baZ^N_t=N^{-1}\sum_{i=1}^N Z^{i,N}_t$. We mention in the following remark, that
we will prove later, that the number of actions per individual increases linearly in the subcritical case.

\begin{rk}\label{rksc} 
Assume $H(1)$. Then for all $\e>0$,
$$
\lim_{(N,t)\to (\infty,\infty)} \Pr \Big(\Big| \frac{\bar Z^N_t}t- \frac \mu {1-\Lambda p}\Big|\geq \e \Big)=0.
$$
\end{rk}

We next introduce
$$
\cE^N_t= \frac{\baZ^N_{2t} - \baZ^N_t}{t}, \qquad
\cV^N_t= \sum_{i=1}^N \Big(\frac{Z^{i,N}_{2t}-Z^{i,N}_t}{t} - \cE^N_t \Big)^2 - \frac Nt \cE^N_t,
$$
$$
\cW^N_{\Delta,t} = 2 \cZ^N_{2\Delta,t}-\cZ^N_{\Delta,t}, \quad \hbox{where} 
\quad \cZ^N_{\Delta,t}=\frac Nt \sum_{k=t/\Delta+1}^{2t/\Delta}\Big(\baZ^N_{k \Delta} - \baZ^N_{(k-1) \Delta}  
-\Delta \cE^N_t\Big)^2 .
$$
In the last expression, $\Delta \in (0,t)$ is required to be such that
$t/(2\Delta) \in \nn^*$.

\begin{thm}\label{mr}
Assume $H(q)$ for some $q > 3$. 
For $t\geq 1$, put $\Delta_t= t/(2 \lfloor t^{1-4/(q+1)}\rfloor)$: it holds that $t/(2\Delta_t) \in \nn^*$
and that $\Delta_t \sim t^{4/(q+1)}/2$ as $t\to \infty$.
There is a constant $C$ depending only on $p$, $\mu$, $\varphi$ and $q$
such that for all $\e \in (0,1)$, all $N\geq1$, all $t\geq 1$,
\begin{gather*}
\Pr\Big(\Big|\cE^N_t - \frac{\mu}{1-\Lambda p} \Big| \geq \e\Big) \leq \frac 
C\e \Big(\frac 1N +\frac1{\sqrt{Nt}} +
\frac 1{t^{q}} \Big),\\
\Pr\Big(\Big|\cV^N_t - \frac{\mu^2\Lambda^2p(1-p)}{(1-\Lambda p)^2} \Big| \geq \e \Big) \leq \frac C\e 
\Big(\frac {\sqrt N}t + \frac 1{\sqrt N} \Big),\\
\Pr\Big(\Big|\cW^N_{\Delta_t,t} - \frac{\mu}{(1-\Lambda p)^3}  \Big| \geq \e \Big)
\leq  \frac C \e \Big(\frac 1 {N} + \frac N{t^2} + \frac{1}{\sqrt{t^{1-4/(q+1)}}}\Big).
\end{gather*}
\end{thm}

We will easily deduce the following corollary.

\begin{cor}\label{mc}
Assume $H(q)$ for some $q > 3$. 
For $t\geq 1$, put $\Delta_t= t/(2 \lfloor t^{1-4/(q+1)}\rfloor)$.
There is a constant $C$ depending only on $p$, $\mu$, $\varphi$ and $q$
such that for all $\e \in (0,1)$, all $N\geq1$, all $t\geq 1$,
\begin{align*}
\Pr\Big(\Big\|\Psi\Big(\cE^N_t,\cV^N_t,\cW^N_{\Delta_t,t}\Big) - (\mu,\Lambda,p) \Big\| \geq \e\Big)
\leq& \frac C\e \Big(\frac 1{\sqrt N} + \frac{\sqrt N}t + \frac1{\sqrt{t^{1-4/(q+1)}}}  \Big)\\
\leq& \frac {2C}\e \Big(\frac 1{\sqrt N} + \frac{\sqrt N}{t^{1-4/(q+1)}}  \Big),
\end{align*}
where $\Psi = \indiq_{D}\Phi$ with $D=\{(u,v,w) \in \rr^3 \; : \;  w>u>0 \; \hbox{and} \; v\ge0\}$ and 
$\Phi : D \mapsto \rr^3$ defined by
$$
\Phi_1(u,v,w)=u\sqrt{\frac u w}, \quad
\Phi_2(u,v,w)= \frac{v + [u-\Phi_1(u,v,w)]^2}{u[u-\Phi_1(u,v,w)]},
\quad \Phi_3(u,v,w)=\frac{1-u^{-1}\Phi_1(u,v,w)}{\Phi_2(u,v,w)}.
$$
\end{cor}

We did not optimize the dependence in $q$:
in many applications, $H(q)$ holds for all $q\geq 1$.

\subsection{Main result in the supercritical case}

For $N\geq 1$ and for $((Z^{i,N}_t)_{t\geq 0})_{i=1,\dots,N}$ the solution to \eqref{sN}, we set
$\baZ^N_t=N^{-1}\sum_{i=1}^N Z^{i,N}_t$. We will check later the following remark, which states
that the mean number of actions per individual increases exponentially in the supercritical case.

\begin{rk}\label{rkaz} 
Assume $A$ and consider $\alpha_0>0$ uniquely 
defined by $p \int_0^\infty e^{-\alpha_0 t}\varphi(t)dt =1$. Then
$$ 
\hbox{for all $\eta>0$,}\quad
\lim_{t\to \infty} \lim_{N\to\infty} \Pr \Big( \baZ^N_t \in [e^{(\alpha_0-\eta)t},e^{(\alpha_0+\eta)t} ]\Big)=1.
$$
\end{rk}

We next introduce

\begin{align*}
\cU^N_t= \Big[\sum_{i=1}^N\Big(\frac{Z^{i,N}_t-\baZ^N_{t}}{\baZ^N_t}\Big)^2 - \frac{N}{\baZ^N_t} 
\Big]\indiq_{\{\baZ^N_t>0\}}
\quad \hbox{and}
\quad \cP^N_t=\frac{1}{\cU^N_t+1}\indiq_{\{\cU^N_t\geq 0\}}.
\end{align*}

\begin{thm}\label{mr2}
Assume $A$ and consider $\alpha_0>0$ defined in Remark \ref{rkaz}. For all $\eta>0$, there
is a constant $C_\eta>0$ (depending only on $p,\mu,\varphi,\eta$) such that for all 
$N\geq 1$, all $t\geq 1$, all $\e \in (0,1)$,
$$
\Pr\Big(|\cP^N_t - p|\geq \e \Big) \leq \frac{C_\eta e^{\eta t}}\e \Big(\frac {\sqrt N} {e^{\alpha_0 t}} 
+ \frac 1 {\sqrt N} 
\Big).
$$
\end{thm}

\subsection{Detecting subcriticality and supercriticality}

In practise, we may of course not know if we are in the subcritical or supercritical case.

\begin{prop}\label{decid}
(i) Under $H(1)$, there are some constants $0<c<C$ depending only on $p,\mu,\varphi$ such that 
for all $N\geq 1$, all $t\geq 1$,
$\Pr(\log(\bar Z^N_t) \geq (\log t)^2)\leq C(e^{-c N} + t^{1/2}e^{-(\log t)^2})$.

\vip

(ii) Under $A$, for all $\eta>0$, there is a constant $C_\eta$ depending only on $p,\mu,\varphi, \eta$ such that 
for all $N\geq 1$, all $t\geq 1$, $\Pr(\log(\bar Z^N_t) \leq (\log t)^2)
\leq C_\eta e^{\eta t}(N^{-1/2} + e^{-\alpha_0 t})$.
\end{prop}

It is then not hard to check that, with the notation of Corollary \ref{mc} and Theorem \ref{mr2},
under $H(q)$ (for some $q>3$) or $A$, the estimator
$$
\hat p^N_t = \indiq_{\{\log(\bar Z^N_t) < (\log t)^2\}} \Psi_3(\cE^N_{t/2},\cV^N_{t/2},\cW^N_{\Delta_{t/2},t/2})
+  \indiq_{\{\log(\bar Z^N_t) > (\log t)^2\}} \cP^N_t,
$$ 
which is based on the observation of $(Z^{i,N}_s)_{s \in [0,t], i=1,\dots,N}$,
converges in probability to $p$,
with the same speed of convergence as in Corollary \ref{mc} (under $H(q)$ for some $q>3$) 
or as in Theorem \ref{mr2} (under $A$).

\subsection{About optimality}  

In Subection \ref{sbi}, we will see on a toy model that there is no real hope to find
an estimator of $p$ with a better precision than $N^{-1/2}+ N^{1/2}m_t^{-1}$, where $m_t$ is something 
like the mean number
of jumps per individual during $[0,t]$. 
Consequently, we believe that the precision we found in Corollary \ref{mc} is almost optimal,
since then $m_t\simeq t$ by Remark \ref{rksc} and since we reach the precision 
$N^{-1/2}+ N^{1/2}t^{\alpha-1}$ for any $\alpha>0$ (if $\varphi$ has a fast decay), 
so that the loss is arbitrarily small.
Similarly, the precision found in Theorem \ref{mr2} is rather satisfying, since then 
$m_t \simeq e^{\alpha_0t}$ by Remark \ref{rkaz} and since we reach the precision 
$e^{\eta t}(N^{-1/2}+ N^{1/2}e^{-\alpha_0 t})$ for any $\eta>0$,
so that the loss is, here also, arbitrarily small.

\vip

The main default of the present paper is that the constants in Corollary \ref{mc}
and in Theorem \ref{mr2} strongly depend on the parameters
$\mu,\Lambda,p$. They also depend on $q$ in the subcritical case. 
In particular, it would be quite delicate to understand how they behave
when approaching, from below or from above, the critical case.

\subsection{About the modeling}
There are two main limitations in our setting.
\vip
Assuming that the $\theta_{ij}$'s are i.i.d. is of course a strong assumption. What we really need
is that the family $(\theta_{ij})_{i,j=1,\dots,N}$ satisfies similar properties as 
those shown in Subsection \ref{smat} (in the subcritical case) and 
in Subsection \ref{rm2} (in the supercritical case).
This clearly requires that the family $(\theta_{ij})_{i,j=1,\dots,N}$ is not too far from being i.i.d.,
and it does {\it not} suffice that $\lim_{N\to\infty} N^{-2}\sum_{i,j=1}^N \theta_{ij}=p$.
However, we believe that all the conclusions of the present paper are still true if one assumes
that $(\theta_{ij})_{1\leq i\leq j\leq N}$ is i.i.d. and that $\theta_{ji}=\theta_{ij}$ for all
$1\leq i < j\leq N$, which might be the case in some applications where the interactions
are symmetric. A rigorous proof would require some work but should not be too hard.
We will study this problem numerically at the end of the paper.

\vip

Assuming that we observe all the population is also rather stringent.
It would be interesting to study what happens if one observes only $(Z^{i,N}_s)_{i=1,\dots,K,s\in [0,t]}$,
for some $K$ large but smaller than $N$. It is not difficult to guess how to adapt the estimators
to such a situation (see Section \ref{num} for precise formulae).  
The theoretical analysis would require a careful and tedious study.
Again, we will discuss this numerically.

\subsection{Notation}

We denote by $\Prt$  the conditional probability knowing
$(\theta_{ij})_{i,j=1,\dots,N}$. We introduce $\Et$, $\Vart$ and $\Covt$ accordingly.

\vip

For two functions $f,g: [0,\infty) \mapsto \rr$, we introduce (if it exists) 
$(f \star g) (t) = \intot f(t-s)g(s)ds$. The functions $\varphi^{\star n}$
will play an important role in the paper. 
Observe that, since $\int_0^\infty \varphi(s)ds = \Lambda$,
$\int_0^\infty \varphi^{\star n}(s)ds = \Lambda^n$.
We adopt the conventions $\varphi^{\star 0}(s)ds=\delta_0(ds)$ and $\varphi^{\star 0}(t-s)ds=\delta_t(ds)$. 
We also adopt the convention that $\varphi^{\star n} (s)=0$ for $s<0$.

\vip

All the finite constants used in the upperbounds are denoted by $C$, 
the positive constants used in the lowerbounds are denoted by $c$ and
their values change from line to line.
They are allowed to depend only on $\mu$, $p$ and $\varphi$ (and on $q$ under $H(q)$),
but never on $N$ nor on $t$. Any other dependence will be indicated in subscript. For example,
$C_\eta$ is a finite constant depending only on $\mu$, $p$, $\varphi$ and $\eta$
(and on $q$ under $H(q)$).

\subsection{Plan of the paper}

In the next section, we try to give the main reasons why our estimators should be convergent,
which should help the reader to understand the strategies of the proofs. We also briefly and formally introduce 
a Gaussian
toy model in Section \ref{sbi} to show that the rates of convergence we obtain are not far
from being the best we can hope for.
In Section \ref{swpmf}, we prove 
Proposition \ref{wp} (strong existence and uniqueness of the process)
and check a few more or less explicit formulae concerning $(Z^{i,N}_t)_{i=1,\dots,N,t\geq 0}$ of constant use.
Section \ref{souc} is devoted to the proof of Theorem \ref{mr} and Corollary \ref{mc} (main results
in the subcritical case). Theorem \ref{mr2} (main result in the supercritical case) is proved in 
Section \ref{surc}. We check Proposition \ref{decid} in Section \ref{htd}.
Finally, we illustrate numerically the results of the paper and some possible extensions in the last section.

\section{Heuristics}\label{heu}

This section is completely informal and the symbol $\simeq$ means nothing precise.
For example,
``$Z^{i,N}_t \simeq \Et[Z^{i,N}_t]$ for $t$ large'' should be understood as 
``we hope that $Z^{i,N}_t /\Et[Z^{i,N}_t]$ tends to $1$ as $t\to \infty$ in probability or in another sense.''

\subsection{The subcritical case}\label{heusou}
We assume that $\Lambda p \in [0,1)$ and try to explain the asymptotics of $(Z^{i,N}_t)_{i=1,\dots,N,t\geq 0}$
and where the three estimators $\cE^N_t$, $\cV^N_t$ and $\cW^N_{\Delta,t}$ come from.
We introduce the matrices $A_N(i,j)=N^{-1}\theta_{ij}$ and $Q_N=(I-\Lambda A_N)^{-1}$, which
exists with high probability because $\Lambda p <1$. We also set
$\ell_N(i)=\sum_{j=1}^N Q_N(i,j)$ and $c_N(i)=\sum_{j=1}^N Q_N(j,i)$.

\vip

Fixing $N$ and knowing $(\theta_{ij})_{i,j=1,\dots,N}$, we expect that $Z^{i,N}_t \simeq \Et[Z^{i,N}_t]$
for $t$ large by a law of large numbers. Next, it is not hard to check that
$\Et[Z^{i,N}_t]=\mu t + N^{-1}\sum_{j=1}^N \theta_{ij} \intot \varphi(t-s)
\Et[Z^{j,N}_s]ds$. Assume now that $\gamma_N(i)=\lim_{t\to \infty} t^{-1}\Et[Z^{i,N}_t]$
exists for each $i=1,\dots,N$. 
Then, using that  $\intot \varphi(t-s) s ds \simeq \Lambda t$ for $t$ large, we find that the vector
$\gamma_N$ must solve $\gamma_N = \mu \bun + \Lambda A_N \gamma_N$, where $\bun$ is the $N$-dimensional 
vector with all coordinates equal to $1$. This implies that $\gamma_N=\mu(I-\Lambda A_N)^{-1}\bun=\mu \ell_N$.
We thus expect that $Z^{i,N}_t \simeq \Et[Z^{i,N}_t]\simeq \mu \ell_N(i) t$.

\vip

Based on this and setting $\bar \ell_N=N^{-1}\sum_{i=1}^N \ell_N(i)$, 
we expect that $\bar Z^N_t \simeq \mu \bar \ell_N t$ for large values of $t$,
whence $\tilde{\cE}^N_t:=t^{-1}\bar Z^N_t \simeq  \mu \bar \ell_N$.

\vip

Knowing $(\theta_{ij})_{i,j=1,\dots,N}$, 
$Z^{1,N}_t$ should resemble, roughly, a Poisson process, so that it should approximately hold true
that $\Vart (Z^{1,N}_t) \simeq \Et[Z^{1,N}_t]$.
Thus 
$N^{-1} \sum_{i=1}^N( Z^{i,N}_t - \bar Z^N_t)^2$ should resemble
$\Var(Z^{1,N}_t)=\Var(\Et[Z^{1,N}_t]) + \E[\Vart(Z^{1,N}_t)]\simeq \Var(\Et[Z^{1,N}_t]) + \E[Z^{1,N}_t]$,
which itself resembles 
$N^{-1}\sum_{i=1}^N(\Et[Z^{i,N}_t]-\Et[\baZ^{N}_t])^2+\baZ^N_t
\simeq N^{-1} \mu^2 t^2 \sum_{i=1}^N( \ell_N(i) - \bar \ell_N)^2+\baZ^N_t$. Consequently, we expect that
$\tilde {\cV}^N_t:=t^{-2}[\sum_{i=1}^N( Z^{i,N}_t - \bar Z^N_t)^2 - N \baZ^N_t]
\simeq \mu^2 \sum_{i=1}^N( \ell_N(i) - \bar \ell_N)^2$ for $t$ large.

\vip

Finally, the {\it temporal} empirical variance 
$\Delta t^{-1} \sum_{k=1}^{t/\Delta} [\baZ^N_{k\Delta}-\baZ^N_{(k-1)\Delta}-\Delta t^{-1}\baZ^N_t]^2$
should resemble $\Vart[\baZ^N_\Delta]$ if $1\ll\Delta\ll t$.
Thus $\tilde{\cW}^N_{\Delta,t}:=N t^{-1} \sum_{k=1}^{t/\Delta} [\baZ^N_{k\Delta}-\baZ^N_{(k-1)\Delta}-\Delta t^{-1}\baZ^N_t]^2
\simeq N \Delta^{-1} \Vart[\baZ^N_\Delta]$.
Introducing the martingales $M^{i,N}_t=Z^{i,N}_t-C^{i,N}_t$ (where $C^{i,N}$ is the compensator of $Z^{i,N}$),
the centered processes $U^{i,N}_t=Z^{i,N}_t-\Et[Z^{i,N}_t]$, 
and the $N$-dimensional vectors $\bU^N_t$ and $\bM^N_t$ with coordinates $U^{i,N}_t$ and $M^{i,N}_t$,
we will see in Section \ref{swpmf} that
$\bU^N_t=\bM^N_t+ A_N \intot \varphi(t-s) \bU^N_s ds$, so that for large times, $\bU^N_{t}\simeq \bM^N_t+
\Lambda A_N \bU^N_t$ and thus $\bU^N_t \simeq Q_N \bM^N_t$.
Consequently, we hope that $\bar U^N_t \simeq \overline{Q_N \bM^N_t}$, where
$\bar U^N_t=N^{-1}\sum_{i=1}^NU^{i,N}_t$ and $\overline{Q_N \bM^N_t}=N^{-1}\sum_{i=1}^N (Q_N\bM^N_t)_i$.
A little study shows that the martingales $M^{j,N}_t$ are orthogonal and that $[M^{j,N},M^{j,N}]_t=Z^{j,N}_t\simeq
\mu \ell_N(j)t$, so that $\Vart (\overline{Q_N \bM^N_t})\simeq
\mu t N^{-2} \sum_{j=1}^N (\sum_{i=1}^N Q_N(i,j))^2 \ell_N(j)
=\mu t N^{-2} \sum_{j=1}^N (c_N(j))^2\ell_N(j)$. Finally, $\Vart [\baZ^N_{t}]=\Vart [\bar U^N_{t}]\simeq 
\mu t N^{-2} \sum_{j=1}^N (c_N(j))^2\ell_N(j)$ and we hope that
$\tilde{\cW}^N_{\Delta,t}\simeq N \Delta^{-1} \Vart[\baZ^N_\Delta]\simeq \mu  N^{-1} \sum_{j=1}^N (c_N(j))^2\ell_N(j)$
if $1\ll\Delta\ll t$.

\vip

We thus need to find the limits of
$\bar \ell_N$,  $\sum_{i=1}^N( \ell_N(i) - \bar \ell_N)^2$ and $N^{-1} \sum_{i=1}^N \ell_N(i)(c_N(i))^2$ 
as $N\to\infty$.
It is not easy to make rigorous, but it holds true
that ${\ell_N(i) \simeq 1 + \Lambda (1-\Lambda p)^{-1} L_N(i)}$, where $L_N(i)=\sum_{j=1}^N A_N(i,j)$.
This comes from $\sum_{j=1}^N A_N^2(i,j)
=\sum_{j=1}^N A_N(i,j)\sum_{k=1}^N A_N(j,k)\simeq p\sum_{j=1}^N A_N(i,j)=pL_N(i)$,
$\sum_{j=1}^N A_N^3(i,j)\simeq p^2 L_N(i)$ for similar reasons, etc. It is very rough, but it will imply that
$\ell_N(i)=\sum_{n\geq 0} \Lambda^n \sum_{j=1}^N A_N^n(i,j)\simeq 
1+ \sum_{n\geq 1} \Lambda^n p^{n-1} L_N(i)=1+\Lambda(1-\Lambda p)^{-1} L_N(i)$.
Once this is seen (as well as a similar fact for the columns), we get convinced,
$NL_N$ being a vector of $N$ i.i.d. Binomial$(N,p)$-distributed random variables, that
$\bar \ell_N\simeq 1/(1-\Lambda p)$, that  $\sum_{i=1}^N( \ell_N(i) - \bar \ell_N)^2\simeq \Lambda^2 
p(1-p)/(1-\Lambda p)^2$ and 
that $N^{-1} \sum_{i=1}^N \ell_N(i)(c_N(i))^2\simeq 1/(1-\Lambda p)^3$.

\vip

At the end, it should be more or less true that, for $t,\Delta$ and $N$ large enough and in a suitable regime, 
$\tilde{\cE}^N_t\simeq \mu/(1-\Lambda p)$, $\tilde{\cV}^N_t\simeq \mu^2 \Lambda^2 p(1-p)/(1-\Lambda p)^2$,
and $\tilde{\cW}^N_{\Delta,t}\simeq \mu/(1-\Lambda p)^3$.
Of course, all this is completely informal and many points have to be clarified. 

\vip

Observe that concerning
$\tilde\cV^N_t$, we use that $Z^{1,N}_t$ resembles a Poisson process, while concerning $\tilde{\cW}^N_{\Delta,t}$,
we use that $\baZ^N_t$ does not resemble a Poisson process.

\vip

The three estimators $\cE^N_t,\cV^N_t,\cW^N_{\Delta, t}$ we study in the paper
resemble much $\tilde \cE^N_t,\tilde \cV^N_t,\tilde \cW^N_{\Delta, t}$ and should converge to the same limits.
Let us explain why we have modified the expressions.
We started this subsection 
by the observation that $\Et[Z^{i,N}_t] \simeq \mu \ell_N(i) t$, on which
the construction of the estimators relies.
A detailed study shows that, under $H(q)$, $\Et[Z^{i,N}_t]= \mu \ell_N(i) t + \chi^N_i \pm t^{1-q}$,
for some finite random variable $\chi^N_i$.
As a consequence, $t^{-1}\Et[Z^{i,N}_{2t}-Z^{i,N}_t]$ converges to $\mu \ell_N(i)$ considerably faster 
(with an error in $t^{-q}$) than $t^{-1}\Et[Z^{i,N}_{t}]$ (for which the error is of order $t^{-1}$).
This explains our modifications and why these modifications are crucial.

\vip
Let us conclude this subsection with a technical issue. 
If $\Lambda>1$ (which is not forbidden even in the subcritical case), 
there is a positive probability
that an anomalously high proportion of the $\theta_{ij}$'s equal $1$, so that $I-\Lambda A_N$ is not invertible
and our multivariate Hawkes process is supercritical (on this event with small probability).
We will thus work on an event $\Omega_N^1$ on which such problems do not occur
and show that this event has a high probability.

\subsection{The supercritical case}\label{heusur}
We now assume that $\Lambda p >1$ and explain the asymptotics of $(Z^{i,N}_t)_{i=1,\dots,N,t\geq 0}$
and where the estimator $\cU^N_t$ comes from. We introduce $A_N(i,j)=N^{-1}\theta_{ij}$.

\vip

Fixing $N$ and knowing $(\theta_{ij})_{i,j=1,\dots,N}$, we 
expect that $Z^{i,N}_t \simeq H_N \Et[Z^{i,N}_t]$, for some random $H_N>0$ not depending on $i$
(and with $H_N$ almost constant for $N$ large). This is typically a supercritical phenomenon, 
that can already be observed on Galton-Watson processes.
Fortunately, we will not really need to check it nor to study $H_N$, essentially 
because we will use the ratios $Z^{i,N}_t/\baZ^N_t$, which makes disappear $H_N$.

\vip

Next, we believe that $\Et[Z^{i,N}_t]\simeq \gamma_N(i) e^{\alpha_N t}$ for $t$ large, for some vector
$\gamma_N$ with positive entries and some exponent $\alpha_N>0$.
Inserting this into
$\Et[Z^{i,N}_t]=\mu t + N^{-1}\sum_{j=1}^N \theta_{ij} \intot \varphi(t-s)
\Et[Z^{j,N}_s]ds$, we find that $\gamma_N = A_N \gamma_N \int_0^\infty e^{-\alpha_Ns}\varphi(s)ds$.
The vector $\gamma_N$ being positive, it is necessarily a Perron-Frobenius eigenvector
of $A_N$,
so that $\rho_N=(\int_0^\infty e^{-\alpha_N s}\varphi(s)ds)^{-1}$ is its Perron-Frobenius eigenvalue 
(i.e. its spectral radius). 
We now consider the normalized Perron-Frobenius eigenvector $V_N$ such that
$\sum_{i=1}^N (V_N(i))^2=N$ and conclude that $Z^{i,N}_t \simeq K_N V_N(i) e^{\alpha_N t}$
for all $i=1,\dots,N$, where $K_N= [N^{-1}\sum_{i=1}^N (\gamma_N(i))^2]^{1/2} H_N$.

\vip

Exactly as in the subcritical case, the empirical variance
$N^{-1} \sum_{i=1}^N( Z^{i,N}_t - \bar Z^N_t)^2$ should resemble
$N^{-1}\sum_{i=1}^N(\Et[Z^{i,N}_t]-\Et[\baZ^{N}_t])^2+\baZ^N_t
\simeq N^{-1}K_N^2e^{2\alpha_N t} \sum_{i=1}^N( V_N(i) - \bar V_N)^2+\baZ^N_t$.
Since we also guess that
$\baZ^N_t\simeq K_N \bar V_N e^{\alpha_N t}$, where $\bar V_N=N^{-1}\sum_{i=1}^N V_N(i)$, we
expect that for $t$ large, 
$\cU^N_t=(\baZ^N_t)^{-2}[ \sum_{i=1}^N( Z^{i,N}_t - \bar Z^N_t)^2 -N \baZ^N_t]\simeq   
(\bar V_N)^{-2}\sum_{i=1}^N( V_N(i) - \bar V_N)^2$.

\vip

We now search for the limit of $(\bar V_N)^{-2}\sum_{i=1}^N( V_N(i) - \bar V_N)^2$ as $N\to\infty$.
Roughly, $A_N^2(i,j)\simeq p^2/N$, whence, starting from
$A_N^2 V_N = \rho_N^2 V_N$, we see that $\rho_N^2 V_N \simeq p^2 \bar V_N \bun$, where $\bun$ is the $N$-dimensional
vector with all coordinates equal to $1$. Consequently, $V_N=(A_N V_N)/\rho_N \simeq \kappa_N A_N\bun$, where
$\kappa_N= (p^2/\rho_N^3)\bar V_N$. In other words, $V_N$ is almost colinear to $L_N:=A_N \bun$, and $NL_N$ 
is a vector
of $N$ i.i.d. Binomial$(N,p)$-distributed random variables.
It is thus reasonable to expect that $(\bar V_N)^{-2}\sum_{i=1}^N( V_N(i) - \bar V_N)^2
\simeq (\bar L_N)^{-2}\sum_{i=1}^N( L_N(i) - \bar L_N)^2
\simeq p^{-2}p(1-p)=1/p-1$.

\vip

All in all, we hope that for $N$ and $t$ large and in a suitable regime, $\cU^N_t\simeq 1/p-1$.

\vip

Finally, let us mention that $\alpha_N\simeq \alpha_0$ (see Remark \ref{rkaz}) because
$\int_0^\infty e^{-\alpha_N s}\varphi(s)ds=1/\rho_N$, because $\int_0^\infty e^{-\alpha_0 s}\varphi(s)ds=1/p$
and because $\rho_N \simeq p$. This last assertion follows from the fact that $A_N^2(i,j)\simeq p^2/N$,
so that the largest eigenvalue of $A_N^2$ should resemble $p^2$, whence that of $A_N$ should resemble $p$.

\vip

Of course, all this is not clear and has to be made rigorous. Let us mention that we will use
a quantified version of the Perron-Frobenius of G. Birkhoff \cite{birk}. As we will see, the projection
onto the eigenvector $V_N$ will be very fast (almost immediate for $N$ very large).

\vip

As in the subcritical case, we will have to work on an event $\Omega_N^2$, of high probability,
on which the $\theta_{ij}$'s behave reasonably. For example, 
to apply the Perron-Frobenius theorem, we have to be sure that the matrix $A_N$ is irreducible, which is not a.s. true.

\subsection{About optimality: a related toy model}\label{sbi}

Consider $\alpha_0\geq 0$ and two unknown parameters $\Gamma>0$ and $p \in (0,1]$. For $N\geq 1$, 
consider an i.i.d. family  
$(\theta_{ij})_{i,j=1,\dots,N}$ of Bernoulli$(p)$-distributed random 
variables, put $\lambda^{i,N}_t=N^{-1} \Gamma e^{\alpha_0 t} \sum_{j=1}^N \theta_{ij}$ and, conditionally on 
$(\theta_{ij})_{i,j=1,\dots,N}$, consider a family $(Z^{1,N}_t)_{t\geq 0},\dots,(Z^{N,N}_t)_{t\geq 0}$ 
of independent inhomogeneous Poisson processes with intensities $(\lambda^{1,N}_t)_{t\geq 0},\dots,
(\lambda^{N,N}_t)_{t\geq 0}$.
We observe $(Z^{i,N}_s)_{s \in [0,t], i=1,\dots,N}$ and we want to estimate $p$
in the asymptotic $(N,t)\to(\infty,\infty)$.

\vip

This problem can be seen as a strongly simplified version of the one studied in the present paper,
with $\alpha_0=0$ in the subcritical case and $\alpha_0>0$ in the supercritical case. Roughly, the mean number
of jumps per individual resembles $m_t=\intot e^{\alpha_0 s} ds$, which is of order 
$t$ when $\alpha_0=0$ and $e^{\alpha_0 t}$ else.

\vip

There is classically no loss of information, since $\alpha_0$ is known,
if we only observe $(Z^{i,N}_t)_{i=1,\dots,N}$: 
after a (deterministic and known) change of time, the processes $(Z^{i,N}_t)_{i=1,\dots,N}$ become
{\it homogeneous} Poisson processes with unkown parameters (conditionally on $(\theta_{ij})_{i,j=1,\dots,N}$),
and the conditional law of a Poisson process on $[0,t]$ knowing its value at time 
$t$ does not depend on its parameter.

\vip

We next proceed to a Gaussian approximation: we have $\lambda^{i,N}_t \simeq \Gamma e^{\alpha_0 t} [p+ 
\sqrt{N^{-1}p(1-p)}]G_i$ and $Z^{i,N}_t\simeq \intot \lambda^{i,N}_s ds  
+ \sqrt{\intot \lambda^{i,N}_s ds}H_i$,
for two independent i.i.d. families $(G_i)_{i=1,\dots,N}$,
$(H_i)_{i=1,\dots,N}$ of $\cN(0,1)$-distributed random variables. Using finally that $(m_t)^{-1}N^{-1/2}\ll(m_t)^{-1}$
in our asymptotic, we conclude that 
$(m_t)^{-1}Z^{i,N}_t \simeq \Gamma p + \Gamma \sqrt{N^{-1}p(1-p)}G_i + \sqrt{(m_t)^{-1}\Gamma p} H_i$,
which is  $\cN( \Gamma p, N^{-1}\Gamma^2p(1-p)+ (m_t)^{-1}\Gamma p )$-distributed.

\vip

Our toy problem is thus the following: estimate $p$ when observing a $N$-sample
$(X^{i,N}_t)_{i=1,\dots,N}$ of the $\cN( \Gamma p, N^{-1}\Gamma^2p(1-p)+ (m_t)^{-1}\Gamma p )$-distribution.
We assume that $\Gamma p$ is known, which can only make easier the estimation of $p$.
As is well-known the statistic $S^N_t=N^{-1}\sum_{i=1}^N(X^{i,N}_t-\Gamma p)^2$
is then {\it sufficient} and is the best estimator (in all the usual senses),
for $N\geq 1$ and $t\geq 1$ fixed, of $N^{-1}\Gamma^2 p(1-p)+(m_t)^{-1}\Gamma p$,
so that $T^N_t=N(\Gamma p)^{-2}(S^N_t-m_t^{-1}\Gamma p)$ is more or less the best estimator
of $(1/p-1)$. But $\Var S^N_t = 2 N^{-1}(N^{-1}\Gamma^2p(1-p)+ (m_t)^{-1}\Gamma p)^2$, whence
$\Var T^N_t = 2(\Gamma p)^{-4}(N^{-1/2}\Gamma^2p(1-p)+ N^{1/2}(m_t)^{-1}\Gamma p)^2$.
It is thus not possible to estimate $(1/p-1)$ with a better precision than
$N^{-1/2}+N^{1/2}(m_t)^{-1}$. This of course implies that we cannot estimate $p$
with a better precision than $N^{-1/2}+N^{1/2}(m_t)^{-1}$.

\section{Well-posedness and explicit formulae}\label{swpmf}

We first give the 

\begin{proof}[Proof of Proposition \ref{wp}]
Conditionally on $(\theta_{ij})_{i,j=1,\dots,N}$, we can apply directly \cite[Theorem 6]{dfh}, of which 
the assumption is satisfied here, see \cite[Remark 5-(i)]{dfh}:
conditionally on $(\theta_{ij})_{i,j=1,\dots,N}$, there is a unique solution $(Z^{i,N}_t)_{t\geq 0,i=1,\dots,N}$
to \eqref{sN} such that $\sum_{i=1}^N \Et [Z^{i,N}_t]<\infty$ for all $t\geq 0$.
Since now $(\theta_{ij})_{i,j=1,\dots,N}$ can only take a finite number of values, 
we immediately deduce that indeed $\sum_{i=1}^N \E[Z^{i,N}_t]<\infty$ for all $t\geq 0$.
\end{proof}

We carry on with a classical lemma.
Recall that $\varphi^{\star 0}(t-s)ds=\delta_t(ds)$ by convention.

\begin{lem}\label{soluce}
Consider $d\geq 1$, $A\in \cM_{d\times d}(\rr)$, $m,g : [0,\infty) \mapsto \rr^d$ locally bounded
and assume that $\varphi:[0,\infty) \mapsto [0,\infty)$ is locally integrable. 
If $m_t=g_t+\intot \varphi(t-s) A m_s ds$ for all $t\geq 0$, then
$m_t=\sum_{n\geq 0} \intot \varphi^{\star n}(t-s) A^n g_sds$.
\end{lem}

\begin{proof}
The equation $m_t=g_t+\intot \varphi(t-s) Am_s ds$ with unkown $m$ has at most one locally bounded solution.
Indeed, consider two such solutions $m,\tilde m$, observe that $u=|m-\tilde m|$ satisfies
$u_t \leq |A| \intot \varphi(t-s) u_s ds$, and conclude that $u=0$
by the generalized Gronwall lemma, see e.g. \cite[Lemma 23-(i)]{dfh}. We thus just have to prove that
$m_t:=\sum_{n\geq 0} \intot \varphi^{\star n}(t-s)A^n g_s ds$ is locally bounded and solves $m=g+A\varphi \star m$.
We introduce $k^n_t= |A|^n\intot \varphi^{\star n}(s)ds$, which is locally bounded because $\varphi$ 
is locally integrable
and which satisfies $k^{n+1}_t \leq |A|\intot k^n_s \varphi(t-s)ds$. 
We use \cite[Lemma 23-(ii)]{dfh} to conclude that
$\sum_{n\geq 0} k^n_t$ is locally bounded. Consequently, $|m_t| \leq \sup_{[0,t]}|g_s| \times 
\sum_{n\geq 0} k^n_t$ is locally bounded. Finally, we write
$m=g+\sum_{n\geq 1} A^n \varphi^{\star n } \star g = g + A\varphi \star \sum_{n\geq 0} A^n\varphi^{\star n }\star g=
g+A\varphi\star m$
as desired.
\end{proof}

We next introduce a few processes.

\begin{nota}\label{processes} 
Assume only that $\varphi$ is locally integrable, fix $N\geq 1$
and consider the solution $(Z^{i,N}_t)_{t\geq 0,i=1,\dots,N}$ to \eqref{sN}.
For each $i=1,\dots, N$, we introduce the martingale (recall that $\lambda^{i,N}$ was
defined in \eqref{sN})
$$
M^{i,N}_t = \intot\int_0^\infty \indiq_{\{z \leq \lambda^{i,N}_{s}\}} \tilde \pi^i(ds,dz),
$$
where $\tilde \pi^i(ds,dz)=\pi^i(ds,dz)-dsdz$ is the compensated Poisson measure
associated to $\pi^i$. 
We also introduce $M^{i,N,*}_t=\sup_{[0,t]} |M^{i,N}_s|$, as well as the (conditionally) centered process
$$
U^{i,N}_t = Z^{i,N}_t - \Et[Z^{i,N}_t].
$$
For each $t\geq 0$, we denote by $\bZ^N_t$ (resp. $\bM^N_t$, $\bM^{N,*}_t$, $\bU^N_t$) the $N$-dimensional vector 
with coordinates $Z^{i,N}_t$ (resp. $M^{i,N}_t$,  $M^{i,N,*}_t$, $U^{i,N}_t$).
We also set $\bar Z^N_t=N^{-1}\sum_{i=1}^N Z^{i,N}_t$, $\bar M^N_t=N^{-1}\sum_{i=1}^N M^{i,N}_t$
and $\bar U^N_t=N^{-1}\sum_{i=1}^N U^{i,N}_t$.
\end{nota}

We refer to Jacod-Shiryaev \cite[Chapter 1, Section 4e]{js} for definitions and properties
of pure jump martingales and of their quadratic variations.

\begin{rk}\label{mo}
Since the Poisson measures $\pi^i$ are independent, the martingales $M^{i,N}$ are orthogonal.
More precisely, we have $[M^{i,N},M^{j,N}]_t=0$ if $i\ne j$, while $[M^{i,N},M^{i,N}]_t=Z^{i,N}_t$
(because $Z^{i,N}_t$ counts the jumps of $M^{i,N}$, which are all of size $1$).
Consequently, $ \Et[M^{i,N}_s M^{j,N}_t]= \indiq_{\{i=j\}}\Et[Z^{i,N}_{s\land t}]$.
\end{rk}

We now give some more or less explicit formulas. We denote by $\bun$ the $N$-dimensional
vector with all entries equal to $1$ and we set $A_N(i,j)=N^{-1}\theta_{ij}$ for $i,j=1,\dots,N$.

\begin{lem}\label{fonda} 
Assume only that $\varphi$ is locally integrable.
We have (recall that $\varphi^{\star 0}(t-s)ds = \delta_t(ds)$):
\begin{align}
\label{funda}\bZ^N_t=& \bM^N_t + \mu \bun t + \intot \varphi(t-s) A_N \bZ^N_s ds, \\
\label{fundb}\Et[\bZ^N_t]=& \mu \sum_{n\geq 0} \Big[\intot s \varphi^{\star n}(t-s) ds \Big] A_N^n 
\bun,\\
\label{fundc}\bU^N_t=& \sum_{n\geq 0} \intot \varphi^{\star n}(t-s) A_N^n \bM^N_s ds.
\end{align}
\end{lem}

\begin{proof}
The first expression is not difficult: starting from \eqref{sN},
$$
Z^{i,N}_t=M^{i,N}_t+\intot \lambda^{i,N}_s ds = M^{i,N}_t+ \mu t + \sum_{j=1}^N A_N(i,j) 
\intot \int_0^s \varphi(s-u)dZ^{j,N}_u ds.
$$
Using \cite[Lemma 22]{dfh}, we see that 
$\intot \int_0^s \varphi(s-u)dZ^{j,N}_u ds=\intot \varphi(t-s) Z^{j,N}_sds$, whence indeed,
$$
Z^{i,N}_t= M^{i,N}_t+ \mu t + \intot \varphi(t-s) \sum_{j=1}^N A_N(i,j) Z^{j,N}_s ds,
$$
which is nothing but \eqref{funda}.
Taking conditional expectations in \eqref{funda}, we find that $\Et[\bZ^N_t]=\mu \bun t 
+ \intot \varphi(t-s) A_N \Et[\bZ^N_s] ds$ and thus also $\bU^N_t=\bM^N_t+\intot \varphi(t-s) A_N \bU^N_s ds$.
Since now $\varphi$ is (a.s.) locally integrable, since $\mu\bun t$ and $\bM^N_t$ are (a.s.) locally bounded,
as well as $\Et[\bZ^N_t]$ and $\bU^N_t$, \eqref{fundb} and \eqref{fundc} directly follow from Lemma \ref{soluce}. 
\end{proof}

\section{The subcritical case}\label{souc}

Here we consider the subcritical case.
We first study the large $N$-asymptotic of the matrix $Q_N= (I-\Lambda A_N)^{-1}$,
which plays a central role in the rest of the section. 
In Subsection \ref{sconv}, we finely study the behavior of $\varphi^{\star n}$.
In Subsection \ref{ssa}, we handle a few computations to be used several times later.
Subsections \ref{se1}, \ref{se2} and \ref{se3} are devoted to the studies of the three estimators
$\cE^N_t$, $\cV^N_t$ and $\cW^N_{\Delta,t}$. We conclude the proofs of Theorem \ref{mr}
and Corollary \ref{mc} in Subsection \ref{sconc}.

\subsection{Study of a random matrix}\label{smat}

We use the following standard notation: for $x=(x_1,\dots,x_N) \in \rr^N$ and $r\in [1,\infty)$, we set
$||x||_r=(\sum_{i=1}^N |x_i|^r)^{1/r}$ and $||x||_\infty = \max_{i=1,\dots,N}|x_i|$. 
For $r \in [1,\infty]$, we denote by $|||\,\cdot\,|||_r$ the operator norm on $\cM_{N\times N}(\rr)$ 
associated to $|| \, \cdot \, ||_r$. We recall that
$$
|||M|||_1=\sup_{j=1,\dots,N}\sum_{i=1}^N |M_{ij}|,\quad |||M|||_\infty=\sup_{i=1,\dots,N}\sum_{j=1}^N |M_{ij}| 
$$
and that for all $r\in (1,\infty)$,
\begin{equation}\label{holder}
|||M|||_r \leq |||M|||_1^{1/r} |||M|||_\infty^{1-1/r}.
\end{equation}

\begin{nota}\label{ssimp}
We assume that $\Lambda p < 1$.
For each $N\geq 1$, we introduce the $N\times N$ random matrix $A_N$ defined by
$A_N(i,j)=N^{-1}\theta_{ij}$, as well as the event
\begin{equation}\label{omegan}
\Omega_N^1=\Big\{\Lambda |||A_N|||_r \leq a \; \hbox{ for all $r\in [1,\infty]$}\Big\}, 
\hbox{ where } a=\frac{1+\Lambda p}{2} \in (\Lambda p,1).
\end{equation}
On $\Omega_N^1$, the $N\times N$ matrix $Q_N=\sum_{n\geq 0} \Lambda^n A_N^n = (I-\Lambda A_N)^{-1}$
is well-defined and we introduce, for each $i=1,\dots,N$, $\ell_N(i)=\sum_{j=1}^N Q_N(i,j)$,
$c_N(i)=\sum_{j=1}^N Q_N(j,i)$, as well as $\bar \ell_N= N^{-1}\sum_{i=1}^N \ell_N(i)$ and 
$\bar c_N= N^{-1}\sum_{i=1}^N c_N(i)$. We of course have $\bar \ell_N= \bar c_N$.
\end{nota}

Let us remark once for all that, with $C=1/(1-a)<\infty$,
\begin{gather}
\label{omegancons1}
\Omega_N^1\subset \Big\{|||Q_N|||_r \leq C \; \hbox{ for all $r\in [1,\infty]$}\Big\}
\subset \Big\{ \sup_{i=1,\dots,N} \max\{\ell_N(i),c_N(i)\} \leq C \Big\},\\
\label{omegancons2}
\Omega_N^1 \subset \Big\{ \indiq_{\{i=j\}} \leq Q_N(i,j)\leq \indiq_{\{i=j\}} + \Lambda C N^{-1} 
\;\hbox{ for all $i,j=1,\dots,N$}\Big\}.
\end{gather}
Indeed, \eqref{omegancons1} is straightforward since $Q_N=\sum_{n\geq 0}\Lambda^n A_N^n$. 
To check \eqref{omegancons2}, 
we first observe that $Q_N(i,j) \geq \Lambda^0 A_N^0(i,j)=\indiq_{\{i=j\}}$.
Next, we use that 
$A_N(i,j)\leq N^{-1}$ while, for $n\geq 2$, $A_N^n(i,j) = \sum_{k=1}^N A_N(i,k)A_N^{n-1}(k,j) 
\leq N^{-1} \sum_{k=1}^N A_N^{n-1}(k,j)
\leq N^{-1} |||A_N^{n-1}|||_1 \leq N^{-1} |||A_N|||_1^{n-1}$. 
Thus $A_N^n(i,j)\leq N^{-1}|||A_N|||_1^{n-1}$ for all $n\geq 1$.
Hence on $\Omega_N^1$, it holds that
$Q_N(i,j)\leq \indiq_{\{i=j\}} + N^{-1} \sum_{n\geq 1} \Lambda^n|||A_N|||_1^{n-1}
\leq \indiq_{\{i=j\}} + N^{-1} \Lambda/(1-a)$ as desired.

\begin{lem}\label{pom} 
Assume that $\Lambda p <1$. It holds that $\Pr(\Omega_N^1) \geq 1 - C\exp(-c N)$.
\end{lem}

\begin{proof} 
By \eqref{holder}, it suffices to prove that $\Pr(\Lambda |||A_N|||_1 >a)\leq C\exp(-cN)$ and
$\Pr(\Lambda |||A_N|||_\infty >a)\leq C\exp(-cN)$. 
Since $|||A_N|||_\infty=|||A_N^t|||_1$ and since $A_N^t$ (the transpose of $A_N$)
has the same law as $A_N$, it actually suffices to verify the first inequality.
First, $N|||A_N|||_1=\max\{X_1^N,\dots,X_N^N\}$, where $X_i^N=\sum_{j=1}^N \theta_{ij}$ is
Binomial$(N,p)$-distributed for each $i$. 
Consequently, $\Pr(\Lambda |||A_N|||_1 >a) \leq N \Pr(X_1^N \geq N a/\Lambda)
\leq N \Pr(|X_1^N-Np| \geq N (a/\Lambda-p))$.
Since $a/\Lambda>p$, we can use the Hoeffding inequality \cite{hoe} to obtain
$\Pr(\Lambda |||A_N|||_1 >a) \leq 2 N \exp(-2N(a/\Lambda-p)^2) \leq C \exp(-N(a/\Lambda-p)^2)$ as desired.
\end{proof}

The next result is much harder but crucial.

\begin{prop}\label{interro}
Assume that $\Lambda p <1$. It holds that 
\begin{gather*}
\E\Big[\indiq_{\Omega_N^1} \Big|\bar\ell_N - \frac{1}{1-\Lambda p} \Big|^2\Big] \leq \frac C {N^2},\\
\E\Big[\indiq_{\Omega_N^1}\Big|\frac 1 N \sum_{i=1}^N \ell_N(i)(c_N(i))^2  - \frac 1 {(1-\Lambda p)^3}
\Big|^2 \Big]  \leq  \frac C  {N^2}, \\
\E\Big[\indiq_{\Omega_N^1}\Big|\sum_{i=1}^N (\ell_N(i)-\bar\ell_N)^2  
- \frac{\Lambda^2p(1-p)}{(1-\Lambda p)^2} \Big| \Big]  \leq  \frac C {\sqrt N}.
\end{gather*}
\end{prop}

\begin{proof} Recall that $\bun$ is the $N$-dimensional vector of which all the coordinates equal $1$.
Let $\ell_N$ (resp. $c_N$) be the vector with coordinates $\ell_N(1), \dots,\ell_N(N)$ 
(resp. $c_N(1),\dots,c_N(N)$).
We also introduce, for all $i=1,\dots,N$, $L_N(i)=\sum_{j=1}^N A_N(i,j)$ and 
$C_N(i)=\sum_{j=1}^N A_N(j,i)$, as well as the corresponding vectors $L_N$ and $C_N$. 
Let us observe that, with obvious notation, $\bar \ell_N=\bar c_N$ and $\bar L_N=\bar C_N$.
Finally, we introduce the vectors 
$$
x_N=\ell_N-\bar\ell_N\bun, \quad y_N=c_N-\bar c_N\bun, \quad 
X_N=L_N-\bar L_N\bun, \quad Y_N=C_N-\bar C_N\bun.
$$ 
We recall that $a=(1+\Lambda p)/2 \in (0,1)$ and we introduce $b=(2+\Lambda p)/3 \in (a,1)$.

\vip

{\it Step 1.} We introduce the event
$$
\cA_N = \Big\{ ||L_N - p\bun||_2 + ||C_N - p\bun||_2 \leq N^{1/4} \Big\}
\subset \Big\{ ||X_N||_2 + ||Y_N||_2 \leq N^{1/4} \Big\}.
$$
The inclusion comes from the fact that a.s., $||X_N||_2 =||L_N-\bar L_N \bun||_2 \leq ||L_N- x \bun||_2$ 
for any $x\in\rr$.
Since $N L_N=(Z_1^N,\dots,Z_N^N)$ with $Z_i^N$ i.i.d. and Binomial$(N,p)$-distributed, it is very classical
that for any $\alpha>0$, $\E[||L_N - p\bun||_2^\alpha]\leq C_\alpha$ (uniformly in $N$), 
we have similarly $\E[||C_N - p\bun||_2^\alpha]\leq C_\alpha$, so that 
$$
\Pr(\cA_N) \geq 1 - C_\alpha N^{-\alpha/4}.
$$

{\it Step 2.} We now check the following points:  (i) $\E[|\bar L_N - p |^2] \leq C N^{-2}$,
(ii) $\E[||X_N||_2^4]\leq C$, (iii) $\E[(||X_N||_2^2 - p(1-p))^2] \leq CN^{-1}$
and (iv) $\E[||A_N X_N||_2^2]\leq CN^{-1}$.

\vip

Point (i) is clear, because $\bar L_N=N^{-2}\sum_{i,j=1}^N \theta_{ij}$ 
is nothing but the empirical mean of $N^2$ independent
Bernoulli$(p)$-random variables. Points (ii) and (iii) are very classical, since
$N||X_N||_2^2$ is the empirical variance of $N$ independent Binomial$(N,p)$-random variables. We now prove (iv):
$$
\E[||A_N X_N||_2^2]=\sum_{i=1}^N \E\Big[\Big(\sum_{j=1}^N \frac{\theta_{ij}}N (L_N(j)-\bar L_N)\Big)^2\Big]
=\frac 1N \E\Big[\Big(\sum_{j=1}^N \theta_{1j} (L_N(j)-\bar L_N)\Big)^2\Big]
$$
by symmetry. We now write $\E[||A_N X_N||_2^2]\leq 4 N^{-1}(I_N+J_N+K_N)$, where 
$$
I_N = \E\Big[(\bar L_N-p)^2 \Big(\sum_{j=1}^N \theta_{1j} \Big)^2\Big], \quad
J_N=\E\Big[\Big(\theta_{11} (L_N(1)- p)\Big)^2\Big],\quad 
K_N=\E\Big[\Big(\sum_{j=2}^N \theta_{1j}(L_N(j)-p) \Big)^2\Big].
$$
First, $I_N \leq N^2 \E[(\bar L_N-p)^2 ] \leq C$ by (i). Next, it is obvious 
that $J_N \leq 1$ (because $\theta_{11}\in \{0,1\}$ and $L_N(1) \in [0,1]$). 
Finally, the random variables $\theta_{1j}(L_N(j)-p)$ being i.i.d. and centered (for $j=2,\dots,N$),
we may write
$$
K_N=(N-1) \E\Big[\Big(\theta_{12}(L_N(2)-p) \Big)^2\Big] \leq (N-1) \E\Big[(L_N(2)-p)^2\Big] \leq C,
$$
since $NL_N(2)$ follows a Binomial$(N,p)$-distribution. This completes the step.

\vip

{\it Step 3.} We next prove that
(i) $x_N=\Lambda A_N x_N - \Lambda r_N \bun + \Lambda \bar\ell_N X_N$ on $\Omega_N^1$, where 
$r_N= N^{-2}\sum_{i,j=1}^N (\theta_{ij}-p)x_N(j)$ and that
(ii) $|r_N| \leq N^{-3/4}||x_N||_2$ on $\Omega_N^1\cap\cA_N$.

\vip

We start from $\ell_N=Q_N\bun=(I-\Lambda A_N)^{-1}\bun$, whence $\ell_N=\bun+\Lambda A_N\ell_N$.
Since  $\bar \ell_N=N^{-1}(\ell_N,\bun)$, we see that $\bar \ell_N= 1 + \Lambda N^{-1}(A_N \ell_N,\bun)$ 
(here $(\cdot,\cdot)$ is the usual scalar product on $\rr^N$) and thus
\begin{align*}
x_N =& \Lambda A_N \ell_N -\Lambda N^{-1}( A_N \ell_N,\bun)\bun \\
=& \Lambda A_N x_N - \Lambda N^{-1}( A_N x_N,\bun)\bun +  \bar \ell_N \Lambda A_N \bun  - \bar \ell_N 
\Lambda N^{-1}( A_N \bun,\bun)\bun.
\end{align*}
It only remains to check that $N^{-1}( A_N x_N,\bun)=r_N$, which follows from
$N^{-1}( A_N x_N,\bun)=N^{-2}\sum_{i,j=1}^N \theta_{ij}x_N(j)$ and the fact that $\sum_{j=1}^N x_N(j)=0$;
and that $A_N \bun  - N^{-1}( A_N \bun,\bun)\bun= X_N$, which is clear since $A_N\bun = L_N$.

\vip

To verify (ii), we observe that $r_N= N^{-1}\sum_{j=1}^N (C_N(j)-p)x_N(j)$, whence, by the Cauchy-Schwarz inequality,
$|r_N| \leq  N^{-1}||x_N||_2 ||C_N-p\bun||_2 \leq N^{-3/4}||x_N||_2$ on $\Omega_N^1\cap\cA_N$.

\vip

{\it Step 4.} Let $N_0$ be the smallest integer such that
$a + \Lambda N_0^{-1/4} \leq b$. We check that for all $N\geq N_0$,
$$
\indiq_{\Omega_N^1\cap\cA_N}||x_N||_2 \leq C ||X_N||_2.
$$ 
Using Step 3 and that $||\bun||_2=N^{1/2}$, we write
$||x_N||_2 \leq \Lambda |||A_N|||_2||x_N||_2 + \Lambda N^{-1/4}||x_N||_2+ \Lambda |\bar\ell_N|||X_N||_2$.
But on $\Omega_N^1$, $\Lambda |||A_N|||_2 \leq a$ and
$|\bar\ell_N| \leq C$, see \eqref{omegan} and \eqref{omegancons1}. 
Hence, for $N\geq N_0$, on $\Omega_N^1\cap\cA_N$, we have 
$||x_N||_2 \leq (a+\Lambda N^{-1/4}) ||x_N||_2 + C||X_N||_2 \leq b ||x_N||_2 + C||X_N||_2$. 
Since $b<1$, the conclusion follows.

\vip

{\it Step 5.} We now prove that for $N\geq N_0$, 
$$
\E\Big[\indiq_{\Omega_N^1\cap\cA_N}\Big|\bar\ell_N - \frac 1{1-\Lambda p}\Big|^2 \Big]
\leq \frac C {N^{2}}.
$$
Using Step 3, we know that on $\Omega_N^1\cap\cA_N$, $\ell_N= \bun + \Lambda A_N \ell_N$,
whence 
$$
\bar \ell_N = 1+\frac\Lambda{N}\sum_{i,j=1}^N A_N(i,j) \ell_N(j)=1+\frac\Lambda{N}\sum_{j=1}^N C_N(j)\ell_N(j)
=1+ \Lambda p \bar \ell_N + S_N,
$$
where $S_N=\Lambda N^{-1} \sum_{j=1}^N (C_N(j)-p)\ell_N(j)$. Consequently, 
$\bar \ell_N =(1-\Lambda p)^{-1} (1+S_N)$, and we only have to prove that 
$\E[\indiq_{\Omega_N^1\cap\cA_N} S_N^2] \leq C N^{-2}$. To this end, we write 
$S_N= \Lambda N^{-1} (a_N+b_N)$, where $a_N=\sum_{j=1}^N (C_N(j)-p)x_N(j)$ and $b_N=\bar \ell_N \sum_{j=1}^N (C_N(j)-p)$.
First, since $|\bar \ell_N| \leq C$ on $\Omega_N^1$ by \eqref{omegancons1}, we can write 
$\E[\indiq_{\Omega_N^1}b_N^2] \leq C \E[( \sum_{j=1}^N (C_N(j)-p))^2]=C N^2 \E[(\bar C_N -p)^2 ] \leq C$, the last 
inequality
coming from Step 2-(i) since $\bar C_N=\bar L_N$.
Next, we use the Cauchy-Schwarz inequality: 
$a_N^2 \leq ||C_N-p \bun||_2  ||x_N||_2\leq C ||C_N-p \bun||_2 ||X_N||_2$ 
on $\Omega_N^1\cap\cA_N$ by Step 4. Consequently, $\E[\indiq_{\Omega_N^1\cap\cA_N}a_N^2] 
\leq C \E[||X_N||_2^2]^{1/2}\E[||C_N-p\bun||_2^2]^{1/2}$.
But $\E[||X_N||_2^2]\leq C$ by Step 2-(ii) and we have seen at the end of Step 1 that 
$\E[||C_N-p\bun||_2^2]\leq C$.

\vip

{\it Step 6.} Here we verify that, still for $N\geq N_0$, 
$$
\E\Big[\indiq_{\Omega_N^1\cap\cA_N}\Big| \frac 1 N\sum_{i=1}^N \ell_N(i)(c_N(i))^2 - \frac1{(1-\Lambda p)^{3}}\Big|^2 
\Big]\leq \frac C {N^{2}}.
$$
We write, using that $\bar c_N=\bar \ell_N$,
\begin{align*}
\frac1N\sum_{i=1}^N \ell_N(i)(c_N(i))^2=&\frac1N\sum_{i=1}^N \ell_N(i)(c_N(i)-\bar c_N)^2+(\bar\ell_N)^3+
\frac2N \bar\ell_N \sum_{i=1}^N \ell_N(i)(c_N(i)-\bar c_N).
\end{align*}
First, since $|\bar\ell_N|\leq C$ on $\Omega_N^1$, we have $|(\bar\ell_N)^3- (1-\Lambda p)^{-3}| \leq C 
|\bar\ell_N- (1-\Lambda p)^{-1}|$, whence $\E[\indiq_{\Omega_N^1\!\cap\cA_N}|(\bar\ell_N)^3- (1-\Lambda p)^{-3}|^2]
\leq C N^{-2}$ by Step 5. It thus suffices to verify that $\E[\indiq_{\Omega_N^1\cap \!\cA_N}((a_N')^2+(b_N')^2)]\leq C$,
where $a_N'=\sum_{i=1}^N \ell_N(i)(c_N(i)-\bar c_N)^2$ and $b_N'=\sum_{i=1}^N \ell_N(i)(c_N(i)-\bar c_N)$.

\vip

First, it holds that $b_N'=\sum_{i=1}^N \ell_N(i) y_N(i) =\sum_{i=1}^N x_N(i) y_N(i)$
because $\sum_{i=1}^N y_N(i)=0$. Hence
$|b_N'| \leq ||x_N||_2||y_N||_2$. But on $\Omega_N^1\cap\cA_N$, we know from Step 4 that 
$||x_N||_2\leq C ||X_N||_2$, and it obviously also holds true that $||y_N||_2\leq C ||Y_N||_2$. We thus conclude that
$\E[ \indiq_{\Omega_N^1\cap\cA_N} (b_N')^2]\leq C \E[||X_N||_2^4 ]^{1/2}\E[||Y_N||_2^4 ]^{1/2}=\E[||X_N||_2^4 ] $ by 
symmetry.
Using finally Step 2-(ii), we deduce that indeed, $\E[ \indiq_{\Omega_N^1\cap\cA_N} (b_N')^2]\leq C$.
Next, since $|\ell_N(i)|\leq C$ on $\Omega_N^1$ by \eqref{omegancons1}, 
we can write $|a_N'| \leq C ||c_N-\bar c_N\bun||_2^2=C||y_N||_ 2^2$.
We conclude as previously that  $\E[ \indiq_{\Omega_N^1\cap\cA_N} (a_N')^2]\leq C$.

\vip

{\it Step 7.} The goal of this step is to establish that, for all $N\geq N_0$,
$$
\E\Big[\indiq_{\Omega_N^1\cap\cA_N}\Big|||x_N||_2^2 -  \frac{\Lambda^2 p(1-p)}{(1-\Lambda p)^{2}}\Big| \Big]
\leq \frac{C}{\sqrt N}.
$$
Starting from Step 3, we write 
$$
x_N - \Lambda \bar \ell_N X_N = \Lambda A_N x_N - \Lambda r_N \bun = \Lambda A_N( x_N - \Lambda \bar \ell_N X_N)
+ \Lambda^2 \bar \ell_N A_NX_N - \Lambda r_N \bun.
$$
Thus
$$
||x_N - \Lambda \bar \ell_N X_N||_2 \leq \Lambda |||A_N|||_2 ||x_N - \Lambda \bar \ell_N X_N||_2
+ \Lambda^2 |\bar \ell_N| ||A_NX_N||_ 2 + \Lambda N^{-1/2} ||C_N-p\bun||_2 ||x_N||_2,
$$
where we used that $||\bun||_2=N^{1/2}$ and that 
$|r_N| \leq N^{-1} ||C_N-p\bun||_2 ||x_N||_2$ on $\Omega_N^1\cap\cA_N$, 
as checked at the end of Step 3. Using now that $\Lambda |||A_N|||_ 2\leq a<1$ and
$|\bar \ell_N|\leq C$ on $\Omega_N^1$ and that  $||x_N||_2\leq C ||X_N||_2$ on $\Omega_N^1\cap\cA_N$
by Step 4, we conclude that, still on $\Omega_N^1\cap\cA_N$, 
$$
||x_N - \Lambda \bar \ell_N X_N||_2^2 \leq C(||A_NX_N||_2^2 + C N^{-1} ||C_N-p\bun||_2^2 ||X_N||_2^2).
$$
Since now $\E[||A_NX_N||_2^2] \leq C N^{-1}$ by Step 2-(iv), since $\E[||X_N||_2^4]\leq C$ by Step 2-(ii)
and since $\E[||C_N-p\bun||_2^4] \leq C$ (see the end of Step 1), we deduce that
$$
\E\Big[\indiq_{\Omega_N^1\cap\cA_N}||x_N -  \Lambda \bar \ell_N X_N||_2^2 \Big]\leq \frac{C}{N}.
$$
Next, we observe that $\big|||x_N||_2^2-(\Lambda \bar \ell_N)^2 ||X_N||_2^2\big| \leq ||x_N -  \Lambda \bar \ell_N X_N||_2
(||x_N||_2+ \Lambda |\bar \ell_N| ||X_N||_2 ) \leq C ||x_N -  \Lambda \bar \ell_N X_N||_2 ||X_N||_2$
on $\Omega_N^1\cap\cA_N$ by Step 4 and since $\bar \ell_N$ is bounded on $\Omega_N^1$. Hence
$$
\E\Big[\indiq_{\Omega_N^1\cap\cA_N}\Big| ||x_N||_2^2 - (\Lambda \bar \ell_N)^2 ||X_N||_2^2\Big| \Big]
\leq \frac{C}{\sqrt N} \E[||X_N||_ 2^2]^{1/2} \leq \frac{C}{\sqrt N}
$$
by Step 2-(ii). To complete the step, it only remains to verify that 
$$
d_N=\E\Big[\indiq_{\Omega_N^1\cap\cA_N}\Big|(\bar \ell_N)^2 ||X_N||_2^2 - p(1-p)(1-\Lambda p)^{-2}\Big| \Big]
\leq \frac{C}{\sqrt N}.
$$
We naturally write $d_N \leq a_N''+b_N''$, where
\begin{align*}
a_N''=&\E\Big[\indiq_{\Omega_N^1\cap\cA_N}\Big|(\bar \ell_N)^2-(1-\Lambda p)^{-2}\Big| ||X_N||_2^2\Big],\\
b_N''=&(1-\Lambda p)^{-2}\E\Big[\indiq_{\Omega_N^1\cap\cA_N}\Big|||X_N||_2^2 - p(1-p)\Big|\Big].
\end{align*}
Step 2-(iii) directly implies that $b_N'' \leq C N^{-1/2}$. Using that $\bar \ell_N$ is bounded on $\Omega_N^1$,
we deduce that $|(\bar \ell_N)^2-(1-\Lambda p)^{-2}| \leq C |\bar\ell_N-(1-\Lambda p)^{-1}|$.
Thus 
$$
a_N'' \leq C \E\Big[\indiq_{\Omega_N^1\cap\cA_N}\Big|\bar \ell_N-(1-\Lambda p)^{-1}\Big|^2\Big]^{1/2}\E[ ||X_N||_2^4]^{1/2}.
$$
Step 2-(ii) and Step 5 imply that $a_N'' \leq C N^{-1}\leq C N^{-1/2}$ as desired.

\vip

{\it Step 8.} It remains to conclude. It clearly suffices to treat the case where
$N\geq N_0$, because $\ell_N(i)$ and $c_N(i)$ are uniformly bounded on $\Omega_N^1$
by \eqref{omegancons1}, so that the inequalities of the statement are trivial when 
$N\leq N_0$ (if the constant $C$ is large enough).
Since $\bar \ell_N$ is (uniformly) bounded on $\Omega_N^1$, we have 
$$
\E\Big[\indiq_{\Omega_N^1} \Big|\bar\ell_N - \frac{1}{1-\Lambda p} \Big|^2\Big] \leq 
\E\Big[\indiq_{\Omega_N^1 \cap \cA_N} \Big|\bar\ell_N - \frac{1}{1-\Lambda p} \Big|^2\Big]
+ C \Pr((\cA_N)^c).
$$ 
The first term is bounded by $CN^{-2}$ (by Step 5), as well as the second one (use the last inequality
of Step 1 with $\alpha=8$).

\vip

Similarly, using Step 6 and that $\ell_N(i)$ and $c_N(i)$ are (uniformly) bounded on $\Omega_N^1$,
we see that
$$
\E\Big[\indiq_{\Omega_N^1}\Big|\frac 1 N \sum_{i=1}^N \ell_N(i)(c_N(i))^2  - \frac 1 {(1-\Lambda p)^3}
\Big|^2 \Big] \leq \frac C {N^2} + C \Pr((\cA_N)^c) \leq \frac C {N^2} .
$$ 

Finally, observe that $\sum_{i=1}^N (\ell_N(i)-\bar\ell_N)^2=||x_N||_2^2$ is bounded by $CN$ on $\Omega_N^1$,
so that by Step 7,
$$
\E\Big[\indiq_{\Omega_N^1}\Big|\sum_{i=1}^N (\ell_N(i)-\bar\ell_N)^2  
- \frac{\Lambda^2p(1-p)}{(1-\Lambda p)^2} \Big| \Big]  \leq  \frac C {\sqrt N} + C N \Pr((\cA_N)^c)
\leq  \frac C {\sqrt N}.
$$
We used the last inequality of Step 1 with $\alpha=6$.
\end{proof}

\subsection{Preliminary analytic estimates}\label{sconv}

In view of \eqref{fundb} and \eqref{fundc}, it will be necessary  for our purpose to study 
very precisely the behavior of $\varphi^{\star n}$, which we now do. The following
statements may seem rather tedious, but they are exactly the ones we need.
Recall that 
$\varphi^{\star 0}(t-s)ds=\delta_t(ds)$
and that $\varphi^{\star n} (s)=0$ for $s<0$ by convention.

\begin{lem}\label{conv}
Recall that $\varphi:[0,\infty)\mapsto [0,\infty)$ and that $\Lambda=\int_0^\infty \varphi(s)ds$.
Assume that there is $q\geq 1$ such that $\int_0^\infty s^q \varphi(s)ds <\infty$ and set
$\kappa = \Lambda^{-1}\int_0^\infty s\varphi(s)ds$.

\vip

(i) For $n\geq 0$ and $t\geq 0$, we have 
$\intot  s \varphi^{\star n} (t-s)  ds = \Lambda^n t - n\Lambda^n \kappa + \e_n(t)$, where
$$
0 \leq \e_n(t) \leq C n^q \Lambda^n t^{1-q} \quad \hbox{and} \quad \e_n(t) \leq n\Lambda^n \kappa.
$$

(ii) For $n\geq 0$, for $0\leq t\leq z$ and $s\in [0,z]$, we set $\beta_n(t,z,s)=\varphi^{\star n} (z-s)
-\varphi^{\star n} (t-s)$. Then $\int_0^z |\beta_n(t,z,s)| ds\leq 2 \Lambda^n$
and for all $0\leq \Delta \leq t$ and all $z\in [t,t+\Delta]$,
$$
\Big|\int_0^z \beta_n(t,z,s)ds \Big| \leq C n^q \Lambda^n t^{-q} \;\; \hbox{and}\;\;
\int_0^{t-\Delta} |\beta_n(t,z,s)|ds + 
\Big|\int_{t-\Delta}^z \beta_n(t,z,s)ds \Big| \leq C n^q \Lambda^n \Delta^{-q}.
$$

(iii) For $m,n\geq 0$, for $0\leq t\leq z$, we put $\gamma_{m,n}(t,z)=\int_0^z \int_0^z
(s \land u) \beta_m(t,z,s)\beta_n(t,z,u)duds$. 
It holds that $0 \leq \gamma_{m,n}(t,t+\Delta) \leq \Lambda^{m+n} \Delta$, for all $t\geq 0$, all $\Delta \geq 0$.
Furthermore, there is 
a family $\kappa_{m,n}$ satisfying $0\leq \kappa_{m,n}\leq (m+n)\kappa$ such that, 
for all $0 \leq \Delta \leq t$,
$$
\gamma_{m,n}(t,t+\Delta) = \Delta \Lambda^{m+n} - \kappa_{m,n} \Lambda^{m+n} 
+ \e_{m,n}(t,t+\Delta),
$$
with $|\e_{m,n}(t,t+\Delta)|\leq C (m+n)^q \Lambda^{m+n} t \Delta^{-q}$.
\end{lem}

\begin{proof}
We introduce some i.i.d. random variables $X_1,X_2,\dots$ with density $\Lambda^{-1}\varphi$ and
set $S_0=0$ as well as $S_n=X_1+\dots+X_n$ for all $n\geq 1$.
We observe that, by the Minkowski inequality, $\E[S_n^q] \leq n^q \E[X_1^q]\leq C n^q$,
since $\E[X_1^q]=\Lambda^{-1}\int_0^\infty s^q \varphi(s)ds<\infty$ by assumption.

\vip

To check (i), we use that $S_n$ has for density $\Lambda^{-n}\varphi^{\star n}$, so that we can write
$$
\intot s \varphi^{\star n} (t-s)  ds=\intot (t-s)\varphi^{\star n}(s)ds=\Lambda^n \E[(t-S_n)_+]
=\Lambda^nt-\Lambda^n\E[S_n]+\e_n(t),
$$ 
where $\e_n(t)=\Lambda^n\E[(S_n-t)\indiq_{\{S_n \geq t\}}]$.
We clearly have that $\E[S_n]=n\kappa$, that $\e_n(t)\geq 0$ and that 
$\e_n(t) \leq \Lambda^n\E[S_n]= n\Lambda^n \kappa$.
Finally, $\e_n(t) \leq \Lambda^n\E[S_n \indiq_{\{S_n \geq t\}}] \leq \Lambda^n t^{1-q} \E[S_n^q] 
\leq C n^q \Lambda^n t^{1-q}$.

\vip

To check (ii), we observe that $\int_0^z |\beta_n(t,z,s)| ds\leq 2 \Lambda^n$ is obvious because $\int_0^\infty 
\varphi^{\star n}(s)ds =\Lambda^n$ and that, 
since $\E[S_n^q] \leq C n^q$,
$$
\int_r^\infty \varphi^{\star n}(u)du=\Lambda^n \Pr(S_n \geq r) \leq C n^q\Lambda^n  r^{-q}.
$$
We write $\int_0^z \beta_n(t,z,s)ds = \int_0^z \varphi^{\star n}(z-s)ds - \int_0^t \varphi^{\star n}(t-s)ds=
\int_t^z \varphi^{\star n}(u)du$, which implies 
that $|\int_0^z \beta_n(t,z,s) ds | \leq 
\int_t^\infty \varphi^{\star n}(u)du \leq C n^q\Lambda^n  t^{-q}$.
Next, we see that $\int_0^{t-\Delta} |\beta_n(t,z,s)|ds \leq \int_0^{t-\Delta}\varphi^{\star n}(z-u)du + \int_0^{t-\Delta}
\varphi^{\star n}(t-u)du \leq 2 \int_\Delta^\infty
\varphi^{\star n}(u)du \leq C n^q\Lambda^n  \Delta^{-q}$. 
Finally, using the two previous bounds, 
$|\int_{t-\Delta}^z \beta_n(t,z,s)ds| \leq |\int_0^z \beta_n(t,z,s)ds| + |\int_0^{t-\Delta} \beta_n(t,z,s)ds|
\leq C n^q\Lambda^n  t^{-q} +  
Cn^q\Lambda^n  \Delta^{-q} \leq Cn^q\Lambda^n  \Delta^{-q}$ because $\Delta \in [0,t]$ by assumption.

\vip

We finally prove (iii) and thus consider $0\leq\Delta\leq t$ and $m,n\geq 0$. We start from 
\begin{align*}
\gamma_{m,n}(t,t+\Delta)=& \int_0^{t+\Delta} \int_0^{t+\Delta} (s\land u) 
\Big[\varphi^{\star m}(t+\Delta-s) \varphi^{\star n}(t+\Delta-u)
+\varphi^{\star m}(t-s) \varphi^{\star n}(t-u) \\
&\hskip2cm-\varphi^{\star m}(t+\Delta-s) \varphi^{\star n}(t-u)-
\varphi^{\star m}(t-s) \varphi^{\star n}(t+\Delta-u)\Big] duds.
\end{align*}
Using another (independent) i.i.d. family $Y_1,Y_2,\dots$ of random variables
with density $\Lambda^{-1}\varphi$
and setting $T_m=Y_1+\dots+Y_m$ (or $T_m=0$ if $m=0$), we may write 
\begin{align*}
\gamma_{m,n}(t,t+\Delta)=& \Lambda^{m+n}\E\Big[(t+\Delta-T_m)_+\land(t+\Delta-S_n)_+ +(t-T_m)_+\land(t-S_n)_+ \\
&\hskip 2cm -(t+\Delta-T_m)_+\land(t-S_n)_+ 
-(t-T_m)_+\land(t+\Delta-S_n)_+   \Big].
\end{align*}
This precisely rewrites 
$\gamma_{m,n}(t,t+\Delta)=\Lambda^{m+n}\E[((t+\Delta-T_m\lor S_n)_+ - (t -T_m\land S_n)_+)_+]$, which
implies that $0\leq \gamma_{m,n}(t,t+\Delta) \leq \Lambda^{m+n}\Delta$.
We next introduce 
$$
\delta_{m,n}(t,t+\Delta)=\Lambda^{m+n}\E[(t+\Delta-T_m\lor S_n) - (t -T_m\land S_n )],
$$
which is nothing but $\delta_{m,n}(t,t+\Delta)=\Lambda^{m+n}(\Delta -\kappa_{m,n})$,
where $\kappa_{m,n}=\E[|T_m-S_n|]$ obviously satisfies 
$0 \leq \kappa_{m,n} \leq \kappa(m+n)$. Thus $\gamma_{m,n}(t,t+\Delta)=\Lambda^{m+n}(\Delta -\kappa_{m,n})
+\e_{m,n}(t,t+\Delta)$, where $\e_{m,n}(t,t+\Delta)=\gamma_{m,n}(t,t+\Delta)-\delta_{m,n}(t,t+\Delta)$.
Finally, it is clear that, since $0\leq \Delta \leq t$,
\begin{align*}
|\e_{m,n}(t,t+\Delta)|
\leq& \Lambda^{m+n}(t+\Delta) \Pr(T_m\lor S_n \geq t+\Delta \hbox{ or } T_m \land S_n \geq t
\hbox{ or } |T_m-S_n| \geq \Delta ) \\
\leq & 2\Lambda^{m+n}t \Pr(T_m\geq \Delta \hbox{ or } S_n \geq \Delta).
\end{align*}
This is, as usual, bounded by $C \Lambda^{m+n}t (m^q + n^q)\Delta^{-q}$.
\end{proof}

\subsection{Preliminary stochastic analysis}\label{ssa}

We handle once for all a number of useful computations concerning the processes introduced in
Notation \ref{processes}.

\begin{lem}\label{if}
We assume $H(q)$ for some $q\geq 1$. Recall that $\Omega_N^1$ and $\ell_N$ were defined
in Notation \ref{ssimp} and that all the processes below have been introduced in Notation
\ref{processes}.

\vip

(i) For any $r\in[1,\infty]$, for all $t\geq 0$,
$$
\indiq_{\Omega_N^1} \| \Et [ \bZ^N_{t} ] \|_r \leq C t ||\bun||_r.
$$ 

(ii) For any $r\in[1,\infty]$, for all $t\geq s \geq 0$,
$$
\indiq_{\Omega_N^1} \Big\| \Et\Big[ {\bZ^N_{t}-\bZ^N_s}\Big] - \mu(t-s) \ell_N \Big\|_r \leq C (1 \land s^{1-q}) 
||\bun||_r.
$$ 

(iii) For all $t\geq s+1 \geq 1$, 
$$
\indiq_{\Omega_N^1} \sup_{i=1,\dots,N} \Et \Big[ (Z^{i,N}_{t}-Z^{i,N}_s)^2
+ \sup_{[s,t]} |M^{i,N}_r-M^{i,N}_s|^4\Big] + \indiq_{\Omega_N^1} \Et \Big[ (\bar Z^{N}_{t}-\bar Z^{N}_s)^2\Big]
\leq C (t-s)^2.
$$
\end{lem}

\begin{proof}
Recall \eqref{fundb}, which asserts that $\Et[\bZ^N_t]= 
\mu \sum_{n\geq 0} [\intot s \varphi^{\star n}(t-s) ds] A_N^n \bun$.
Using that $\intot s \varphi^{\star n}(t-s) ds\leq t \Lambda^n$, we deduce that
$||\Et[\bZ^N_t]||_r \leq \mu t \sum_{n\geq 0} \Lambda^n  |||A_N|||_r^n ||\bun||_r$.
This is clearly bounded, on $\Omega_N^1$, by $C t ||\bun||_r$, which proves (i).

\vip

Using next Lemma \ref{conv}-(i), $\Et[\bZ^N_t]= 
\mu \sum_{n\geq 0} [\Lambda^n t - n \Lambda^n \kappa + \e_n(t)] A_N^n \bun$, where 
$0\leq \e_n(t) \leq C n^q \Lambda^n (t^{1-q} \land 1)$. Hence
$$
\Et[\bZ^N_t]-\Et[\bZ^N_s]= \mu (t-s) \sum_{n\geq 0} \Lambda^n A_N^n \bun + \mu\sum_{n\geq 0} [\e_n(t)-\e_n(s)]A_N^n\bun.
$$
But $\sum_{n\geq 0} \Lambda^n A_N^n \bun=Q_N \bun=\ell_N$ on $\Omega_N^1$. Thus, still on $\Omega_N^1$, since $s\leq t$
and $q\geq 1$,
$$
\Big\|\Et\Big[ {\bZ^N_{t}-\bZ^N_s}\Big] - \mu(t-s) \ell_N \Big\|_r \leq C(1 \land s^{1-q})
\sum_{n\geq 0}  n^q \Lambda^n |||A_N|||_r^n ||\bun||_r \leq C (1\land s^{1-q}) ||\bun||_r.
$$

Since $[M^{i,N},M^{i,N}]_t=Z^{i,N}_t$ by Remark \ref{mo}, the Doob inequality implies that
$\Et [\sup_{[s,t]} |M^{i,N}_r-M^{i,N}_s|^4\Big]\leq C \Et [ (Z^{i,N}_{t}-Z^{i,N}_s)^2]$.
Also, the Cauchy-Schwarz inequality tells us that $\Et[ (\bar Z^{N}_{t}-\bar Z^{N}_s)^2]\leq
N^{-1}\sum_{i=1}^N \Et[(Z^{i,N}_t-Z^{i,N}_s)^2]\leq \sup_{i=1,\dots,N}\Et[(Z^{i,N}_t-Z^{i,N}_s)^2]$.
Hence we just have to prove that $\sup_{i=1,\dots,N}\Et[(Z^{i,N}_t-Z^{i,N}_s)^2] \leq C (t-s)^2$.
Recalling that $Z^{i,N}_t= U^{i,N}_t+ \Et[Z^{i,N}_t]$, we have to show that, on $\Omega_N^1$,
(a) $(\Et[Z^{i,N}_t]-\Et[Z^{i,N}_s])^2 \leq C (t-s)^2$ and (b) $\Et[(U^{i,N}_t-U^{i,N}_s)^2] \leq C(t-s)^2$.

\vip

To prove (a), we use (ii) with $r=\infty$ and find that, on $\Omega_N^1$,
$\Et[Z^{i,N}_t]-\Et[Z^{i,N}_s]\leq \mu (t-s) ||\ell_N||_\infty + C ||\bun||_\infty \leq C (t-s)$,
since $\ell_N$ is bounded on $\Omega_N^1$ and since $t-s\geq 1$ by assumption.

\vip

To prove (b), we use \eqref{fundc} to write $U^{i,N}_t-U^{i,N}_s=\sum_{n\geq 0} \intot \beta_n(s,t,r)
\sum_{j=1}^N A_N^n(i,j)M^{j,N}_r dr$, 
where we have set $\beta_n(s,t,r)=\varphi^{\star n}(t-r)-\varphi^{\star n}(s-r)$ as in Lemma 
\ref{conv}. We deduce that
$$
\E[(U^{i,N}_t-U^{i,N}_s)^2]= \sum_{m,n\geq 0} \intot\intot\beta_m(s,t,u)\beta_n(s,t,v)\sum_{j,k=1}^N A_N^m(i,j) A_N^n(i,k)
\Et[M^{j,N}_u M^{k,N}_v] dvdu.
$$
By Remark \ref{mo}, $\Et[M^{j,N}_u M^{k,N}_v]=\indiq_{\{j=k\}}\Et[Z^{j,N}_{u\land v}]$.
Using now (ii) with $s=0$ and $r=\infty$, we see that $x^{j,N}_{t}:=\Et[Z^{j,N}_{t}]-\mu t\ell_N(j)$
satisfies $\sup_{t\geq 0,j=1,\dots,N} |x^{j,N}_{t}| \leq C$ on $\Omega_N^1$.
We thus write
$\Et[(U^{i,N}_t-U^{i,N}_s)^2]=I+J$, where
\begin{align*}
I=& \mu \sum_{m,n\geq 0} \intot\intot\beta_m(s,t,u)\beta_n(s,t,v)\sum_{j=1}^N A_N^m(i,j) A_N^n(i,j) 
(u\land v)\ell_N(j) dudv,\\
J=& \sum_{m,n\geq 0} \intot\intot\beta_m(s,t,u)\beta_n(s,t,v)\sum_{j=1}^N A_N^m(i,j) A_N^n(i,j)x^{j,N}_{u\land v}dudv.
\end{align*}
First, using only that $x^{j,N}_{t}$ is uniformly bounded on $\Omega_N^1$ and that  
$\intot|\beta_m(s,t,u)| du \leq 2 \Lambda^m$, we find
$|J| \leq C \sum_{m,n\geq 0} \Lambda^{m+n}\sum_{j=1}^N A_N^m(i,j) A_N^n(i,j)=C \sum_{j=1}^N(Q_N(i,j))^2$
on $\Omega_N^1$, whence $|J|\leq C\sum_{j=1}^N (\indiq_{\{i=j\}}+N^{-1})^2$ by \eqref{omegancons2}. We conclude that 
$|J| \leq C \leq C(t-s)^2$. Next, we realize that, with the notation of Lemma \ref{conv}-(iii),
$$
I=\mu\sum_{m,n\geq 0} \gamma_{m,n}(s,t)\sum_{j=1}^N A_N^m(i,j) A_N^n(i,j) \ell_N(j).
$$
But we know that $0 \leq \gamma_{m,n}(s,t)\leq \Lambda^{m+n}(t-s)$. Hence
$I \leq \mu(t-s) \sum_{j=1}^N (Q_N(i,j))^2 \ell_N(j)\leq C (t-s)$, since $\ell_N$ is bounded on 
$\Omega_N^1$ and since, as already seen,  $\sum_{j=1}^N(Q_N(i,j))^2$ is also bounded on $\Omega_N^1$.
We conclude that  $\Et[(U^{i,N}_t-U^{i,N}_r)^2] \leq C(t-s) \leq C (t-s)^2$ on $\Omega_N^1$, as desired.
\end{proof}

\subsection{First estimator}\label{se1}

We recall that $\cE^N_t=(\bar Z^N_{2t}-\bar Z^N_{t})/t$, that the matrices $A_N$ and $Q_N$ 
and the event $\Omega_N^1$ were defined in Notation \ref{ssimp}, as well as 
$\ell_N(i)=\sum_{j=1}^N Q_N(i,j)$ and $\bar\ell_N = N^{-1}\sum_{i=1}^N\ell_N(i)$. The goal of this subsection
is to establish the following estimate.

\begin{prop}\label{esti1}
Assume $H(q)$ for some $q\geq 1$. Then for $t\geq 1$,
$$
\indiq_{\Omega_N^1}\Et\Big[\Big|\cE^N_t - \mu \bar \ell_N\Big|^2\Big] 
\leq C\Big(\frac 1 {t^{2q}} + \frac 1 {Nt} \Big).
$$
\end{prop}

We start with the following lemma (recall that $\bar U^N$ was defined in Notation
\ref{processes}).

\begin{lem}\label{esti1lem}
Assume $H(q)$ for some $q\geq 1$. Then on $\Omega_N^1$, for $t\geq 1$,
$$
\Big|\Et[\cE^N_t] - \mu \bar \ell_N \Big| \leq C t^{-q}
\quad \hbox{and}\quad  \Et [ |\bar U^N_t|^2] \leq C t N^{-1}.
$$
\end{lem}

\begin{proof}
Applying Lemma \ref{if}-(ii) with $r=1$, we immediately find, on $\Omega_N^1$,
$$
\Big|\Et[\cE^N_t] - \mu \bar \ell_N \Big| \leq
N^{-1}\Big\| \Et\Big[ \frac{\bZ^N_{2t}-\bZ^N_t}{t}\Big] - \mu \ell_N \Big\|_1 \leq CN^{-1} t^{-q} ||\bun||_1=Ct^{-q}.
$$
Next, we deduce from \eqref{fundc} that 
$\bar U^N_t = N^{-1} \sum_{n\geq 0} \intot \varphi^{\star n}(t-s) \sum_{i,j=1}^N A_N^n(i,j)M^{j,N}_sds$, whence
$$
\Et [ |\bar U^N_t|^2]^{1/2} \leq N^{-1} \sum_{n\geq 0} \intot \varphi^{\star n}(t-s) 
\Et\Big[\Big(\sum_{i,j=1}^N A_N^n(i,j)M^{j,N}_s\Big)^2\Big]^{1/2} ds
$$
by the Minkowski inequality. But recalling Remark 
\ref{mo}, i.e. $\Et[M^{j,N}_sM^{l,N}_s]=\indiq_{\{j=l\}}\Et[Z^{j,N}_s]$,
\begin{align*}
\Et\Big[\Big(\sum_{i,j=1}^N A_N^n(i,j)M^{j,N}_s\Big)^2\Big]=& \sum_{j=1}^N \Big( \sum_{i=1}^N A_N^n(i,j)\Big)^2 
\Et[Z^{j,N}_s] \leq |||A_N|||_1^{2n} \sum_{j=1}^N\Et[Z^{j,N}_s].
\end{align*}
We know from Lemma \ref{if}-(i) with $r=1$  that $\sum_{j=1}^N\Et[Z^{j,N}_s]\leq C N s$ on $\Omega_N^1$. Hence,
still on $\Omega_N^1$,
$$
\Et [ |\bar U^N_t|^2]^{1/2} \leq \frac C {N} \sum_{n\geq 0} |||A_N|||_1^{n}\intot \sqrt{Ns} \varphi^{\star n}(t-s) ds
\leq \frac {Ct^{1/2}} {N^{1/2}} \sum_{n\geq 0} \Lambda^n |||A_N|||_1^{n} \leq \frac{C t^{1/2}}{N^{1/2}}
$$
as desired.
\end{proof}

We can now give the

\begin{proof}[Proof of Proposition \ref{esti1}]
It suffices to write
$$
\Et\Big[\Big|\cE^N_t - \mu \bar \ell_N\Big|^2\Big] 
\leq 2 \Et\Big[\Big|\cE^N_t - \Et[\cE^N_t]\Big|^2\Big] + 2 \Big|\Et[\cE^N_t] - \mu \bar \ell_N \Big|^2
$$
and to observe that $|\cE^N_t - \Et[\cE^N_t]|=| \bar U^{N}_{2t} - \bar U^N_t|/t \leq |\bar U^{N}_{2t}|/t 
+|\bar U^{N}_{t}|/t$, whence finally
$$
\Et\Big[\Big|\cE^N_t - \mu \bar \ell_N\Big|^2\Big] 
\leq \frac4{t^2} (\Et[|\bar U^{N}_{2t}|^2] + \Et[|\bar U^{N}_{t}|^2]) + 2 \Big|\Et[\cE^N_t] - \mu \bar \ell_N \Big|^2.
$$
Then the proposition immediately follows from Lemma \ref{esti1lem}.
\end{proof}

\subsection{Second estimator}\label{se2}

We recall that
$\cV^N_t=\sum_{i=1}^N [(Z^{i,N}_{2t}-Z^{i,N}_{t})/t - \cE^N_t]^2 - N\cE^N_t /t$ 
where $\cE^N_t=(\bar Z^N_{2t}-\bar Z^N_{t})/t$, 
that the matrices $A_N$ and $Q_N$ 
and the event $\Omega_N^1$ were defined in Notation \ref{ssimp}, as well as
$\ell_N(i)=\sum_{j=1}^N Q_N(i,j)$ and $\bar\ell_N = N^{-1}\sum_{i=1}^N\ell_N(i)$. 
We also introduce $\cV^N_\infty=\mu^2\sum_{i=1}^N [\ell_N(i) - \bar \ell_N]^2$.

\begin{prop}\label{esti2}
Assume $H(q)$ for some $q\geq 1$. Then for $t\geq 1$, a.s.,
$$
\indiq_{\Omega_N^1}\Et\Big[\Big|\cV^N_t - \cV^N_\infty\Big|\Big] 
\leq C \Big(1+\sum_{i=1}^N \Big[\ell_N(i)-\bar\ell_N\Big]^2\Big)^{1/2}
\Big(\frac N {t^{q}} + \frac {\sqrt N}t + \frac 1 {\sqrt t} \Big).
$$
\end{prop}

Observe that the term $\sum_{i=1}^N [\ell_N(i)-\bar\ell_N]^2$ will not cause any problem,
since its expectation (restricted to $\Omega_N^1$) is uniformly bounded, see Proposition \ref{interro}.

\vip

We write $|\cV^N_t - \cV^N_\infty| \leq \Delta^{N,1}_t+\Delta^{N,2}_t+\Delta^{N,3}_t$, where
\begin{align*}
\Delta^{N,1}_t = & \Big|\sum_{i=1}^N [(Z^{i,N}_{2t}-Z^{i,N}_{t})/t - \cE^N_t]^2-
\sum_{i=1}^N [(Z^{i,N}_{2t}-Z^{i,N}_{t})/t - \mu\bar\ell_N]^2 \Big|,\\
\Delta^{N,2}_t = & \Big|\sum_{i=1}^N [(Z^{i,N}_{2t}-Z^{i,N}_{t})/t - \mu\ell_N(i)]^2- N\cE^N_t /t \Big|,\\
\Delta^{N,3}_t = & 2 \Big|\sum_{i=1}^N [(Z^{i,N}_{2t}-Z^{i,N}_{t})/t - \mu\ell_N(i)][\mu\ell_N(i)-\mu\bar\ell_N]\Big|.
\end{align*}
We next write $\Delta^{N,2}_t\leq \Delta^{N,21}_t+\Delta^{N,22}_t+\Delta^{N,23}_t$, where
\begin{align*}
\Delta^{N,21}_t = & \Big|\sum_{i=1}^N \Big[(Z^{i,N}_{2t}-Z^{i,N}_{t})/t - \Et[(Z^{i,N}_{2t}-Z^{i,N}_{t})/t]\Big]^2
- N\cE^N_t /t \Big|,\\
\Delta^{N,22}_t = & \sum_{i=1}^N \Big[ \Et[(Z^{i,N}_{2t}-Z^{i,N}_{t})/t]- \mu\ell_N(i)\Big]^2,\\
\Delta^{N,23}_t = & 2 \Big|\sum_{i=1}^N \Big[(Z^{i,N}_{2t}-Z^{i,N}_{t})/t - \Et[(Z^{i,N}_{2t}-Z^{i,N}_{t})/t]\Big]
\Big[ \Et[(Z^{i,N}_{2t}-Z^{i,N}_{t})/t]- \mu\ell_N(i)\Big] \Big|.
\end{align*}
We will also need to write, recalling that $U^{i,N}_t=Z^{i,N}_t-\Et[Z^{i,N}_t]$,
$$
\Delta^{N,21}_t=\Big|\sum_{i=1}^N \Big[(U^{i,N}_{2t}-U^{i,N}_{t})/t\Big]^2
- N\cE^N_t /t \Big|  \leq \Delta^{N,211}_t+\Delta^{N,212}_t+\Delta^{N,213}_t,
$$
where
\begin{align*}
\Delta^{N,211}_t =& \Big|\sum_{i=1}^N \Big\{ \Big((U^{i,N}_{2t}-U^{i,N}_{t})/t\Big)^2 
- \Et\Big[\Big((U^{i,N}_{2t}-U^{i,N}_{t})/t\Big)^2\Big]\Big\} \Big| ,\\
\Delta^{N,212}_t = &\Big|\sum_{i=1}^N \Et\Big[\Big((U^{i,N}_{2t}-U^{i,N}_{t})/t\Big)^2\Big] -\Et[N\cE^N_t /t] \Big| ,\\
\Delta^{N,213}_t = & \Big|N\cE^N_t /t-\Et[N\cE^N_t /t] \Big|.
\end{align*}
Finally, we will use that $\Delta^{N,3}_t\leq \Delta^{N,31}_t+\Delta^{N,32}_t$, where
\begin{align*}
\Delta^{N,31}_t = & 2 \Big|\sum_{i=1}^N \Big[(Z^{i,N}_{2t}-Z^{i,N}_{t})/t - \Et[(Z^{i,N}_{2t}-Z^{i,N}_{t})/t]\Big]
\Big[\mu\ell_N(i)-\mu\bar\ell_N\Big]\Big|, \\
\Delta^{N,32}_t = & 2 \Big|\sum_{i=1}^N \Big[\Et[(Z^{i,N}_{2t}-Z^{i,N}_{t})/t] - \mu\ell_N(i)\Big]
\Big[\mu\ell_N(i)-\mu\bar\ell_N\Big]\Big|.
\end{align*}
To summarize, we have to bound $\Delta^{N,1}_t$, $\Delta^{N,211}_t$, $\Delta^{N,212}_t$, $\Delta^{N,213}_t$,
$\Delta^{N,22}_t$, $\Delta^{N,23}_t$, $\Delta^{N,31}_t$ and $\Delta^{N,32}_t$. 
Only the term $\Delta^{N,211}_t$ is really difficult.

\vip

In the following lemma, we treat the easy terms. We do not try to be optimal when not useful: 
for example in (iv) below, some sharper estimate could probably be obtained with more work,
but since we already have a term in $N^{1/2}t^{-1}$ (see Lemma \ref{esti2l4}), this would be useless.
We also recall that we do not really try to optimize the dependence in $q$: 
it is likely that $t^{-q}$ could be replaced by  $t^{-2q}$ here and there.

\begin{lem}\label{esti2l1}
Assume $H(q)$ for some $q\geq 1$. Then a.s. on $\Omega_N^1$, for $t\geq 1$,

(i) $\Et[\Delta^{N,1}_t] \leq C (N t^{-2q} + t^{-1})$,

(ii) $\Et[\Delta^{N,22}_t] \leq C N t^{-2q}$,

(iii) $\Et[\Delta^{N,23}_t] \leq C N t^{-q}$,

(iv) $\Et[\Delta^{N,213}_t] \leq C N^{1/2} t^{-3/2}$,

(v) $\Et[\Delta^{N,32}_t] \leq C N t^{-q}$.
\end{lem}

\begin{proof} We work on $\Omega_N^1$ during the whole proof.

\vip

Using that $\cE^N_t= N^{-1}\sum_{i=1}^N [(Z^{i,N}_{2t}-Z^{i,N}_{t})/t]$, one easily checks that
$\Delta^{N,1}_t = N | \cE^N_t- \mu \bar \ell_N|^2$. Thus point (i) follows from
Proposition \ref{esti1}. 

\vip

Next, we observe that
$\Delta^{N,22}_t= ||\Et[(\bZ^{N}_{2t}-\bZ^{N}_{t})/t]- \mu\ell_N||_2^2$. Applying Lemma \ref{if}-(ii)
with $r=2$, we conclude that indeed, $\Delta^{N,22}_t \leq C t^{-2q}||\bun||_2^2=C N t^{-2q}$.

\vip

We write
$$
\Delta^{N,23}_t \leq 2 \Big\| (\bZ^{N}_{2t}-\bZ^{N}_{t})/t - \Et[(\bZ^{N}_{2t}-\bZ^{N}_{t})/t]\Big\|_1 
\Big\| \Et[(\bZ^{N}_{2t}-\bZ^{N}_{t})/t] - \mu\ell_N\Big\|_\infty.
$$
Applying Lemma \ref{if}-(ii) with $r=\infty$, we deduce that
$\| \Et[(\bZ^{N}_{2t}-\bZ^{N}_{t})/t] - \mu\ell_N\|_\infty \leq C t^{-q}$.
Lemma \ref{if}-(i) with $r=1$ gives us that
$\E_\theta[\| (\bZ^{N}_{2t}-\bZ^{N}_{t})/t - \Et[(\bZ^{N}_{2t}-\bZ^{N}_{t})/t]\|_1]\leq 2t^{-1} 
\|\E_\theta[\bZ^{N}_{2t}+\bZ^{N}_{t}]\|_1
\leq C N$. We thus find that indeed, $\Et[\Delta^{N,23}_t] \leq C N t^{-q}$.

\vip

Since $\Delta^{N,213}_t=(N/t)|\cE^N_t-\Et[\cE^N_t]|=N t^{-2}|\bar U^N_{2t}-\bar U^N_t| 
\leq Nt^{-2}(|\bar U^N_{2t}|+|\bar U^N_t|)$, we deduce from Lemma \ref{esti1lem} that 
$\Et[\Delta^{N,213}_t] \leq C N t^{-2}\sqrt{t/N}=C N^{1/2}t^{-3/2}$.

\vip

Finally, starting from $\Delta^{N,32}_t \leq 2\mu ||\Et[(\bZ^{N}_{2t}-\bZ^{N}_{t})/t]- \mu\ell_N||_\infty
||\ell_N-\bar \ell_N \bun||_1$ and using that, as already seen when studying $\Delta^{N,23}_t$,
$\| \Et[(\bZ^{N}_{2t}-\bZ^{N}_{t})/t] - \mu\ell_N\|_\infty \leq C t^{-q}$, we conclude that
$\Delta^{N,32}_t \leq C t^{-q}||\ell_N-\bar \ell_N \bun||_1 \leq  C N t^{-q}$, since 
$\ell_N$ is bounded (see \eqref{omegancons1}) on $\Omega_N^1$.
\end{proof}

Next, we treat the term $\Delta^{N,212}_t$.

\begin{lem}\label{esti2l2}
Assume $H(q)$ for some $q\geq 1$. Then a.s. on $\Omega_N^1$, for $t\geq1$,  $\Et[\Delta^{N,212}_t]\leq C t^{-1}$.
\end{lem}

\begin{proof} We work on $\Omega_N^1$.
Recalling that $N\cE^N_t = t^{-1}\sum_{i=1}^N (Z^{i,N}_{2t} -Z^{i,N}_{t})$, we may write
$\Et [\Delta^{N,212}_t] \leq t^{-2}\sum_{i=1}^N a_i$, where $a_i=|\Et[(U^{i,N}_{2t} -U^{i,N}_{t})^2-(Z^{i,N}_{2t} -Z^{i,N}_{t})]|$.
Now we infer from \eqref{fundc} that $U^{i,N}_{t}=M^{i,N}_t + \sum_{n \geq 1} \intot \varphi^{\star n}(t-s) \sum_{j=1}^N
A_N^n(i,j) M^{j,N}_s ds$, so that $U^{i,N}_{2t}-U^{i,N}_t=M^{i,N}_{2t}-M^{i,N}_t + R^{i,N}_t$, where
$$
R^{i,N}_t=\sum_{n \geq 1}\int_0^{2t} \beta_n(t,2t,s) \sum_{j=1}^N A_N^n(i,j) M^{j,N}_s ds.
$$
We have set $\beta_n(t,2t,s)=\varphi^{\star n}(2t-s) - \varphi^{\star n}(t-s)$
as in Lemma \ref{conv} and the only thing we will use is that $\int_0^{2t} |\beta_n(t,2t,s)|ds \leq 2 \Lambda^n$.
Recalling that $M^{i,N}$ is a martingale with quadratic variation 
$[M^{i,N},M^{i,N}]_t=Z^{i,N}_t$, see Remark \ref{mo}, 
we deduce that $\Et[(M^{i,N}_{2t}-M^{i,N}_t)^2]=\Et[Z^{i,N}_{2t} -Z^{i,N}_{t}]$. Hence 
$$
a_i=\Et[(R^{i,N}_t)^2] + 2 \Et[(M^{i,N}_{2t}-M^{i,N}_t) R^{i,N}_t]=b_i+d_i,
$$
the last equality standing for a definition. We first write
$$
b_i=\sum_{m,n\geq 1} \int_0^{2t}\int_0^{2t}\beta_m(t,2t,s)\beta_n(t,2t,u)\sum_{j,k=1}^NA_N^m(i,j)A_N^n(i,k)
\Et[M^{j,N}_sM^{k,N}_u ] duds.
$$
But we know that $\Et[M^{j,N}_sM^{k,N}_u ]=\indiq_{\{j=k\}}\Et[Z^{j,N}_{s\land u}]$ by Remark \ref{mo}
and that $\Et[Z^{j,N}_{s\land u}] \leq C t$ on $\Omega_N^1$ by 
Lemma \ref{if}-(i) (with $r=\infty$). Hence
$$
b_i \leq C t \sum_{m,n\geq 1} \Lambda^{m+n} \sum_{j=1}^NA_N^m(i,j)A_N^n(i,j)=C t \sum_{j=1}^N 
\Big(\sum_{n\geq 1} \Lambda^n A_N^n(i,j) \Big)^2.
$$
But $\sum_{n\geq 1} \Lambda^n A_N^n(i,j)=Q_N(i,j)-\indiq_{\{i=j\}} \leq C N^{-1}$ on  $\Omega_N^1$ by \eqref{omegancons2},
so that $b_i \leq C t N^{-1}$.

\vip

Next, we start from
$$
d_i=2 \sum_{n \geq 1}\int_0^{2t} \beta_n(t,2t,s) \sum_{j=1}^N A_N^n(i,j) \Et[(M^{i,N}_{2t}-M^{i,N}_t)M^{j,N}_s] ds.
$$
As previously, we see that $\Et[(M^{i,N}_{2t}-M^{i,N}_t)M^{j,N}_s]=0$ if $i\ne j$ and that 
$\Et[(M^{i,N}_{2t}-M^{i,N}_t)M^{i,N}_s]=\Et[Z^{i,N}_{2t\land s} -Z^{i,N}_{t\land s} ] \leq C t$ on $\Omega_N^1$
(by Lemma \ref{if}-(i)), whence
$$
d_i \leq C t \sum_{n \geq 1}\Lambda^n A_N^n(i,i)=Ct(Q_N(i,i)-1) \leq C t N^{-1}
$$
on $\Omega_N^1$ by \eqref{omegancons2} again. Finally, $a_i \leq C t N^{-1}$,
so that $\Et [\Delta^{N,212}_t] \leq t^{-2}\sum_{i=1}^N a_i \leq C t^{-1}$
on  $\Omega_N^1$.
\end{proof}

We next compute some covariances in the following tedious lemma.

\begin{lem}\label{covted}
Assume $H(q)$ for some $q\geq 1$. Then a.s., on $\Omega_N^1$, for all $t\geq 1$, all $k,l,a,b \in \{1,\dots,N\}$,
all $r,s,u,v \in [0,t]$,

\vip

(i) $|\Covt (Z^{k,N}_r,Z^{l,N}_s)|=|\Covt (U^{k,N}_r,U^{l,N}_s)| \leq C t (N^{-1}+\indiq_{\{k=l\}} )$,

\vip
(ii) $|\Covt (Z^{k,N}_r,M^{l,N}_s)|=|\Covt (U^{k,N}_r,M^{l,N}_s)| \leq C t (N^{-1}+\indiq_{\{k=l\}} )$,

\vip
(iii) $|\Covt (Z^{k,N}_r,\int_0^s M^{l,N}_{\tau-}dM^{l,N}_\tau)|
=|\Covt (U^{k,N}_r,\int_0^s M^{l,N}_{\tau-}dM^{l,N}_\tau)| \leq C t^{3/2} (N^{-1}+\indiq_{\{k=l\}} )$,

\vip
(iv) $|\Et[M^{k,N}_rM^{k,N}_sM^{l,N}_u]| \leq  C N^{-1} t $ if $\#\{k,l\}=2$,

\vip
(v) $|\Covt (M^{k,N}_rM^{l,N}_s,M^{a,N}_uM^{b,N}_v)|=0$ if $\#\{k,l,a,b\}=4$,

\vip
(vi) $|\Covt (M^{k,N}_rM^{k,N}_s,M^{a,N}_uM^{b,N}_v)|\leq C N^{-2}t$ if $\#\{k,a,b\}=3$,

\vip
(vii) $|\Covt (M^{k,N}_rM^{k,N}_s,M^{a,N}_uM^{a,N}_v)|\leq C N^{-1}t^{3/2}$ if $\#\{k,a\}=2$,

\vip
(viii) $|\Covt (M^{k,N}_rM^{l,N}_s,M^{a,N}_uM^{b,N}_v)| \leq C t^2$ without condition.
\end{lem}

\begin{proof}
We work on $\Omega_N^1$ and start with point (i). First, it is clear, since $U^{k,N}_t=Z^{k,N}_t-\Et[Z^{k,N}_t]$, 
that 
$\Covt (Z^{k,N}_r,Z^{l,N}_s)=\Covt (U^{k,N}_r,U^{l,N}_s)$. Then we infer from \eqref{fundc} that
$$
\Covt(U^{k,N}_r,U^{l,N}_s)\!=\!\sum_{m,n\geq 0} \!\int_0^r\!\int_0^s\!\! \varphi^{\star m}(r-x)\varphi^{\star n}(s-y)
\sum_{i,j=1}^N A_N^m(k,i)A_N^n(l,j) \Covt(M^{i,N}_x,M^{j,N}_y) dydx.
$$
But we know (see Remark \ref{mo}) that  
$\Covt(M^{i,N}_x,M^{j,N}_y)=\indiq_{\{i=j\}}\Et[Z^{i,N}_{x\land y}]\leq C \indiq_{\{i=j\}} t$
by Lemma \ref{if}-(i) (with $r=\infty$). Thus
$$
|\Covt(U^{k,N}_r,U^{l,N}_s)|\leq C t\sum_{m,n\geq 0} \Lambda^{m+n} \sum_{i=1}^NA_N^m(k,i)A_N^n(l,i)= 
Ct \sum_{i=1}^NQ_N(k,i)Q_N(l,i).
$$
Recalling \eqref{omegancons2}, $\sum_{i=1}^NQ_N(k,i)Q_N(l,i)
\leq C\sum_{i=1}^N(N^{-1}+\indiq_{\{k=i\}})(N^{-1}+\indiq_{\{l=i\}} )
\leq C(N^{-1}+\indiq_{\{k=l\}})$. Point (i) is checked.

\vip

For point (ii), we again have $\Covt (Z^{k,N}_r,M^{l,N}_s)=\Covt (U^{k,N}_r,M^{l,N}_s)$ and, 
using again \eqref{fundc}, 
$$
\Covt(U^{k,N}_r,M^{l,N}_s)=\sum_{n\geq 0} \int_0^r \varphi^{\star n}(r-x)
\sum_{i=1}^N A_N^n(k,i)\Covt(M^{i,N}_x,M^{l,N}_s) dx.
$$
Since $|\Covt(M^{i,N}_x,M^{l,N}_s)|\leq C \indiq_{\{i=l\}} t$ as in (i), we conclude that
$$
|\Covt(U^{k,N}_r,M^{l,N}_s)| \leq C t \sum_{n\geq 0} \Lambda^n A_N^n(k,l)=
C t Q_N(k,l) \leq C t (N^{-1}+\indiq_{\{k=l\}}).
$$

Point (iii) is checked similarly as (ii), provided we verify that 
$|\Covt(M^{i,N}_x,\int_0^s M^{l,N}_{\tau-}dM^{l,N}_\tau)|\leq C \indiq_{\{i=l\}} t^{3/2}$.
This is obvious if $i\ne l$ because the martingales $M^{i,N}$ and $\int_0^{\cdot} M^{l,N}_{\tau-}dM^{l,N}_\tau$
are orthogonal, and relies on the fact, if $i=l$, that 
$$
\Big|\Covt\Big( M^{i,N}_x,\int_0^s M^{i,N}_{\tau-}dM^{i,N}_\tau\Big)\Big| \leq \Et[|M^{i,N}_x|^2]^{1/2}
\Et\Big[\Big|\int_0^s M^{i,N}_{\tau-}dM^{i,N}_\tau\Big|^2\Big]^{1/2}\leq C t^{3/2}. 
$$
The last inequality uses that 
$\Et[|M^{i,N}_x|^2]=\Et[Z^{i,N}_x]\leq C t$ by Remark \ref{mo} and Lemma \ref{if}-(i) and that 
$\Et[|\int_0^s M^{i,N}_{\tau-}dM^{i,N}_\tau|^2]\leq C t^2$.
Indeed, we have
$[\int_0^{\cdot} M^{i,N}_{\tau-}dM^{i,N}_\tau,\int_0^{\cdot} M^{i,N}_{\tau-}dM^{i,N}_\tau ]_s=
\int_0^s (M^{i,N}_{\tau-})^2d Z^{i,N}_{\tau-}\leq (M^{i,N,*}_s)^2 Z^{i,N}_s$, whence
$$
\Et\Big[\Big|\int_0^s M^{i,N}_{\tau-}dM^{i,N}_\tau\Big|^2\Big]
\leq \Et[(M^{i,N,*}_s)^2 Z^{i,N}_s]
\leq \Et[(M^{i,N,*}_s)^4]^{1/ 2}\Et[(Z^{i,N}_s)^2]^{1/2}, 
$$
which is bounded by $C t^2$ by Lemma \ref{if}-(iii). 
 
\vip

For point (iv), we assume e.g. that $r\leq s$ and first note that
$$
\Et[M^{k,N}_rM^{k,N}_sM^{l,N}_u]=\Et[M^{k,N}_r\Et[M^{k,N}_sM^{l,N}_u |\cF_r]]=\Et[(M^{k,N}_r)^2M^{l,N}_{u\land r}]
$$
because the martingales $M^{k,N}$ and $M^{l,N}$ are orthogonal.
Since $[M^{k,N},M^{k,N}]_r=Z^{k,N}_r$, it holds that
$(M^{k,N}_r)^2=2\int_0^r M^{k,N}_{\tau-}dM^{k,N}_\tau + Z^{k,N}_r$. Using that
$\int_0^{\cdot} M^{k,N}_{\tau-}dM^{k,N}_\tau$ and $M^{l,N}$ are orthogonal, we conclude that
$\Et[(M^{k,N}_r)^2M^{l,N}_{u\land r}]=\Et[Z^{k,N}_r M^{l,N}_{u\land r}]=\Covt(Z^{k,N}_r,M^{l,N}_{u\land r})$.
Since $k\ne l$, we conclude using point (ii).

\vip

Point (v) is obvious, since when $k,l,a,b$ are pairwise different, the martingales
$M^{k,N}$, $M^{l,N}$, $M^{a,N}$ and $M^{b,N}$ are orthogonal.

\vip

Point (vi) is harder. Recall that $\#\{k,a,b\}=3$, so that clearly,
$\Covt (M^{k,N}_rM^{k,N}_s,M^{a,N}_uM^{b,N}_v)=\Et[M^{k,N}_rM^{k,N}_sM^{a,N}_uM^{b,N}_v]$.
We assume e.g. that $r\leq s$ and we observe that
$$
\Et[M^{k,N}_rM^{k,N}_sM^{a,N}_uM^{b,N}_v]=\Et[ M^{k,N}_r \Et[M^{k,N}_sM^{a,N}_uM^{b,N}_v |\cF_r] ]
=\Et[(M^{k,N}_r)^2M^{a,N}_{u\land r}M^{b,N}_{v\land r}]
$$
because $M^{k,N}$, $M^{a,N}$ and $M^{b,N}$ are orthogonal.
We thus have to prove that for all $r,u,v \in [0,t]$ with $u,v \leq r$, 
$|\Et[(M^{k,N}_r)^2M^{a,N}_{u}M^{b,N}_{v}]|\leq C N^{-2}t$.
We write $(M^{k,N}_r)^2=2\int_0^r M^{k,N}_{\tau-}dM^{k,N}_\tau + Z^{k,N}_r$ as in the proof 
of (iv).
The three martingales $\int_0^{\cdot} M^{k,N}_{\tau-}dM^{k,N}_\tau$, $M^{a,N}$ and $M^{b,N}$ being orthogonal,
we find $\Et[(M^{k,N}_r)^2M^{a,N}_{u}M^{b,N}_{v}]= \Et[Z^{k,N}_rM^{a,N}_{u}M^{b,N}_{v}]=\Et[U^{k,N}_rM^{a,N}_{u}M^{b,N}_{v}]$.
We next write, starting again from \eqref{fundc},
$$
\Et[U^{k,N}_rM^{a,N}_{u}M^{b,N}_{v}]=\sum_{n\geq 0} \int_0^r \varphi^{\star n}(r-x) \sum_{j=1}^NA_N^n(k,j) 
\Et[M^{j,N}_x M^{a,N}_{u}M^{b,N}_{v}] dx.
$$
But $|\Et[M^{j,N}_x M^{a,N}_{u}M^{b,N}_{v}]|$ is zero if $j\notin\{a,b\}$
because the martingales $M^{j,N}$, $M^{a,N}$ and $M^{b,N}$ are orthogonal,
and is bounded by $CN^{-1}t$ else by point (iv). As a consequence,
$$
|\Et[U^{k,N}_rM^{a,N}_{u}M^{b,N}_{v}]|\leq C N^{-1}t \sum_{n\geq 0} \Lambda^n (A_N^n(k,a)+A_N^n(k,b))
=C N^{-1}t (Q_N(k,a)+Q_N(k,b)).
$$
Since $k\ne a$ and $k\ne b$, this is bounded by $C N^{-2}t$ by \eqref{omegancons2}.

\vip

For (vii), we assume e.g. that $r\leq s$ and $u\leq v$ and we recall that $k \ne a$.
We have
\begin{align*}
&\Covt (M^{k,N}_rM^{k,N}_s,M^{a,N}_uM^{a,N}_v)\\
=&\Covt((M^{k,N}_r)^2,(M^{a,N}_u)^2)+\Covt(M^{k,N}_r(M^{k,N}_s-M^{k,N}_r),(M^{a,N}_u)^2) \\
&+\!  \Covt((M^{k,N}_r)^2,M^{a,N}_u(M^{a,N}_v-M^{a,N}_u ))\! 
+\! \Covt(M^{k,N}_r(M^{k,N}_s-M^{k,N}_r),M^{a,N}_u(M^{a,N}_v-M^{a,N}_u ))\\
=&I+J+K+L.
\end{align*}
First, $L=0$. Indeed, assuming e.g. that $r\geq u$, we have 
\begin{align*}
L=& \Et[M^{k,N}_r(M^{k,N}_s-M^{k,N}_r)M^{a,N}_u(M^{a,N}_v-M^{a,N}_r+M^{a,N}_r-M^{a,N}_u)]   \\
=&\Et[M^{k,N}_rM^{a,N}_u  \Et[(M^{k,N}_s-M^{k,N}_r)(M^{a,N}_v-M^{a,N}_r )|\cF_r]]\\
&+\Et[M^{k,N}_rM^{a,N}_u (M^{a,N}_r-M^{a,N}_u ) \Et[M^{k,N}_s-M^{k,N}_r|\cF_r]]
\end{align*}
and in both terms, the conditional expectation vanishes. Next, we write as usual
$(M^{k,N}_r)^2=2\int_0^r M^{k,N}_{\tau-}dM^{k,N}_\tau + Z^{k,N}_r$ and
$(M^{a,N}_u)^2=2\int_0^u M^{a,N}_{\tau-}dM^{a,N}_\tau + Z^{a,N}_u$. By orthogonality of the martingales
$\int_0^{\cdot} M^{k,N}_{\tau-}dM^{k,N}_\tau$ and $\int_0^{\cdot}M^{a,N}_{\tau-}dM^{a,N}_\tau$, we find
$$
I= \Covt(Z^{k,N}_r,Z^{a,N}_u)+2\Covt\Big(Z^{k,N}_r,\int_0^u M^{a,N}_{\tau-}dM^{a,N}_\tau\Big)+
2\Covt\Big(\int_0^r M^{k,N}_{\tau-}dM^{k,N}_\tau,Z^{a,N}_u\Big).
$$
We deduce from points (i) and (iii), since $k \ne a$, that $|I| 
\leq C(N^{-1}t+N^{-1}t^{3/2})\leq C N^{-1}t^{3/2}$.
We now treat $K$. It vanishes if $u\geq r$, because $\Et[M^{a,N}_v-M^{a,N}_u \vert \cF_u]=0$. 
We thus assume that $u<r$.
We write as usual $(M^{k,N}_r)^2=(M^{k,N}_u)^2 + 2\int_u^r M^{k,N}_{\tau-}dM^{k,N}_\tau + Z^{k,N}_r-Z^{k,N}_{u}$ and
\begin{align*}
K=&\Et[(M^{k,N}_u)^2M^{a,N}_u(M^{a,N}_v-M^{a,N}_u )]+2\E\Big[\Big(\int_u^r M^{k,N}_{\tau-}dM^{k,N}_\tau\Big) 
M^{a,N}_u(M^{a,N}_v-M^{a,N}_u )\Big]\\
&+ \Et[(Z^{k,N}_r-Z^{k,N}_{u}) M^{a,N}_u(M^{a,N}_v-M^{a,N}_u )].
\end{align*}
The first term vanishes (because $\Et[M^{a,N}_v-M^{a,N}_u|\cF_u]=0$), as well as the second one
(because $\Et[(\int_u^r M^{k,N}_{\tau-}dM^{k,N}_\tau)(M^{a,N}_v-M^{a,N}_u)|\cF_u]=0$ by orthogonality of the involved 
martingales). Consequently, 
$$
K=\Et[(Z^{k,N}_r-Z^{k,N}_{u}) M^{a,N}_u(M^{a,N}_v-M^{a,N}_u )]= \Et[(U^{k,N}_r-U^{k,N}_{u}) M^{a,N}_u(M^{a,N}_v-M^{a,N}_u )].
$$
Using \eqref{fundc} and recalling that $\beta_n(u,r,x)=\varphi^{\star n}(r-x)-\varphi^{\star n}(u-x)$, we find
$$
K=\sum_{n\geq 0} \int_0^r \beta_n(u,r,x)\sum_{j=1}^N A_N^n(k,j)\Et[M^{j,N}_x M^{a,N}_u(M^{a,N}_v-M^{a,N}_u ) ] dx.
$$
But $|\Et[M^{j,N}_x M^{a,N}_u(M^{a,N}_v-M^{a,N}_u ) ]|\leq C N^{-1} t$ if $a\ne j$ by (iv), while 
$|\Et[M^{j,N}_x M^{a,N}_u(M^{a,N}_v-M^{a,N}_u ) ]|\leq C t^{3/2}$ if $a=j$ by Lemma \ref{if}-(iii). Thus
$$
|K| \leq C \sum_{n\geq 0} \Lambda^n \Big[A_N^n(k,a) t^{3/2} + \sum_{j=1}^N  A_N^n(k,j) N^{-1}t\Big]
\leq C \Big[Q_N(k,a)t^{3/2} + N^{-1}\sum_{j=1}^N Q_N(k,j) t \Big].
$$
But $k\ne a$ implies that $Q_N(k,a) \leq C N^{-1}$ by \eqref{omegancons2}, while 
$N^{-1}\sum_{j=1}^N Q_N(k,j)\leq C N^{-1} |||Q_N|||_\infty\leq CN^{-1}$. As a conclusion,
$|K| \leq C N^{-1}(t^{3/2} + t) \leq C N^{-1} t^{3/2}$. Of course, $J$ is treated similarly,
and this completes the proof of point (vii).

\vip

Point (viii) is obvious: it suffices to use the H\"older inequality to find
$$
|\Covt (M^{k,N}_rM^{l,N}_s,M^{a,N}_uM^{b,N}_v)| \leq \Et[(M^{k,N}_r)^4]^{1/4}\Et[(M^{l,N}_s)^4]^{1/4}\Et[(M^{a,N}_u)^4]^{1/4}
\Et[(M^{b,N}_v)^4]^{1/4},
$$
which is bounded by $C t^2$ by Lemma \ref{if}-(iii).
\end{proof}

We can now easily bound $\Delta^{N,31}_t$.

\begin{lem}\label{esti2l3}
Assume $H(q)$ for some $q\geq 1$. Then a.s., on $\Omega_N^1$, for $t\geq 1$,
$$
\Et[(\Delta^{N,31}_t)^2]\leq C t^{-1}\sum_{i=1}^N \Big[\ell_N(i)-\bar\ell_N\Big]^2.
$$
\end{lem}

\begin{proof}
We first note that 
$$
\Delta^{N,31}_t = 2\mu t^{-1} \Big|\sum_{i=1}^N \Big[U^{i,N}_{2t}-U^{i,N}_{t}\Big]
\Big[\ell_N(i)-\bar\ell_N\Big]\Big|.
$$
Since $U^{i,N}_{2t}-U^{i,N}_{t}$ is centered (its conditional expectation $\Et$
vanishes),
$$
\Et[(\Delta^{N,31}_t)^2]=4\mu^2t^{-2} \sum_{i,j=1}^N \Big[\ell_N(i)-\bar\ell_N\Big]\Big[\ell_N(j)-\bar\ell_N\Big]
\Covt(U^{i,N}_{2t}-U^{i,N}_{t},U^{j,N}_{2t}-U^{j,N}_{t}).
$$
Using now Lemma \ref{covted}-(i), we deduce that 
$|\Covt(U^{i,N}_{2t}-U^{i,N}_{t},U^{j,N}_{2t}-U^{j,N}_{t})|\leq C t (\indiq_{\{i=j\}}+N^{-1})$ on $\Omega_N^1$.
Using furthermore that $[\ell_N(i)-\bar\ell_N][\ell_N(j)-\bar\ell_N]\leq 
[\ell_N(i)-\bar\ell_N]^2+[\ell_N(j)-\bar\ell_N]^2$
and a symmetry argument, we conclude that
$$
\Et[(\Delta^{N,31}_t)^2]\leq C t^{-1} \sum_{i,j=1}^N \Big[\ell_N(i)-\bar\ell_N\Big]^2
(\indiq_{\{i=j\}}+N^{-1} )= C t^{-1} \sum_{i=1}^N \Big[\ell_N(i)-\bar\ell_N\Big]^2,
$$
which was our goal.
\end{proof}

We can finally estimate $\Delta^{N,211}_t$.

\begin{lem}\label{esti2l4}
Assume $H(q)$ for some $q\geq 1$. Then a.s., on $\Omega_N^1$, for $t\geq 1$,
$\Et[(\Delta^{N,211}_t)^2]\leq C  N t^{-2}.$
\end{lem}

\begin{proof}
We as usual work on $\Omega_N^1$.
We first note that $\Et[(\Delta^{N,211}_t)^2]=t^{-4}\sum_{i,j=1}^N a_{ij}$, where 
$$
a_{ij}=\Covt((U^{i,N}_{2t}-U^{i,N}_{t})^2,(U^{j,N}_{2t}-U^{j,N}_{t})^2).
$$
But recalling \eqref{fundc} and setting $\alpha_N(s,t,i,k)=\sum_{n\geq 0}A_N^n(i,k)[\varphi^{\star n}(2t-s)-
\varphi^{\star n}(t-s) ]$ for all $0\leq s \leq 2t$ and $i,k\in \{1,\dots,N\}$,
\begin{equation}\label{suru}
U^{i,N}_{2t}-U^{i,N}_t = \int_0^{2t} \sum_{k=1}^N \alpha_N(s,t,i,k)M^{k,N}_sds.
\end{equation}
Concerning $\alpha_N$, we will only use that, on $\Omega_N^1$,
\begin{equation}\label{suralpha}
\int_0^{2t} |\alpha_N(s,t,i,k)|ds \leq 2 \sum_{n\geq 0} \Lambda^n A_N^n(i,k) = 2 Q_N(i,k)\leq C(\indiq_{\{i=k\}}+N^{-1}),
\end{equation}
the last inequality coming from \eqref{omegancons2}. A direct computation starting from \eqref{suru} shows that
\begin{align*}
a_{ij}=&\sum_{k,l,a,b=1}^N \int_0^{2t}\int_0^{2t}\int_0^{2t}\int_0^{2t} \alpha_N(r,t,i,k)\alpha_N(s,t,i,l)
\alpha_N(u,t,j,a)\alpha_N(v,t,j,b)\\
&\hskip7cm \Covt(M^{k,N}_rM^{l,N}_s,M^{a,N}_uM^{b,N}_v) dvdudsdr.
\end{align*}
Let us now denote by $\Gamma_{k,l,a,b}(t)=\sup_{r,s,u,v \in [0,2t]}|\Covt(M^{k,N}_rM^{l,N}_s,M^{a,N}_uM^{b,N}_v)|$.
We can write, recalling \eqref{suralpha},
$$
\sum_{i,j=1}^N a_{ij} \leq C \!\! \sum_{i,j,k,l,a,b=1}^N  (\indiq_{\{i=k\}}+N^{-1})(\indiq_{\{i=l\}}+N^{-1})
(\indiq_{\{j=a\}}+N^{-1})
(\indiq_{\{j=b\}}+N^{-1})\Gamma_{k,l,a,b}(t).
$$
Using some symmetry arguments, we find that $\sum_{i,j=1}^N a_{ij} \leq C [R_1+\dots+R_6]$, where
\begin{gather*}
R_1= N^{-4} \sum_{i,j,k,l,a,b=1}^N\Gamma_{k,l,a,b}(t)=N^{-2} \sum_{k,l,a,b=1}^N \Gamma_{k,l,a,b}(t), \\
R_2= N^{-3} \sum_{i,j,k,l,a,b=1}^N \indiq_{\{i=k\}} \Gamma_{k,l,a,b}(t)= N^{-2} \sum_{k,l,a,b=1}^N \Gamma_{k,l,a,b}(t),\\
R_3= N^{-2} \sum_{i,j,k,l,a,b=1}^N \indiq_{\{i=k\}}\indiq_{\{j=a\}} \Gamma_{k,l,a,b}(t)=N^{-2} \sum_{k,l,a,b=1}^N 
\Gamma_{k,l,a,b}(t),\\
R_4= N^{-2} \sum_{i,j,k,l,a,b=1}^N \indiq_{\{i=k\}}\indiq_{\{i=l\}} \Gamma_{k,l,a,b}(t)=N^{-1} \sum_{k,a,b=1}^N 
\Gamma_{k,k,a,b}(t),\\
R_5= N^{-1} \sum_{i,j,k,l,a,b=1}^N \indiq_{\{i=k\}}\indiq_{\{i=l\}}\indiq_{\{j=a\}} \Gamma_{k,l,a,b}(t)
= N^{-1} \sum_{k,a,b=1}^N\Gamma_{k,k,a,b}(t), \\
R_6= \sum_{i,j,k,l,a,b=1}^N \indiq_{\{i=k\}}\indiq_{\{i=l\}}\indiq_{\{j=a\}}\indiq_{\{j=b\}} \Gamma_{k,l,a,b}(t)
=\sum_{k,a=1}^N\Gamma_{k,k,a,a}(t) .
\end{gather*}
Using Lemma \ref{covted}-(v)-(viii), from which
$\Gamma_{k,l,a,b}(t) \leq C t^2 \indiq_{\{\#\{k,l,a,b\}<4\}}$, we deduce that 
$R_1=R_2=R_3 \leq C N t^2$.
Next we use Lemma \ref{covted}-(vi)-(viii), that is
$\Gamma_{k,k,a,b}(t) \leq C (\indiq_{\{\#\{k,a,b\}=3\}} N^{-2}t+\indiq_{\{\#\{k,a,b\}<3\}}t^2)$, whence
$R_4=R_5 \leq C t +CN t^2 \leq C N t^2$. Finally,
we use Lemma \ref{covted}-(vii)-(viii), i.e. $\Gamma_{k,k,a,a}(t) \leq C (\indiq_{\{\#\{k,a\}=2\}} N^{-1}t^{3/2}
+\indiq_{\{\#\{k,a\}=1\}}t^2)$ and find that $R_6 \leq C Nt^{3/2}+CN t^2 \leq C N t^2$.
All in all, we have proved that $\sum_{i,j=1}^N a_{ij} \leq CN t^2$, which completes the proof.
\end{proof}

We can finally give the

\begin{proof}[Proof of Proposition \ref{esti2}]
It suffices to recall that $|\cV^N_t-\cV^N_\infty|\leq \Delta^{N,1}_t+\Delta^{N,211}_t+\Delta^{N,212}_t+\Delta^{N,213}_t
+\Delta^{N,22}_t+\Delta^{N,23}_t+\Delta^{N,31}_t+\Delta^{N,32}_t$ and to use Lemmas
\ref{esti2l1}, \ref{esti2l2}, \ref{esti2l3} and \ref{esti2l4}:
this gives, on $\Omega_N^1$,
$$
\Et[|\cV^N_t-\cV^N_\infty|]\leq C\Big(\frac{N}{t^{2q}}+\frac 1t +\frac N {t^q} + \frac{N^{1/2}}{t^{3/2}}
+ \Big[\sum_{i=1}^N [\ell_N(i)-\bar \ell_N]^2\Big]^{1/2}\frac{1}{t^{1/2}} + \frac{N^{1/2}}t
 \Big).
$$
Recalling that $t\geq 1$, the conclusion immediately follows.
\end{proof}

\subsection{Third estimator}\label{se3}

We recall that, for $\Delta>0$ such that $t/(2\Delta)$ is an integer, we have set
$\cE^N_t=(\bar Z^N_{2t}-\bar Z^N_{t})/t$, 
$\cZ^N_{\Delta,t}=(N/t)\sum_{a=t/\Delta+1}^{2t/\Delta}[\bar Z^N_{a\Delta}-\bar Z^N_{(a-1)\Delta} -\Delta \cE^N_t]^2$ 
and $\cW^N_{\Delta,t}=2 \cZ^{N}_{2\Delta,t}-\cZ^N_{\Delta,t}$.
The matrices $A_N$ and $Q_N$ 
and the event $\Omega_N^1$ were defined in Notation \ref{ssimp}, as well as  
$\ell_N(i)=\sum_{j=1}^N Q_N(i,j)$ and $c_N(i)=\sum_{j=1}^N Q_N(j,i)$.
We finally introduce $\cW^N_{\infty,\infty}=\mu N^{-1}\sum_{i=1}^N \ell_N(i)(c_N(i))^2$.
The aim of the subsection is to verify the following result.

\begin{prop}\label{esti3}
Assume $H(q)$ for some $q\geq 3$. Then a.s., for $t\geq 4$ and $\Delta \in [1,t/4]$
such that $t/(2\Delta)$ is a positive integer,
$$
\indiq_{\Omega_N^1}\Et\Big[\Big|\cW^N_{\Delta,t} - \cW^N_{\infty,\infty}\Big|\Big] 
\leq C \Big(\sqrt{\frac \Delta t} + \frac N {\Delta^{(q+1)/2}}+ \frac {t}{\Delta^{q/2+1}} \Big).
$$
\end{prop}

Recall that we do not try to optimize the dependence in $q$.
We first write 
$$
|\cW^N_{\Delta,t} - \cW^N_{\infty,\infty}|\leq D^{N,1}_{\Delta,t}+2 D^{N,1}_{2\Delta,t}
+ D^{N,2}_{\Delta,t}+2 D^{N,2}_{2\Delta,t}
+ D^{N,3}_{\Delta,t}+2 D^{N,3}_{2\Delta,t}
+D^{N,4}_{\Delta,t}, 
$$
where
$$D^{N,1}_{\Delta,t}=\frac Nt \Big|\sum_{a=t/\Delta+1}^{2t/\Delta}\Big[\bar Z^N_{a\Delta}-\bar Z^N_{(a-1)\Delta} 
-\Delta \cE^N_t\Big]^2-
\sum_{a=t/\Delta+1}^{2t/\Delta}\Big[\bar Z^N_{a\Delta}-\bar Z^N_{(a-1)\Delta} - \Delta \mu \bar \ell_N\Big]^2 \Big|,
$$
$$D^{N,2}_{\Delta,t}=\frac Nt \Big|\sum_{a=t/\Delta+1}^{2t/\Delta}\Big[\bar Z^N_{a\Delta}-\bar Z^N_{(a-1)\Delta} 
-\Delta \mu \bar \ell_N\Big]^2-
\!\!\sum_{a=t/\Delta+1}^{2t/\Delta}\Big[\bar Z^N_{a\Delta}-\bar Z^N_{(a-1)\Delta} 
- \Et[\bar Z^N_{a\Delta}-\bar Z^N_{(a-1)\Delta}]  \Big]^2 \Big|,
$$
\begin{align*}
D^{N,3}_{\Delta,t}=&\frac Nt \Big|\sum_{a=t/\Delta+1}^{2t/\Delta}\Big[\bar Z^N_{a\Delta}-\bar Z^N_{(a-1)\Delta} 
-\Et[\bar Z^N_{a\Delta}-\bar Z^N_{(a-1)\Delta}]\Big]^2\\
&\hskip3cm -
\Et\Big[\sum_{a=t/\Delta+1}^{2t/\Delta}\Big[\bar Z^N_{a\Delta}-\bar Z^N_{(a-1)\Delta} - \Et[\bar Z^N_{a\Delta}
-\bar Z^N_{(a-1)\Delta}]  \Big]^2\Big] \Big|,
\end{align*}
\begin{align*}
D^{N,4}_{\Delta,t}=&\Big|\frac {2N}t \Et\Big[\sum_{a=t/(2\Delta)+1}^{t/\Delta}\Big[\bar Z^N_{2a\Delta}-\bar Z^N_{2(a-1)\Delta} 
- \Et[\bar Z^N_{2a\Delta}-\bar Z^N_{2(a-1)\Delta}]  \Big]^2\Big]\\ &
\hskip2cm- \frac {N}t \Et\Big[\sum_{a=t/\Delta+1}^{2t/\Delta}\Big[\bar Z^N_{a\Delta}-\bar Z^N_{(a-1)\Delta} 
- \Et[\bar Z^N_{a\Delta}-\bar Z^N_{(a-1)\Delta}]  \Big]^2\Big] 
- \cW^N_{\infty,\infty}\Big|.
\end{align*}

We treat these four terms one by one.

\begin{lem}\label{esti3l1}
Assume $H(q)$ for some $q\geq 1$. Then a.s. on $\Omega_N^1$, for $1\leq \Delta \leq t$,
$\Et[D^{N,1}_{\Delta,t}] \leq C \Delta [t^{-1}+Nt^{-2q}]$.
\end{lem}

\begin{proof}
Using that $(\Delta/t)\sum_{a=t/\Delta+1}^{2t/\Delta} (\bar Z^N_{a\Delta}-\bar Z^N_{(a-1)\Delta})=\Delta \cE^N_t$, we find that
$$
D^{N,1}_{\Delta,t} = \frac Nt \frac t\Delta (\Delta\mu\bar\ell_N-\Delta\cE^N_t)^2 =  N 
\Delta (\mu\bar\ell_N- \cE^N_t)^2,
$$ 
whence, on $\Omega_N^1$, see Proposition \ref{esti1},
$\Et[D^{N,1}_{\Delta,t}] \leq C N\Delta(t^{-2q}+(Nt)^{-1}) \leq C \Delta (N t^{-2q}+t^{-1})$.
\end{proof}

The second term is also easy.

\begin{lem}\label{esti3l2}
Assume $H(q)$ for some $q\geq 1$. Then on $\Omega_N^1$, for $1\leq\Delta\leq t$,
$\Et[D^{N,2}_{\Delta,t}] \leq C N t^{1-q}$.
\end{lem}

\begin{proof}
Using that $|(A-x)^2 - (A-y)^2| \leq |x-y|(|x|+|y|+2|A|)$,
$$
D^{N,2}_{\Delta,t}\leq \frac N t  \sum_{a=t/\Delta+1}^{2t/\Delta} \Big|\Delta \mu \bar \ell_N - 
\Et[\bar Z^N_{a\Delta}-\bar Z^N_{(a-1)\Delta}] \Big| 
\Big[\Delta \mu \bar \ell_N + \Et[\bar Z^N_{a\Delta}-\bar Z^N_{(a-1)\Delta}] +2(Z^N_{a\Delta}-\bar Z^N_{(a-1)\Delta}) \Big],
$$
whence
$$
\Et[D^{N,2}_{\Delta,t}]\leq \frac N t  \sum_{a=t/\Delta+1}^{2t/\Delta} \Big|\Delta \mu \bar \ell_N - 
\Et[\bar Z^N_{a\Delta}-\bar Z^N_{(a-1)\Delta}] \Big| 
\Big[\Delta \mu \bar \ell_N + 3\Et[\bar Z^N_{a\Delta}-\bar Z^N_{(a-1)\Delta}]\Big].
$$
But we deduce from Lemma \ref{if}-(ii) with $r=1$ that, since $(a-1)\Delta \geq t$, 
$$
\Big|\Delta \mu \bar \ell_N - \Et[\bar Z^N_{a\Delta}-\bar Z^N_{(a-1)\Delta}] \Big| \leq C t^{1-q},
$$
whence also $\Et[\bar Z^N_{a\Delta}-\bar Z^N_{(a-1)\Delta}] \leq \Delta \mu \bar \ell_N +Ct^{1-q}
\leq  \Delta \mu \bar \ell_N +C$.
We conclude that
$$
\Et[D^{N,2}_{\Delta,t}]\leq C \frac N {t}  \sum_{a=t/\Delta+1}^{2t/\Delta} t^{1-q}\Big[4\Delta \mu \bar \ell_N + C\Big].
$$
Since $\bar \ell_N$ is bounded on $\Omega_N^1$ and since $\Delta \geq 1 \geq t^{1-q}$,
we find $\Et[D^{N,2}_{\Delta,t}]\leq C (N/t)(t/\Delta)t^{1-q}\Delta \leq CN t^{1-q}$.
\end{proof}

To treat $D^{N,4}_{\Delta,t}$, we need the following lemma.

\begin{lem}\label{esti3l3}
Assume $H(q)$ for some $q\geq 1$.
Almost surely on $\Omega_N^1$, for all $1 \leq \Delta \leq x/2$,
$$
\Vart (\bar U^N_{x+\Delta}-\bar U^N_x) = \frac\Delta N \cW^N_{\infty,\infty}
- \cX^N + r_N(x,\Delta), 
$$
where $\cX_N$ is a $\sigma((\theta_{ij})_{i,j=1,\dots,N})$-measurable finite random variable
and where $r_N$ satisfies, for some deterministic constant $C$, the inequality 
$|r_N(x,\Delta)| \leq C x \Delta^{-q}N^{-1}$.
\end{lem}

\begin{proof}
We set $V^N_{x,\Delta}= \Vart (\bar U^N_{x+\Delta}-\bar U^N_x)$.

\vip

{\it Step 1.}
Recalling \eqref{fundc} and setting $\beta_n(x,x+\Delta,s)=\varphi^{\star n}(x+\Delta-s) 
- \varphi^{\star n}(x-s)$ as in Lemma \ref{conv}, we get
$$
\bar U^N_{x+\Delta}-\bar U^N_x=\sum_{n\geq 0} \int_0^{x+\Delta} \beta_n(x,x+\Delta,s) 
N^{-1}\sum_{i,j=1}^N A_N^n(i,j)M^{j,N}_sds.
$$
Hence
\begin{align*}
V^N_{x,\Delta}=&\sum_{m,n\geq 0} \int_0^{x+\Delta}\!\int_0^{x+\Delta}\! \beta_m(x,x+\Delta,r)
\beta_n(x,x+\Delta,s)N^{-2}\!\sum_{i,j,k,l=1}^N A_N^m(i,j)A_N^n(k,l)\\
&\hskip9cm
\Covt(M^{j,N}_r,M^{l,N}_s)drds.
\end{align*}
Using Remark \ref{mo}, we find
\begin{align*}
V^N_{x,\Delta}=\sum_{m,n\geq 0} \int_0^{x+\Delta}\!\int_0^{x+\Delta} \!\beta_m(x,x+\Delta,r)
\beta_n(x,x+\Delta,s)
N^{-2}\! \sum_{i,j,k=1}^N A_N^m(i,j)A_N^n(k,j)\Et[Z^{j,N}_{r\land s}]drds.
\end{align*}

{\it Step 2.} Here we show that $\Et[Z^{j,N}_s]= \mu \ell_N(j) s -X_j^N + R_j^N(s)$,
with, for some constant $C$, for all $j=1,\dots,N$,
$$
0 \leq X_j^N \leq C \quad \hbox{and} \quad |R_j^N(s)|\leq C (s^{1-q} \land 1).
$$
By \eqref{fundb}, we have $\Et[Z^{j,N}_s]=\mu \sum_{n\geq 0} (\int_0^s r \varphi^{\star n}
(s-r) dr) \sum_{l=1}^N A_N^n(j,l)$, whence by Lemma \ref{conv}-(i),
$$
\Et[Z^{j,N}_s]=\mu \sum_{n\geq 0} (\Lambda^n s - n \Lambda^n \kappa + \e_n(s))
\sum_{l=1}^N A_N^n(j,l)= \mu \ell_N(j) s -X_j^N + R_j^N(s).
$$ 
We have used that $ \sum_{n\geq 0} \Lambda^n \sum_{l=1}^N A_N^n(j,l)= \sum_{l=1}^N Q_N(j,l)= \ell_N(j)$
and we have set 
$X_j^N= \mu\kappa \sum_{n\geq 0} n \Lambda^n \sum_{l=1}^N A_N^n(j,l)$
and $R_j^N(s) = \mu \sum_{n \geq 0} \e_n(s)  \sum_{l=1}^N A_N^n(j,l)$.
We obviously have $0\leq X^N_j \leq \mu\kappa \sum_{n\geq 0} n \Lambda^n |||A_N|||_\infty^n \leq C$ on
$\Omega_N^1$ and, since $\e_n(s) \leq C n^q \Lambda^n (s^{1-q} \land 1)$ by Lemma \ref{conv}-(i),
$|R^N_j(s)|\leq C (s^{1-q} \land 1) \sum_{n \geq 0} n^q \Lambda^n |||A_N|||_\infty^n \leq C (s^{1-q} \land 1)$,
still on $\Omega_N^1$.

\vip

{\it Step 3.} Gathering Steps 1 and 2, we now write $V_{x,\Delta}^N=I-J+K$, where 
\begin{align*}
I=&\!\!\sum_{m,n\geq 0}\int_0^{x+\Delta}\!\!\int_0^{x+\Delta} \!\!\beta_m(x,x+\Delta,r)
\beta_n(x,x+\Delta,s)
N^{-2}\!\! \sum_{i,j,k=1}^N A_N^m(i,j)A_N^n(k,j)\mu\ell_N(j)(r\land s)drds,\\
J=&\!\!\sum_{m,n\geq 0}\int_0^{x+\Delta}\!\int_0^{x+\Delta} \!\!\beta_m(x,x+\Delta,r)
\beta_n(x,x+\Delta,s)
N^{-2}\! \sum_{i,j,k=1}^N A_N^m(i,j)A_N^n(k,j) X_j^N drds, \\
K= &\!\!\sum_{m,n\geq 0}\int_0^{x+\Delta}\!\int_0^{x+\Delta} \!\!\beta_m(x,x+\Delta,r)
\beta_n(x,x+\Delta,s)
N^{-2}\! \sum_{i,j,k=1}^N A_N^m(i,j)A_N^n(k,j)R_j^N(r \land s) drds.
\end{align*}

{\it Step 4.} Here we verify that $|J|\leq C x^{-2q} N^{-1}$ on $\Omega_N^1$.
Using that $|\int_0^{x+\Delta}\beta_m(x,x+\Delta,r)dr| \leq C n^q \Lambda^n x^{-q}$ by
Lemma \ref{conv}-(ii) and that $X_j^N$ is bounded by Step 2 (and does not depend on time),
\begin{align*}
|J| \leq &C \sum_{m,n \geq 0} m^q n^q\Lambda^{m+n} x^{-2q} N^{-2} \sum_{i,j,k=1}^N 
A_N^m(i,j)A_N^n(k,j) \\
\leq & C x^{-2q} N^{-1}  \sum_{m,n \geq 0} m^q n^q\Lambda^{m+n}|||A_N|||_1^{m+n}.
\end{align*}
The conclusion follows, since $\Lambda |||A_N|||_1 \leq a <1$ on $\Omega_N^1$.

\vip

{\it Step 5.} We next check that $|K| \leq C x \Delta^{-q} N^{-1}$ on $\Omega_N^1$.
Using the bound on $R^N_j$ (see Step 2), we start from 
\begin{align*}
|K| \leq &C \sum_{m,n\geq 0} \int_0^{x+\Delta}\!\int_0^{x+\Delta} \!\!|\beta_m(x,x+\Delta,r)|
|\beta_n(x,x+\Delta,s)| N^{-1} |||A_N|||_1^{m+n} [(r\land s)^{1-q} \land 1]  drds\\
\leq & C(K_1+K_2),
\end{align*}
where, using that $x-\Delta \geq x/2$ (whence $(r\land s)^{1-q}\leq C x^{1-q}$ if $r\land s \geq x-\Delta$) 
and a symmetry argument,
\begin{align*}
K_1 =& x^{1-q} \sum_{m,n\geq 0} \int_{x-\Delta}^{x+\Delta}\!\int_{x-\Delta}^{x+\Delta} \!\!|\beta_m(x,x+\Delta,r)|
|\beta_n(x,x+\Delta,s)| N^{-1} |||A_N|||_1^{m+n} drds,\\
K_2 =& \sum_{m,n\geq 0} \int_0^{x-\Delta}\!\int_0^{x+\Delta} \!\!|\beta_m(x,x+\Delta,r)|
|\beta_n(x,x+\Delta,s)| N^{-1} |||A_N|||_1^{m+n} drds.
\end{align*}
First, on $\Omega_N^1$,
$$
K_1 \leq C x^{1-q} \sum_{m,n\geq 0} \Lambda^{m+n} N^{-1} |||A_N|||_1^{m+n} \leq C N^{-1} x^{1-q} \leq  C x \Delta^{-q} N^{-1}
$$
since $x \geq \Delta$. Next, using that 
$\int_0^{x-\Delta}|\beta_m(x,x+\Delta,r)|dr \leq Cm^q \Lambda^m \Delta^{-q}$
by Lemma \ref{conv}-(ii)
and that $\int_0^{x+\Delta} |\beta_n(x,x+\Delta,s)| ds \leq 2 \Lambda^n$, still on $\Omega_N^1$,
$$
K_2 \leq C \Delta^{-q} \sum_{m,n\geq 0} m^q\Lambda^{m+n} N^{-1} |||A_N|||_1^{m+n} 
\leq C \Delta^{-q} N^{-1}\leq  C x \Delta^{-q} N^{-1},
$$
since $x\geq 1$ by assumption.

\vip

{\it Step 6.} Finally recall that $\gamma_{m,n}(x,x+\Delta)=\int_0^{x+\Delta} \int_0^{x+\Delta}
(s \land u) \beta_m(x,x+\Delta,s)\beta_n(x,x+\Delta,u)duds=
\Delta \Lambda^{m+n}-\kappa_{m,n}\Lambda^{m+n}+\e_{m,n}(x,x+\Delta)$ with the notation of Lemma \ref{conv}-(iii).
We thus may write 
$$
I= \mu \sum_{m,n\geq 0} \gamma_{m,n}(x,x+\Delta)
N^{-2} \sum_{i,j,k=1}^N A_N^m(i,j)A_N^n(k,j)\ell_N(j)=
I_1-I_2+I_3,
$$ 
where
\begin{align*}
I_1=&\mu \Delta \sum_{m,n\geq 0} \Lambda^{m+n} N^{-2} \sum_{i,j,k=1}^N A_N^m(i,j)A_N^n(k,j)\ell_N(j),\\
I_2=&\mu \sum_{m,n\geq 0} \kappa_{m,n}\Lambda^{m+n} N^{-2} \sum_{i,j,k=1}^N A_N^m(i,j)A_N^n(k,j)\ell_N(j),\\
I_3=&\mu \sum_{m,n\geq 0} \e_{m,n}(x,x+\Delta) N^{-2} \sum_{i,j,k=1}^N A_N^m(i,j)A_N^n(k,j)\ell_N(j).
\end{align*}
First, we clearly have
$$
I_1= \mu \Delta N^{-2} \sum_{i,j,k=1}^N Q_N(i,j)Q_N(k,j)\ell_N(j)=
\mu \Delta N^{-2} \sum_{j=1}^N (c_N(j))^2\ell_N(j)=
\Delta N^{-1} \cW^N_{\infty,\infty}.
$$
We next simply set $\cX^N=I_2$, which is clearly
$\sigma((\theta_{i,j})_{i,j=1,\dots,N})$-measurable and well-defined on $\Omega_N^1$. 
Finally, since $\e_{m,n}(x,x+\Delta)
\leq C(m+n)^q \Lambda^{m+n} x \Delta^{-q}$ by Lemma \ref{conv}-(iii), since $\ell_N$ is bounded on
$\Omega_N^1$ and since, as already seen, $\sum_{i,j,k=1}^N A_N^m(i,j)A_N^n(k,j) \leq N |||A_N|||_1^{m+n}$,
$$
|I_3| \leq C  x \Delta^{-q} N^{-1}  \sum_{m,n\geq 0} (n+m)^q \Lambda^{m+n} |||A_N|||_1^{m+n}
\leq C x \Delta^{-q} N^{-1}.
$$
All this implies that 
$|I- \Delta N^{-1} \cW^N_{\infty,\infty} + \cX^N|\leq C x \Delta^{-q} N^{-1}$.
Since $V^N_{x,\Delta} = I-J+K$ by Step 3 and since we have seen in Steps 4 and 5 that
$|J|\leq C x^{-2q} N^{-1}\leq C x \Delta^{-q} N^{-1}$
and $|K| \leq C x \Delta^{-q} N^{-1}$, we conclude that, on $\Omega_N^1$,
$|V^N_{x,\Delta}- \Delta N^{-1} \cW^N_{\infty,\infty} + \cX^N|\leq C x \Delta^{-q} N^{-1}$
as desired.
\end{proof}

We can now study the term $D^{N,4}_{\Delta,t}$.

\begin{lem}\label{esti3l4}
Assume $H(q)$ for some $q\geq 1$. Then a.s. on $\Omega_N^1$, for $1 \leq \Delta \leq t/4$,
$D^{N,4}_{\Delta,t} \leq C t \Delta^{-1-q}$.
\end{lem}

\begin{proof}
We clearly have
\begin{align*}
D^{N,4}_{\Delta,t}= \Big| \frac {2N} t\sum_{a=t/(2\Delta)+1}^{t/\Delta} 
\Vart (\bar U^N_{2a\Delta}-\bar U^N_{2(a-1)\Delta})
-  \frac {N} t\sum_{a=t/\Delta+1}^{2t/\Delta} 
\Vart (\bar U^N_{a\Delta}-\bar U^N_{(a-1)\Delta})
- \cW^N_{\infty,\infty}\Big|.
\end{align*}
Using Lemma \ref{esti3l3}, (observe that for $a \in \{t/(2\Delta)+1,\dots,t/\Delta\}$,
$x= 2(a-1)\Delta \geq t$ satisfies $2\Delta \leq x/2$ and, 
for $a \in \{t/\Delta+1,\dots,2t/\Delta\}$,
$x= (a-1)\Delta \geq t$ satisfies $\Delta \leq x/2$), we get
\begin{align*}
D^{N,4}_{\Delta,t}=& \Big| \frac {2N} t\sum_{a=t/(2\Delta)+1}^{t/\Delta} 
\Big[\frac{2\Delta}N \cW^N_{\infty,\infty} -\cX^N+r_N(2(a-1)\Delta,2\Delta) \Big]\\
&\hskip2cm-  \frac {N} t\sum_{a=t/\Delta+1}^{2t/\Delta} 
\Big[\frac{\Delta}N \cW^N_{\infty,\infty} -\cX^N+r_N((a-1)\Delta,\Delta) \Big]
- \cW^N_{\infty,\infty}\Big|.
\end{align*}
This rewrites
$$
D^{N,4}_{\Delta,t} = \Big| \frac {2N} t\sum_{a=t/(2\Delta)+1}^{t/\Delta} r_N(2(a-1)\Delta,2\Delta)
-  \frac {N} t\sum_{a=t/\Delta+1}^{2t/\Delta} r_N((a-1)\Delta,\Delta) \Big|.
$$
Since $r_N(x,\Delta)\leq C x \Delta^{-q}N^{-1}$,
we find that
$D^{N,4}_{\Delta,t} \leq C (N/t)(t/\Delta)(t \Delta^{-q}N^{-1}) = C t \Delta^{-1-q}$.
\end{proof}

The following tedious lemma will allow us to treat the last term $D^{N,3}_{\Delta,t}$.

\begin{lem}\label{esti3l5}
Assume $H(q)$ for some $q\geq 1$. On $\Omega_N^1$, for all $t,x, \Delta \geq 1$
with $t/2 \leq x-\Delta \leq x+\Delta \leq 2t$,
$$
\Vart ( (\bar U^N_{x+\Delta}-\bar U^N_x)^2) \leq C \Big(\frac {\Delta^2} {N^2} + \frac{t^2}
{N^2\Delta^{4q}}\Big) 
$$
and, if $t/2\leq y-\Delta \leq y+\Delta \leq x-2\Delta \leq x+\Delta \leq 2t$,
$$
\Covt ((\bar U^N_{x+\Delta}-\bar U^N_x)^2,(\bar U^N_{y+\Delta}-\bar U^N_y)^2) \leq
C \Big( \frac{t^{1/2}}{N\Delta^{q-1}} + \frac{t^2}{N^2\Delta^{4q}}+
\frac{t^{1/2}}{N^2\Delta^{q-3/2}} \Big).
$$
\end{lem}

\begin{proof}
We divide the proof in several steps. We work on $\Omega_N^1$.
\vip
{\it Step 1.} For $i=1,\dots,N$ and $z \in [x,x+\Delta]$, we can write, recalling \eqref{fundc} 
and that $\beta_n(x,z,r)=\varphi^{\star n}(z-r)-\varphi^{\star n}(x-r)$,
$$
U^{i,N}_z-U^{i,N}_x = \sum_{n\geq 0} \int_0^z \beta_n(x,z,r) \sum_{j=1}^N A_N^n(i,j) M^{j,N}_rdr
=\Gamma^{i,N}_{x,z} + X^{i,N}_{x,z},
$$
where
\begin{align*}
\Gamma^{i,N}_{x,z}=&  \sum_{n\geq 0} \int_{x-\Delta}^z \beta_n(x,z,r) \sum_{j=1}^N A_N^n(i,j) 
(M^{j,N}_r  -M^{j,N}_{x-\Delta})dr,\\
X^{i,N}_{x,z}=&  \sum_{n\geq 0} \Big(\int_{x-\Delta}^z \beta_n(x,z,r) dr \Big)\sum_{j=1}^N A_N^n(i,j) 
M^{j,N}_{x-\Delta}  + \sum_{n\geq 0} \int_0^{x-\Delta} \beta_n(x,z,r) \sum_{j=1}^N A_N^n(i,j) M^{j,N}_rdr,
\end{align*}
and we set as usual $\bar \Gamma^N_{x,z}= N^{-1}\sum_{i=1}^N \Gamma^{i,N}_{x,z}$
and $\bar X^N_{x,z}= N^{-1}\sum_{i=1}^N X^{i,N}_{x,z}$.

\vip

{\it Step 2.} We now show that, on $\Omega_N^1$, for $z \in [x,x+\Delta]$,
$$
\sup_{i=1,\dots,N} \Et[(X^{i,N}_{x,z})^4] \leq C t^2 \Delta^{-4q}
\quad \hbox{and}\quad  \Et[(\bar X^{N}_{x,z})^4] \leq C t^2 N^{-2}\Delta^{-4q}.
$$ 
Using that $|\int_{x-\Delta}^z \beta_n(x,z,r)dr| +  \int_0^{x-\Delta} |\beta_n(x,z,r) | dr 
\leq C n^q \Lambda^n \Delta^{-q}$ by Lemma \ref{conv}-(ii)
$$
|X^{i,N}_{x,z}| \leq C \sum_{n \geq 0} n^q \Lambda^n \Delta^{-q} \sum_{j=1}^N A_N^n(i,j) \sup_{[0,2t]} |M^{j,N}_r|. 
$$
But we now from Lemma \ref{if}-(iii) that $\sup_{j=1,\dots,N}\Et[\sup_{[0,2t]} |M^{j,N}_r|^4] \leq C t^2$.
We thus deduce from the Minkowski inequality that, still on $\Omega_N^1$,
$$
\Et[(X^{i,N}_{x,z})^4]^{1/4}\leq Ct^{1/2}\Delta^{-q} \sum_{n \geq 0} n^q \Lambda^n \sum_{j=1}^N A_N^n(i,j)
\leq C t^{1/2}\Delta^{-q} \sum_{n \geq 0} n^q \Lambda^n |||A_N|||_\infty^n \leq C t^{1/2}\Delta^{-q}.
$$
We next observe that
$$
\bar X^{N}_{x,z}=  \sum_{n\geq 0} \Big(\int_{x-\Delta}^z \beta_n(x,z,r) dr\Big) O^{N,n}_{x-\Delta}
+ \sum_{n\geq 0} \int_0^{x-\Delta} \beta_n(x,z,r) O^{N,n}_r dr,
$$
where the martingale
$$
O^{N,n}_r= N^{-1}\sum_{i,j=1}^N A_N^n(i,j) M^{j,N}_r
$$
has for quadratic variation $[O^{N,n},O^{N,n}]_r=N^{-2} \sum_{j=1}^N (\sum_{i=1}^N A_N^n(i,j))^2 
Z^{j,N}_r \leq N^{-1} |||A_N|||_1^{2n} \bar Z^N_r$ by Remark \ref{mo}. By Lemma \ref{if}-(iii),
we conclude that, on $\Omega_N^1$, 
$$
\Et\Big[\sup_{[0,2t]} (O^{N,n}_r)^4\Big] \leq C N^{-2} |||A_N|||_1^{4n}\Et[(\bar Z^N_{2t})^2]
\leq C N^{-2} |||A_N|||_1^{4n} t^2. 
$$
Using again that $|\int_{x-\Delta}^z \beta_n(x,z,r)dr| +  \int_0^{x-\Delta} |\beta_n(x,z,r) | dr 
\leq C n^q \Lambda^n \Delta^{-q}$ by Lemma \ref{conv}-(ii),
$$
|\bar X^{N}_{x,z}| \leq C   \sum_{n\geq 0} n^q \Lambda^n \Delta^{-q} \sup_{[0,2t]} |O^{N,n}_r|.
$$
Thus, we infer from the Minkowski inequality that, still on $\Omega_N^1$,
$$
\E[(\bar X^{N}_{x,z})^4]^{1/4} \leq C \sum_{n\geq 0} n^q \Lambda^n \Delta^{-q} 
N^{-1/2} |||A_N|||_1^{n} t^{1/2} \leq C \Delta^{-q}N^{-1/2}t^{1/2}.
$$

{\it Step 3.} We next check that $\Et[(\bar \Gamma^N_{x,z})^4]\leq C \Delta^2 N^{-2}$
for any $z\in[x,x+\Delta]$, on $\Omega_1^N$. Using the same martingale $O^{N,n}$ as in Step 2,
$$
\bar \Gamma^N_{x,z} = \sum_{n\geq 0} \int_{x-\Delta}^z \beta_n(x,z,r) 
[O^{N,n}_r-O^{N,n}_{x-\Delta}]dr.
$$
Recalling that $[O^{N,n},O^{N,n}]_r= N^{-2} \sum_{j=1}^N (\sum_{i=1}^N A_N^n(i,j))^2 
Z^{j,N}_r$ with $\sum_{i=1}^N A_N^n(i,j) \leq |||A_N|||_1^n$,
\begin{align*}
\Et\Big[\sup_{[x-\Delta,z]}(O^{N,n}_r-O^{N,n}_{x-\Delta})^4\Big] \leq& 
C N^{-4} |||A_N|||_1^{4n} \Et\Big[\Big(\sum_{j=1}^N (Z^{j,N}_z-Z^{j,N}_{x-\Delta})\Big)^2 \Big]\\
=&C N^{-2} |||A_N|||_1^{4n} \Et\Big[\Big(\bar Z^{N}_z- \bar Z^{N}_{x-\Delta})\Big)^2 \Big].
\end{align*}
We conclude from Lemma \ref{if}-(iii) that (recall that $z\in [x,x+\Delta]$)
$$
\Et\Big[\sup_{[x-\Delta,z]}(O^{N,n}_r-O^{N,n}_{x-\Delta})^4\Big] \leq C \Delta^2 N^{-2} |||A_N|||_1^{4n}.
$$
Using that $\int_{x-\Delta}^z |\beta_n(x,z,r)|dr \leq 2\Lambda^n$ and the Minkowski inequality,
$$
\E[(\bar \Gamma^N_{x,z})^4 ]^{1/4} \leq C \sum_{n\geq 0} \Lambda^n \Delta^{1/2} N^{-1/2} 
|||A_N|||_1^{n} \leq C \Delta^{1/2} N^{-1/2}.
$$

{\it Step 4.} Recalling Step 1, $(\bar U^N_{x+\Delta}-\bar U^N_{x})^4=(\bar \Gamma^N_{x,x+\Delta}
+ \bar X^N_{x,x+\Delta})^4 \leq 8(\bar \Gamma^N_{x,x+\Delta})^4 +8(\bar X^N_{x,x+\Delta})^4$.
We deduce from Steps 2 and 3 that
$\Vart((\bar U^N_{x+\Delta}-\bar U^N_{x})^2) \leq C(\Delta^2 N^{-2} + t^2 N^{-2}\Delta^{-4q})$.

\vip

{\it Step 5.} Here we show that 
\begin{align*}
\Big|\Covt((\bar U^N_{x+\Delta} -\bar U_x^N)^2,(\bar U^N_{y+\Delta} -\bar U_y^N)^2)\Big|
\leq \Big|\Covt ((\bar \Gamma^N_{x,x+\Delta})^2,(\bar \Gamma^N_{y,y+\Delta})^2)\Big|
+ \frac C{N^2}\Big(\frac{t^2}{\Delta^{4q}}+ \frac{t^{1/2}}{\Delta^{q-3/2}}\Big).
\end{align*}
It suffices to write that 
$(\bar U^N_{x+\Delta} -\bar U_x^N)^2= (\bar \Gamma^N_{x,x+\Delta})^2 + (\bar X^N_{x,x+\Delta})^2
+ 2 \bar \Gamma^N_{x,x+\Delta} \bar X^N_{x,x+\Delta}$, the same formula
with $y$ instead of $x$, and to use the bilinearity of the covariance: we have the term
$\Covt ((\bar \Gamma^N_{x,x+\Delta})^2,(\bar \Gamma^N_{y,y+\Delta})^2)$, and the other ones
are bounded by
\begin{align*}
&\Et\Big[(\bar \Gamma^N_{x,x+\Delta})^2(\bar X^N_{y,y+\Delta})^2 + 2(\bar \Gamma^N_{x,x+\Delta})^2
|\bar \Gamma^N_{y,y+\Delta} \bar X^N_{y,y+\Delta}| + (\bar X^N_{x,x+\Delta})^2(\bar \Gamma^N_{y,y+\Delta})^2\\
&\hskip2cm+ (\bar X^N_{x,x+\Delta})^2(\bar X^N_{y,y+\Delta})^2+ 2(\bar X^N_{x,x+\Delta})^2
|\bar \Gamma^N_{y,y+\Delta} \bar X^N_{y,y+\Delta}|+2 |\bar \Gamma^N_{x,x+\Delta} \bar X^N_{x,x+\Delta}|
(\bar \Gamma^N_{y,y+\Delta})^2\\
&\hskip4cm+2 |\bar \Gamma^N_{x,x+\Delta} \bar X^N_{x,x+\Delta}|(\bar X^N_{y,y+\Delta})^2
+4 |\bar \Gamma^N_{x,x+\Delta} \bar X^N_{x,x+\Delta} \bar \Gamma^N_{y,y+\Delta} \bar X^N_{y,y+\Delta}|\Big].
\end{align*}
We bound all these terms, using only the H\"older inequality and recalling that
$\E[(\bar \Gamma^N_{x,x+z})^4 ]\leq C \Delta^2 N^{-2}$ and $\E[(\bar X^N_{x,x+z})^4 ] \leq C t^2 N^{-2}\Delta^{-4q}$ and 
that the same bounds hold with $y$ instead of $x$. 
We finally remove a few terms using the inequality
$a+a^{3/4}b^{1/4}+a^{1/2}b^{1/2}+ a^{1/4}b^{3/4} \leq 4 (a+a^{1/4}b^{3/4})$
with $a=t^2 N^{-2}\Delta^{-4q}$ and $b=\Delta^2N^{-2}$.

\vip

{\it Step 6.} Recall that $y+\Delta \leq x-2\Delta$. We check here that for any
$r,s \in [x-\Delta,x+\Delta]$, any $u,v \in [y-\Delta,y+\Delta]$, 
any $i,j,k,l \in \{1,\dots,N\}$, 
$$
\Big|\Covt\Big((M^{i,N}_r-M^{i,N}_{x-\Delta})(M^{j,N}_s-M^{j,N}_{x-\Delta}),
(M^{k,N}_u-M^{k,N}_{y-\Delta})(M^{l,N}_v-M^{l,N}_{y-\Delta})\Big)\Big| 
\leq C \indiq_{\{i=j\}} t^{1/2}\Delta^{1-q}.
$$
First, $i\ne j$ implies that the covariance vanishes, since 
$\Et[(M^{i,N}_r-M^{i,N}_{x-\Delta})(M^{j,N}_s-M^{j,N}_{x-\Delta}) \vert \cF_{x-\Delta}]=0$
and since $u,v \leq y+\Delta \leq x-\Delta$.
We next assume that $i=j$ and w.l.o.g. that $r\leq s$. Conditioning with respect to $\cF_r$, we easily find,
since $u, v \leq x-\Delta\leq r$,
\begin{align*}
K:=&\Covt\Big((M^{i,N}_r-M^{i,N}_{x-\Delta})(M^{i,N}_s-M^{i,N}_{x-\Delta}),
(M^{k,N}_u-M^{k,N}_{y-\Delta})(M^{l,N}_v-M^{l,N}_{y-\Delta})\Big)\\
=& \Covt\Big((M^{i,N}_r-M^{i,N}_{x-\Delta})^2,
(M^{k,N}_u-M^{k,N}_{y-\Delta})(M^{l,N}_v-M^{l,N}_{y-\Delta})\Big).
\end{align*}
We write as usual $(M^{i,N}_r-M^{i,N}_{x-\Delta})^2=2\int_{x-\Delta}^r M^{i,N}_{\tau-}
dM^{i,N}_\tau +  Z^{i,N}_r-Z^{i,N}_{x-\Delta}$, because $[M^{i,N},M^{i,N}]_\tau=Z^{i,N}_\tau$ by Remark \ref{mo}.
Since $\E[\int_{x-\Delta}^r M^{i,N}_{\tau-}
dM^{i,N}_\tau  \vert \cF_{x-\Delta}]=0$ and since $u,v \leq x-\Delta$,
we find that
\begin{align*}
K=& \Covt\Big(Z^{i,N}_r-Z^{i,N}_{x-\Delta},
(M^{k,N}_u-M^{k,N}_{y-\Delta})(M^{l,N}_v-M^{l,N}_{y-\Delta})\Big) \\
=& \Covt\Big(U^{i,N}_r-U^{i,N}_{x-\Delta},
(M^{k,N}_u-M^{k,N}_{y-\Delta})(M^{l,N}_v-M^{l,N}_{y-\Delta})\Big) \\
=& \Covt \Big(\Gamma^{i,N}_{x-\Delta,r}+ X^{i,N}_{x-\Delta,r},
(M^{k,N}_u-M^{k,N}_{y-\Delta})(M^{l,N}_v-M^{l,N}_{y-\Delta})\Big)
\end{align*}
with the notation of Step 1. But $\Gamma^{i,N}_{x-\Delta,r}$ involves
only increments of martingales of the form $M^{j,N}_\tau-M^{j,N}_{x-2\Delta}$,
of which the conditional expectation knowing $\cF_{x-2\Delta}$ vanishes.
Since now $u,v \leq y+\Delta \leq x-2\Delta$, we deduce that
\begin{align*}
K=& \Covt \Big(X^{i,N}_{x-\Delta,r},
(M^{k,N}_u-M^{k,N}_{y-\Delta})(M^{l,N}_v-M^{l,N}_{y-\Delta})\Big),
\end{align*}
whence
$$
|K| \leq \Et[(X^{i,N}_{x-\Delta,r})^2]^{1/2} \E[(M^{k,N}_u-M^{k,N}_{y-\Delta})^4]^{1/4}
\E[(M^{l,N}_v-M^{l,N}_{y-\Delta})^4]^{1/4}.
$$
Using Step 2, Lemma \ref{if}-(iii) and that $u-(y-\Delta)\leq 2 \Delta$
and $v-(y-\Delta)\leq 2 \Delta$, we easily conclude that indeed,
$|K| \leq C t^{1/2} \Delta^{-q} \Delta$.

\vip

{\it Step 7.} We now show, recalling that $y+\Delta \leq x-2\Delta$, that
$$
\Big|\Covt ((\bar \Gamma^N_{x,x+\Delta})^2,(\bar \Gamma^N_{y,y+\Delta})^2)\Big| 
\leq C N^{-1}t^{1/2} \Delta^{1-q}.
$$
We denote by $|I|$ the left hand side and we start from 
$$
\bar \Gamma^N_{x,x+\Delta}= \sum_{n\geq 0} \int_{x-\Delta}^{x+\Delta} \beta_n(x,x+\Delta,r) N^{-1}\sum_{i,j=1}^N
A_N^n(i,j)(M^{j,N}_r-M^{j,N}_{x-\Delta}) dr,
$$
whence 
\begin{align*}
I=&\!\!\!\sum_{m,n,a,b \geq 0} \int_{x-\Delta}^{x+\Delta} \int_{x-\Delta}^{x+\Delta} 
\int_{y-\Delta}^{y+\Delta} \int_{y-\Delta}^{y+\Delta} \beta_m(x,x+\Delta,r)
\beta_n(x,x+\Delta,s)\beta_a(y,y+\Delta,u)\beta_b(y,y+\Delta,v)\\
& \hskip0.5cm N^{-4} \sum_{i,j,k,l=1}^N \sum_{\alpha,\delta,\gamma,\zeta=1}^N A_N^m(i,j)A_N^n(k,l)
A_N^a(\alpha,\delta)A_N^b(\gamma,\zeta) \\
& \hskip1cm \Covt\Big((M^{j,N}_r-M^{j,N}_{x-\Delta})(M^{l,N}_s-M^{l,N}_{x-\Delta}),
(M^{\delta,N}_u-M^{\delta,N}_{y-\Delta})(M^{\zeta,N}_v-M^{\zeta,N}_{y-\Delta})\Big)dvdudsdr.
\end{align*}
Using that $\int_{x-\Delta}^{x+\Delta} |\beta_m(x,x+\Delta,r)|dr \leq 2\Lambda^m$ (and the same formula for the three
other integrals),
Step 6 and that $\sum_{i=1}^N A_N^m(i,j) \leq |||A_N|||_1^m$ (and the same formula for the sums
in $k,\alpha,\gamma$), we find that, still on $\Omega_N^1$,
\begin{align*}
|I| \leq C \sum_{m,n,a,b \geq 0} \Lambda^{m+n+a+b} |||A_N|||_1^{m+n+a+b} 
N^{-4}  \sum_{j,l,\delta,\zeta=1}^N t^{1/2} \Delta^{1-q} \indiq_{\{j=l\}}
\leq C N^{-1}t^{1/2}\Delta^{1-q}.
\end{align*}

{\it Step 8.} Gathering Steps 5 and 7, we find that
$$
\Big|\Covt((\bar U_{x+\Delta} -\bar U_x^N)^2,(\bar U_{y+\Delta} -\bar U_y^N)^2)\Big|
\leq C (N^{-1}t^{1/2}\Delta^{1-q}+ N^{-2}t^2 \Delta^{-4q}+N^{-2}t^{1/2}\Delta^{3/2-q}),
$$
which completes the proof.
\end{proof}

We can finally treat the last term.

\begin{lem}\label{esti3l6}
Assume $H(q)$ for some $q \geq 1$. On $\Omega_N^1$,
for all $1\leq \Delta \leq t/2$,
$$
\Et[(D^{N,3}_{\Delta,t})^2] \leq C \Big(\frac \Delta t + \frac{t}{\Delta^{4q+1}}  
+ \frac {N t^{1/2}}{\Delta^{q+1}}+ \frac{t^2}{\Delta^{4q+2}}+\frac{t^{1/2}}{\Delta^{q+1/2}}
\Big).
$$
\end{lem}

\begin{proof}
First note that by definition of $D^{N,3}_{\Delta,t}$ and since $\bar U^N_r=\bar Z^N_r-\Et[\bar Z^N_r]$,
\begin{align*}
\Et[(D^{N,3}_{\Delta,t})^2]=&\frac{N^2}{t^2} \Vart \Big(\sum_{a=t/\Delta+1}^{2t/\Delta}
(\bar U^N_{a\Delta}-\bar U^N_{(a-1)\Delta})^2 \Big) = \frac{N^2}{t^2} \sum_{a,b=t/\Delta+1}^{2t/\Delta}
K_{a,b},
\end{align*}
where $K_{a,b}=\Covt((\bar U^N_{a\Delta}-\bar U^N_{(a-1)\Delta})^2,
(\bar U^N_{b\Delta}-\bar U^N_{(b-1)\Delta})^2)$.
If $|a-b|\leq 2$, we only use that
\begin{align*}
|K_{a,b}|\leq \Big( 
\Vart \Big((\bar U^N_{a\Delta}-\bar U^N_{(a-1)\Delta})^2\Big)
\Vart \Big((\bar U^N_{b\Delta}-\bar U^N_{(b-1)\Delta})^2\Big)
\Big)^{1/2}
\leq C \Big(\frac{\Delta^2}{N^2} + \frac{t^2}{N^2\Delta^{4q}} \Big).
\end{align*}
We finally used the first estimate of Lemma \ref{esti3l5},
which is valid since $x=(a-1)\Delta$ satisfies $x\geq t$ and thus $t/2\leq x-\Delta \leq x+\Delta \leq 2t$
and $x=(b-1)\Delta$ satisfies the same conditions.
If now $|a-b| \geq 3$ and w.l.o.g. $a>b$, we use the second estimate
of Lemma \ref{esti3l5}, which is valid since $x=(a-1)\Delta$ and $y=(b-1)\Delta$
satisfy the required conditions (in particular, $y+\Delta \leq x-2\Delta$).
This gives
\begin{align*}
|K_{a,b}| \leq C \Big( \frac{t^{1/2}}{N\Delta^{q-1}} + \frac{t^2}{N^2\Delta^{4q}}+
\frac{t^{1/2}}{N^2\Delta^{q-3/2}} \Big).
\end{align*}
We end with
\begin{align*}
\Et[(D^{N,3}_{\Delta,t})^2] \leq C \frac{N^2}{t^2}\frac t\Delta 
\Big(\frac{\Delta^2}{N^2} + \frac{t^2}{N^2\Delta^{4q}} \Big)
+  C \frac{N^2}{t^2}\frac {t^2}{\Delta^2} 
\Big( \frac{t^{1/2}}{N\Delta^{q-1}} + \frac{t^2}{N^2\Delta^{4q}}+
\frac{t^{1/2}}{N^2\Delta^{q-3/2}} \Big). 
\end{align*}
The conclusion follows.
\end{proof}

We can at last give the

\begin{proof}[Proof of Proposition \ref{esti3}]
Gathering Lemmas \ref{esti3l1}, \ref{esti3l2}, \ref{esti3l4} and \ref{esti3l6},
we see that, on $\Omega_N^1$, if $1\leq \Delta \leq t/4$,
\begin{align*}
\Et[|\cW^N_{\Delta,t}-\cW^N_{\infty,\infty}|] \leq& \Et[D^{N,1}_{\Delta,t}+2 D^{N,1}_{2\Delta,t}
+ D^{N,2}_{\Delta,t}+2 D^{N,2}_{2\Delta,t}+ D^{N,3}_{\Delta,t}+2 D^{N,3}_{2\Delta,t}+D^{N,4}_{\Delta,t}] \\
\leq& C\Big(\frac{\Delta}{t}+ \frac{N\Delta}{t^{2q}} 
+ \frac{N}{t^{q-1}} + \frac{t}{\Delta^{q+1}}
+ \sqrt{\frac \Delta t \!+\! \frac{t}{\Delta^{4q+1}}  
\!+\! \frac {N t^{1/2}}{\Delta^{q+1}}\!+\! \frac{t^2}{\Delta^{4q+2}}\!+\!\frac{t^{1/2}}{\Delta^{q+1/2}}}
\Big).
\end{align*}
Using that $q \geq 3$ (whence in particular $2q-1 \geq q-1 \geq (q+1)/2$) 
and that $1\leq \Delta \leq t$,
we easily deduce that $\Delta/t \leq (\Delta/t)^{1/2}$, that $N\Delta t^{-2q} \leq N \Delta^{1-2q}
\leq N \Delta^{-(q+1)/2}$, that $N t^{1-q} \leq N \Delta^{1-q} \leq N \Delta^{-(q+1)/2}$,
that $t \Delta^{-q-1} \leq t \Delta^{-q/2-1}$, that $t^{1/2}\Delta^{-2q-1/2}\leq t\Delta^{-2q-1}  \leq t \Delta^{-q/2-1}$,
that $N^{1/2}t^{1/4}\Delta^{-(q+1)/2}\leq N \Delta^{-(q+1)/2} + t^{1/2}\Delta^{-(q+1)/2} \leq  N \Delta^{-(q+1)/2} 
+ t \Delta^{-q/2-1}$, that $t \Delta^{-2q-1} \leq  t \Delta^{-q/2-1}$ and that 
$t^{1/4}\Delta^{-q/2-1/4} \leq t \Delta^{-q/2-1}$. This gives, still on $\Omega_N^1$,
\begin{align*}
\Et[|\cW^N_{\Delta,t}-\cW^N_{\infty,\infty}|] \leq C \Big(\sqrt{\frac \Delta t} + \frac N {\Delta^{(q+1)/2}}
+ \frac {t}{\Delta^{q/2+1}} \Big)
\end{align*}
as desired.
\end{proof}

\subsection{Conclusion}\label{sconc}

We now have all the weapons to check our main result.

\begin{proof}[Proof of Theorem \ref{mr}] Recall that we assume $H(q)$ for some $q>3$ and that
$\Delta_t= t/(2 \lfloor t^{1-4/(q+1)}\rfloor) \sim t^{4/(q+1)}/2$ (for $t$ large).
We can of course assume that $t\geq 4$ is large enough so that $\Delta_t\in [1,t/4]$,
because else the inequalities of the statement are trivial.
Using Propositions \ref{interro} and \ref{esti1}, we find
$$
\E\Big[\indiq_{\Omega_N^1}\Big|\cE^N_t- \frac{\mu}{1-\Lambda p}\Big| \Big]
\leq \E\Big[\indiq_{\Omega_N^1}\Big|\cE^N_t- \mu \bar\ell_N\Big| \Big]+
\mu \E\Big[\indiq_{\Omega_N^1}\Big|\bar \ell_N - \frac{1}{1-\Lambda p}\Big| \Big]
\leq C \Big(\frac 1N + \frac{1}{\sqrt {Nt}} + \frac 1 {t^q}\Big) .
$$
Since now $\Pr((\Omega_N^1)^c) \leq C e^{-c N}$ by Lemma \ref{pom}, we conclude that for any $\e\in(0,1)$,
$$
\Pr\Big(\Big|\cE^N_t- \frac{\mu}{1-\Lambda p}\Big|\geq \e \Big)
\leq Ce^{-cN} + \frac C \e \Big(\frac 1N + \frac{1}{\sqrt {Nt}}+\frac 1 {t^q}\Big) 
\leq \frac C \e \Big(\frac 1N + \frac{1}{\sqrt {Nt}}+\frac 1 {t^q}\Big).
$$

Similarly, Propositions \ref{interro} and \ref{esti2} imply, since 
$\cV^N_\infty=\mu^2\sum_{i=1}^N|\ell_N(i)-\bar \ell_N|^2$,
\begin{align*}
\E\Big[\indiq_{\Omega_N^1}\Big|\cV^N_t- \frac{\mu^2\Lambda^2 p(1-p)}{(1-\Lambda p)^2}\Big| \Big]
\leq & \E\Big[\indiq_{\Omega_N^1}\Big|\cV^N_t- \cV^N_\infty\Big| \Big] 
+ \mu^2 \E\Big[\indiq_{\Omega_N^1}\Big|\sum_{i=1}^N|\ell_N(i)-\bar \ell_N|^2 
- \frac{\Lambda^2 p(1-p)}{(1-\Lambda p)^2}\Big| \Big] \\
\leq& \frac C {\sqrt N} + C \E\Big[  \indiq_{\Omega_N^1}\Big(1+\sum_{i=1}^N \Big[\ell_N(i)-\bar\ell_N\Big]^2
\Big)^{1/2}\Big]
\Big(\frac N {t^{q}} + \frac {\sqrt N}t + \frac 1 {\sqrt t} \Big) \\
\leq& C\Big( \frac 1{\sqrt N}+ \frac N {t^{q}} + \frac {\sqrt N}t + \frac 1 {\sqrt t} \Big).
\end{align*}
The last inequality uses a second time Proposition \ref{interro}.
We conclude, using 
Lemma \ref{pom} as previously, that for any $\e\in(0,1)$,
$$
\Pr\Big(\Big|\cV^N_t- \frac{\mu^2\Lambda^2 p(1-p)}{(1-\Lambda p)^2}\Big|\geq \e \Big)
\leq Ce^{-cN}+\frac C \e \Big( \frac 1{\sqrt N}+ \frac N {t^{q}} + \frac {\sqrt N}t + \frac 1 {\sqrt t} \Big)
\leq \frac C \e \Big( \frac 1{\sqrt N}+ \frac N {t^{q}} + \frac {\sqrt N}t\Big)
$$
because $t^{-1/2}=(N^{1/4}t^{-1/2})N^{-1/4} \leq N^{1/2}t^{-1}+N^{-1/2}$.
This implies that
$$
\Pr\Big(\Big|\cV^N_t- \frac{\mu^2\Lambda^2 p(1-p)}{(1-\Lambda p)^2}\Big|\geq \e \Big)
\leq \frac C \e \Big( \frac 1{\sqrt N}+ \frac {\sqrt N}t\Big).
$$
Indeed, either $\sqrt N > t$ and the inequality is trivial
or $\sqrt N \leq t$ and then $N t^{-q} \leq N t^{-2} \leq N^{1/2} t^{-1}$.

\vip

Finally, we infer from Propositions \ref{interro} and \ref{esti3},
since $\cW^N_{\infty,\infty}=\mu N^{-1}\sum_{i=1}^N\ell_N(i)(c_N(i))^2$, that
\begin{align*}
&\E\Big[\indiq_{\Omega_N^1}\Big|\cW^N_{\Delta_t,t}- \frac{\mu}{(1-\Lambda p)^3}\Big| \Big]\\
\leq& \E\Big[\indiq_{\Omega_N^1}\Big|\cW^N_{\Delta_t,t}- \cW^N_{\infty,\infty}\Big| \Big]
+\mu \E\Big[\indiq_{\Omega_N^1}\Big|N^{-1}\sum_{i=1}^N\ell_N(i)(c_N(i))^2- \frac{1}{(1-\Lambda p)^3}\Big| \Big]\\
\leq &C \Big(\frac 1N + \sqrt{\frac{\Delta_t}{t}} + \frac {N} {\Delta_t^{(q+1)/2}} + \frac t {\Delta_t^{q/2+1}}\Big), 
\end{align*}
whence as usual by Lemma \ref{pom}, for $\e \in (0,1)$,
\begin{align*}
\Pr\Big(\Big|\cW^N_{\Delta_t,t}- \frac{\mu}{(1-\Lambda p)^3}\Big| \geq \e\Big) \leq& Ce^{-cN}+\frac C \e
\Big(\frac 1N + \sqrt{\frac{\Delta_t}{t}} + \frac {N} {\Delta_t^{(q+1)/2}} + \frac t {\Delta_t^{q/2+1}}\Big)\\
\leq& \frac C \e
\Big(\frac 1N + \frac 1{\sqrt {t^{1-4/(q+1)}}}+ \frac N {t^2}\Big).
\end{align*}
We finally used that $\Delta_t \sim t^{4/(q+1)}/2$, which implies that 
$\sqrt{\Delta_t/t}\sim \sqrt{1/(2t^{1-4/(q+1)})}$, that $N/\Delta_t^{(q+1)/2} \sim 2^{(q+1)/2}N t^{-2}$ and
$t/\Delta_t^{q/2+1}\sim 2^{q/2+1} t^{-(q+3)/(q+1)}\leq 2^{q/2+1}/\sqrt{t^{1-4/(q+1)}}$.
\end{proof}

\begin{proof}[Proof of Corollary \ref{mc}] Recall that we assume $H(q)$ for some $q>3$.
We fix $\mu>0$, $\Lambda>0$ and $p\in (0,1]$ such that $\Lambda p \in (0,1)$. We define
$u= \mu/(1-\Lambda p)$, $v=\mu^2\Lambda^2 p(1-p)/(1-\Lambda p)^2$ and $w=\mu/(1-\Lambda p)^3$.
It holds that $(u,v,w) \in D$ (which would not be the case if $\Lambda p=0$) 
and $\Psi(u,v,w)=(\mu,\Lambda,p)$. Furthermore, $\Psi$ is obviously of class
$C^\infty$ on $D$, it is in particular locally Lipschitz continuous. As a consequence,
there is a constant $c \in (0,1)$ (depending on $\mu,\Lambda,p$) such that for any $N\geq 1$, any $t\geq 1$,
any $\e \in (0,1/c)$,
\begin{gather*}
\Pr\Big(\Big\|\Psi(\cE^N_t,\cV^N_t,\cW^N_{\Delta_t,t})-(\mu,\Lambda,p)\Big\|\geq \e \Big)
\leq \Pr\Big(|\cE^N_t-u|+|\cV^N_t-v|+|\cW^N_{\Delta_t,t}-w|\geq c \e \Big) \\
\leq  \frac C \e \Big(\frac 1 N + \frac1{\sqrt{Nt}}+ 
\frac 1 {t^{q}} +  \frac{\sqrt N}t+ \frac 1 {\sqrt N}
+\frac 1 N + \frac N {t^2} + \frac 1 {\sqrt{t^{1-4/(q+1)}}}\Big)
\end{gather*}
by Theorem \ref{mr}. 
Using next that $q>3$, that $t\geq 1$ and $N\geq 1$,
and that either $N^{1/2}t^{-1}\geq 1$ (whence the inequality below is trivial)
or $N^{1/2}t^{-1}<1$ (whence $Nt^{-2} \leq N^{1/2}t^{-1}$), we find
$$
\Pr\Big(\Big\|\Psi(\cE^N_t,\cV^N_t,\cW^N_{\Delta_t,t})-(\mu,\Lambda,p)\Big\|\geq \e \Big)
\leq  \frac C \e \Big(\frac 1 {\sqrt N}+ \frac{\sqrt N}t + \frac 1 {\sqrt{t^{1-4/(q+1)}}}\Big).
$$
Noting that 
$t^{-(1-4/(q+1))/2}=[N^{1/4}t^{-(1-4/(q+1))/2}]N^{-1/4} \leq N^{-1/2} + N^{1/2}t^{-(1-4/(q+1))}$
concludes the proof.
\end{proof}

We finally give the

\begin{proof}[Proof of Remark \ref{rksc}]
Lemma \ref{if}-(ii) with $r=1$ and $s=0$ tells us that on $\Omega_N^1$,
$|\Et[\bar Z^N_t]- \mu \bar \ell_N t|\leq C$. By Lemma \ref{esti1lem}, 
we know that $\Et[|\bar Z^N_t- \Et[\bar Z^N_t]|]=\Et[|\bar U^N_t|]\leq C (t/N)^{1/2}$,
still on $\Omega_N^1$ and, by Proposition \ref{interro},
$\E[\indiq_{\Omega_N^1}|\bar \ell_N - 1/(1-\Lambda p)|] \leq C N^{-1}$. We easily
deduce that 
$$
\E[\indiq_{\Omega_N^1}| \bar Z^N_t - \mu (1-\Lambda p)^{-1}t|] \leq C(1+(t/N)^{1/2}+t/N)\leq
C(1+t/N).
$$
Since $\Pr(\Omega_N^1)\geq 1-Ce^{-cN}$ by 
Lemma \ref{pom}, we find that for any $\e>0$,
$$
\Pr \Big(\Big| \frac{\bar Z^N_t}t- \frac \mu {1-\Lambda p}\Big|\geq \e \Big) \leq Ce^{-cN} + \frac{C}\e 
\Big( \frac 1t + \frac 1N\Big).
$$
The conclusion follows.
\end{proof}

\section{The supercritical case}\label{surc}

The goal of this section is to prove Theorem \ref{mr2}. In Subsection \ref{rm2}, we study precisely
the Perron-Frobenius eigenvalue and eigenvector of the matrix with nonnegative entries $A_N(i,j)=N^{-1}\theta_{ij}$.
In Subsection \ref{ana2}, we state and prove a few results on some series involving $\varphi^{\star n}$.
A few preliminary stochastic analysis is handled in Subsection \ref{sa2}.
We finally conclude the proof in Subsection \ref{con2}.

\subsection{Perron-Frobenius analysis of the random matrix $A_N$}\label{rm2}

We recall that the norms $|| \cdot ||_r$ on $\rr^N$ and $||| \cdot |||_r$ on $\cM_{N\times N}(\rr)$ 
were defined in Subsection \ref{smat}.
We denote by $(e_1,\dots,e_N)$ the canonical basis of $\rr^N$
and by $\bun=\sum_{i=1}^N e_i$ the vector will all entries equal to $1$.

\begin{nota} 
We consider the matrix $A_N(i,j)=N^{-1}\theta_{ij}$ and the event
$$
\Omega_N^2=\Big\{\frac1N\sum_{i,j=1}^N A_N(i,j) > \frac p2 \hbox{ and for all } i,j=1,\dots,N, \; 
|N A_N^2(i,j)-p^2|<\frac{p^2}{2N^{3/8}} \Big\}.
$$
\end{nota}

Actually, $3/8$ could be replaced by any other exponent in $[3/8,1/2)$.
We first show that $\Omega_N^2$ has a high probability.

\begin{lem}\label{pome}
Assume that $p\in (0,1]$. It holds that $\Pr(\Omega_N^2)\geq 1- Ce^{-cN^{1/4}}$.
\end{lem}

\begin{proof} We recall the Hoeffding inequality \cite{hoe} for a Binomial$(n,q)$-distributed random variable 
$X$: for all $x\geq 0$, it holds that $\Pr(|X-nq|\geq x)\leq 2\exp(-2x^2/n)$.

\vip

Since $N\sum_{i,j=1}^NA_N(i,j)=\sum_{i,j=1}^N\theta_{ij}\sim$ Binomial$(N^2,p)$,
$\Pr(N^{-1}\sum_{i,j=1}^N A_N(i,j) \leq p/2) \leq \Pr( |N\sum_{i,j=1}^NA_N(i,j)-N^2p|\geq N^2p/2)\leq
2\exp(-N^2p^2/2)$.

\vip

For each $i\ne j$, we write
$N^2A_N^2(i,j)=\sum_{k=1}^N \theta_{ik}\theta_{kj}=Z^N_{ij}+\theta_{ii}\theta_{ij}+\theta_{ij}\theta_{jj}$,
where $Z^N_{ij}$ follows a Binomial$(N-2,p^2)$ distribution. We thus have $|N^2A_N^2(i,j)-Z^N_{ij}|\leq 2$.
This obviously extends to the case where $i=j$.
Hence for any $i,j$,  
$|N A_N^2(i,j)-p^2|\geq p^2/(2N^{3/8})$ implies that $|Z^N_{ij}-(N-2)p^2|\geq p^2N^{5/8}/2 -4$ and thus,
if $N\geq (16/p^2)^{8/5}$, that $|Z^N_{ij}-(N-2)p^2|\geq p^2N^{5/8}/4$. By
the Hoeffding inequality, $\Pr(|N A_N^2(i,j)-p^2|\geq p^2/(2N^{3/8}) )\leq 2\exp(-p^4N^{5/4}/(8(N-2))
\leq 2 \exp (-p^4N^{1/4}/8)$.

\vip

All this shows that $\Pr((\Omega_N^2)^c)\leq 2\exp(-N^2p^2/2)+ 2N^2\exp (-p^4N^{1/4}/8)$ for all 
$N\geq (16/p^2)^{8/5}$.
The conclusion easily follows: we can find $0<c<C$ depending only on $p$ such that for all
$N\geq 1$, $\Pr((\Omega_N^2)^c) \leq C e^{-cN^{1/4}}$.
\end{proof}

Next, we apply the Perron-Frobenius theorem.

\begin{lem}\label{rhoV} Assume that $p\in(0,1]$. On $\Omega_N^2$, 
the spectral radius $\rho_N$ of $A_N$ is a simple eigenvalue
of $A_N$ and $\rho_N \in [p(1-1/(2N^{3/8})),p(1+1/(2N^{3/8}))]$. There is a unique eigenvector
$V_N \in (\rr_+)^N$ of $A_N$ for the eigenvalue $\rho_N$ such that $||V_N||_2=\sqrt N$. 
We also have $V_N(i)>0$ for all $i=1,\dots,N$.
\end{lem}

\begin{proof}
The matrix $A_N$ has nonnegative entries and is irreducible on $\Omega_N^2$ since
$A_N^2$ has positive entries. We thus infer from the Perron-Frobenius theorem 
that on $\Omega_N^2$, $\rho_N$ is a simple eigenvalue of $A_N$,
that there is a unique corresponding eigenvector $V_N$ with nonnegative entries
such that $||V_N||_2=\sqrt N$ and that $V_N(i)>0$ for all $i=1,\dots,N$.

\vip

Since $N A_N^2(i,j) \in [p^2(1-1/(2N^{3/8})),p^2(1+1/(2N^{3/8}))]$ for all $i,j=1,\dots,N$ on $\Omega_N^2$
we deduce from $\rho_N^2V_N=A_N^2V_N$ that
$\rho_N^2 ||V_N||_1 =\sum_{i,j=1}^N A_N^2(i,j)V_N(j)\leq  p^2(1+1/(2N^{3/8}))||V_N||_1$, whence
$\rho_N^2 \leq p^2(1+1/(2N^{3/8}))$ and thus $\rho_N \leq p(1+1/(2N^{3/8}))$. Similarly,
$\rho_N^2 ||V_N||_1 =\sum_{i,j=1}^N A_N^2(i,j)V_N(j)\geq  p^2(1-1/(2N^{3/8}))||V_N||_1$, whence
$\rho_N^2 \geq p^2(1-1/(2N^{3/8}))$ and thus $\rho_N \geq p(1-1/(2N^{3/8}))$.
\end{proof}

We now gather a number of important facts.

\begin{lem}\label{vrac} Assume that $p\in(0,1]$. There is $N_0\geq 1$ (depending only on $p$) such that for
all $N\geq N_0$, on $\Omega_N^2$, the following properties hold true for all $i,j,k,l=1,\dots,N$:
\vip
(i) for all $n\geq 2$, $A_N^n(i,j)\leq (3/2)A_N^n(k,l)$,
\vip
(ii)  $V_N(i)\in[1/2,2]$,
\vip
(iii) for all $n\geq 0$, $||A_N^n \bun||_2 \in [\sqrt N \rho_N^n /2,2\sqrt N \rho_N^n ]$,
\vip
(iv) for all $n\geq 2$, $A_N^n(i,j)\in [\rho_N^n/(3N),3\rho_N^n/N]$,
\vip
(v) for all $n\geq 0$, all $r\in[1,\infty]$, $\big|\big| ||A_N^n\bun||_r^{-1}A_N^n\bun-||V_N||_r^{-1}V_N
\big|\big|_r \leq 3(2N^{-3/8})^{\lfloor n/2 \rfloor +1}$,
\vip
(vi) for all $n\geq 0$, all $r\in[1,\infty]$,
$\big|\big| ||A_N^n e_j||_r^{-1}A_N^ne_j-||V_N||_r^{-1}V_N\big|\big|_r \leq 12(2N^{-3/8})^{\lfloor n/2 \rfloor}$,
\vip
(vii) for all $n\geq 1$, $||A_N^n e_j||_2
\leq  3 \rho_N^n/(p \sqrt N)$ and for all $n\geq 0$, $||A_N^n\bun||_\infty \leq 3 \rho_N^n/p$.
\end{lem}

The proof requires a quantitative version of the Perron-Frobenius theorem due to G. Birkhoff \cite{birk}.
It is based on the use of the Hilbert projective distance.

\begin{nota}
For $x=(x_i)_{i=1,\dots,N}$ and $y=(y_i)_{i=1,\dots,N}$ in $(0,\infty)^N$, we set 
$$
d_N(x,y)=\log\Big(\frac{\max_{i=1,\dots,N}(x_i/y_i)}{\min_{i=1,\dots,N}(x_i/y_i)} \Big).
$$
We have $d_N(x,y)=d_N(y,x)=d_N(x,\lambda y)$ for all $\lambda>0$ and
$d_N(x,y)\leq d_N(x,z)+d_N(z,y)$. Finally, $d_N(x,y)=0$ if and only if $x$ and $y$ are colinear.
\end{nota}

The result of Birkhoff quantifies the projection on the Perron-Frobenius vector.

\begin{thm}[Birkhoff \cite{birk}, Cavazos-Cadena \cite{cc}]\label{gb}
For any $A \in \cM_{N\times N}(\rr)$ with positive entries and any $x$ and $y$ in $(0,\infty)^N$, we have
$d_N(Ax,Ay)\leq k_A d_N(x,y)$, where
$$
\Gamma_A=\max_{i,j,k,l=1,\dots,N} \frac{A(i,k)A(j,l)}{A(i,l)A(j,k)}\geq 1 \quad \hbox{and}\quad
k_A=\frac{\sqrt {\Gamma_A}-1}{\sqrt {\Gamma_A}+1} \leq \frac{\Gamma_A-1}{4}.
$$
\end{thm}

In our context, this gives the following estimates.

\begin{rk}\label{gbfu}
Assume that $p\in (0,1]$. Then on $\Omega_N^2$, it holds that for all $x,y \in (0,\infty)^N$, we have 
(i) $d_N(A_Nx,A_Ny)\leq d_N(x,y)$ and (ii) $d_N(A_N^2x,A_N^2y)\leq 2N^{-3/8}d_N(x,y)$.
\end{rk}

\begin{proof}
On $\Omega_N^2$, we have
\begin{align}\label{tbru}
A_N^2(i,j) \in [p^2N^{-1}(1-1/(2N^{3/8})), p^2N^{-1}(1+1/(2N^{3/8}))].
\end{align}
This implies that for each $i=1,\dots,N$, $\sum_{k=1}^N A_N(i,k)>0$
(because else, $A_N^2(i,j)$ would vanish for all $j=1,\dots,N$).
Thus for $x,y \in (0,\infty)^N$, we have $A_N x, A_N y \in(0,\infty)^N$ so that $d_N(A_Nx,A_Ny)$
is well-defined. We put $m=\min_i (x_i/y_i)$ and $M=\max_i(x_i/y_i)$. We then have
$m (A_N y)_i\leq(A_N x)_i \leq M (A_N y)_i$ for all $i$, whence $d_N(A_Nx,A_Ny)\leq \log(M/m)=d_N(x,y)$,
which proves (i).
For point (ii), it suffices to use Theorem \ref{gb}, and to note that, by \eqref{tbru},
$$
\Gamma_{A_N^2}=\max_{i,j,k,l=1,\dots,N} \frac{A_N^2(i,k)A_N^2(j,l)}{A_N^2(i,l)A_N^2(j,k)}
\leq \frac{(1+1/(2N^{3/8}))^2}{(1-1/(2N^{3/8}))^2}\leq 1+8N^{-3/8},
$$
whence $k_{A_N^2}\leq (\Gamma_{A_N^2}-1)/4 \leq 2N^{-3/8}$.
\end{proof}

We will also use the following easy remark.

\begin{lem}\label{hildr}
For all $r\in[1,\infty]$ and all $x,y\in(0,\infty)^N$ such that $d_N(x,y)\leq 1$,
we have the inequality $\big|\big| ||x||_r^{-1}x-||y||_r^{-1}y\big|\big|_r  \leq 3d_N(x,y)$.
\end{lem}

\begin{proof}
We fix $r\in[1,\infty]$ and assume without loss of generality that $||x||_r=||y||_r=1$. 
We set $m=\min_i(x_i/y_i)$ and $M=\max_i(x_i/y_i)$. Since $||x||_r=||y||_r$, it holds that
$m\leq 1\leq M$. Using that $1 \geq d_N(x,y)=\log(1+(M-m)/m)$, we deduce that
$(M-m)/m \leq e-1\leq 2$.
Since $\log(1+u)\geq u/3$ on $[0,2]$, we conclude that $d_N(x,y)\geq (M-m)/(3m)\geq (M-m)/3$.
But for all $i$, we have $x_i\in[my_i,My_i]$, whence $|x_i-y_i|\leq (M-m)y_i$.
Thus $||x-y||_r \leq (M-m)||y||_r=(M-m) \leq 3d_N(x,y)$. 
\end{proof}

We can now give the

\begin{proof}[Proof of Lemma \ref{vrac}] We work on $\Omega_N^2$ during the whole proof.

\vip

{\it Step 1.} We first check that $d_N(\bun,V_N)\leq 2N^{-3/8}$. We start from $A_N^2V_N=\rho_NV_N$,
so that for all $i$, $V_N(i)=\rho_N^{-2}\sum_{j=1}^N A_N^2(i,j)V_N(j)$. 
But using \eqref{tbru} and setting $\kappa_N=p^2 \rho_N^{-2}N^{-1}\sum_{j=1}^N V_N(j)$, we find that
$V_N(i) \in [\kappa_N(1-1/(2N^{3/8})),\kappa_N(1+1/(2N^{3/8}))]$. Consequently,
$\max_i V_N(i) / \min_i V_N(i) \leq (1+1/(2N^{3/8}))/(1-1/(2N^{3/8}))\leq 1+2N^{-3/8}$.
Hence 
$$
d_N(\bun,V_N) \leq \log[(1+1/(2N^{3/8}))/(1-1/(2N^{3/8}))] \leq \log(1 + 2 N^{-3/8})\leq 2 N^{-3/8}.
$$

{\it Step 2.} Here we show that for all $i$, $V_N(i) \in [(1+2N^{-3/8})^{-1}, (1+2N^{-3/8})]$. 
This will imply point (ii) (for $N$ large enough so that $2N^{-3/8}\leq 1$). 
We introduce $m=\min_i V_N(i)$ and $M=\max_i V_N(i)$. We have seen
in Step 1 that $M/m \leq 1+2N^{-3/8}$. Recalling that $||V_N||_2=\sqrt N$ by definition, we deduce that
$N=\sum_{i=1}^N (V_N(i))^2\leq N M^2 \leq N (1+2N^{-3/8})^2 m^2$, whence $m\geq (1+2N^{-3/8})^{-1}$.
Similarly, $N=\sum_{i=1}^N (V_N(i))^2\geq N m^2 \geq N (1+2N^{-3/8})^{-2} M^2$, whence $M\leq (1+2N^{-3/8})$.

\vip

{\it Step 3.} We verify that for all $n\geq 0$,
$d_N(A_N^n\bun,V_N)\leq (2N^{-3/8})^{\lfloor n/2\rfloor +1}$. By Lemma \ref{hildr},
this will imply point (v) for all $N$ large enough so that $2N^{-3/8}\leq 1$.
Using that $A_N^n V_N=\rho_N^n V_N$, we deduce that $d_N(A_N^n\bun,V_N)=d_N(A_N^n\bun,A_N^nV_N)$.
Hence for all $n$ even, we deduce from Remark \ref{gbfu}-(ii) and Step 1
that $d_N(A_N^n\bun,V_N)\leq (2N^{-3/8})^{n/2} d_N(\bun,V_N)\leq (2N^{-3/8})^{n/2+1}$.
When $n$ is odd, we simply use that $d_N(A_N^n\bun,V_N)=d_N(A_N^n\bun,A_NV_N)\leq d_N(A_N^{n-1}\bun,V_N)$
by Remark \ref{gbfu}-(i).

\vip

{\it Step 4.} We now prove (vi). We fix $r\in[1,\infty]$ and $j\in \{1,\dots,N\}$. 
The result is obvious if $n=0$ or $n=1$ because then
$\big|\big| ||A_N^n e_j||_r^{-1}A_N^ne_j-||V_N\bun||_r^{-1}V_N\big|\big|_r \leq 2\leq 
12(2N^{-3/8})^{\lfloor n/2 \rfloor}$. 

\vip

By Remark \ref{gbfu}-(ii),
$d_N(A_N^{2k}e_j,V_N)=d_N(A_N^{2k}e_j,A_N^{2k}V_N)\leq(2N^{-3/8})^{k-1}d_N(A_N^2e_j,V_N)$ for all $k\geq 1$.

\vip

We next write $d_N(A_N^2e_j,V_N)\leq d_N(A_N^2e_j,\bun)+d_N(\bun,V_N)$. By Step 1, we have
$d_N(\bun,V_N)\leq \log[(1+N^{-3/8}/2)/(1+N^{-3/8}/2)]$. Furthermore, we deduce from \eqref{tbru}
that $d_N(A_N^2e_j,\bun)=
\log[\max_i(A_N^2(i,j))/\min_i(A_N^2(i,j))] \leq \log[(1+N^{-3/8}/2)/(1+N^{-3/8}/2)]$.
All in all, we find that $d_N(A_N^2e_j,V_N)\leq\log[(1+N^{-3/8}/2)^2/(1-N^{-3/8}/2)^2]\leq 
\log(1+8N^{-3/2})\leq 8N^{-3/2}$.

\vip

Hence for all $k\geq 1$, $d_N(A_N^{2k}e_j,V_N)\leq 8N^{-3/8}(2N^{-3/8})^{k-1}= 4(2N^{-3/8})^k$.
We also have, by Remark \ref{gbfu}-(i),
$d_N(A_N^{2k+1}e_j,V_N)=d_N(A_N^{2k+1}e_j,A_NV_N)\leq d_N(A_N^{2k}e_j,V_N)$.
Thus for all $n\geq 2$, $d_N(A_N^ne_j,V_N)\leq 4(2N^{-3/8})^{\lfloor n/2\rfloor}$.
This implies that indeed, $\big|\big| ||A_N^n e_j||_r^{-1}A_N^ne_j-||V_N||_r^{-1}V_N\big|\big|_r \leq
12(2N^{-3/8})^{\lfloor n/2 \rfloor}$ by Lemma \ref{hildr}, if $N$ is large enough so that $2N^{-3/8}\leq 1/4$.

\vip

{\it Step 5.} We check (i). Using Step 2, we see that 
for all $j=1,\dots,N$, all $n\geq 2$,
$$
d_N(A_N^ne_j,V_N) =\log\Big(\frac{\max_i (A_N^n(i,j)/V_N(i))}{\min_i (A_N^n(i,j)/V_N(i))}  \Big) 
\geq \log\Big(\frac{\max_i A_N^n(i,j) }{\min_i A_N^n(i,j)} \times (1+2N^{-3/8})^{-2} \Big).
$$
But for all $n\geq 2$, using Remark \ref{gbfu}-(i),
we see that $d_N(A_N^n e_j,V_N)=d_N(A_N^n e_j,A_N^{n-2}V_N)\leq d_N(A_N^2e_j,V_N)
\leq \log(1+8N^{-3/8})$ as seen in Step 4. We conclude that 
$$
\frac{\max_i A_N^n(i,j) }{\min_i A_N^n(i,j)} \leq (1+2N^{-3/8})^{2}(1+8N^{-3/8}).
$$
Using the same arguments with the transpose matrix $A_N^t$ (which satisfies exactly the same assumptions
as $A_N$ on $\Omega_N^2$), we see that for all $i=1,\dots,N$,
$$
\frac{\max_j A_N^n(i,j) }{\min_j A_N^n(i,j)} \leq (1+2N^{-3/8})^{2}(1+8N^{-3/8}).
$$
Finally, we conclude that for all $n\geq 2$,
$$
\frac{\max_{i,j} A_N^n(i,j) }{\min_{i,j} A_N^n(i,j)} \leq (1+2N^{-3/8})^{4}(1+8N^{-3/8})^2.
$$
This is indeed smaller than $3/2$ if $N$ is large enough.

\vip

{\it Step 6.} We now verify (iii). We write $A_N^n \bun = ||A_N^n\bun||_2(N^{-1/2}V_N + Z_{N,n})$,
where $Z_{N,n}=||A_N^n\bun||_2^{-1}A_N^n\bun - N^{-1/2}V_N$. We know by (v) (with $r=2$) that 
$||Z_{N,n}||_2 \leq 3(2N^{-3/8})^{\lfloor n/2\rfloor+1}$.
We next write, for each $n\geq 0$, $A_N^{n+1} \bun= ||A_N^n\bun||_2(N^{-1/2} \rho_N V_N + A_NZ_{N,n})$.
Using that $||V_N||_2=\sqrt N$ and that $|||A_N|||_2 \leq 1$ 
(which immediately follows from the fact that $0\leq A_N(i,j)\leq 1/N$), we conclude that
$\big| ||A_N^{n+1}\bun||_2- \rho_N ||A_N^{n}\bun||_2\big| \leq 3 ||A_N^{n}\bun||_2 (2N^{-3/8})^{\lfloor n/2\rfloor+1}$.

\vip

We now set $x_n=||A_N^{n}\bun||_2 / (\sqrt N \rho_N^n)$. For all $n\geq 0$, we have
$$
|x_{n+1}-x_n|\leq 3x_n (2N^{-3/8})^{\lfloor n/2\rfloor+1} /\rho_N \leq  6 x_n (2N^{-3/8})^{\lfloor n/2\rfloor+1} /p,
$$
because $\rho_N \geq p/2$ on $\Omega_N^2$, see Lemma \ref{rhoV}. If now $N$ is large enough so that
$6 (2N^{-3/8})^{1/2} /p \leq 1/2$, we easily conclude, using that $x_0=1$, that, for all $n\geq 1$,
$$
x_n \in \Big[\prod_{k=1}^n(1-6(2N^{-3/8})^{\lfloor k/2 \rfloor+1} /p),
\prod_{k=1}^n(1+6(2N^{-3/8})^{\lfloor k/2\rfloor +1} /p)\Big],
$$
which is included in $[1/2,2]$ 
if $N$ is large enough (depending only on $p$). Since $x_0=1$, we thus have 
$x_n\in [1/2,2]$ for all $n\geq 0$,
and thus $||A_N^{n}\bun||_2\in [\sqrt N \rho_N^n/2,2\sqrt N \rho_N^n]$ for all $n\geq 0$.

\vip

{\it Step 7.} Here we prove (iv). We fix $n\geq 2$ and set $m=\min_{i,j}A_N^n(i,j)$ and $M=\max_{i,j}A_N^n(i,j)$.
We know from (i) that $M/m \leq 3/2$. Starting from point (iii), we write
$\sqrt N \rho_N^n/2 \leq ||A_N^{n}\bun||_2=(\sum_{i=1}^N (\sum_{j=1}^N A_N^n(i,j))^2)^{1/2}\leq N^{3/2}M \leq 3 N^{3/2}m/2$,
whence $m \geq \rho_N^n/(3N)$. By the same way, $2 \sqrt N \rho_N^n \geq ||A_N^{n}\bun||_2
\geq N^{3/2} m  \geq 2 N^{3/2}M/3$, whence $M  \leq 3 \rho_N^n/N$.

\vip

{\it Step 8.} It only remains to check (vii). We know from (iv) that for all $n\geq 2$, 
$A_N^n(i,j) \leq  3 \rho_N^n/N \leq  3 \rho_N^n/(pN)$. And for $n=1$, $A_N(i,j) \leq 1/N \leq 3 \rho_N/(pN)$
because $\rho_N\geq p/3$ on $\Omega_N^2$, see Lemma \ref{rhoV}.
We conclude that for all $n\geq 1$, $A_N^n(i,j) \leq 3 \rho_N^n/(pN)$.
This immediately implies that for all $n\geq 1$, $||A_N^n e_j||_2=(\sum_{i=1}^N (A_N^n(i,j))^2)^{1/2}
\leq 3 \rho_N^n/(p\sqrt N)$ and $||A_N^{n}\bun||_\infty = \max_i \sum_{j=1}^N A_N^n(i,j) \leq  3 \rho_N^n/p$.
Finally, for $n=0$, we of course have $||A_N^{0}\bun||_\infty=1\leq 3\rho_N^0/p$.
\end{proof}

Finally, the following tedious result is crucial for our estimation method.

\begin{prop}\label{basis}
We assume that $p\in (0,1]$ and we introduce, on $\Omega_N^2$, $\bar V_N=N^{-1}\sum_{i=1}^NV_N(i)$ and
$$
\cU^N_\infty = \sum_{i=1}^N \Big(\frac{V_N(i)-\bar V_N}{\bar V_N} \Big)^2.
$$
There is $N_0\geq 1$ and $C>0$ (depending only on $p$) such that for all $N\geq N_0$,
$$
\E\Big[\indiq_{\Omega_N^2}\Big|\cU^N_\infty - \Big(\frac1p -1\Big) \Big|  \Big] \leq \frac C {\sqrt N}
\quad \hbox{and}\quad \E[\indiq_{\Omega_N^2}||V_N - \bar V_N \bun||_2^2]\leq C.
$$
\end{prop}

\begin{proof} We work with $N$ large enough so that we can apply Lemma \ref{vrac}.
We introduce the vectors $L_N=A_N\bun$ and $\cL_N=A_N^6\bun$, we set
$\bar L_N = N^{-1}\sum_{i=1}^N L_N(i)$, $\bar \cL_N = N^{-1}\sum_{i=1}^N \cL_N(i)$,
$$
H_N = \sum_{i=1}^N \Big(\frac{L_N(i)-\bar L_N}{\bar L_N} \Big)^2 \quad \hbox{and} \quad
\cH_N = \sum_{i=1}^N \Big(\frac{\cL_N(i)-\bar \cL_N}{\bar \cL_N} \Big)^2.
$$
We checked in the proof of Proposition \ref{interro}-Step 2 that
(i) $\E[|\bar L_N - p|^2]\leq C N^{-2}$, 
(ii) $\E[||L_N - \bar L_N \bun||_2^4]\leq C$,
(iii) $\E[( ||L_N - \bar L_N \bun||_2^2-p(1-p))^2]\leq CN^{-1}$, 
(iv) $\E[ ||A_N L_N - \bar L_N L_N||_2^2] \leq C N^{-1}$.

\vip

We also recall that $\bar L_N \leq 1$ and $|||A_N|||_2 \leq 1$ (simply because $0\leq A_N(i,j)\leq 1/N$).
Furthermore, on $\Omega_N^2$, it holds that 
$\bar L_N=N^{-1}\sum_{i,j=1}^N A_N(i,j) \geq p/2$, that
$\bar \cL_N =N^{-1}\sum_{i,j=1}^N A_N^6(i,j)\geq \rho_N^6/3 \geq p^6/192$ 
(by Lemma \ref{vrac}-(iv) and because $\rho_N \geq p/2$ by Lemma \ref{rhoV})
and that $\bar V_N \geq 1/2$ (by Lemma \ref{vrac}-(ii)).

\vip

{\it Step 1.} We show that on $\Omega_N^2$, $\Delta_N=|\cU^N_\infty-\cH_N|\leq
C N^{-1/2}$. A simple computation shows that
$$
\Delta_N=\Big| \sum_{i=1}^N \Big[\Big(\frac{V_N(i)}{\bar V_N} \Big)^2
-\Big(\frac{\cL_N(i)}{\bar \cL_N} \Big)^2 \Big]\Big|\leq 
\Big(\sum_{i=1}^N \Big|\frac{V_N(i)}{\bar V_N}-\frac{\cL_N(i)}{\bar \cL_N} \Big| \Big) 
\Big(\max_i \Big(\frac{V_N(i)}{\bar V_N}+\frac{\cL_N(i)}{\bar \cL_N} \Big)\Big)=S_NT_N,
$$
the last equality being a definition.

\vip

Lemma \ref{vrac}-(ii) implies that $\max_i (V_N(i)/\bar V_N)\leq \max_i V_N(i)/ \min_iV_N(i)
\leq 4$ and Lemma \ref{vrac}-(i) implies that $\max_i (\cL_N(i)/\bar \cL_N)\leq
\max_i \cL_N(i) / \min_i \cL_N(i) \leq 3/2$ because $\cL_N=A_N^6 \bun$. Thus 
$T_N \leq 4+3/2\leq 6$.
Next, it holds that $S_N=N\big|\big| ||A_N^6\bun||_1^{-1}A_N^6\bun-||V_N\bun||_1^{-1}V_N\big|\big|_1$.
We thus infer from Lemma \ref{vrac}-(v) that
$S_N \leq 3 N (2N^{-3/8})^{4}= 48 N^{-1/2}$.
The conclusion follows.

\vip

{\it Step 2.} We next prove that $\E[\indiq_{\Omega_N^2}|\cH_N-H_N|]\leq
C N^{-1/2}$.  We first write
\begin{align*}
||\cL_N - (\bar L_N)^5 L_N||_2=||A_N^6\bun- (\bar L_N)^5 A_N\bun||_2
\leq \sum_{k=1}^5 ||(\bar L_N)^{5-k}A_N^{k+1}\bun- (\bar L_N)^{6-k}A_N^{k}\bun||_2.
\end{align*}
Using  that $\bar L_N\leq 1$ and $|||A_N|||_2\leq 1$, we deduce that
$$
||\cL_N - (\bar L_N)^5 L_N||_2\leq 5 ||A_N^{2}\bun- \bar L_N A_N\bun||_2=5 ||A_NL_N - \bar L_N L_N||_2 .
$$
We thus deduce from point (iv) recalled above that $\E[||\cL_N - (\bar L_N)^5 L_N||_2^2]\leq C N^{-1}$.
But it holds that $||\cL_N - (\bar L_N)^5 L_N||_2=I_N+J_N$, where
$I_N=||(\cL_N-\bar\cL_N\bun) - (\bar L_N)^5 (L_N-\bar L_N \bun)||_2$ and 
$J_N=||\bar\cL_N\bun - (\bar L_N)^6\bun ||_2=\sqrt N|\bar\cL_N-(\bar L_N)^6|$.
Consequently, $\E[I_N^2]+\E[J_N^2]\leq C N^{-1}$.
Using now that
$$
H_N=\frac{ ||L_N-\bar L_N \bun||_2^2}{(\bar L_N)^{2}}=
\frac{ ||(\bar L_N)^5 (L_N-\bar L_N \bun)||_2^2}{(\bar L_N)^{12}} \quad \hbox{and}\quad
\cH_N=\frac{ ||\cL_N-\bar \cL_N \bun||_2^2}{(\bar \cL_N)^{2}},
$$
the facts that $\bar\cL_N\geq p^6/192$ and $(\bar L_N)^6 \geq p^6/64$ on $\Omega_N^2$
and that the map $x\mapsto x^{-2}$ is globally Lipschitz and bounded on $[p^6/192,\infty)$, we conclude that,
still on $\Omega_N^2$,
\begin{align*}
|H_N-\cH_N|\leq & C\Big(||(\bar L_N)^5 (L_N-\bar L_N \bun)||_2^2|(\bar L_N)^6-\bar\cL_N|\\
&\hskip3cm+ \Big| ||\cL_N-\bar \cL_N \bun||_2^2 -||(\bar L_N)^5 (L_N-\bar L_N \bun)||_2^2 \Big|\Big).
\end{align*}
Using now the inequality $|a^2-b^2|\leq(a-b)^2+2a|a-b|$ for $a,b\geq 0$, we deduce that
\begin{align*}
|H_N-\cH_N|\leq & C\Big(||(\bar L_N)^5 (L_N-\bar L_N \bun)||_2^2 N^{-1/2}J_N
+  I_N^2 +  ||(\bar L_N)^5 (L_N-\bar L_N \bun)||_2^2  I_N \Big)\\
\leq & C\Big(||L_N-\bar L_N \bun||_2^2 N^{-1/2} J_N+  I_N^2 +  ||L_N-\bar L_N \bun||_2^2  I_N \Big)
\end{align*}
because $\bar L_N \leq 1$. Using the Cauchy-Schwarz inequality, that $\E[I_N^2]+\E[J_N^2]\leq C N^{-1}$
and that $\E[ ||L_N-\bar L_N \bun||_2^4]\leq C$ by point (ii) recalled above, we conclude that
$\E[\indiq_{\Omega_N^2}|\cH^N-H_N|]\leq C N^{-1/2}$.

\vip

{\it Step 3.} Here we check that $\E[\indiq_{\Omega_N^2}|H_N-(1/p-1)|]\leq
C N^{-1/2}$. Since $\bar L_N \geq p/2$ on $\Omega_N^2$ and since $x\mapsto x^{-2}$ is bounded and globally Lipschitz 
continuous on $[p/2,\infty)$, we can write
\begin{align*}
\Big|H_N-\Big(\frac 1p-1\Big)\Big|=& \Big| \frac{ ||L_N-\bar L_N \bun||_2^2}{(\bar L_N)^{2}} 
- \frac{p(1-p)}{p^2}\Big|\\
\leq& C \Big(|\bar L_N-p| p(1-p) +  \Big|||L_N-\bar L_N \bun||_2^2-p(1-p)\Big|\Big).
\end{align*}
The conclusion follows, since as recalled in points (i) and (iii) above, $\E[|\bar L_N-p|]\leq CN^{-1}$ and
$\E[\big|||L_N-\bar L_N \bun||_2^2-p(1-p)\big|]\leq CN^{-1/2}$.

\vip

{\it Step 4.} Gathering Steps 1, 2 and 3, we immediately deduce that
$\E[\indiq_{\Omega_N^2}|\cU^N_\infty-(1/p-1)|]\leq C N^{-1/2}$.
Since now $\bar V_N\leq 2$ on $\Omega_N^2$ by Lemma \ref{vrac}-(ii), 
$||V_N-\bar V_N \bun||_2^2= (\bar V_N)^2 \cU^N_\infty \leq 4 \cU^N_\infty$, whence of course,
$\E[\indiq_{\Omega_N^2}||V_N-\bar V_N \bun||_2^2] \leq C$.
\end{proof}

\subsection{Preliminary analytic estimates}\label{ana2}

We recall the following lemma, relying on some results of Feller \cite{f} on convolution equations, 
that can be found in \cite[Lemma 26-(b)]{dfh}.

\begin{lem}\label{fdfh}
Let $\psi:[0,\infty)\mapsto [0,\infty)$ be integrable and such that $\int_0^\infty\psi(t)dt>1$.
Assume also that $t\mapsto \intot |d\psi(s)|$ has at most polynomial growth
and set $\Gamma_t=\sum_{n\geq 0} \psi^{\star n}(t)$.
Consider $\alpha>0$ such that $\int_0^\infty e^{-\alpha t}\psi(t)dt=1$. 
There are $0<c<C$ such that for all $t\geq 0$, $1+\Gamma_t \in [c e^{\alpha t}, C e^{\alpha t}]$.
\end{lem}

Based on this, it is not hard to verify the following result.

\begin{lem}\label{gamman}
Assume $A$. Recall that $\alpha_0$ was defined in Remark \ref{rkaz} such that 
$p\int_0^\infty e^{-\alpha_0 t}\varphi(t)dt=1$
and that $\rho_N$ was defined, for each $N\geq 1$, in Lemma \ref{rhoV}.
We now set $\Gamma^N_t=\sum_{n\geq 0} \rho_N^n \varphi^{\star n}(t)$. For any $\eta>0$, 
we can find $N_\eta\geq 1$ and $0<c_\eta<C_\eta$ (depending only on $p,\varphi$ and $\eta$)
such that for all $N\geq N_\eta$, on $\Omega_N^2$,
for all $t\geq 0$, $1+\Gamma^N_t \in [c_\eta e^{(\alpha_0-\eta)t},C_\eta e^{(\alpha_0+\eta)t}]$. 
\end{lem}

\begin{proof}
We only prove the result when $\eta\in(0,\alpha_0)$, which of course suffices. We 
consider $\rho_\eta^+>p>\rho_\eta^-$ defined by
$\int_0^\infty e^{-(\alpha_0+\eta) t}\varphi(t)dt=1/\rho_\eta^+$ and $\int_0^\infty e^{-(\alpha_0-\eta) t}\varphi(t)dt
=1/\rho_\eta^-$.
We put $\Gamma^{\eta,+}_t=\sum_{n\geq 0} (\rho_\eta^+)^n \varphi^{\star n}(t)$ and
$\Gamma^{\eta,-}_t=\sum_{n\geq 0} (\rho_\eta^-)^n \varphi^{\star n}(t)$.
Applying Lemma \ref{fdfh} with $\psi=\rho_\eta^+ \varphi$ and with $\psi=\rho_\eta^- \varphi$,
we deduce that there are some constants $0<c_\eta<C_\eta$ such that for all $t\geq 0$,
$c_\eta e^{(\alpha_0-\eta)t}\leq 1+\Gamma^{\eta,-}_t\leq 1+\Gamma^{\eta,+}_t \leq C_\eta e^{(\alpha_0+\eta)t}$.
But on $\Omega_N^2$, we know from Lemma \ref{rhoV} that $\rho_N \in [p(1-N^{-3/8}/2),p(1+N^{-3/8}/2)]$.
Thus for $N$ large enough, we clearly have $\rho_N \in [\rho_\eta^-,\rho_\eta^+]$, so that
$\Gamma^N_t \in [\Gamma^{\eta,-}_t,\Gamma^{\eta,+}_t]$. The conclusion follows.
\end{proof}

We next gather a number of consequences of the above estimate that we will use later.

\begin{lem}\label{list}
Assume $A$. Recall that $\alpha_0$ was defined in Remark \ref{rkaz},
that $\rho_N$ was defined in Lemma \ref{rhoV} and that $\Gamma^N_t=\sum_{n\geq 0} \rho_N^n \varphi^{\star n}(t)$. 
We also put $v^N_t=\mu N^{-1/2} \sum_{n\geq 0} ||A_N^n\bun||_2\intot s \varphi^{\star n}(t-s)ds $.
For any $\eta>0$, 
we can find $N_\eta\geq 1$, $t_\eta>0$ and $0<c_\eta<C_\eta$ (depending only on $p,\mu,\varphi$ and $\eta$)
such that for all $N\geq N_\eta$, on $\Omega_N^2$,

\vip

(i) for all $t\geq 0$, $v^N_t \leq C_\eta e^{(\alpha_0+\eta)t}$,

\vip

(ii) for all $t\geq t_\eta$, $v^N_t \geq c_\eta e^{(\alpha_0-\eta)t}$,

\vip

(iii) for all $t\geq 0$, $\sum_{n\geq 0} \rho_N^n (2N^{-3/8})^{\lfloor n/2\rfloor} 
\intot \varphi^{\star n}(t-s)ds \leq C_\eta$,

\vip

(iv) for all $t\geq 0$, $\sum_{n\geq 0} \rho_N^n \intot e^{(\alpha_0+\eta)s/2}\varphi^{\star n}(t-s)ds \leq C_\eta 
e^{(\alpha_0+\eta)t}$,

\vip

(v) for all $t\geq 0$, $\sum_{n\geq 0} \rho_N^n \intot s \varphi^{\star n}(t-s)ds \leq C_\eta 
e^{(\alpha_0+\eta)t}$,

\vip

(vi) for all $t\geq 0$, $\intot \intot \Gamma^N_{t-r} \Gamma^N_{t-s} e^{(\alpha_0+\eta)(r\land s)} drds \leq C_\eta 
e^{2(\alpha_0+\eta)t}$.
\end{lem}

\begin{proof} We fix $\eta>0$ and work with $N$ large enough and on $\Omega_N^2$, so that we can use
Lemmas \ref{vrac} and \ref{gamman}.

\vip

We start with (i). We know from Lemma \ref{vrac}-(iii) that $||A_N^n\bun||_2\leq 2 \sqrt N \rho_N^n$,
whence $v^N_t \leq 2 \mu \intot s \Gamma^N_{t-s} ds \leq C_\eta \intot s e^{(\alpha_0+\eta)(t-s)}ds
= C_\eta e^{(\alpha_0+\eta)t}\intot s e^{-(\alpha_0+\eta)s}ds \leq C_\eta  e^{(\alpha_0+\eta)t}$.

\vip

The LHS of point (iv) is nothing but $\intot e^{(\alpha_0+\eta)s/2}\Gamma^N_{t-s}ds \leq C_\eta \intot e^{(\alpha_0+\eta)s/2}
e^{(\alpha_0+\eta)(t-s)}ds =C_\eta e^{(\alpha_0+\eta)t} \intot e^{-(\alpha_0+\eta)s/2}ds \leq C_\eta   e^{(\alpha_0+\eta)t} $.

\vip

Point (v) follows from point (iv).

\vip

The LHS of point (vi) is smaller than $C_\eta\intot\intot e^{(\alpha_0+\eta)(t-r)}e^{(\alpha_0+\eta)(t-s)}
e^{(\alpha_0+\eta)(r\land s)} drds$, which equals $2C_\eta \intot e^{(\alpha_0+\eta)(t-s)} \int_0^s e^{(\alpha_0+\eta)(t-r)}
e^{(\alpha_0+\eta)r} drds=2C_\eta e^{2(\alpha_0+\eta)t} \intot se^{-(\alpha_0+\eta)s}ds \leq C_\eta e^{2(\alpha_0+\eta)t} $.

\vip

Setting $\Lambda=\int_0^\infty\varphi(t)dt$, the LHS of (iii) is bounded by
$\sum_{n\geq 0} (\Lambda \rho_N)^n (2N^{-3/8})^{\lfloor n/2\rfloor}$ which is itself bounded by 
$\sum_{n\geq 0} (2 \Lambda p)^n 
(2N^{-3/8})^{\lfloor n/2\rfloor}$ since $\rho_N \leq 2p$ on $\Omega_N^2$ by Lemma \ref{rhoV}.
This is uniformly bounded, as soon as $N$ is large enough so that $2\Lambda p (2N^{-3/8})^{1/2}\leq 1/2$.

\vip

We finally check (ii). We know from Lemma \ref{vrac}-(iii) that, on $\Omega_N^2$, 
$||A_N^n\bun||_2\geq \sqrt N \rho_N^n/2$, whence $v^N_t \geq (\mu/2) \intot s \Gamma^N_{t-s}ds
\geq (\mu/2) \int_1^2 s \Gamma^N_{t-s}ds \geq (\mu/2) \int_{t-2}^{t-1} \Gamma^N_{s}ds$ if $t\geq 2$. 
By Lemma \ref{gamman}, we thus have $v^N_t\geq (\mu/2) \int_{t-2}^{t-1} (c_\eta e^{(\alpha_0-\eta)s}-1)ds
\geq (\mu/2) [c_\eta e^{(\alpha_0-\eta)(t-2)}-1]$. The conclusion easily follows:
we can find $t_\eta\geq 2$ and $c_\eta>0$ such that for all $t\geq t_\eta$, $v^N_t\geq c_\eta e^{(\alpha_0-\eta)t}$.
\end{proof}

\subsection{Preliminary stochastic analysis}\label{sa2}

We now prove a few estimates concerning the processes introduced in Notation \ref{processes}. 
We recall that $\alpha_0$ was defined in Remark \ref{rkaz} and that 
$\rho_N$ and $V_N$ were defined in Lemma \ref{rhoV}.
We start from Lemma \ref{fonda} to write (with as usual $\varphi^{\star 0}(t-s)ds=\delta_t(ds)$)

\begin{gather}
\label{2ex1}
\Et[\bZ^N_t]= \mu \sum_{n\geq 0} \Big[\intot s \varphi^{\star n}(t-s)ds\Big]A_N^n\bun=v^N_tV_N+\bI^N_t,\\
\label{2ex2}
\bU^N_t =\bZ^N_t-\Et[\bZ^N_t]= 
\sum_{n\geq 0}\intot \varphi^{\star n}(t-s) A_N^n \bM^N_s ds = \bM^N_t + \bJ^N_t,
\end{gather}
where
\begin{gather}
v^N_t=\mu  \sum_{n\geq 0} \frac{||A_N^n\bun||_2}{\sqrt N}\intot s \varphi^{\star n}(t-s)ds, \label{2ex3}\\
\bI^N_t=\mu  \sum_{n\geq 0} \Big[\intot s \varphi^{\star n}(t-s)ds\Big]\Big[A_N^n\bun - 
\frac{||A_N^n\bun||_2}{\sqrt N}V_N\Big], \label{2ex4}\\
\bJ^N_t=\sum_{n\geq 1}\intot \varphi^{\star n}(t-s) A_N^n \bM^N_s ds.\label{2ex5}
\end{gather}

As usual, we denote by $I^{i,N}_t$ and $J^{i,N}_t$ the coordinates of $\bI^N_t$ and $\bJ^N_t$
and by $\bar I^N_t$ and $\bar J^N_t$ their empirical mean.
We start with some upperbounds concerning $\bZ^N_t$ and $\bU^N_t$.

\begin{lem}\label{2stoch}
Assume $A$. For all $\eta>0$, there are $N_\eta\geq 1$ and $C_\eta>0$ such that for all $N\geq N_\eta$, all $t\geq 0$,
on $\Omega_N^2$,

\vip

(i) $\max_{i=1,\dots,N} \Et[(Z^{i,N}_t)^2] \leq C_\eta e^{2(\alpha_0+\eta)t}$,

\vip

(ii) $\max_{i=1,\dots,N} \Et[(U^{i,N}_t)^2] \leq C_\eta( N^{-1}  e^{2(\alpha_0+\eta)t} +  e^{(\alpha_0+\eta)t})$,

\vip

(iii) $\Et[(\bar U^N_t)^2] \leq C_\eta N^{-1}e^{2(\alpha_0+\eta)t}$.
\end{lem}

\begin{proof} We fix $\eta>0$ and work with $N$ large enough and on $\Omega_N^2$, so that we can use
Lemmas \ref{vrac} and \ref{list}.

\vip

{\it Step 1.} We first verify that $||\Et[\bZ^N_t]||_\infty \leq C_\eta  e^{(\alpha_0+\eta)t}$. 
Using \eqref{2ex1} and that $||A_N^n \bun||_\infty \leq C \rho_N^n $ for all $n\geq 0$ by Lemma \ref{vrac}-(vii),
we see that  $||\Et[\bZ^N_t]||_\infty \leq C \sum_{n\geq 0} \rho_N^n \intot s \varphi^{\star n}(t-s)ds$,
whence the conclusion by Lemma \ref{list}-(v).

\vip

{\it Step 2.} We next show that for all $i=1,\dots,N$, $\Et[(J^{i,N}_t)^2] \leq C_\eta N^{-1}e^{2(\alpha_0+\eta)t}$.
We start from \eqref{2ex5}, which gives us
\begin{align*}
\Et[(J^{i,N}_t)^2]=\sum_{m,n\geq 1}\intot\intot \varphi^{\star m}(t-r) \varphi^{\star n}(t-s)\sum_{j,k=1}^N 
A_N^m(i,j)A_N^n(i,k)\Et[M^{j,N}_rM^{k,N}_s]drds.
\end{align*}
But we know from Remark \ref{mo} that $\Et[M^{j,N}_rM^{k,N}_s]=\indiq_{\{j=k\}}\Et[Z^{j,N}_{r\land s}] \leq
C_\eta\indiq_{\{j=k\}}  e^{(\alpha_0+\eta)(r\land s)}$ by Step 1. Furthermore, $\sum_{j=1}^N A_N^m(i,j)A_N^n(i,j)
\leq ||A_N^n e_i||_2 ||A_N^m e_i||_2
\leq C N^{-1}\rho_N^{m+n}$ by Lemma \ref{vrac}-(vii) (because $m,n\geq 1$). We thus find, recalling that
$\Gamma^N_t=\sum_{n\geq 0} \varphi^{\star n}(t)$, that
\begin{align*}
\Et[(J^{i,N}_t)^2]=C_\eta N^{-1} \intot\intot \Gamma^N_{t-r} \Gamma^N_{t-s} e^{(\alpha_0+\eta)(r\land s)} drds.
\end{align*}
The conclusion follows from Lemma \ref{list}-(vi).

\vip

{\it Step 3.} Point (ii) follows from the facts that $U^{i,N}_t=M^{i,N}_t+J^{i,N}_t$,
that $\Et[(M^{i,N}_t)^2]=\Et[Z^{i,N}_t] \leq  C_\eta  e^{(\alpha_0+\eta)t}$ by Remark \ref{mo} and Step 1
and that $\Et[(J^{i,N}_t)^2] \leq C_\eta N^{-1}e^{2(\alpha_0+\eta)t}$ by Step 2.

\vip

{\it Step 4.} Since $Z^{i,N}_t=\Et[Z^{i,N}_t]+U^{i,N}_t$, we deduce from Steps 1 and 3 that
$\Et[(Z^{i,N}_t)^2] \leq C_\eta( e^{2(\alpha_0+\eta)t} +  e^{(\alpha_0+\eta)t}+N^{-1}  e^{2(\alpha_0+\eta)t} )
\leq C_\eta e^{2(\alpha_0+\eta)t}$, whence point (i).

\vip

{\it Step 5.} Finally, we write $\bar U^N_t= \bar M^N_t + \bar J^N_t$. It is clear from Step
2 that $\Et[(\bar J^N_t)^2] \leq C_\eta N^{-1}e^{2(\alpha_0+\eta)t}$. Remark \ref{mo}
implies that $\Et[(\bar M^N_t)^2]=N^{-2}\sum_{i=1}^N \Et[Z^{i,N}_t]\leq  C_\eta N^{-1}e^{(\alpha_0+\eta)t}$ by Step 1.
Point (iii) is checked.
\end{proof}

We next show that the term $\bI^N_t$ is very small in the present scales.

\begin{lem}\label{2I}
Assume $A$. For all $\eta>0$, there are $N_\eta\geq 1$ and $C_\eta>0$ such that for all $N\geq N_\eta$, all $t\geq 0$,
on $\Omega_N^2$, $||\bI^N_t||_2 \leq C_\eta N^{1/8} t$.
\end{lem}

\begin{proof}
We fix $\eta>0$ and work with $N$ large enough and on $\Omega_N^2$, so that we can use
Lemmas \ref{vrac} and \ref{list}.
Using the Minkowski inequality and then Lemma \ref{vrac}-(iii)-(v), we find
\begin{align*}
||\bI^N_t||_2 \leq & \mu  \sum_{n\geq 0} \Big[\intot s \varphi^{\star n}(t-s)ds\Big] 
\Big|\Big|A_N^n\bun - \frac{||A_N^n\bun||_2}{\sqrt N}V_N\Big|\Big|_2 \\
\leq & 6\mu t \sum_{n\geq 0} \Big[\intot  \varphi^{\star n}(t-s)ds\Big] N^{1/2} \rho_N^n 
(2N^{-3/8})^{\lfloor n/2\rfloor +1} \\
\leq & 12 \mu t N^{1/8}\sum_{n\geq 0} \rho_N^n (2N^{-3/8})^{\lfloor n/2\rfloor}
\intot  \varphi^{\star n}(t-s)ds.
\end{align*}
The conclusion follows from Lemma \ref{list}-(iii).
\end{proof}

We now study the empirical variance of $\bJ^N_t$.

\begin{lem}\label{2J}
Assume $A$. For all $\eta>0$, there are $N_\eta\geq 1$ and $C_\eta>0$ such that for all $N\geq N_\eta$, all $t\geq 0$,
on $\Omega_N^2$, $\Et[||\bJ^N_t - \bar J^N_t \bun||_2^2] \leq C_\eta [e^{(\alpha_0+\eta) t}+ N^{-1}||V_N-\bar V_N \bun||_2^2
e^{2(\alpha_0+\eta) t}]$.
\end{lem}

\begin{proof}
As usual, we fix $\eta>0$ and work with $N$ large enough and on $\Omega_N^2$, so that we can use
Lemmas \ref{vrac} and \ref{list}. Starting from \eqref{2ex5} and using the Minkowski inequality, we find
$$
\Et[||\bJ^N_t - \bar J^N_t \bun||_2^2]^{1/2} \leq  \sum_{n\geq 1} \intot \varphi^{\star n}(t-s) \Et[||A_N^n M^N_s
-\overline{A_N^n M^N_s}\bun ||_2^2]^{1/2} ds.
$$
But using Remark \ref{mo} and then Lemma \ref{2stoch}-(i), we see that
\begin{align*}
\Et[||A_N^n M^N_s-\overline{A_N^n M^N_s}\bun ||_2^2]=& \sum_{i=1}^N \Et\Big[\Big(\sum_{j=1}^N A_N^n(i,j)M^{j,N}_s
- \frac 1 N\sum_{j,k=1}^N A_N^n(k,j)M^{j,N}_s\Big)^2\Big]\\
=& \sum_{i,j=1}^N \Big(A_N^n(i,j)- \frac 1 N\sum_{k=1}^N A_N^n(k,j)\Big)^2 \Et[Z^{j,N}_s]\\
\leq& C_\eta e^{(\alpha_0+\eta)s} \sum_{j=1}^N ||A_N^n e_j - \overline{A_N^n e_j}\bun||_2^2.
\end{align*}
Using next that, for all $x,y\in \rr^N$, $\big|||x-\bar x \bun||_2-||y-\bar y \bun||_2\big|\leq ||x-y||_2$
(with the notation $\bar x = N^{-1}\sum_{i=1}^N x_i$ and $\bar y = N^{-1}\sum_{i=1}^N y_i$), we write
\begin{align*}
||A_N^n e_j - \overline{A_N^n e_j}\bun||_2\leq & \Big|\Big| A_N^ne_j- \frac{||A_N^ne_j||_2}{\sqrt N}V_N\Big|\Big|_2
+ \frac{||A_N^ne_j||_2}{\sqrt N}||V_N - \bar V_N \bun||_2\\
\leq & ||A_N^ne_j||_2\Big( 12 (2N^{-3/8})^{\lfloor n/2\rfloor} + \frac{||V_N - \bar V_N \bun||_2}{\sqrt N}\Big)
\end{align*}
by Lemma \ref{vrac}-(vi). Since $||A_N^ne_j||_2\leq C \rho_N^n/\sqrt N$ by Lemma \ref{vrac}-(vii)
(because $n\geq 1$), 
we conclude
that
\begin{align*}
\Et[||A_N^n M^N_s-\overline{A_N^n M^N_s}\bun ||_2^2]^{1/2} \leq& C_\eta e^{(\alpha_0+\eta)s/2} \rho_N^n
\Big( (2N^{-3/8})^{\lfloor n/2\rfloor} + \frac{||V_N - \bar V_N \bun||_2}{\sqrt N}\Big).
\end{align*}
Consequently,
\begin{align*}
\Et[||\bJ^N_t - \bar J^N_t \bun||_2^2]^{1/2} \leq&  C_\eta \sum_{n\geq 1}  \rho_N^n
\Big( (2N^{-3/8})^{\lfloor n/2\rfloor} + \frac{||V_N - \bar V_N \bun||_2}{\sqrt N}\Big)
\intot \! \varphi^{\star n}(t-s) e^{(\alpha_0+\eta)s/2} ds\\
\leq & C_\eta e^{(\alpha_0+\eta)t/2}  \sum_{n\geq 1}  \rho_N^n (2N^{-3/8})^{\lfloor n/2\rfloor}\intot\varphi^{\star n}(t-s)ds\\
&+ C_\eta \frac{||V_N - \bar V_N \bun||_2}{\sqrt N} \sum_{n\geq 1}  \rho_N^n\intot 
\varphi^{\star n}(t-s) e^{(\alpha_0+\eta)s/2} ds\\
\leq & C_\eta e^{(\alpha_0+\eta)t/2} +  C_\eta \frac{||V_N - \bar V_N \bun||_2}{\sqrt N}e^{(\alpha_0+\eta)t}
\end{align*}
by Lemma \ref{list}-(iii)-(iv). This completes the proof.
\end{proof}

The last lemma of the subsection concerns the martingale $\bM^N_t$.
In point (ii) below, $(\cdot,\cdot)$ stands for the usual scalar product in $\rr^N$.

\begin{lem}\label{2M}
Assume $A$. For all $\eta>0$, there are $N_\eta\geq 1$ and $C_\eta>0$ such that for all $N\geq N_\eta$, all $t\geq 0$,
on $\Omega_N^2$, 

\vip

(i) $\Et[||\bM^N_t - \bar M^N_t \bun||_2^2] \leq C_\eta N e^{(\alpha_0+\eta) t}$,

\vip

(ii) $\Et[(\bM^N_t - \bar M^N_t \bun,V_N-\bar V_N\bun)^2]\leq C_\eta ||V_N - \bar V_N \bun||_2^2 e^{(\alpha_0+\eta)t}$,

\vip

(iii) setting $X^N_t=||\bM^N_t - \bar M^N_t \bun||_2^2 -N\bar Z^N_t$, we have $\Et[|X^N_t|]\leq C_\eta \sqrt N 
e^{(\alpha_0+\eta) t}$.
\end{lem}

\begin{proof}
We fix $\eta>0$ and work with $N$ large enough and on $\Omega_N^2$, so that we can use
Lemmas \ref{vrac} and \ref{list}.

\vip

To check point (ii), we write $\Et[(\bM^N_t - \bar M^N_t \bun,V_N-\bar V_N\bun)^2]=\Et[(\bM^N_t,V_N-\bar V_N\bun)^2]
=\Et[(\sum_{i=1}^N (V_N(i)-\bar V_N) M^{i,N}_t)^2 ]$. By Remark \ref{mo}, this equals 
$\sum_{i=1}^N (V_N(i)-\bar V_N)^2 \Et[Z^{i,N}_t]$, which is controled by
$C_\eta ||V_N - \bar V_N \bun||_2^2 e^{(\alpha_0+\eta)t}$ by Lemma \ref{2stoch}-(i).

\vip

For point (iii), we first observe that $X^N_t=Y^N_t - N(\bar M^N_t)^2$, where $Y^N_t=||\bM^N_t||_2^2 -N\bar Z^N_t$.
Using as usual Remark \ref{mo}, we deduce that $\Et[N(\bar M^N_t)^2]=N^{-1}\sum_{i=1}^N \Et[Z^{i,N}_t] \leq
C_\eta e^{(\alpha_0+\eta)t}$ by Lemma \ref{2stoch}-(i). Next, we see that 
$||\bM^N_t||_2^2=\sum_{i=1}^N (M^{i,N}_t)^2=2 \sum_{i=1}^N \intot M^{i,N}_{s-} dM^{i,N}_s + \sum_{i=1}^N Z^{i,N}_t$,
since for each $i$, $[M^{i,N},M^{i,N}]_t=Z^{i,N}_t$, see Remark \ref{mo}. Thus 
$Y^N_t=2\sum_{i=1}^N \intot M^{i,N}_{s-} dM^{i,N}_s$. The martingales $\intot M^{i,N}_{s-} dM^{i,N}_s$ are orthogonal
and $[\int_0^\cdot M^{i,N}_{s-} dM^{i,N}_s,\int_0^\cdot M^{i,N}_{s-} dM^{i,N}_s]_t=\intot (M^{i,N}_{s-})^2 dZ^{i,N}_s
\leq Z^{i,N}_t\sup_{[0,t]} (M^{i,N}_s)^2$.
As a conclusion, 
$$
\Et[(Y^N_t)^2]= 4 \sum_{i=1}^N \Et\Big[\intot (M^{i,N}_{s-})^2 dZ^{i,N}_s\Big]
\leq 4 \sum_{i=1}^N\Et\Big[(Z^{i,N}_t)^2\Big]^{1/2}\Et\Big[\sup_{[0,t]} (M^{i,N}_s)^4\Big]^{1/2}.
$$
Using again that $[M^{i,N},M^{i,N}]_t=Z^{i,N}_t$ and the Doob inequality, we see that 
$\Et[\sup_{[0,t]} (M^{i,N}_s)^4]\leq C\Et[(Z^{i,N}_t)^2]$. This shows that  $\Et[(Y^N_t)^2]\leq C \sum_{i=1}^N
\Et[(Z^{i,N}_t)^2]\leq C_\eta N e^{2(\alpha_0+\eta)t}$ by Lemma \ref{2stoch}-(i).
Hence $\Et[|Y^N_t|]\leq C_\eta \sqrt N 
e^{(\alpha_0+\eta) t}$ and $\Et[|X^N_t|]\leq \Et[|Y^N_t|]+\E[N(\bar M^N_t)^2]\leq C_\eta \sqrt N e^{(\alpha_0+\eta) t}$.

\vip

Finally, (i) follows from (iii), since $\Et[||\bM^N_t - \bar M^N_t \bun||_2^2] \leq \Et[|X^N_t|]+N\Et[\bar Z^N_t]$
and since $N\Et[\bar Z^N_t]\leq C_\eta N e^{(\alpha_0 +\eta)t}$ by Lemma \ref{2stoch}-(i) again.
\end{proof}

\subsection{Conclusion}\label{con2}

We now conclude the proof of Theorem \ref{mr2}. We recall that
$$
\cU^N_t= \Big[\sum_{i=1}^N\Big(\frac{Z^{i,N}_t-\baZ^N_{t}}{\baZ^N_t}\Big)^2 - \frac{N}{\baZ^N_t}\Big]\indiq_{\{\baZ^N_t>0\}}
=\Big[\frac{||\bZ^N_t-\bar Z^N_t\bun||_2^2-N\bar Z^N_t}{(\bar Z^N_t)^2}\Big]\indiq_{\{\baZ^N_t>0\}},
$$
that $V_N$ was introduced in Lemma \ref{rhoV} and that
$$
\cU^N_\infty= \sum_{i=1}^N\Big(\frac{V_N(i)-\bar V_N}{\bar V_N}\Big)^2
=\frac{||V_N -\bar V_N \bun||_2^2}{(\bar V_N)^2}.
$$

We first proceed to a suitable decomposition of the error.

\begin{rk}\label{suitdec}
Assume that $p\in (0,1]$. We introduce $\cD^N_t=|\cU^N_t-(1/p-1)|$ and recall that $v^N_t$ 
was defined in \eqref{2ex3}.
There is $N_0$ (depending only on $p$) 
such that for all $N\geq N_0$, on the event $\Omega_N^2 \cap \{\bar Z^N_t \geq v^N_t / 4>0\}$, 
$$
\cD^N_t \leq  16 \cD^{N,1}_t + 128 ||V_N -\bar V_N \bun||_2^2 \cD^{N,2}_t + |\cU^N_\infty-(1/p-1)|,
$$
where
\begin{align*}
\cD^{N,1}_t = \frac{1}{(v^N_t)^2}\Big|||\bZ^N_t-\bar Z^N_t\bun||_2^2-N\bar Z^N_t- 
(v^N_t)^2||V_N -\bar V_N \bun||_2^2  \Big|\quad\hbox{and}\quad
\cD^{N,2}_t= \Big|\frac{\bar Z^N_t}{v^N_t} - \bar V_N\Big|.
\end{align*}
\end{rk}

\begin{proof}
We work with $N$ sufficiently large so that we can apply Lemma \ref{vrac}.
We obviously have $\cD^N_t\leq |\cU^N_t-\cU^N_\infty| +  |\cU^N_\infty -(1/p-1)|$.
We next write, on the event $\Omega_N^2 \cap \{\bar Z^N_t \geq v^N_t / 4>0\}$,
\begin{align*}
|\cU^N_t-\cU^N_\infty| \leq & \frac{1}{(\bar Z^N_t)^2}\Big|||\bZ^N_t-\bar Z^N_t\bun||_2^2-N\bar Z^N_t- 
(v^N_t)^2||V_N -\bar V_N \bun||_2^2  \Big| \\
&+ ||V_N -\bar V_N \bun||_2^2 \Big|\Big(\frac{v^N_t}{\bar Z^N_t}\Big)^2- \frac1{(\bar V_N)^2}\Big| \\
\leq& 16 \cD^{N,1}_t + 128 ||V_N -\bar V_N \bun||_2^2 \cD^{N,2}_t.
\end{align*}
We used that on the present event, $(\bar Z^N_t)^{-2}\leq 16 (v^N_t)^{-2}$, that
$\bar V_N \geq 1/2$ (see Lemma  \ref{vrac}-(ii)), that $(\baZ^N_t/v^N_t)\geq 1/4$
and that, for all $x,y\geq 1/4$, $|x^{-2}-y^{-2}|\leq 128|x-y|$.
\end{proof}

We now treat the term $\cD^{N,2}_t$.

\begin{lem}\label{2d2}
Assume $A$. For all $\eta>0$, there are $N_\eta \geq 1$, $t_\eta\geq 0$ and $C_\eta>0$ such that,
for all $N\geq N_\eta$, all $t\geq t_\eta$, on $\Omega_N^2$,

\vip

(i) $\Et[\cD^{N,2}_t] \leq C_\eta e^{2\eta t}(N^{-1/2} + e^{-\alpha_0t})$,

\vip

(ii) $\Pr_\theta(\bar Z^N_t \leq v^N_t/4) \leq C_\eta e^{2\eta t}(N^{-1/2} + e^{-\alpha_0t})$.
\end{lem}

\begin{proof}
As usual, we fix $\eta>0$ and consider $N\geq N_\eta$ and $t\geq t_\eta$ and we work on $\Omega_N^2$
so that we can apply Lemmas \ref{vrac} and \ref{list}.

\vip

Recalling \eqref{2ex1}-\eqref{2ex2}, we write $\bZ^N_t=\Et[\bZ^N_t]+\bU^N_t=v^N_tV_N+\bI^N_t+\bU^N_t$, whence
$\cD^{N,2}_t\leq (v^N_t)^{-1}(|\bar I^N_t|+|\bar U^N_t|)$. But we infer from
Lemma \ref{2I} that $|\bar I^N_t|\leq N^{-1/2}||\bI^N_t||_2 \leq C_\eta t N^{1/8-1/2}$, which is obviously
bounded by $C_\eta e^{\eta t}$. Next, we know from Lemma \ref{2stoch}-(iii) that
$\Et[|\bar U^N_t|]\leq C_\eta N^{-1/2}e^{(\alpha_0+\eta) t}$. We deduce that
$\Et[\cD^{N,2}_t] \leq C_\eta (v^N_t)^{-1}e^{\eta t}[1+ N^{-1/2}e^{\alpha_0 t}]$.
But since $t\geq t_\eta$, we know from Lemma \ref{list}-(ii) that $v^N_t \geq c_\eta e^{(\alpha_0-\eta)t}$.
This completes the proof of (i).

\vip

By Lemma \ref{vrac}-(ii), $\bar V_N\geq 1/2$. Thus $\bar Z^N_t \leq v^N_t/4$
implies that $\cD^{N,2}_t=|\bar Z^N_t/v^N_t -\bar V_N|\geq 1/4$. Hence 
$\Pr_\theta(\bar Z^N_t \leq v^N_t/4) \leq 4\Et[\cD^{N,2}_t]$ and 
(ii) follows from (i).
\end{proof}

\begin{lem}\label{2d1}
Assume $A$. For all $\eta>0$, there are $N_\eta \geq 1$, $t_\eta\geq 0$ and $C_\eta>0$ such that,
for all $N\geq N_\eta$, all $t\geq t_\eta$, on $\Omega_N^2$,
$$
\Et[\cD^{N,1}_t]\leq C_\eta(1+ ||V_N -\bar V_N \bun||_2^2 )e^{4\eta t}
\Big(\frac 1 {\sqrt N}+ \frac{\sqrt N}{e^{\alpha_0 t}}
+ \Big(\frac{\sqrt N}{e^{\alpha_0 t}}\Big)^{3/2}\Big).
$$
\end{lem}

\begin{proof}
We fix $\eta>0$ and consider $N\geq N_\eta$ and $t\geq t_\eta$ and we work on $\Omega_N^2$
so that we can apply Lemmas \ref{vrac} and \ref{list}.
Recalling \eqref{2ex1}-\eqref{2ex2}, we write $\bZ^N_t=v^N_tV_N + \bI^N_t+M^N_t+\bJ^N_t$ and
\begin{align*}
\cD^{N,1}_t =  \frac{1}{(v^N_t)^2}&\Big|||\bI^N_t-\bar I^N_t\bun+\bJ^N_t-\bar J^N_t\bun ||_2^2 + 
||\bM^N_t-\bar M^N_t\bun||_2^2
-N\bar Z^N_t\\
&+ 2 (\bI^N_t-\bar I^N_t\bun+\bJ^N_t-\bar J^N_t\bun, v^N_t(V_N-\bar V_N\bun)+\bM^N_t-\bar M^N_t\bun )\\
&+2v^N_t(V_N-\bar V_N\bun,\bM^N_t-\bar M^N_t\bun ) \Big|.
\end{align*}
Recalling that $X^N_t= ||\bM^N_t-\bar M^N_t\bun||_2^2-N\bar Z^N_t$, see Lemma \ref{2M}, we deduce that
\begin{align*}
\cD^{N,1}_t \leq \frac{1}{(v^N_t)^2}&\Big| 2||\bI^N_t-\bar I^N_t\bun||_2^2+2||\bJ^N_t-\bar J^N_t\bun ||_2^2 + |X^N_t|\\
&+ 2 (||\bI^N_t-\bar I^N_t\bun||_2+||\bJ^N_t-\bar J^N_t\bun||_2)(v^N_t||V_N-\bar V_N\bun||_2+
||\bM^N_t-\bar M^N_t\bun||_2 )\\
&+2v^N_t|(V_N-\bar V_N\bun,\bM^N_t-\bar M^N_t\bun )| \Big|.
\end{align*}
We know from Lemma \ref{list}-(i)-(ii) that $v^N_t\geq c_\eta e^{(\alpha_0-\eta)t}$ and 
$v^N_t \leq C_\eta e^{(\alpha_0+\eta)t}$, from Lemma \ref{2I}
that
$||\bI^N_t-\bar I^N_t\bun||_2\leq ||\bI^N_t||_2 \leq C_\eta N^{1/8}t\leq C_\eta N^{1/8}e^{\eta t}$ and from Lemma \ref{2J}
that $\Et[||\bJ^N_t-\bar J^N_t\bun ||_2^2]\leq C_\eta [e^{(\alpha_0+\eta) t}+ N^{-1}||V_N-\bar V_N \bun||_2^2
e^{2 (\alpha_0 +\eta) t}]$.
And Lemma \ref{2M} tells us that $\Et[|X^N_t|]\leq C_\eta \sqrt N 
e^{(\alpha_0+\eta) t}$, $\Et[||\bM^N_t - \bar M^N_t \bun||_2^2] \leq C_\eta N e^{(\alpha_0+\eta) t}$ and
$\Et[|(\bM^N_t - \bar M^N_t \bun,V_N-\bar V_N\bun)|]\leq C_\eta ||V_N - \bar V_N \bun||_2 e^{(\alpha_0+\eta)t/2}$.
Using the Cauchy-Schwarz inequality, we find
\begin{align*}
\Et[\cD^{N,1}_t]\leq& C_\eta(1 \!+\! ||V_N \!-\!\bar V_N \bun||_2^2 )e^{-2(\alpha_0-\eta) t}\Big(
N^{1/4}e^{2\eta t} + [e^{(\alpha_0+\eta)t}+N^{-1}e^{2(\alpha_0+\eta) t}] + N^{1/2}e^{(\alpha_0+\eta)t}\\
&\hskip0.5cm+ [N^{1/8}e^{\eta t} +e^{(\alpha_0+\eta)t/2}+N^{-1/2}e^{(\alpha_0+\eta) t} ][e^{(\alpha_0+\eta) t}
+N^{1/2} e^{(\alpha_0+\eta) t/2}] + e^{3(\alpha_0+\eta) t/2}\Big).
\end{align*}
We easily deduce that 
\begin{align*}
\Et[\cD^{N,1}_t]\leq& C_\eta(1 + ||V_N -\bar V_N \bun||_2^2 )e^{4\eta t}\Big(
N^{-1/2}+e^{-\alpha_0 t/2}+ N^{1/2}e^{-\alpha_0 t}+ N^{5/8}e^{-3\alpha_0 t/2}\Big).
\end{align*}
To conclude, it suffices to notice that $e^{-\alpha_0 t/2}\leq N^{-1/2}+N^{1/2}e^{-\alpha_0 t}$
and that $N^{5/8}e^{-3\alpha_0 t/2}\leq N^{3/4}e^{-3\alpha_0 t/2}=(N^{1/2}e^{-\alpha_0 t})^{3/2}$.
\end{proof}

We now have all the weapons to give the

\begin{proof}[Proof of Theorem \ref{mr2}]
We assume $A$ and fix $\eta>0$.

\vip

{\it Step 1.} Starting from Remark \ref{suitdec} and using Lemmas \ref{2d2} and \ref{2d1},
we deduce that there is $N_\eta\geq 1$, $t_\eta\geq 0$ and $C_\eta>0$ such that
for all $N\geq N_\eta$, all $t\geq t_\eta$,
\begin{align*}
\indiq_{\Omega_N^2}\Et\Big[\indiq_{\{\bar Z^N_t\geq v^N_t/4>0\}}\Big|\cU^N_t-&\Big(\frac 1p-1\Big)\Big|\Big]
\leq \indiq_{\Omega_N^2}\Big|\cU^N_\infty-\Big(\frac 1p-1\Big)\Big|\\
&+C_\eta \indiq_{\Omega_N^2}(1 + ||V_N -\bar V_N \bun||_2^2 )e^{4\eta t}
\Big(\frac 1 {\sqrt N}+ \frac{\sqrt N}{e^{\alpha_0 t}}+ \Big(\frac{\sqrt N}{e^{\alpha_0 t}}\Big)^{3/2}\Big),
\end{align*}
which implies, by Proposition \ref{basis}, that
\begin{align*}
\E\Big[\indiq_{\Omega_N^2 \cap \{\bar Z^N_t\geq v^N_t/4>0\}}\Big|\cU^N_t-&\Big(\frac 1p-1\Big)\Big|\Big]
\leq C_\eta e^{4\eta t}
\Big(\frac 1 {\sqrt N}+ \frac{\sqrt N}{e^{\alpha_0 t}}+ \Big(\frac{\sqrt N}{e^{\alpha_0 t}}\Big)^{3/2}\Big)
\end{align*}
and thus, for all $\e\in (0,1)$,
\begin{align*}
\Pr\Big(\Omega_N^2, \bar Z^N_t\geq v^N_t/4>0, \Big|\cU^N_t-\Big(\frac 1p-1\Big)\Big|\geq \e \Big)
\leq& \frac{C_\eta}\e e^{4\eta t}
\Big(\frac 1 {\sqrt N}+ \frac{\sqrt N}{e^{\alpha_0 t}}+ \Big(\frac{\sqrt N}{e^{\alpha_0 t}}\Big)^{3/2}\Big)\\
\leq & \frac{C_\eta}\e e^{4\eta t}
\Big(\frac 1 {\sqrt N}+ \frac{\sqrt N}{e^{\alpha_0 t}}\Big).
\end{align*}
For the last inequality, we used that either $N^{1/2}e^{-\alpha_0 t}\geq 1$ and then the inequality is trivial
or $N^{1/2}e^{-\alpha_0 t}\leq 1$ and then $(N^{1/2}e^{-\alpha_0 t})^{3/2}\leq N^{1/2}e^{-\alpha_0 t}$.
But we know from Lemma \eqref{list}-(ii) that $v^N_t>0$ on $\Omega_N^2$ (because $t\geq t_\eta$) and from 
Lemmas \ref{pome} and \ref{2d2}-(ii) that
\begin{align*}
\Pr((\Omega_N^2)^c \hbox{ or } \bar Z^N_t< v^N_t/4) \leq  
Ce^{-cN^{1/4}} + C_\eta e^{2\eta t}\Big(\frac 1 {\sqrt N}+ e^{-\alpha_0 t}\Big)
\leq  C_\eta e^{2\eta t}\Big(\frac 1 {\sqrt N}+ \frac {\sqrt N}{e^{\alpha_0 t}}\Big),
\end{align*}
whence finally,
\begin{align*}
\Pr\Big(\Big|\cU^N_t-\Big(\frac 1p-1\Big)\Big|\geq \e \Big)
\leq &\frac{C_\eta}\e e^{4\eta t}\Big(\frac 1 {\sqrt N}+ \frac{\sqrt N}{e^{\alpha_0 t}}\Big).
\end{align*}
We have proved this inequality only for $N\geq N_\eta$ and $t\geq t_\eta$, but it obviously
extends, enlarging $C_\eta$ is necessary, to any values of $N\geq 1$ and $t\geq 0$.

\vip

{\it Step 2.} We next recall that $\cP^N_t=\Phi(\cU^N_t)$, where $\Phi(u)=(1+u)^{-1}\indiq_{\{u\geq 0\}}$
and observe that $p=\Phi(1/p-1)$.
The function $\Phi$ is Lipschitz continuous on $[0,\infty)$ with Lipschitz constant $1$.
Thus for any $\e\in (0,1)$, $|\cP^N_t-p|>\e$ implies that either $|\cU^N_t-(1/p-1)|>\e$
or $\cU^N_t<0$, so that in any case, $|\cU^N_t-(1/p-1)|>\min\{\e,(1/p-1)\}$. We thus conclude
from Step 1 that for any $N\geq 1$, any $t\geq 0$, any $\e \in (0,1)$,
\begin{align*}
\Pr(|\cP^N_t-p|\geq \e ) \leq &\frac{C_\eta}{\min\{\e,(1/p-1)\}}
e^{4\eta t}\Big(\frac 1 {\sqrt N}+ \frac{\sqrt N}{e^{\alpha_0 t}}\Big)
\leq \frac{C_\eta}\e
e^{4\eta t}\Big(\frac 1 {\sqrt N}+ \frac{\sqrt N}{e^{\alpha_0 t}}\Big).
\end{align*}
The proof is complete.
\end{proof}

Finally, we handle the

\begin{proof}[Proof of Remark \ref{rkaz}]
We assume $A$ and fix $\eta>0$. We know from Lemma \ref{2d2}-(i) that for all
$N\geq N_\eta$, $t\geq t_\eta$, 
$\indiq_{\Omega_N^2}\Et[|(\bar Z^N_t/v^N_t)- \bar V_N|]\leq C_\eta e^{2\eta t}(N^{-1/2}+e^{-\alpha_0 t})$,
from Lemma \ref{vrac}-(ii) that $\bar V_N \in [1/2,2]$ (on $\Omega_N^2$)
and from Lemma \ref{pome} that $\Pr(\Omega_N^2)\geq 1-Ce^{-cN^{1/4}}$.
We also know from Lemma \ref{list}-(i)-(ii) that on $\Omega_N^2$,
there are $0<a_\eta<b_\eta$ such that $v^N_t \in [a_\eta e^{(\alpha_0-\eta)t}, b_\eta e^{(\alpha_0+\eta)t}]$.
We easily deduce that, still for $N\geq N_\eta$ and $t\geq t_\eta$,
\begin{align*}
\Pr(\bar Z^N_t \notin [(a_\eta/2) e^{(\alpha_0-\eta)t}, 2b_\eta  e^{(\alpha_0+\eta)t}])
\leq Ce^{-cN^{1/4}} + C_\eta e^{2\eta t}(N^{-1/2}+e^{-\alpha_0 t}).
\end{align*}
We conclude that for any $\eta \in (0,\alpha_0/2)$, 
$\lim_{t\to \infty} \lim_{N\to \infty} \Pr(\bar Z^N_t \in [(a_\eta/2) e^{(\alpha_0-\eta)t}, 2b_\eta  e^{(\alpha_0+\eta)t}])=1$.
This of course implies that for any $\eta>0$, 
$\lim_{t\to \infty} \lim_{N\to \infty} \Pr(\bar Z^N_t \in [e^{(\alpha_0-\eta)t},e^{(\alpha_0+\eta)t}])=1$.
\end{proof}

\section{Detecting subcriticality and supercriticality}\label{htd}

\begin{proof}[Proof of Proposition \ref{decid}]
We first assume $H(1)$. We then know from Lemma \ref{if}-(i) with $r=1$ that, on $\Omega_N^1$,
$\Et[\bar Z^N_t]\leq C t < e^{(\log t)^2}/2$ for all $t$ large enough, say for all $t\geq t_0$.
We also know from Lemma \ref{esti1lem} that, still on $\Omega_N^1$, $\Et[|\bar Z^N_t -\Et[\bar Z^N_t]|]
=\Et[|\bar U^N_t|]\leq C (t/N)^{1/2}$ and from Lemma \ref{pom} that $\Pr[(\Omega_N^1)^c]\leq Ce^{-cN}$.
We easily deduce that
\begin{align*}
\Pr(\log(\bar Z^N_t) \geq (\log t)^2) \leq& \Pr((\Omega_N^1)^c)+\Pr(\Omega_N^1, \Et[\bar Z^N_t] \geq 
e^{(\log t)^2}/2 \; \hbox{or}\;|\bar Z^N_t -\Et[\bar Z^N_t]|\geq  e^{(\log t)^2}/2 )\\
\leq & Ce^{-c N}+ C (t/N)^{1/2}e^{-(\log t)^2}
\end{align*}
for all $t\geq t_0$. Enlarging $C$ if necessary, we conclude that for all $t\geq 1$,
$ \Pr(\log(\bar Z^N_t) \geq (\log t)^2) \leq  Ce^{-c N}+ C t^{1/2}e^{-(\log t)^2}$.

\vip

We next assume $A$ and we fix $\eta \in (0,\alpha_0)$. 
We know from Lemma \ref{2d2}-(ii) that for all $N\geq N_\eta$ and $t\geq t_\eta$, 
on $\Omega_N^2$, $\Pr_\theta(\bar Z^N_t \leq v^N_t/4) \leq C_\eta e^{2\eta t}(N^{-1/2}+e^{-\alpha_0 t})$,
from Lemma \ref{list}-(i)-(ii) that, still on  $\Omega_N^2$, $v^N_t \geq c_\eta e^{(\alpha_0-\eta)t}
\geq 4 e^{(\log t)^2}$ (enlarging the value of $t_\eta$ if necessary). Finally, 
Lemma \ref{pome} tells us that $\Pr((\Omega_N^2)^c)\leq C e^{-cN^{1/4}}$. We thus see that
\begin{align*}
\Pr(\log(\bar Z^N_t) \leq (\log t)^2) \leq& \Pr((\Omega_N^2)^c)+\Pr(\Omega_N^2, \bar Z^N_t \leq v^N_t /4)\\
\leq&  Ce^{-c N^{1/4}}+ C_\eta e^{2\eta t}(N^{-1/2}+e^{-\alpha_0 t})\\
\leq & C_\eta e^{2\eta t}(N^{-1/2}+e^{-\alpha_0 t}).
\end{align*}
All this shows that for all $\eta \in (0,\alpha_0)$,
we can find $C_\eta$ and $t\eta$ such that for
 all $t\geq t_\eta$ and all $N\geq N_\eta$,
$\Pr(\log(\bar Z^N_t) \leq (\log t)^2)\leq C_\eta e^{2\eta t}(N^{-1/2}+e^{-\alpha_0 t})$.
We easily conclude that for all $\eta>0$, there is $C_\eta$ such that for all $N\geq 1$ and all $t\geq 1$,
$\Pr(\log(\bar Z^N_t) \leq (\log t)^2) \leq C_\eta e^{2\eta t}(N^{-1/2}+e^{-\alpha_0 t})$ as desired.
\end{proof}

\section{Numerics}\label{num}

We say that we are in the {\it independent case} when the family $(\theta_{ij})_{1\leq i,j\leq N}$ is i.i.d.
and Bernoulli$(p)$-distributed, as in the whole paper. We say we are in the 
{\it symmetric case} when the family $(\theta_{ij})_{1\leq i\leq j\leq N}$ is i.i.d.
and Bernoulli$(p)$-distributed and when $\theta_{ji}=\theta_{ij}$ for all $1\leq i< j\leq N$.
We will see that this does not change much the numerical results (with the very same estimators).
Also, we assume that we observe only  $(Z^{i,N}_s)_{s \in [0,T], i=1,\dots,K}$ for some (large) $K$ smaller than $N$.
The theoretical results of this paper only apply when $K=N$. We adapt the estimators as follows.
We introduce $\bar Z^{N,K}_t=K^{-1}\sum_{i=1}^K Z^{i,N}_t$ and
$$
\cE^{N,K}_t= \frac{\baZ^{N,K}_{2t} - \baZ^{N,K}_t}{t}, \qquad
\cV^{N,K}_t= \frac N K\sum_{i=1}^K \Big(\frac{Z^{i,N}_{2t}-Z^{i,N}_t}{t} - \cE^{N,K}_t \Big)^2 
- \frac N t \cE^{N,K}_t,
$$
$$
\cW^{N,K}_{\Delta,t} = 2 \cZ^{N,K}_{2\Delta,t}-\cZ^{N,K}_{\Delta,t}, \quad \hbox{where} 
\quad \cZ^{N,K}_{\Delta,t}=\frac Nt \sum_{k=t/\Delta+1}^{2t/\Delta}\Big(\baZ^{N,K}_{k \Delta} - \baZ^{N,K}_{(k-1) \Delta}  
-\Delta \cE^{N,K}_t\Big)^2,
$$
as well as
$$
\cP^{sub,N,K}_{\Delta,T}=\Phi_3\Big(\cE^{N,K}_{T/2},\cV^{N,K}_{T/2}, \Big|\cW^{N,K}_{\Delta,T/2} 
- \frac{N-K}K\cE^{N,K}_{T/2} \Big| \Big), 
$$
with $\Phi$ defined in Corollary \ref{mc}. We added the absolute value around the last argument of
$\Phi_3$ for practical reasons: by this way, $\cP^{sub,N,K}_{\Delta,T}$ is always well-defined
(and seems closer to the reality than $\Psi_3$ which is $0$ when $w\leq 0$).
This does not change the theory (since $\cW^{N,K}_{\Delta,T/2} - \frac{N-K}K\cE^{N,K}_{T/2}$
is asymptoticaly positive, at least when $N=K$).
We also put
\begin{align*}
\cU^{N,K}_T= \Big[ \frac N K\sum_{i=1}^K\Big(\frac{Z^{i,N}_T-\baZ^{N,K}_{T}}{\baZ^{N,K}_T}\Big)^2 - \frac{N}{\baZ^{N,K}_T} 
\Big]\indiq_{\{\baZ^{N,K}_T>0\}}
\quad \hbox{and}
\quad \cP^{sup,N,K}_T=\frac{1}{\cU^{N,K}_T+1}\indiq_{\{\cU^{N,K}_T\geq 0\}}.
\end{align*}
We set $\Delta_t=t/(2\lfloor t^{9/13}\rfloor)$, 
which corresponds to the (quite arbitrary) choice $q=12$, and
$$
\hat p^{N,K}_T=\cP^{sub,N,K}_{\Delta_T,T}\indiq_{\{\log (\baZ^{N,K}_T) < (\log T)^2 \}}+ \cP^{sup,N,K}_T\indiq_{\{\log (\baZ^{N,K}_T) 
> (\log T)^2\}}.
$$

\subsection{Choice of the estimators}\label{sssR}
Let us explain briefly how we have modified the estimators when observing only  $(Z^{i,N}_s)_{s \in [0,T], i=1,\dots,K}$.
We adopt the notation of Section \ref{heu}, in particular $A_N(i,j)=N^{-1}\theta_{ij}$,
and we follow the considerations therein.

\vip

In the subcritical case, we recall that $Q_N=(I-\Lambda A_N)^{-1}$ and that 
$\ell_N(i)=\sum_{j=1}^N Q_N(i,j)$.
Following closely the argumentation of Subsection \ref{heusou}, we
expect that, for $t$ (and $\Delta$) large and in a suitable regime,
we should have $\cE^{N,K}_t\simeq \mu \bar \ell_N^K$ (where $\bar \ell_N^K=K^{-1}\sum_{i=1}^K \ell_N(i)$),
$\cV^{N,K}_t\simeq \mu^2 (N/K) \sum_{i=1}^K (\ell_N(i)-\bar \ell_N^K)^2$,
and $\cW^{N,K}_{\Delta,t}\simeq \mu (N/K^2)\sum_{j=1}^N (\sum_{i=1}^K Q_N(i,j))^2\ell_N(j)$.
Recalling now that $\ell_N(i)\simeq 1+ \Lambda(1-\Lambda p)^{-1}L_N(i)$ and that 
$NL_N$ is a vector composed of i.i.d. Binomial$(N,p)$ random variables, we expect that
$\cE^{N,K}_t\simeq \mu\E[\ell_N(1)]\simeq \mu/(1-\Lambda p)$ and 
$\cV^{N,K}_t\simeq \mu^2 N \Var(\ell_N(1))\simeq \mu^2\Lambda^2 p(1-p)/(1-\Lambda p)^2$ for $N$, $K$ and $t$
large.
For the last estimator, one first has to get convinced, following again the
arguments of Subsection \ref{heusou}, that 
$\sum_{i=1}^K Q_N(i,j)\simeq 1+ (K/N) \Lambda p/(1-\Lambda p)$ if $j\in\{1,\dots,K\}$ while 
$\sum_{i=1}^K Q_N(i,j)\simeq (K/N) (\Lambda p/(1-\Lambda p))$ if $j\in\{K+1,\dots,N\}$. 
Since we still have $\ell_N(j)\simeq 1+\Lambda p/(1-\Lambda p)=1/(1-\Lambda p)$, we find that
$$
\cW^{N,K}_{\Delta,t}\simeq \frac{\mu N}{K^2(1-\Lambda p)}\Big(K \Big[1+\frac{K\Lambda p}{N(1-\Lambda p)}\Big]^2
+(N-K)\Big[\frac{K\Lambda p}{N(1-\Lambda p)}\Big]^2 \Big)
= \frac{\mu}{(1-\Lambda p)^3}+ \frac{(N-K)\mu}{K(1-\Lambda p)}.
$$
Recalling that $\cE^{N,K}_t\simeq \mu/(1-\Lambda p)$, we conclude that
$\cW^{N,K}_{\Delta,t} - (N-K) \cE^{N,K}_t/K \simeq \mu/(1-\Lambda p)^3$. 
For $N$, $K$, $t$ and $\Delta$ large, we thus should have 
$$
\Phi_3\Big(\cE^{N,K}_{t},\cV^{N,K}_{t}, \Big|\cW^{N,K}_{\Delta,t} - \frac{N-K}K\cE^{N,K}_{t}\Big|\Big)\simeq
\Phi_3\Big(\frac\mu{1-\Lambda p},\frac{\mu^2\Lambda^2 p(1-p)}{(1-\Lambda p)^2}, \frac \mu{(1-\Lambda p)^3}\Big)=p.
$$
We introduce the conjectured limit of $\cP^{sub,N,K}_{\Delta_t,t}$ as $t\to \infty$:
$$
\cP^{sub,N,K}_{\infty,\infty}=\Phi_3\Big(\mu \bar \ell_N^K, \frac{\mu^2N}K \sum_{i=1}^K (\ell_N(i)-\bar \ell_N^K)^2,
\Big|\frac{\mu N}{K^2}\sum_{j=1}^N (\sum_{i=1}^K Q_N(i,j))^2\ell_N(j)-\frac{N-K}K \mu \bar \ell_N^K\Big|
\Big).
$$

In the supercritical case, we follow Subsection \ref{heusur} and deduce that for $t$ large, we should have
$\cU^{N,K}_t\simeq (N/K)(\bar V^K_N)^{-2}\sum_{i=1}^K(V_N(i)-\bar V_N^K)^2$, where $V_N$ is the Perron-Frobenius
eigenvector of $A_N$ and where $\bar V^K_N=K^{-1}\sum_{i=1}^K V_N(i)$.
Recalling now that $V_N$ is almost colinear to $L_N$ and that 
$NL_N$ is a vector composed of i.i.d. Binomial$(N,p)$ random variables, we conclude that indeed,
it should hold that $\cU^{N,K}_t\simeq N (\E[V_N(1)])^{-2}\Var (V_N(1))\simeq 1/p-1$ for $N$, $K$ and $t$ large,
whence $\cP^{sup,N,K}_t\simeq p$. Here we introduce the conjectured
limit of $\cP^{sup,N,K}_{t}$ as $t\to \infty$:
$$
\cP^{sup,N,K}_{\infty}= \Big(1 + \frac N{ K (\bar V^K_N)^{2}}\sum_{i=1}^K(V_N(i)-\bar V_N^K)^2\Big)^{-1}.
$$

\subsection{Numerical results}
From now on, we assume that $\varphi(t)=a\exp(-bt)$ for some $a,b>0$, which satisfies all
our assumptions and is easy to simulate.
We also always assume that $a=2$ and $b=1$, whence $\Lambda=2$.
We did not find interesting different behaviors when using other values.

\vip

On all the pictures below, we plot the time evolution of the three quartiles, using $1000$ simulations, of
$\hat p^{N,K}_{t}-p$, as a function of time $t\in [0,T]$. We always choose $T$
in such a way that $\bar Z^N_T\simeq 3000$, so that on the right of all the pictures below,
we always have more or less the same quantity of data (for a given value of $K$).
The curves are not smooth because we use only $1000$ simulations, but this already takes a lot of time.

\vip

For a given simulation, we say that the choice is {\it good}
when $\hat p^{N,K}_{t}=\cP^{sub,N,K}_{\Delta_t,t}$ and $\Lambda p<1$ or $\hat p^{N,K}_{t}=\cP^{sup,N,K}_{t}$ and 
$\Lambda p>1$. 
When the choice is almost always good (that is, for a large proportion of the simulations),
we also indicate below the picture the three quartiles of $\hat p^{N,K}_{\infty}-p$,
where $\hat p^{N,K}_\infty$ is the conjectured limit as $t\to \infty$ of
$\hat p^{N,K}_{t}$, given by  $\hat p^{N,K}_\infty=\cP^{sub,N,K}_{\infty,\infty}$ when $\Lambda p<1$
and $\hat p^{N,K}_\infty=\cP^{sup,N,K}_{\infty}$ when $\Lambda p>1$.

\vip

We start with the independent case. As the pictures below show, 
the estimation of $p$ is more precise in the fairly supercritical case.
Also, around the critical case, $\hat p^{N,K}_{t}$ is far from always making the good choice,
but this does not, however, produce too bad results.
On the contrary, the estimation of $p$ when $p$ is very small does not work very well. 

\vip

Observe that the results with $N=K=1000$ are most often not better than those with $N=K=250$.
This is not so surprising for a given value of $T$, since our rate of convergence
resembles $T^{-1} N^{1/2}+N^{-1/2}$ in the subcritical case and something similar in the 
supercritical case.

\vip

Finally, in all the trials below, it seems that $|\hat p^{N,K}_{\infty}-p|$ is much smaller
than $|\hat p^{N,K}_{t}-p|$.

\vip

\noindent\fbox{\begin{minipage}{\textwidth}
{\small Independent case, $p=0.85$, $\mu=1$ (fairly supercritical). 
The choice is always good for $t\in [1,9.7]$.}
\\
\includegraphics[width=4.9cm]{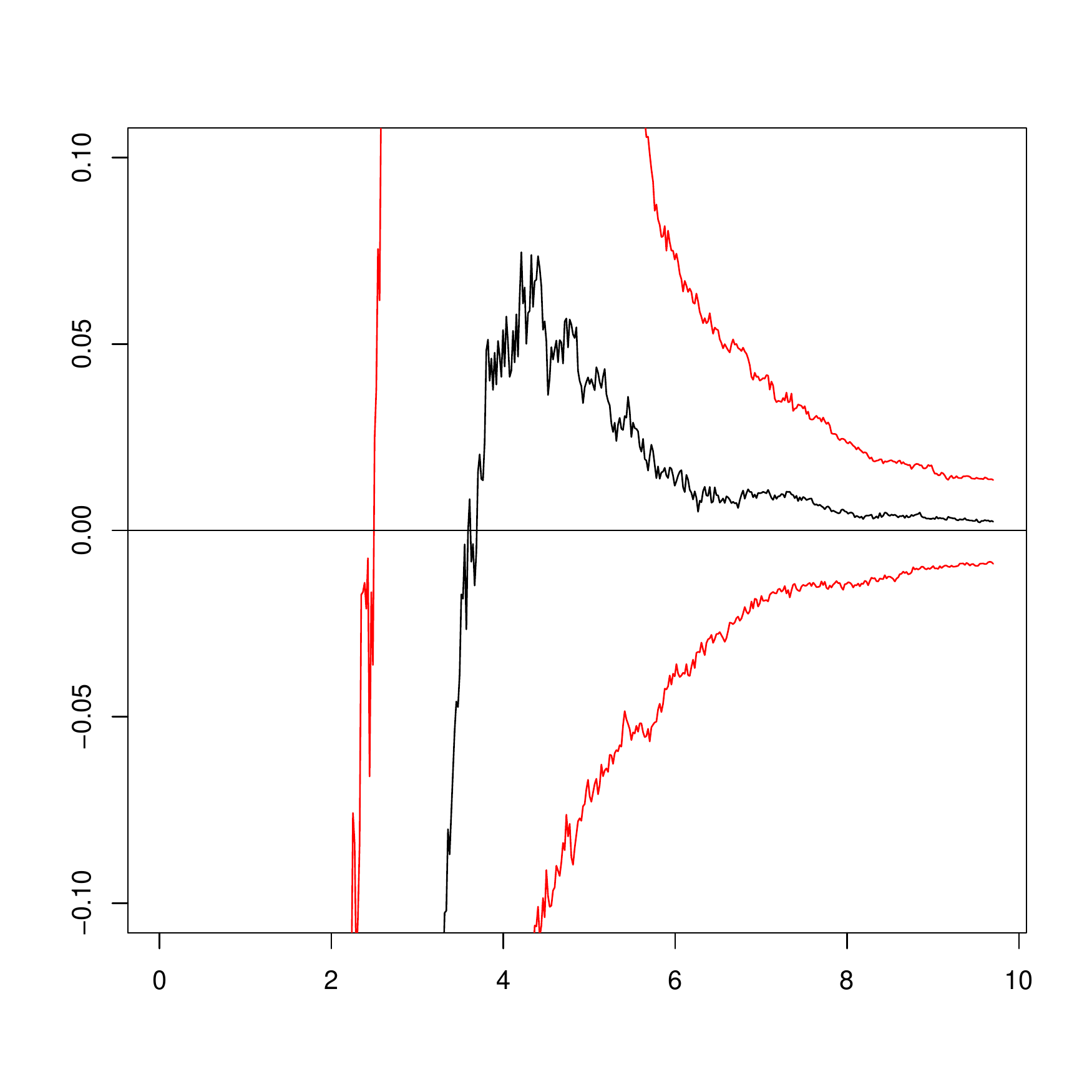}
\includegraphics[width=4.9cm]{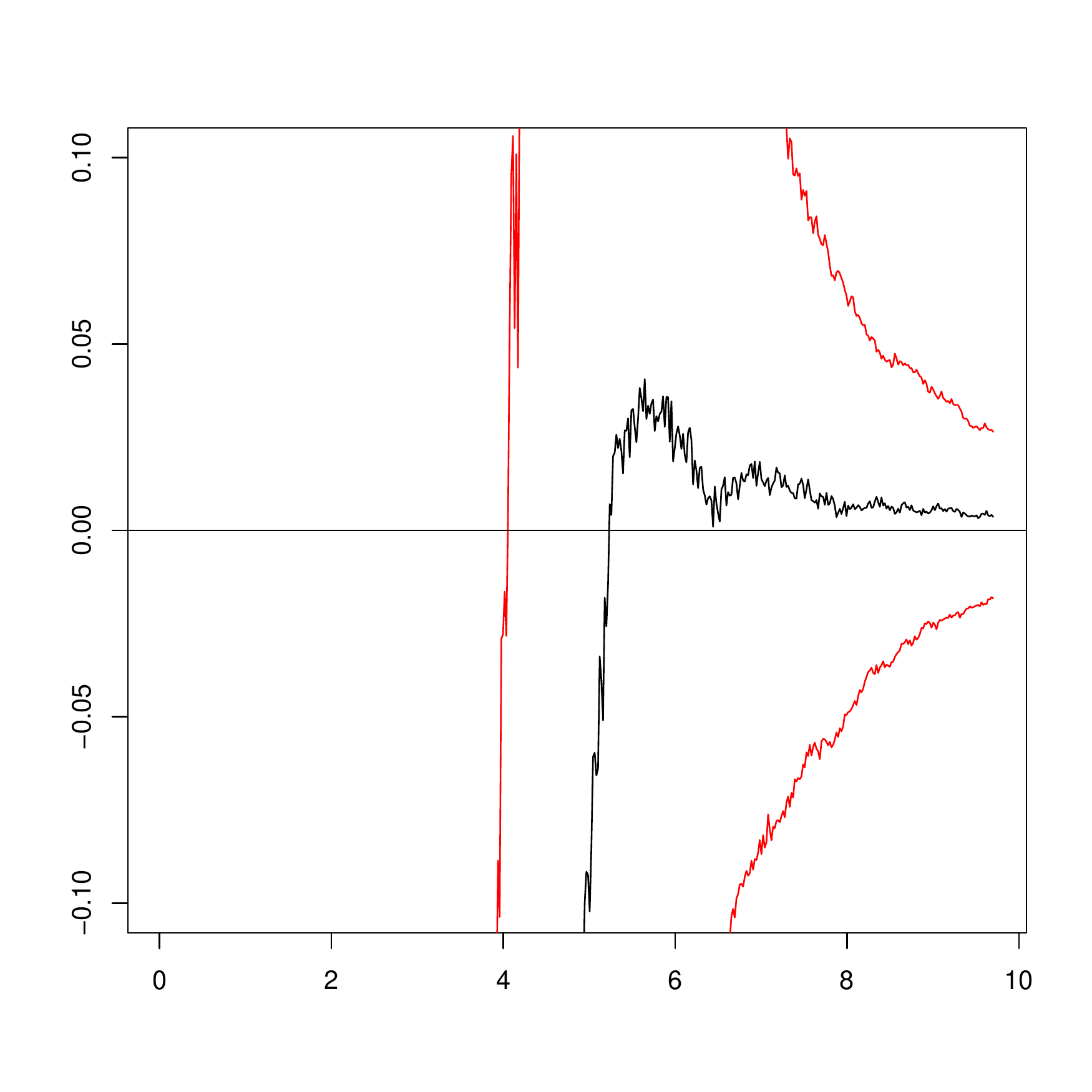}
\includegraphics[width=4.9cm]{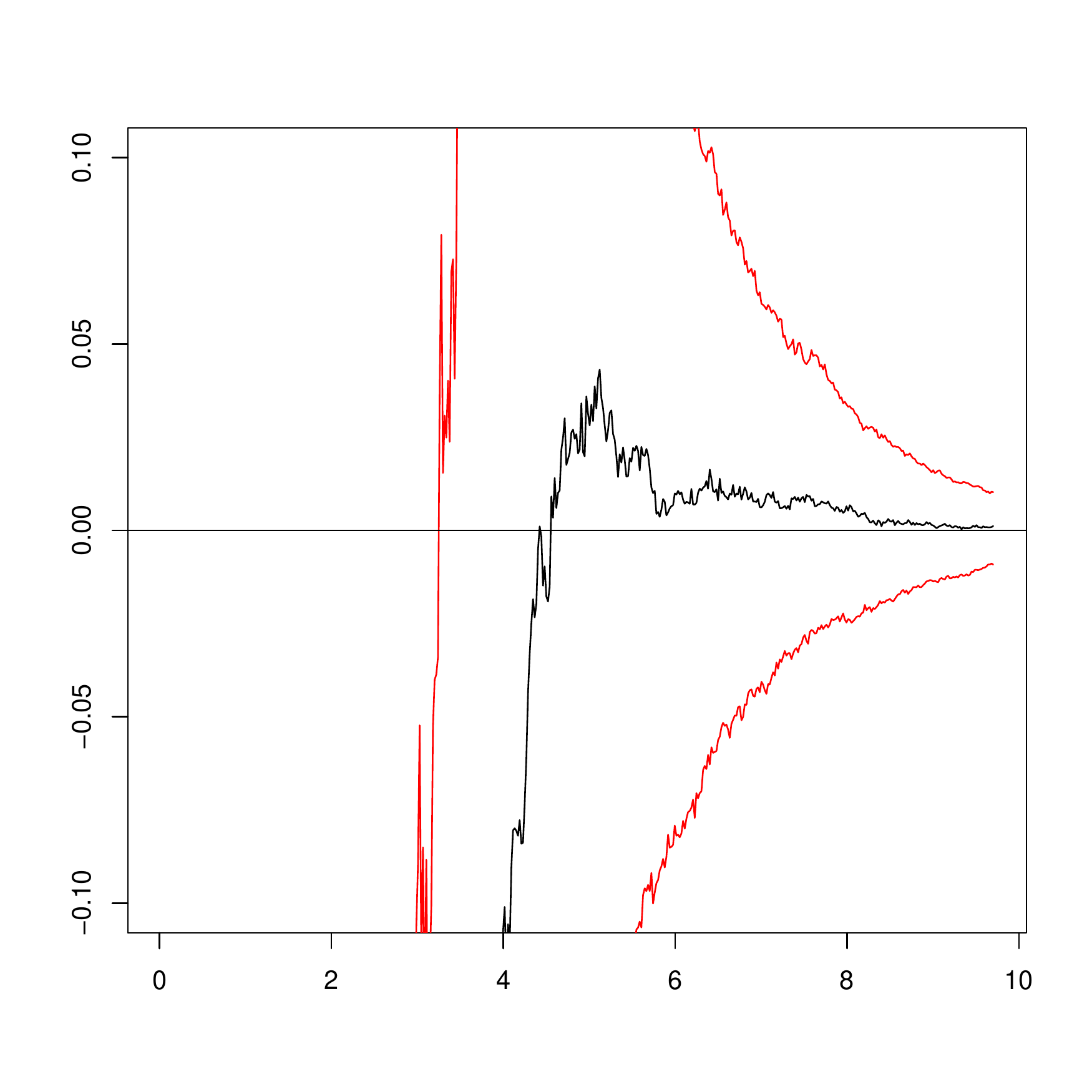}
\\ \vskip-4.9cm  \hskip3.5cm 
{\tiny $N=250$} \hskip3.8cm {\tiny $N=1000$} \hskip3.8cm {\tiny $N=1000$}
\\ \vskip-.6cm \hskip3.5cm 
{\tiny $K=250$} \hskip3.8cm {\tiny $K=250$} \hskip3.9cm {\tiny $K=1000$}
\\ \vskip3cm
\hskip0.9cm {\small $-0.0073,0.00097,0.0091$} \hskip1.6cm{\small $-0.0067,0.0011,0.0080$}\hskip1.6cm
{\small $-0.0038,0.000086,0.0041$}
\\ 
{\small These pictures illustrate that this situation (fairly supercritical) is quite favorable.}
\end{minipage}}

\vip

\noindent\fbox{\begin{minipage}{\textwidth}
{\small Independent case, $p=0.65$, $\mu=1$ (supercritical). 
The choice is always bad for $t\in [14,19]$.}
\\
\includegraphics[width=4.9cm]{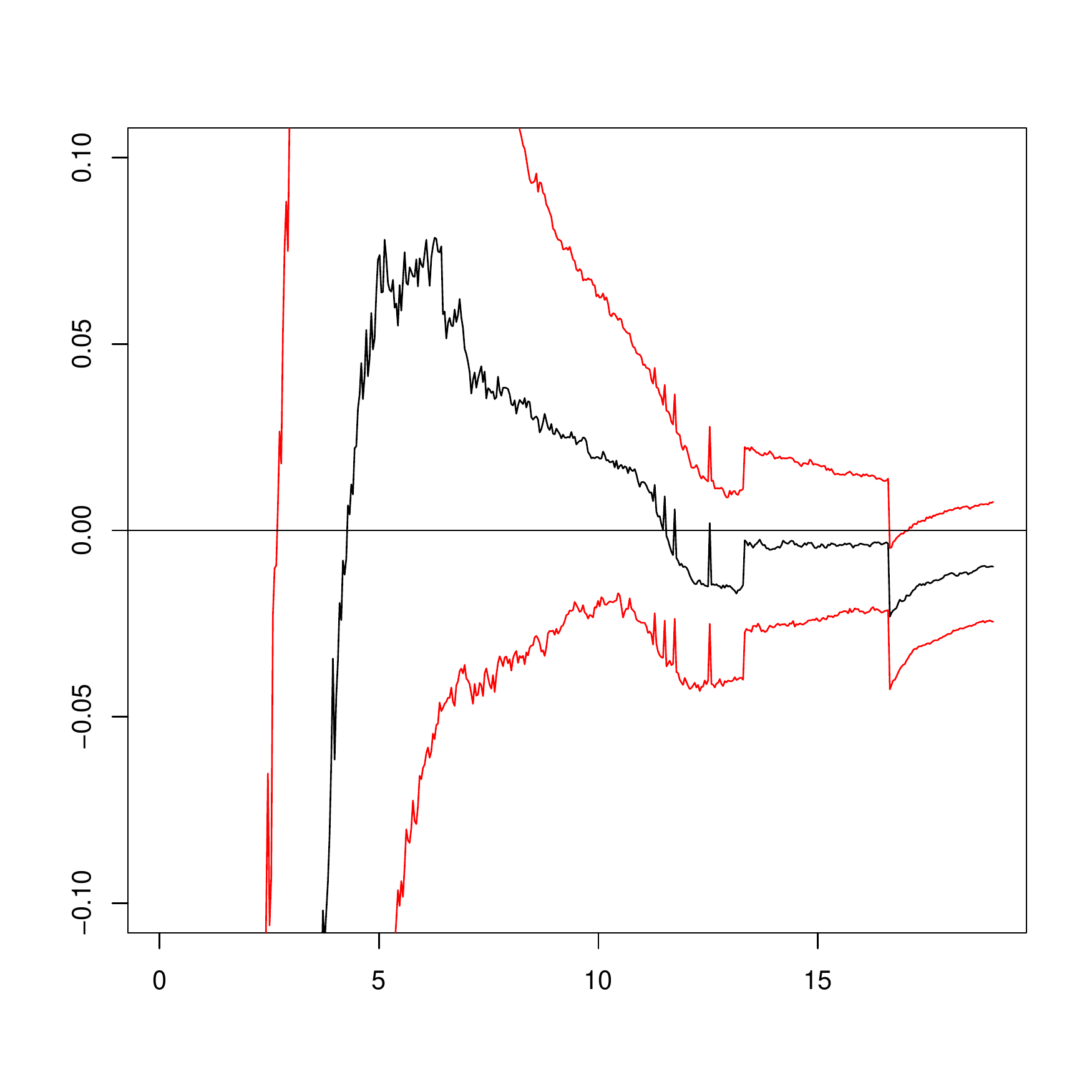}
\includegraphics[width=4.9cm]{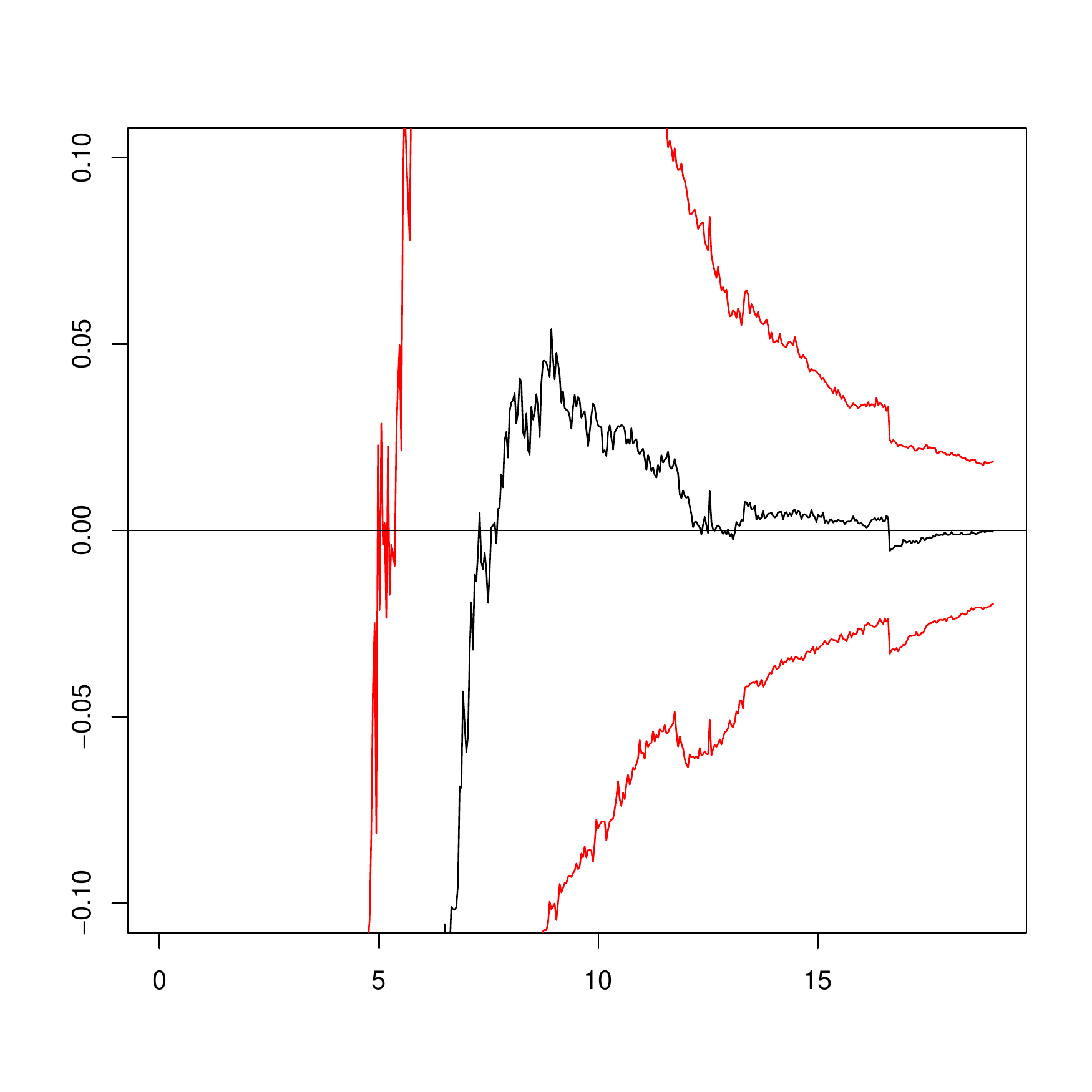}
\includegraphics[width=4.9cm]{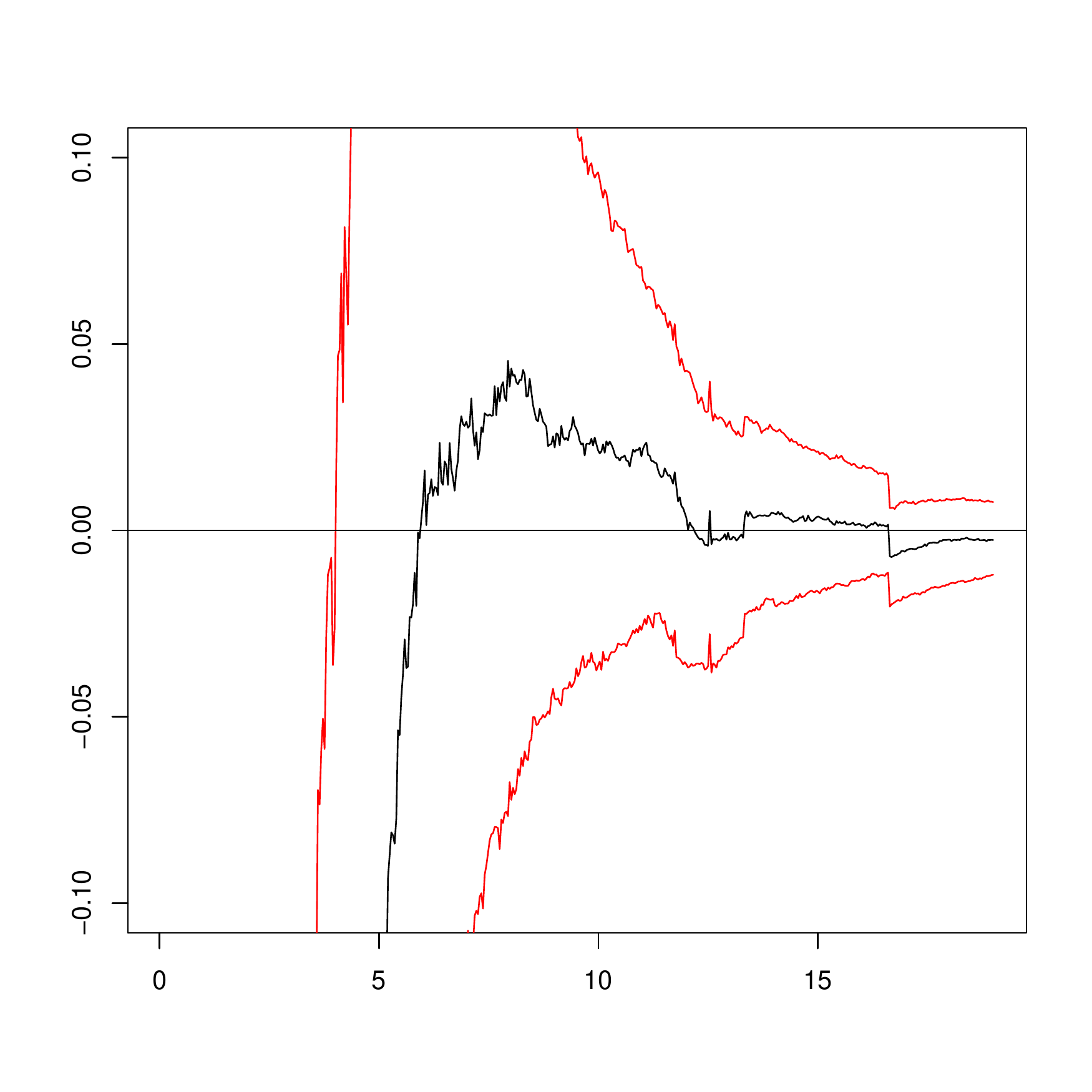}
\\ \vskip-4.9cm  \hskip3.5cm 
{\tiny $N=250$} \hskip3.8cm {\tiny $N=1000$} \hskip3.8cm {\tiny $N=1000$}
\\ \vskip-.6cm \hskip3.5cm 
{\tiny $K=250$} \hskip3.8cm {\tiny $K=250$} \hskip3.9cm {\tiny $K=1000$}
\\ \vskip3cm
{\small However, the error (using the wrong estimator) does not seem so large.
The jumps correspond  to jumps of $t\mapsto \Delta_t$. This is particularly visible on
these pictures because the time intervall is short.
Let us mention that the choice becomes good around $t=22$.}
\end{minipage}}

\vip

\noindent\fbox{\begin{minipage}{\textwidth}
{\small Independent case, $p=0.51$, $\mu=1$ (slightly supercritical). The choice is always bad 
for $t\in [9,62]$.}
\\
\includegraphics[width=4.9cm]{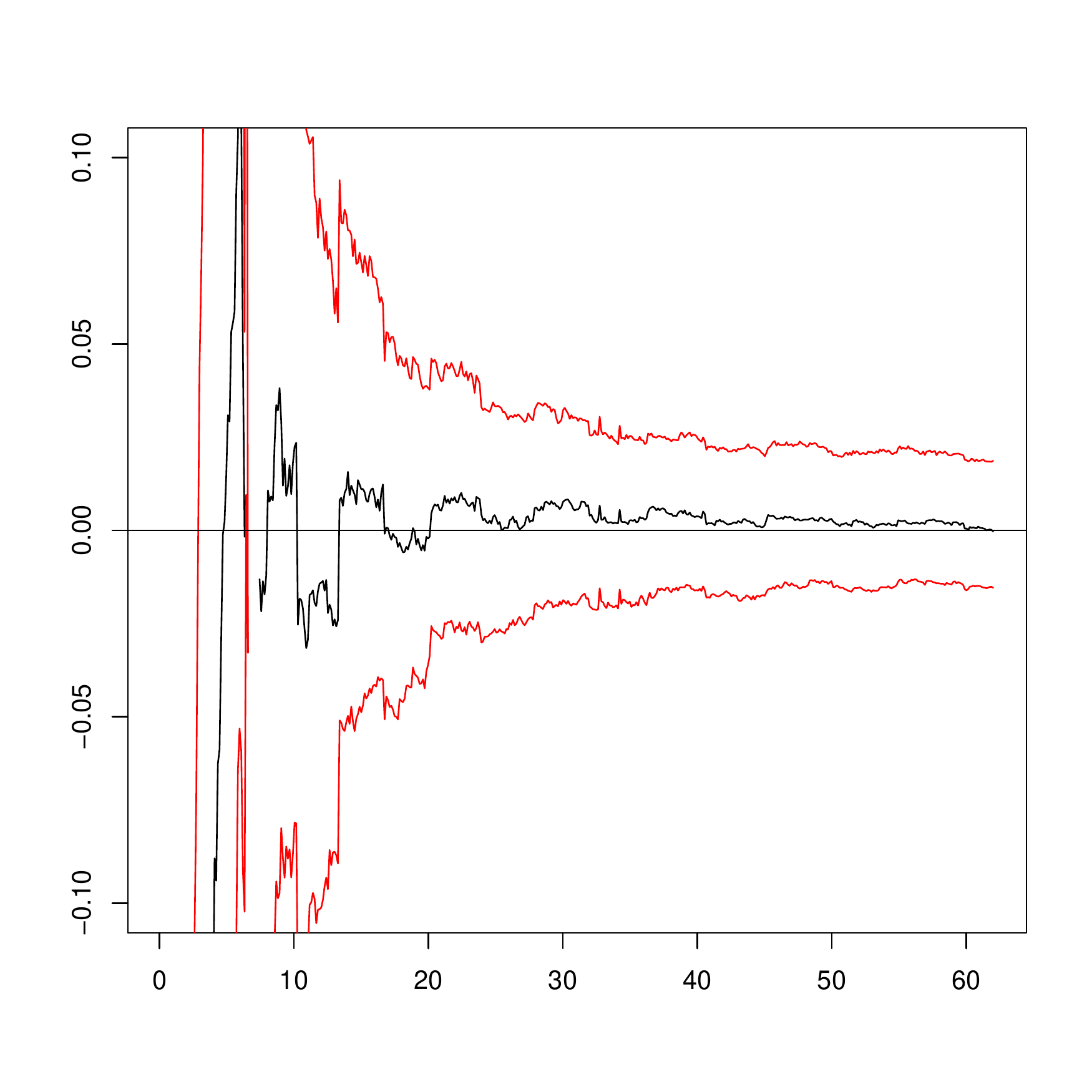}
\includegraphics[width=4.9cm]{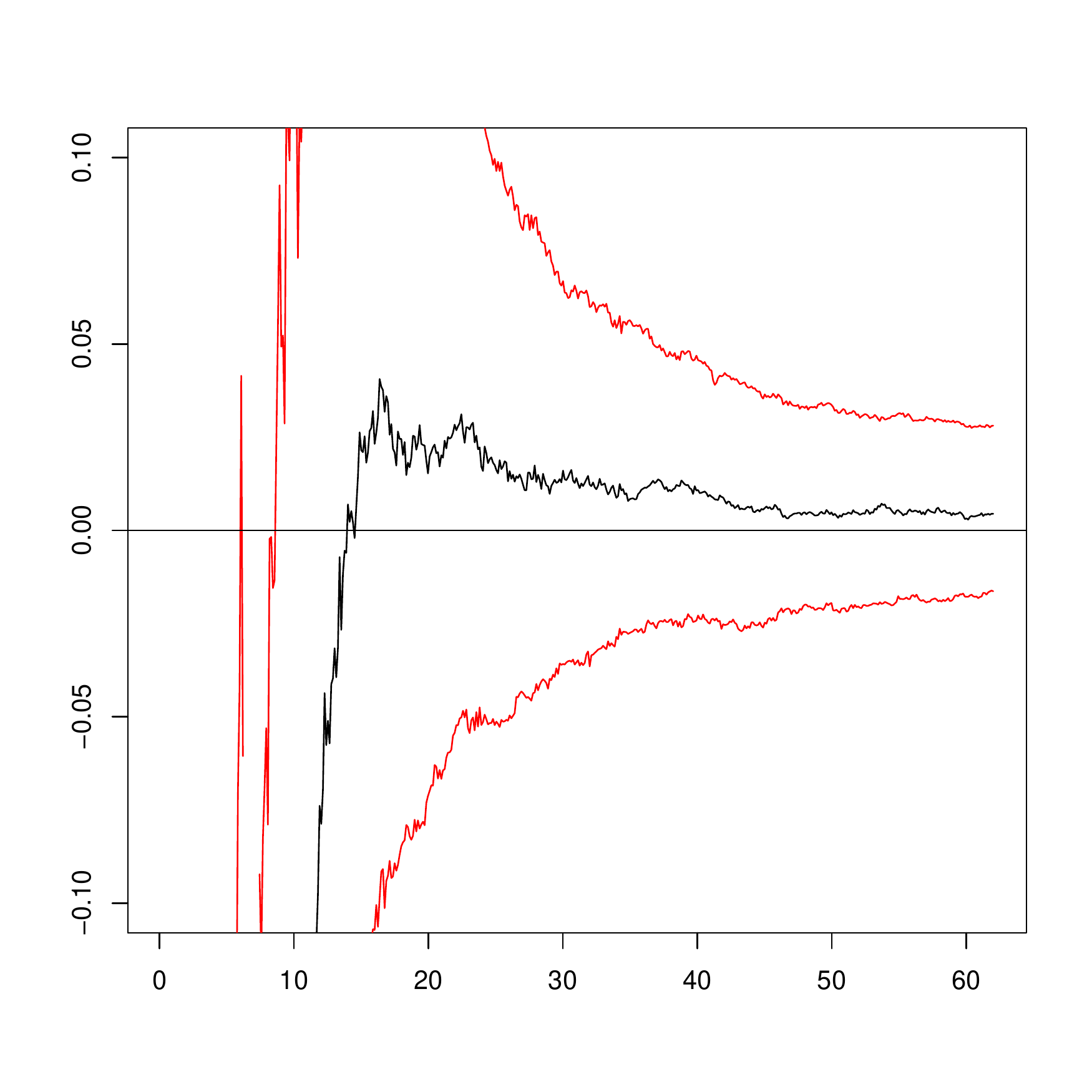}
\includegraphics[width=4.9cm]{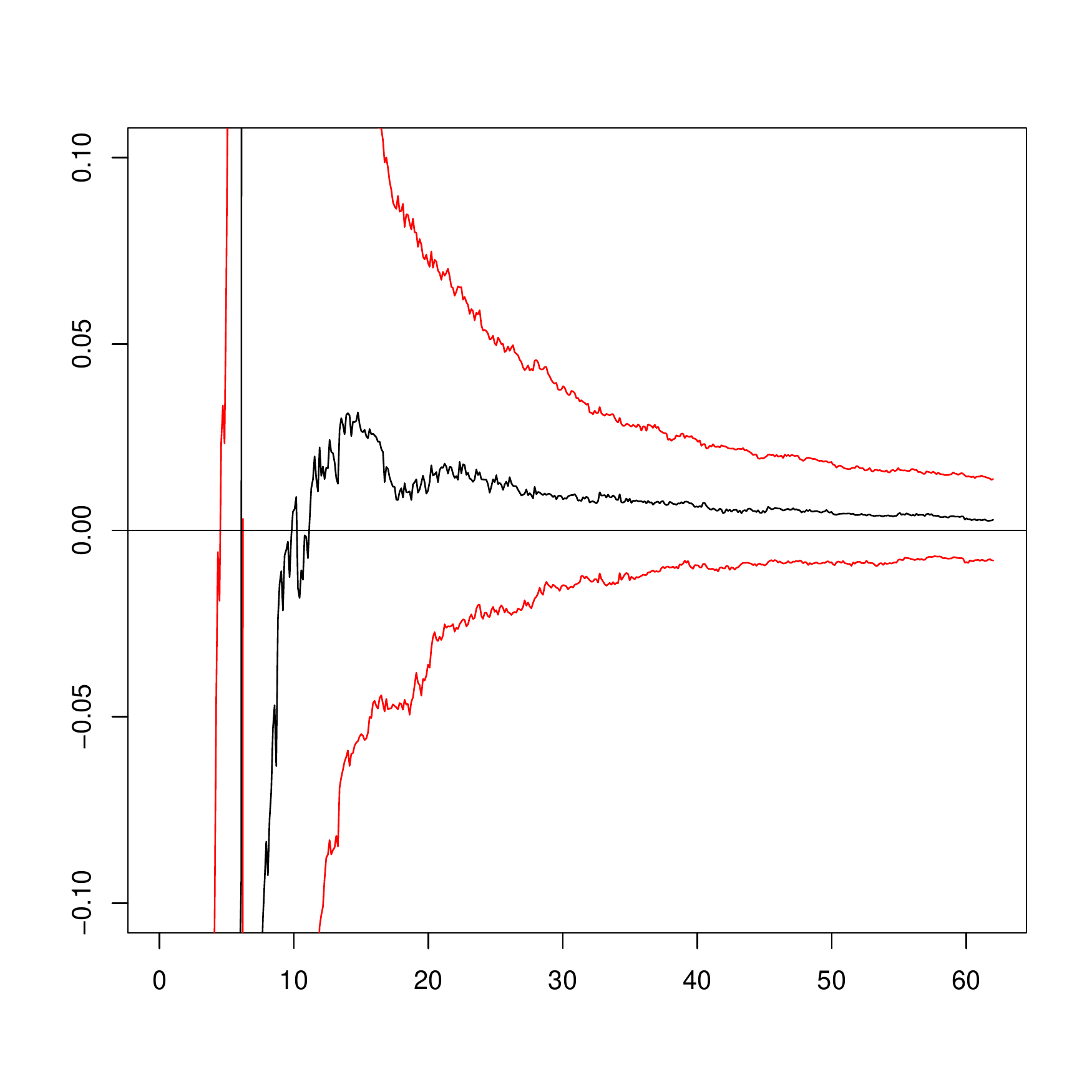}
\\ \vskip-4.9cm  \hskip3.5cm 
{\tiny $N=250$} \hskip3.8cm {\tiny $N=1000$} \hskip3.8cm {\tiny $N=1000$}
\\ \vskip-.6cm \hskip3.5cm 
{\tiny $K=250$} \hskip3.8cm {\tiny $K=250$} \hskip3.9cm {\tiny $K=1000$}
\\ \vskip3cm
{\small However, the error (using the wrong estimator) does not seem so large.}
\end{minipage}}

\vip
\noindent\fbox{\begin{minipage}{\textwidth}
{\small Independent case, $p=0.48$, $\mu=20$ (slightly subcritical). The choice is always bad 
for $t\in [1,15]$ and always good for $t\in [17,20]$.}
\\
\includegraphics[width=4.9cm]{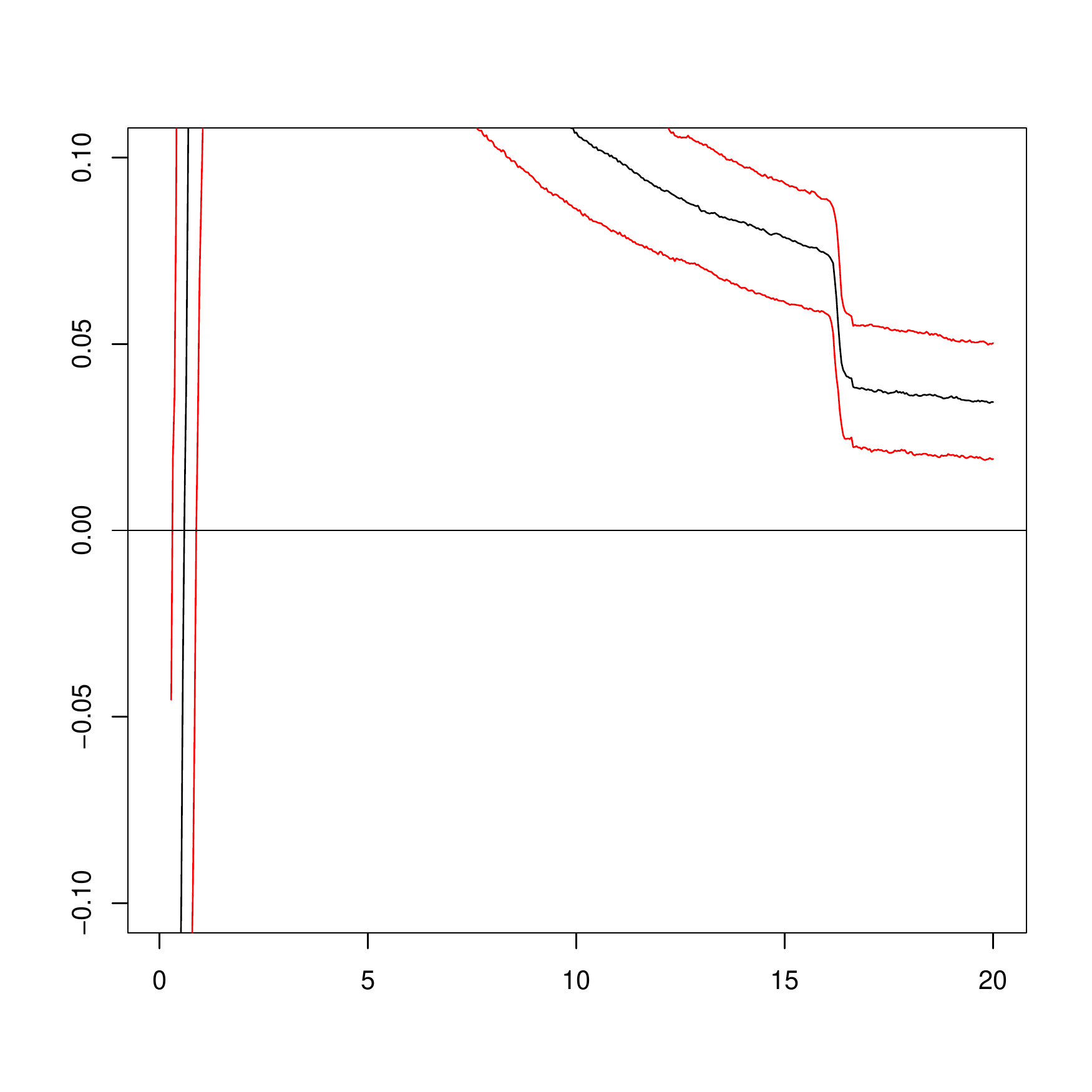}
\includegraphics[width=4.9cm]{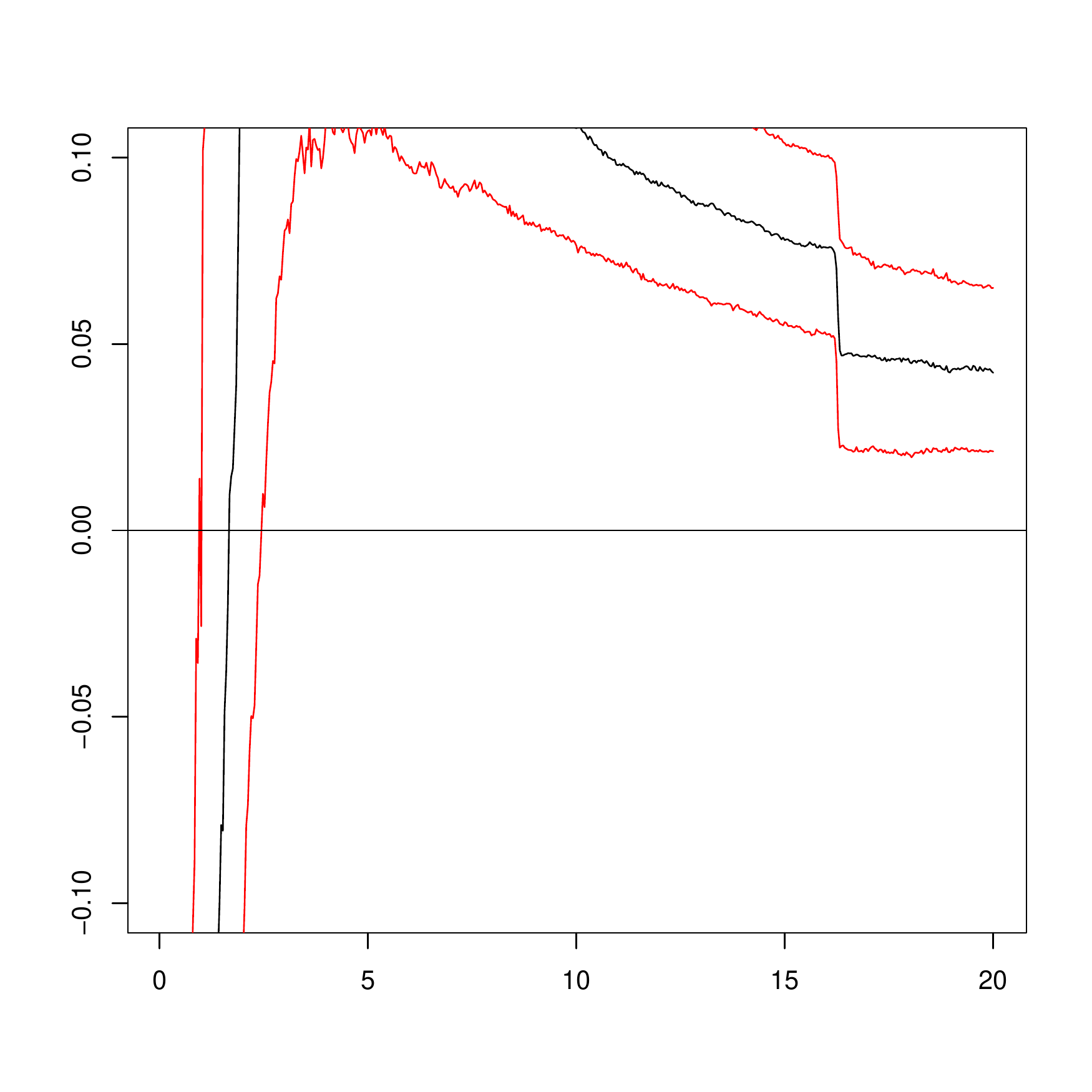}
\includegraphics[width=4.9cm]{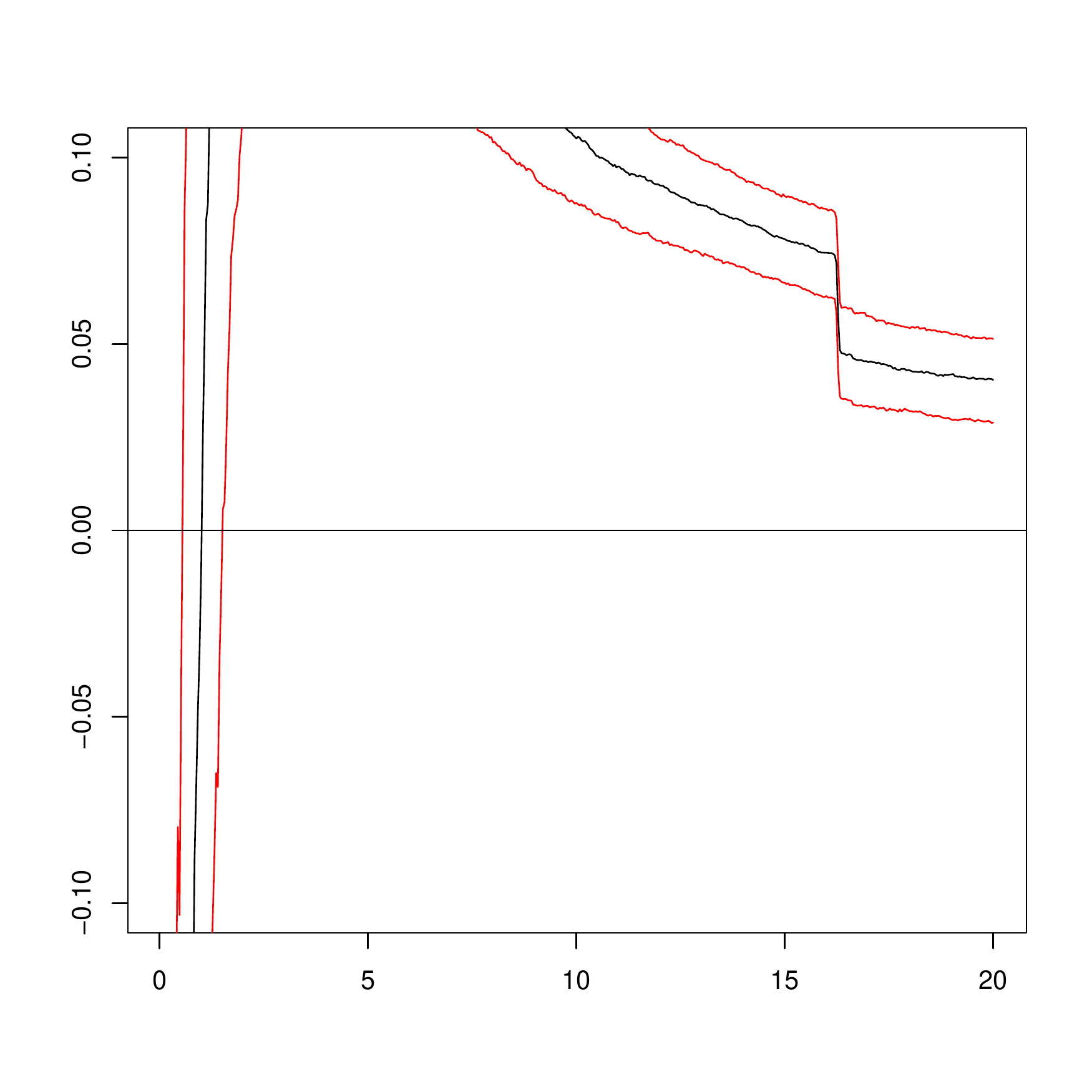}
\\ \vskip-2.9cm  \hskip3.5cm 
{\tiny $N=250$} \hskip3.8cm {\tiny $N=1000$} \hskip3.8cm {\tiny $N=1000$}
\\ \vskip-.6cm \hskip3.5cm 
{\tiny $K=250$} \hskip3.8cm {\tiny $K=250$} \hskip3.9cm {\tiny $K=1000$}\\ 
\vskip0.9cm
\hskip1cm {\small $-0.014,0.0010,0.016$} \hskip1.8cm{\small $-0.015,0.0017,0.016$}\hskip2cm
{\small $-0.0070,0.00078,0.0076$}
\\ 
{\small We clearly see the change of choice around $t=16$.}
\end{minipage}}

\vip

\noindent\fbox{\begin{minipage}{\textwidth}
{\small Independent case, $p=0.35$, $\mu=1$ (fairly subcritical). 
The choice is always good for $t \in (0,900]$.}
\\
\includegraphics[width=4.9cm]{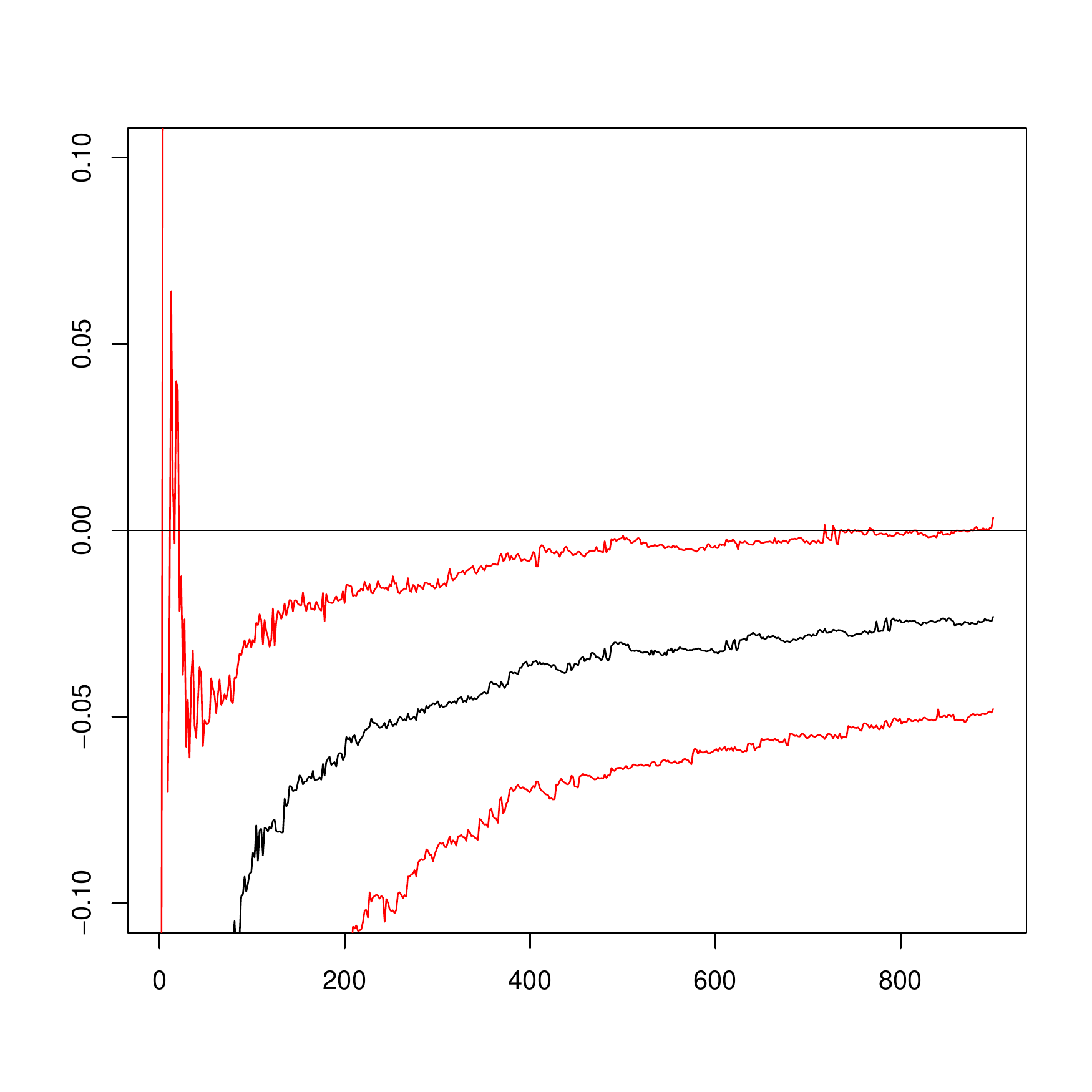}
\includegraphics[width=4.9cm]{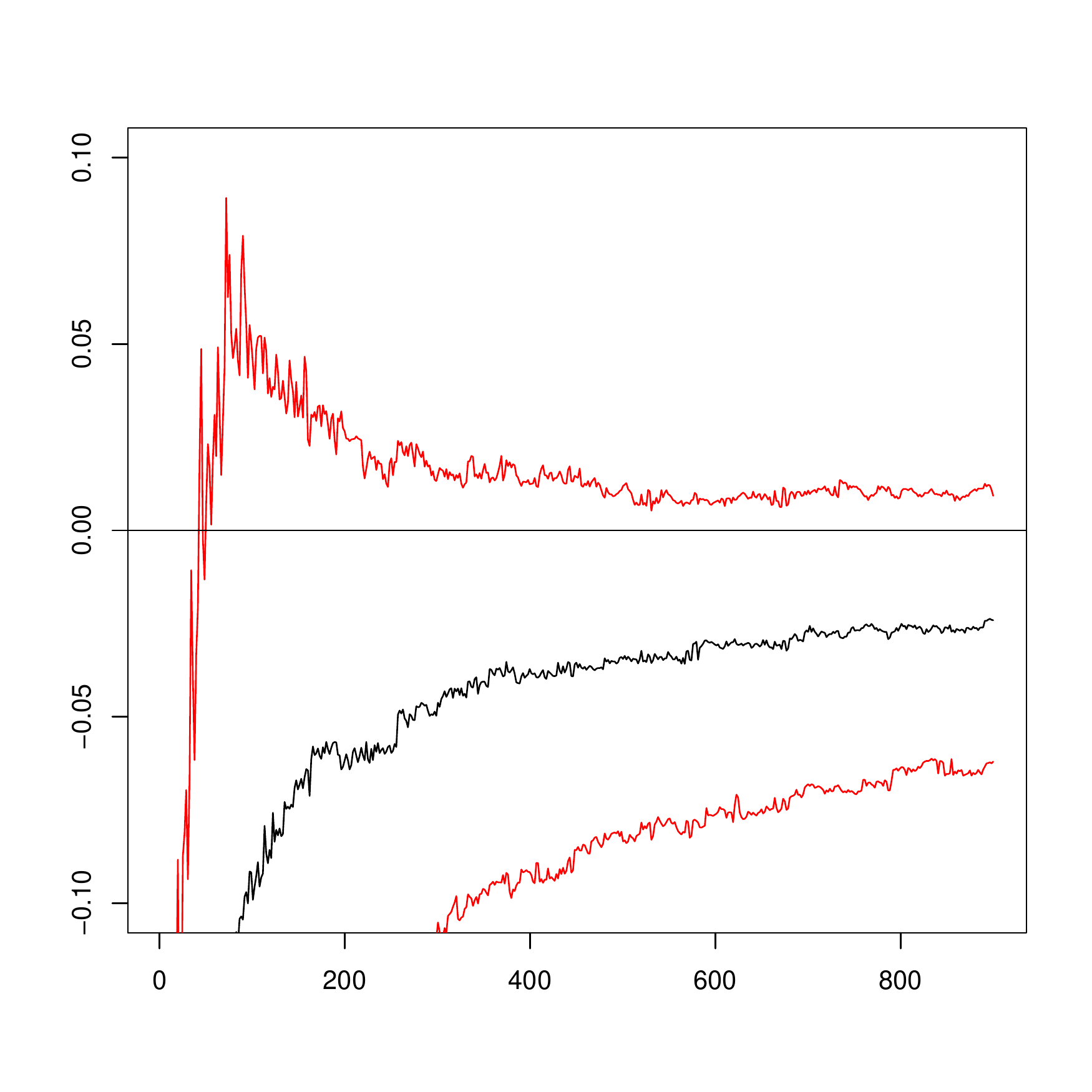}
\includegraphics[width=4.9cm]{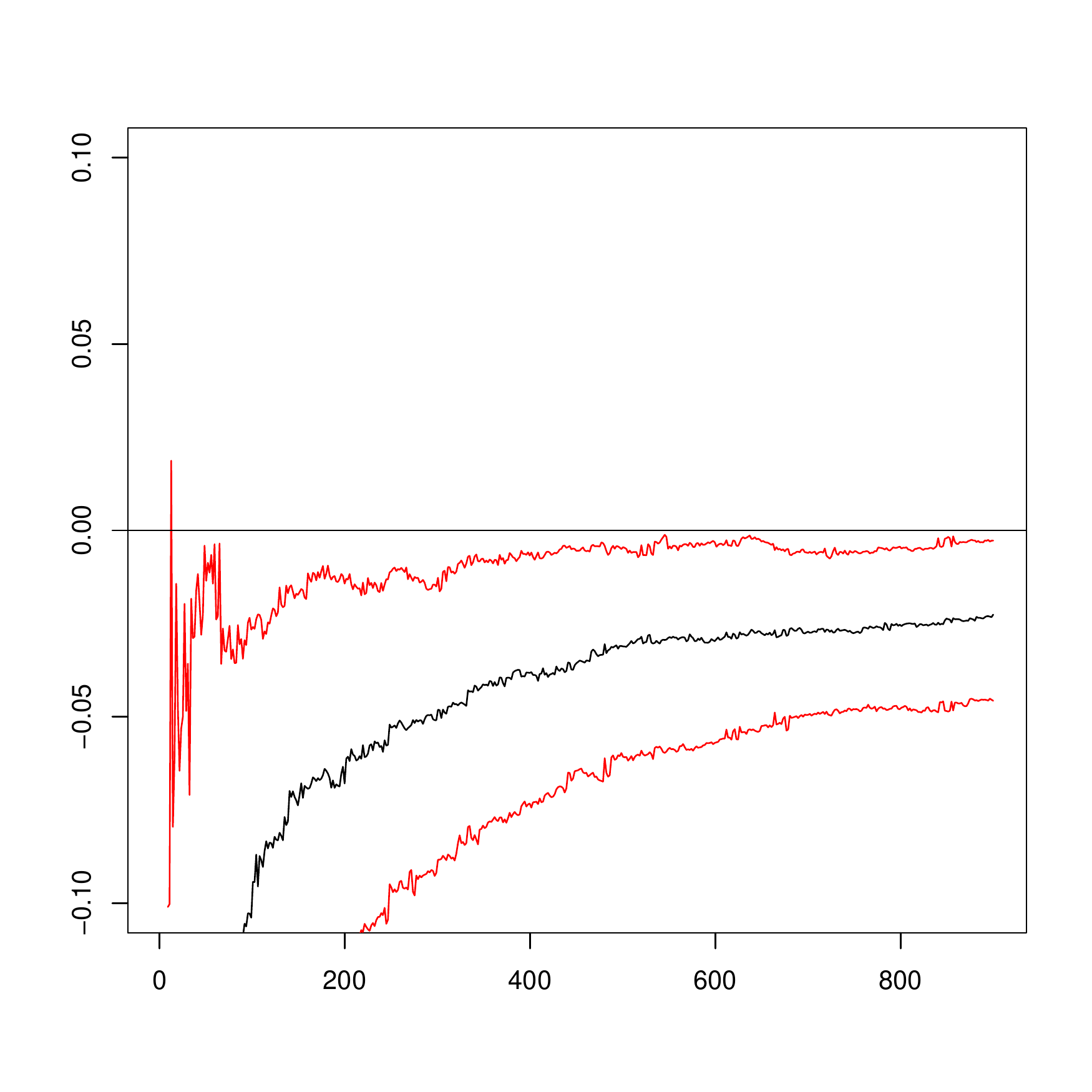}
\\ \vskip-4.9cm  \hskip3.5cm 
{\tiny $N=250$} \hskip3.8cm {\tiny $N=1000$} \hskip3.8cm {\tiny $N=1000$}
\\ \vskip-.6cm \hskip3.5cm 
{\tiny $K=250$} \hskip3.8cm {\tiny $K=250$} \hskip3.9cm {\tiny $K=1000$}
\\ \vskip3cm
\hskip0.9cm {\small $-0.012,0.0012,0.014$} \hskip2cm{\small $-0.012,0.00072,0.015$}\hskip2cm
{\small $-0.0065,0.00021,0.0070$}
\\
{\small These results show that the bias is rather large, of the same order as the standard deviation.}
\end{minipage}}

\vip
\noindent\fbox{\begin{minipage}{\textwidth}
{\small Independent case, $p=0.1$, $\mu=1$ (fairly subcritical). 
The choice is always good.}
\\
\includegraphics[width=4.9cm]{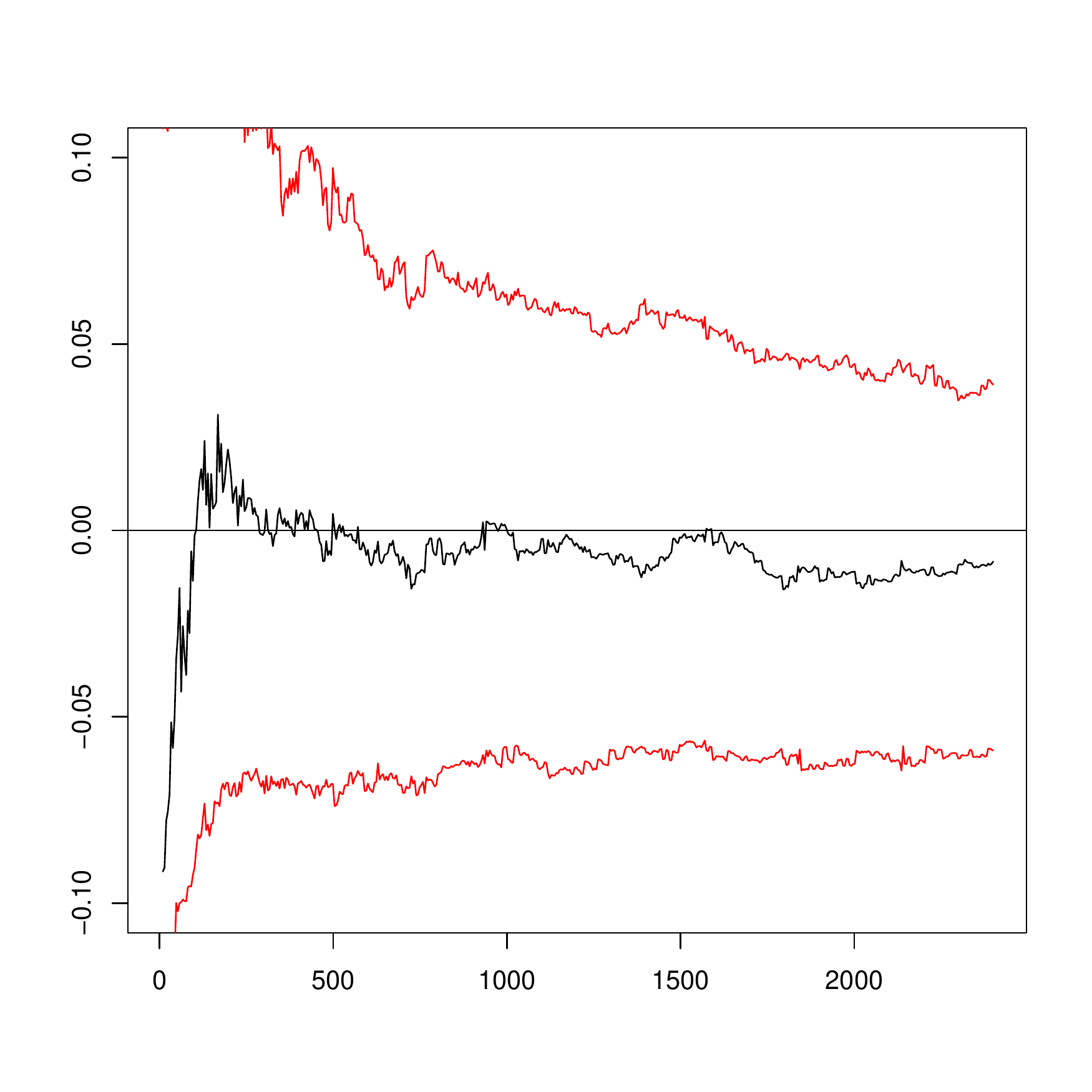}
\includegraphics[width=4.9cm]{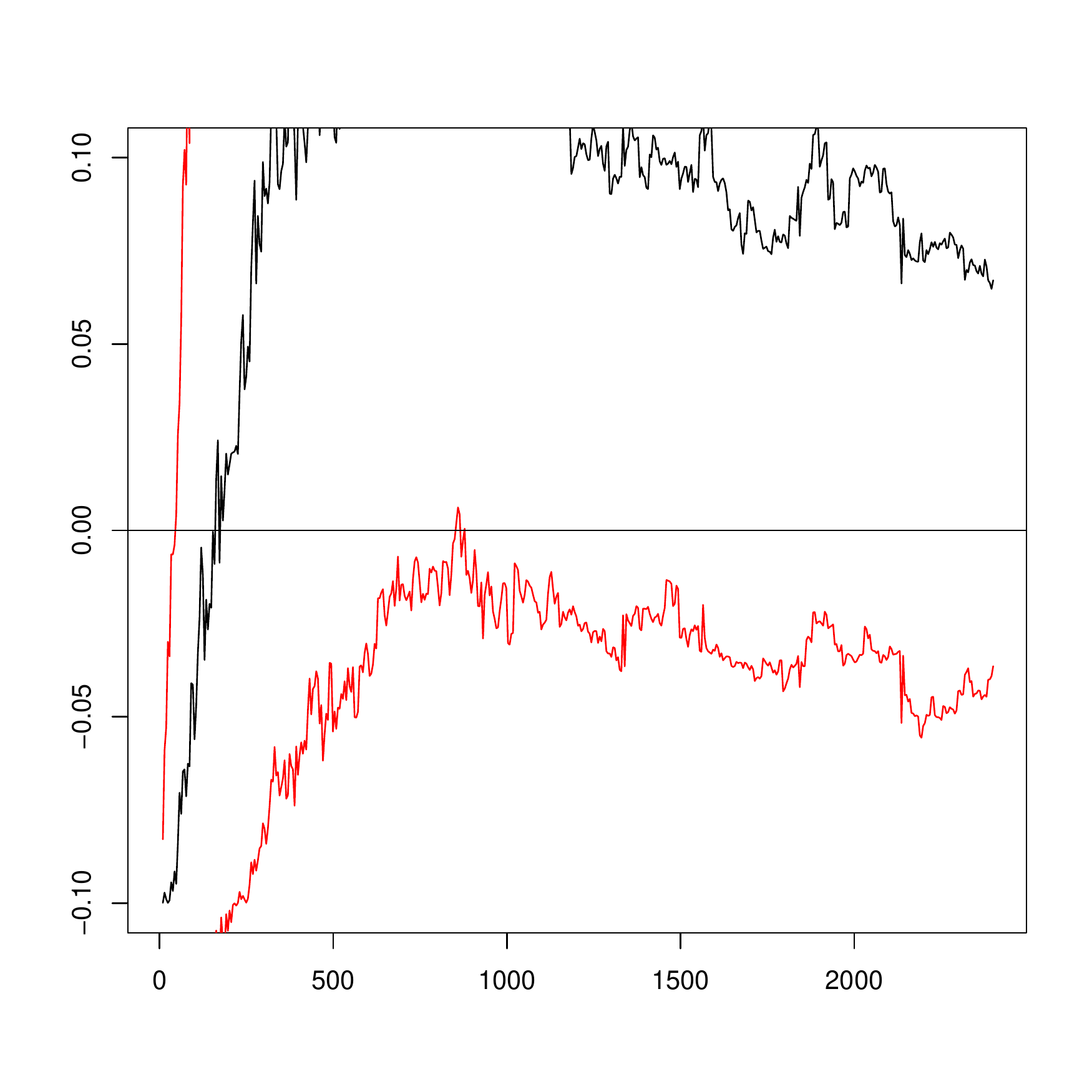}
\includegraphics[width=4.9cm]{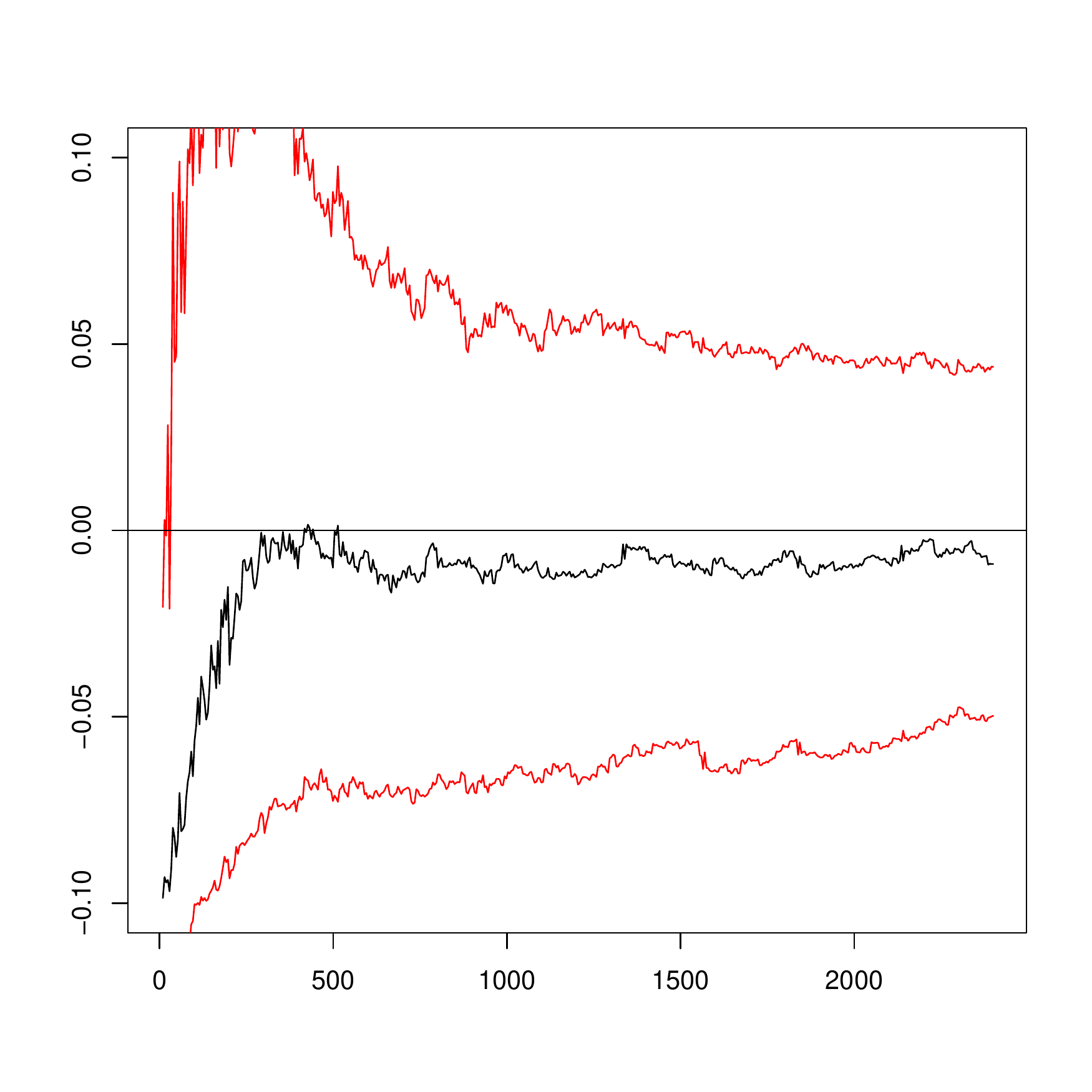}
\\ \vskip-4.9cm  \hskip2.5cm 
{\tiny $N=250$} \hskip2.8cm {\tiny $N=1000$} \hskip4.8cm {\tiny $N=1000$}
\\ \vskip-.6cm \hskip2.5cm 
{\tiny $K=250$} \hskip2.8cm {\tiny $K=250$} \hskip4.9cm {\tiny $K=1000$}
\\ \vskip3cm
\hskip0.9cm {\small $-0.0038,0.0013,0.0067$} \hskip1.6cm{\small $-0.0045,0.00057,0.0061$}\hskip1.6cm
{\small $-0.0021,0.00045,0.0032$}
\\ 
{\small These pictures illustrate that this situation ($p$ small) is not quite favorable.}
\end{minipage}}

\vip

\noindent\fbox{\begin{minipage}{\textwidth}
{\small Independent case, $p=0$, $\mu=1$ (subcritical). The choice is always good.}
\\
\centerline{\includegraphics[width=4.9cm]{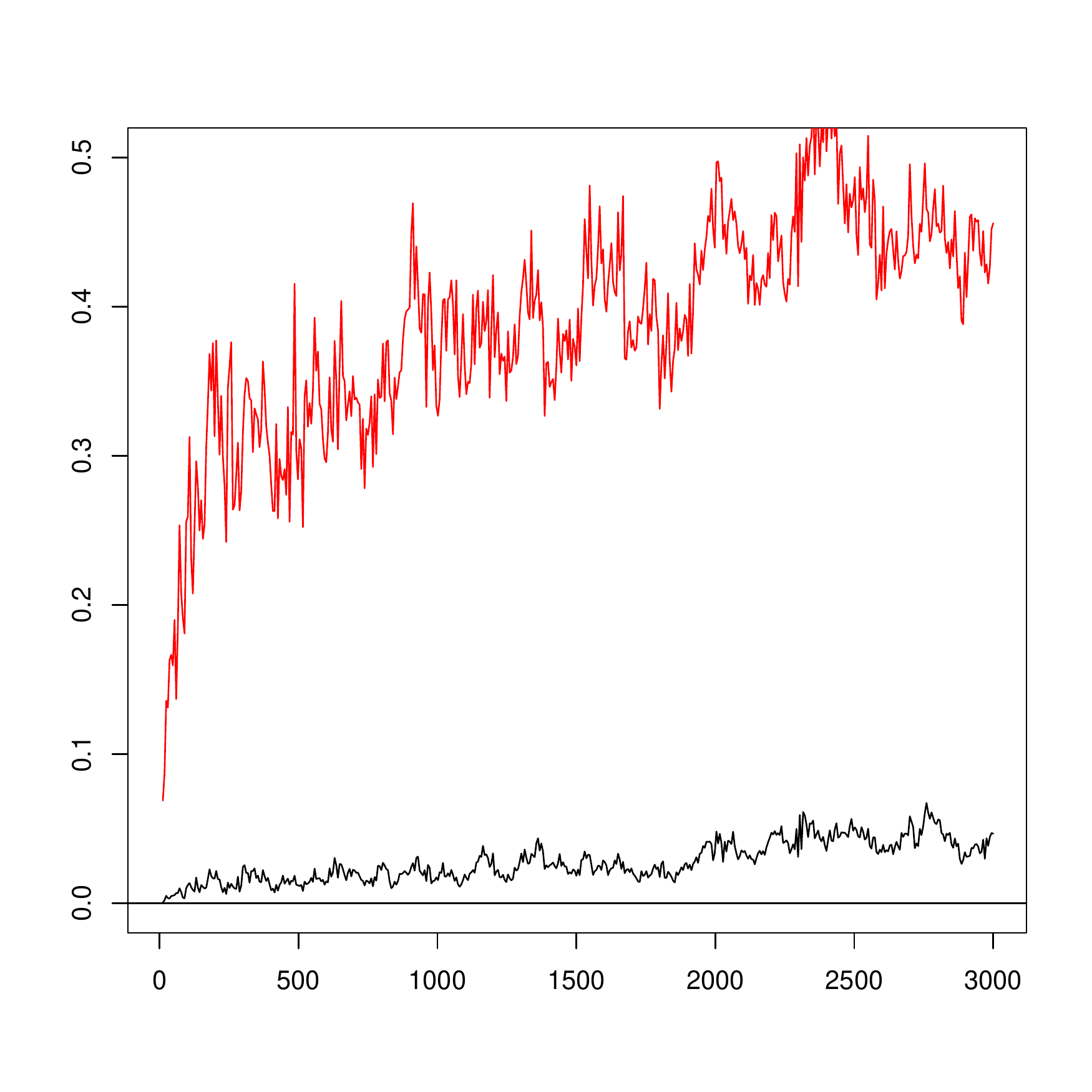}}
\\ \vskip-4.9cm  \hskip5.7cm {\tiny $N=1000$}
\\ \vskip-.6cm \hskip5.7cm 
{\tiny $K=1000$}
\\ \vskip3cm
{\small This is catastrophic.}
\end{minipage}}

\vip\vip\vip

In the symmetric case, we obtain very similar numerical results.

\vip

\noindent\fbox{\begin{minipage}{\textwidth}
{\small Symmetric case, $p=0.85$, $\mu=1$ (fairly supercritical). The choice is always good 
for $t\in [1,9.7]$.}
\\
\includegraphics[width=4.9cm]{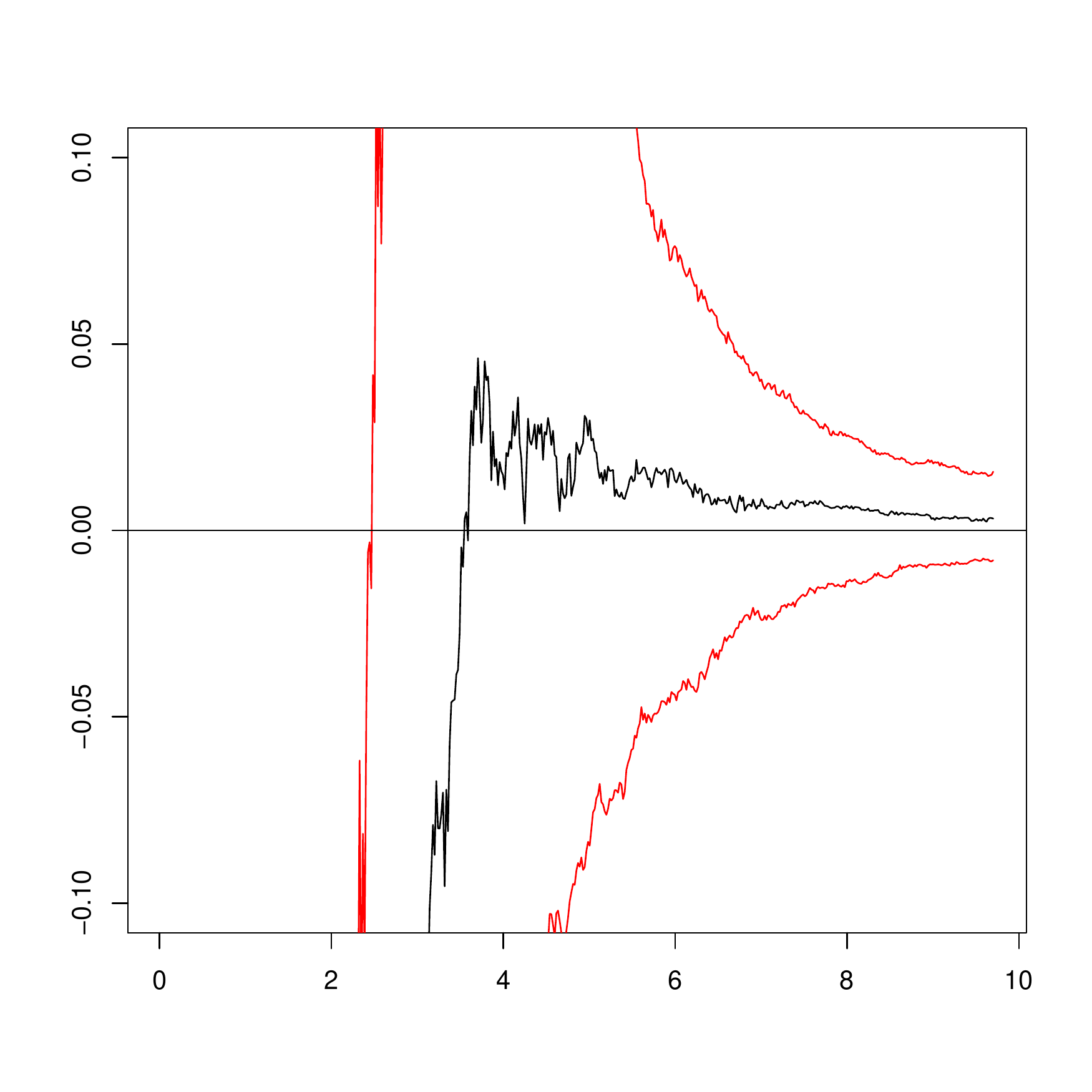}
\includegraphics[width=4.9cm]{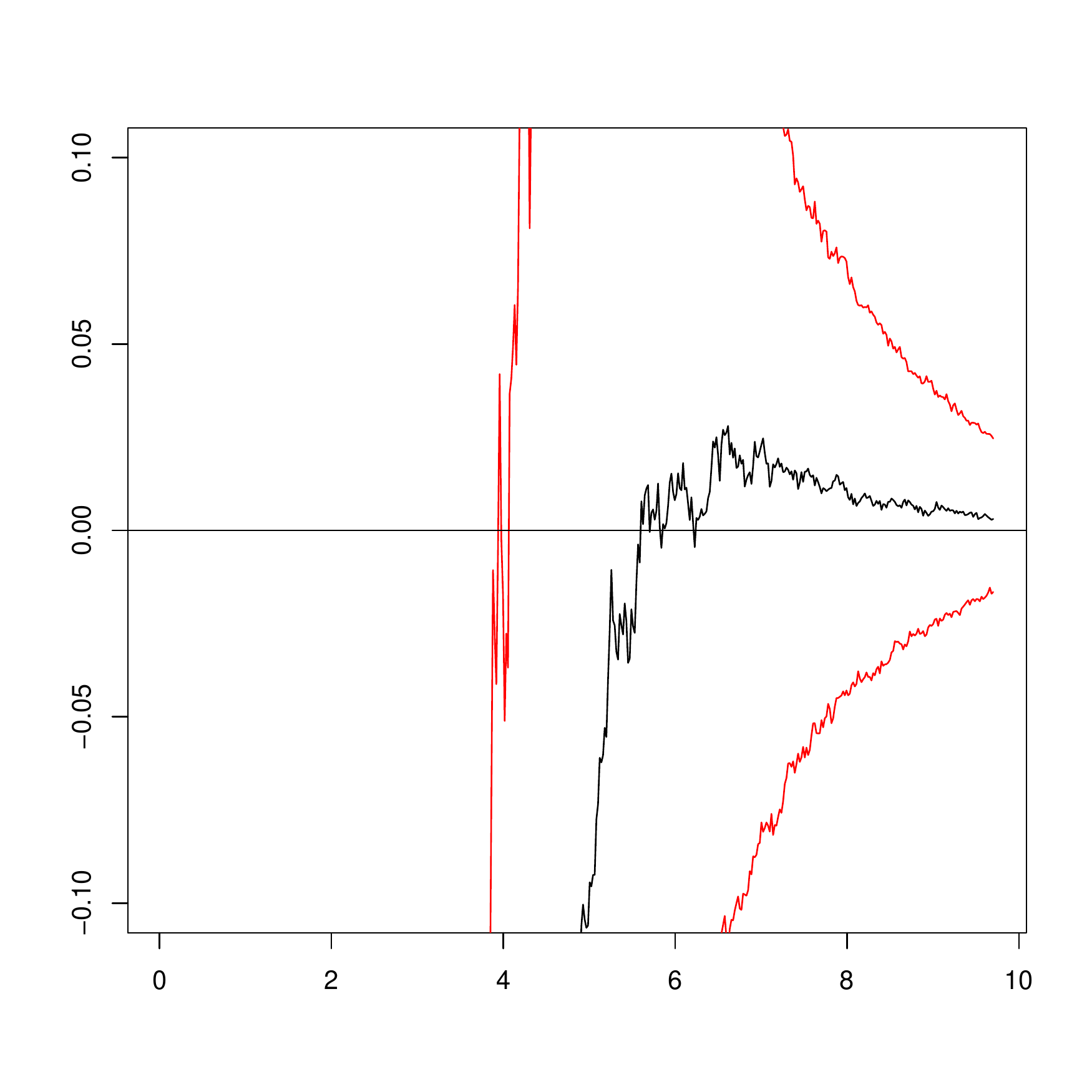}
\includegraphics[width=4.9cm]{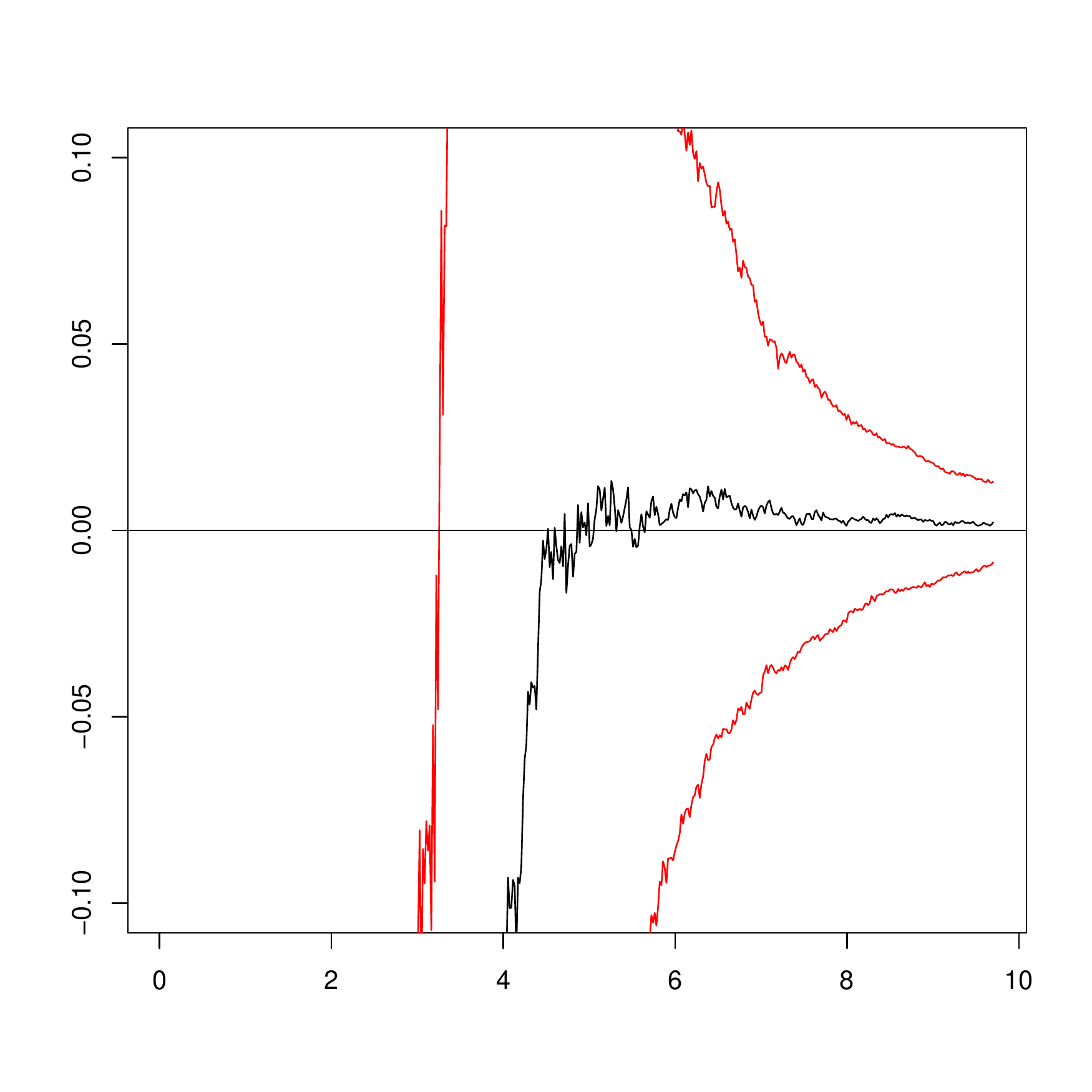}
\\ \vskip-4.9cm  \hskip3.5cm 
{\tiny $N=250$} \hskip3.8cm {\tiny $N=1000$} \hskip3.8cm {\tiny $N=1000$}
\\ \vskip-.6cm \hskip3.5cm 
{\tiny $K=250$} \hskip3.8cm {\tiny $K=250$} \hskip3.9cm {\tiny $K=1000$}
\\ \vskip3cm
\hskip0.9cm {\small $-0.0057,0.0021,0.0091$} \hskip1.7cm{\small $-0.0062,0.0012,0.0084$}\hskip1.8cm
{\small $-0.0029,0.00079,0.0044$}
\end{minipage}}
\vip

\noindent\fbox{\begin{minipage}{\textwidth}
{\small Symmetric case, $p=0.65$, $\mu=1$ (supercritical). 
The choice is always bad for $t\in [14,19]$.}
\\
\includegraphics[width=4.9cm]{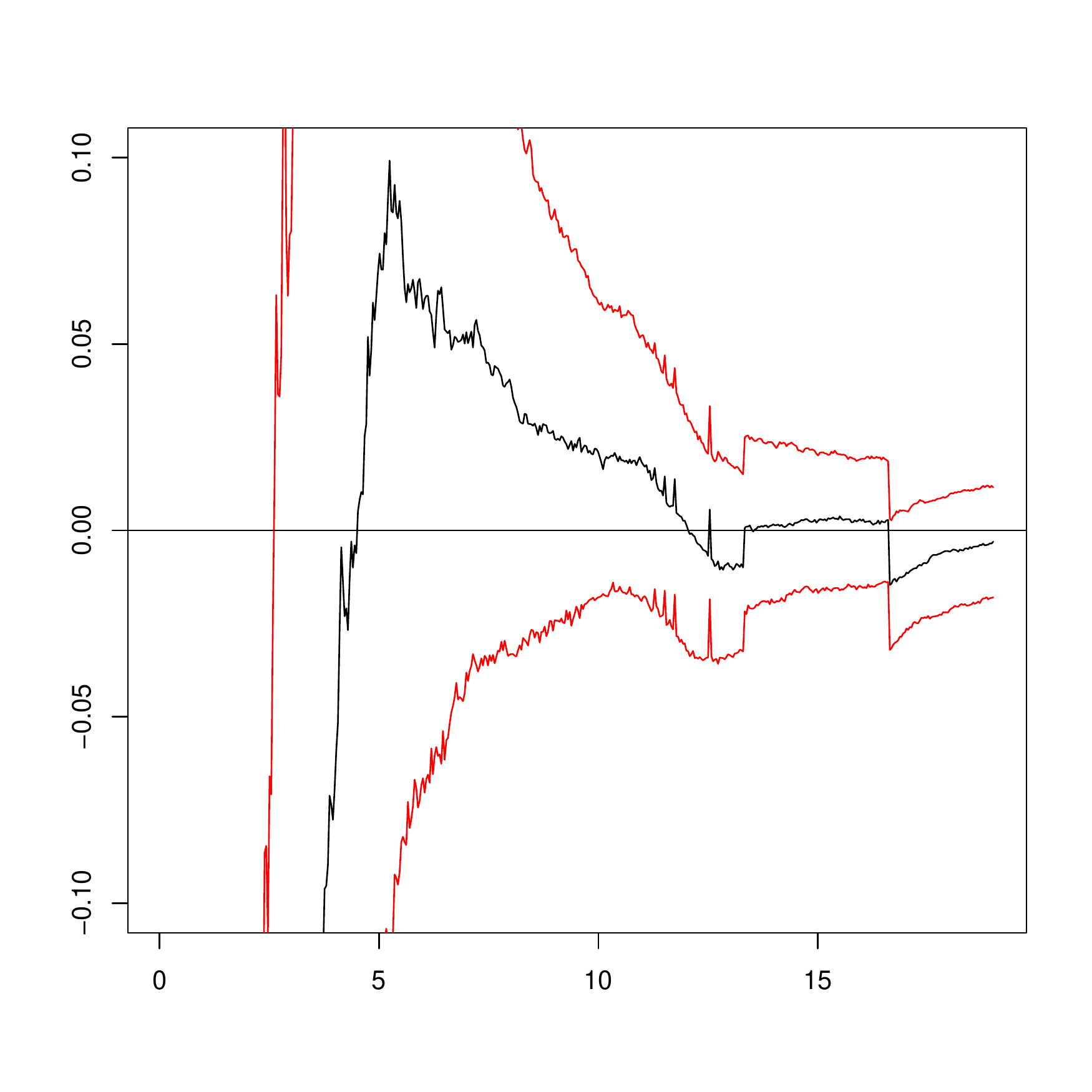}
\includegraphics[width=4.9cm]{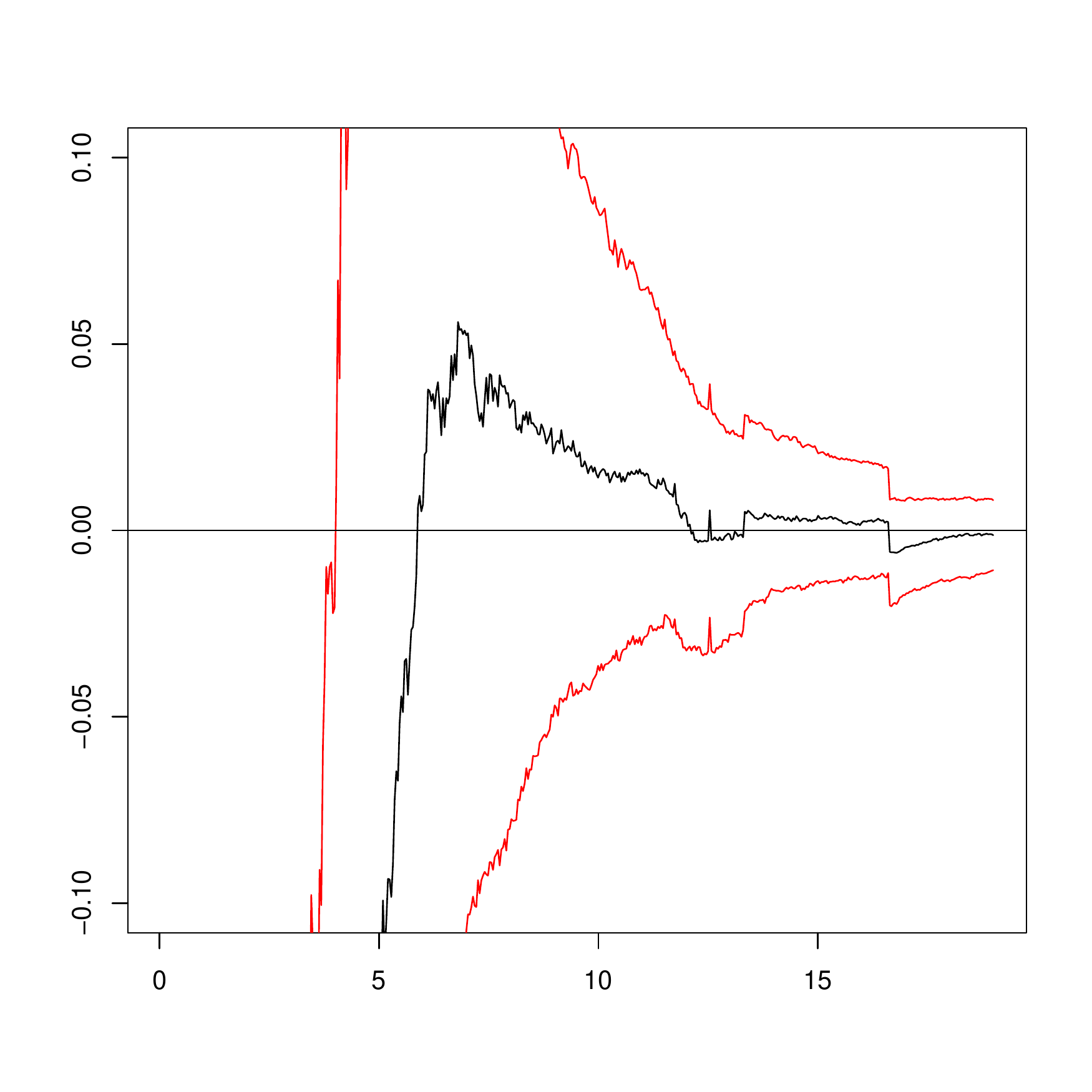}
\includegraphics[width=4.9cm]{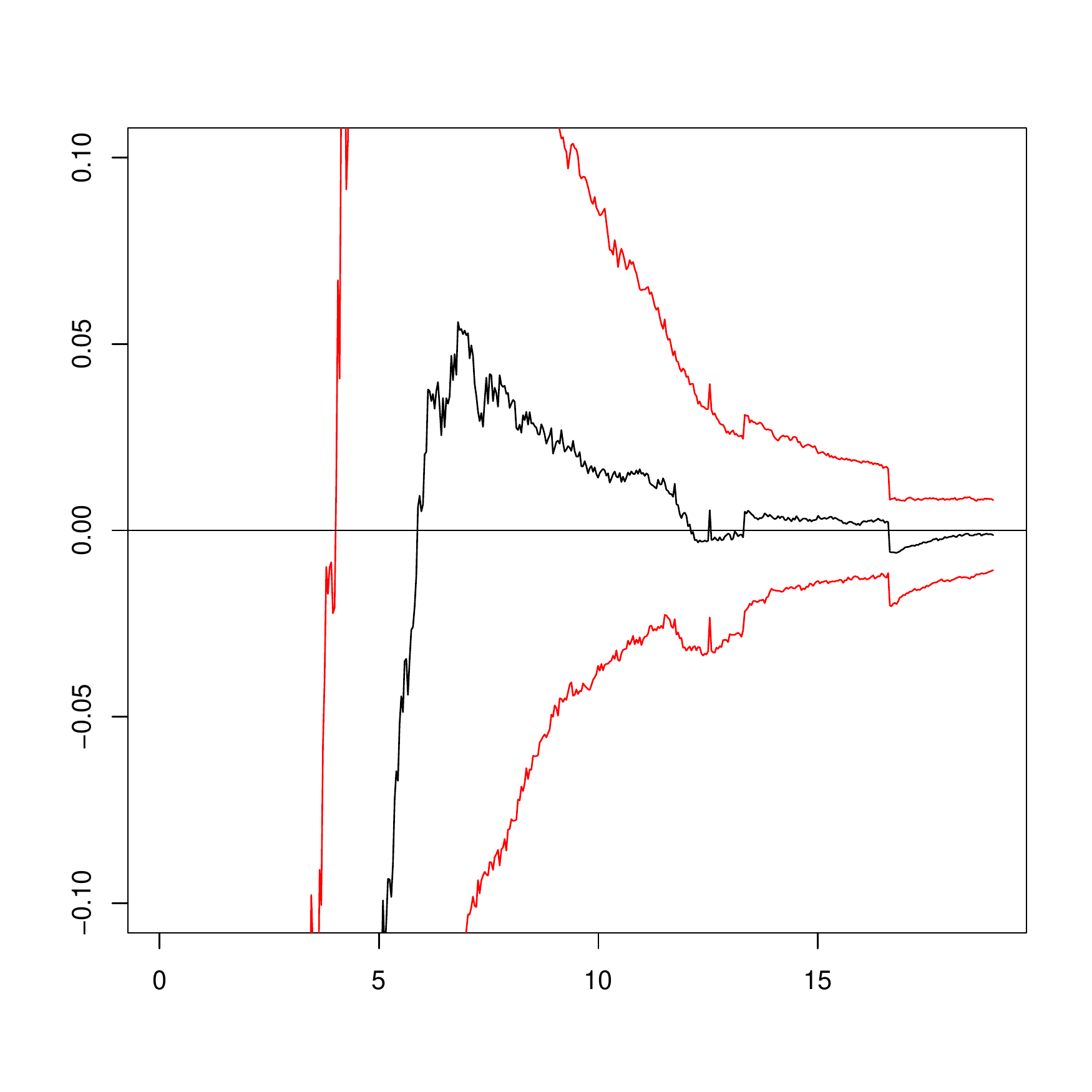}
\\ \vskip-4.9cm  \hskip3.5cm 
{\tiny $N=250$} \hskip3.8cm {\tiny $N=1000$} \hskip3.8cm {\tiny $N=1000$}
\\ \vskip-.6cm \hskip3.5cm 
{\tiny $K=250$} \hskip3.8cm {\tiny $K=250$} \hskip3.9cm {\tiny $K=1000$}
\\ \vskip3cm $\left.\right.$
\end{minipage}}

\vip

\noindent\fbox{\begin{minipage}{\textwidth}
{\small Symmetric case, $p=0.51$, $\mu=1$ (slightly supercritical). The choice is always bad
for $t\in [9,62]$.}
\\
\includegraphics[width=4.9cm]{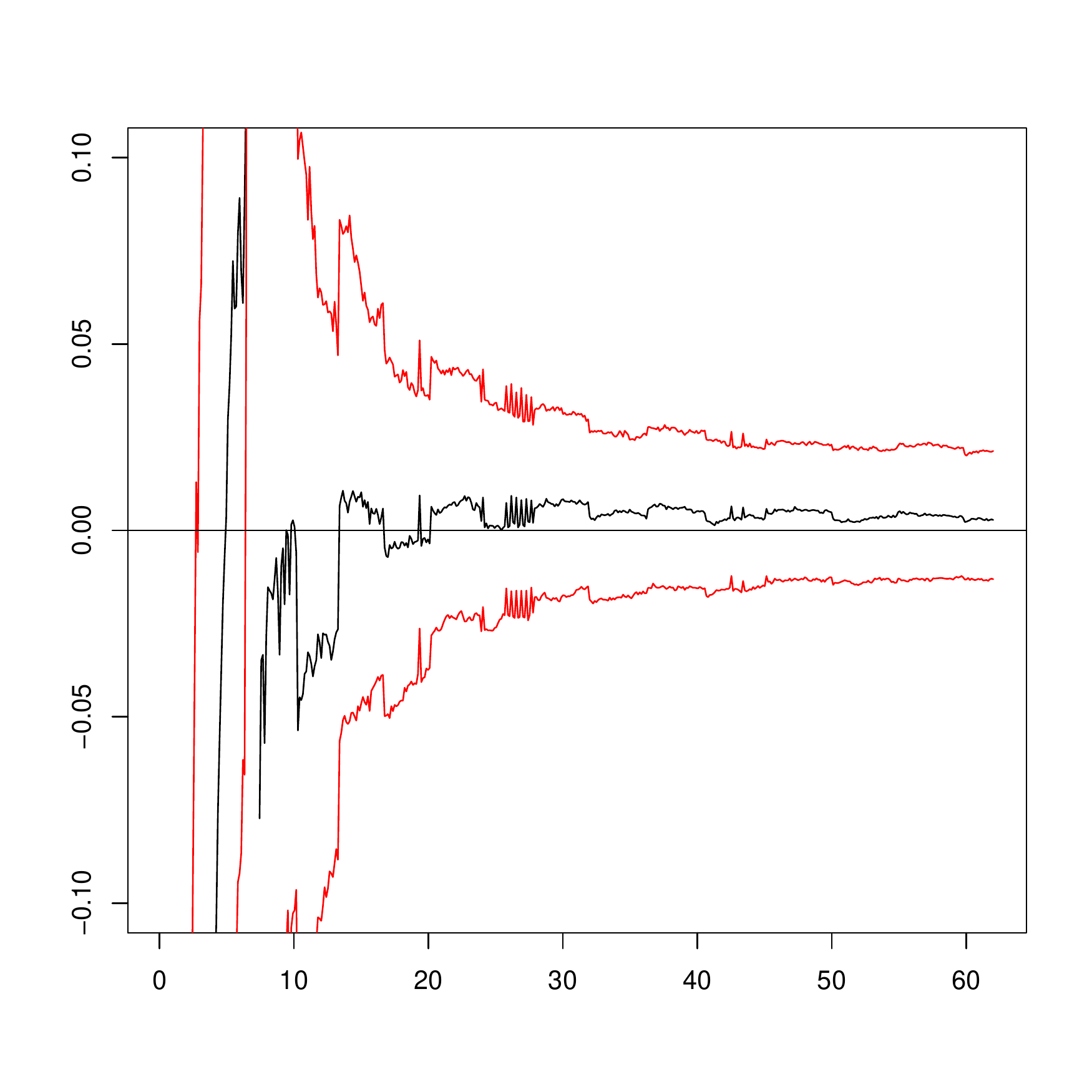}
\includegraphics[width=4.9cm]{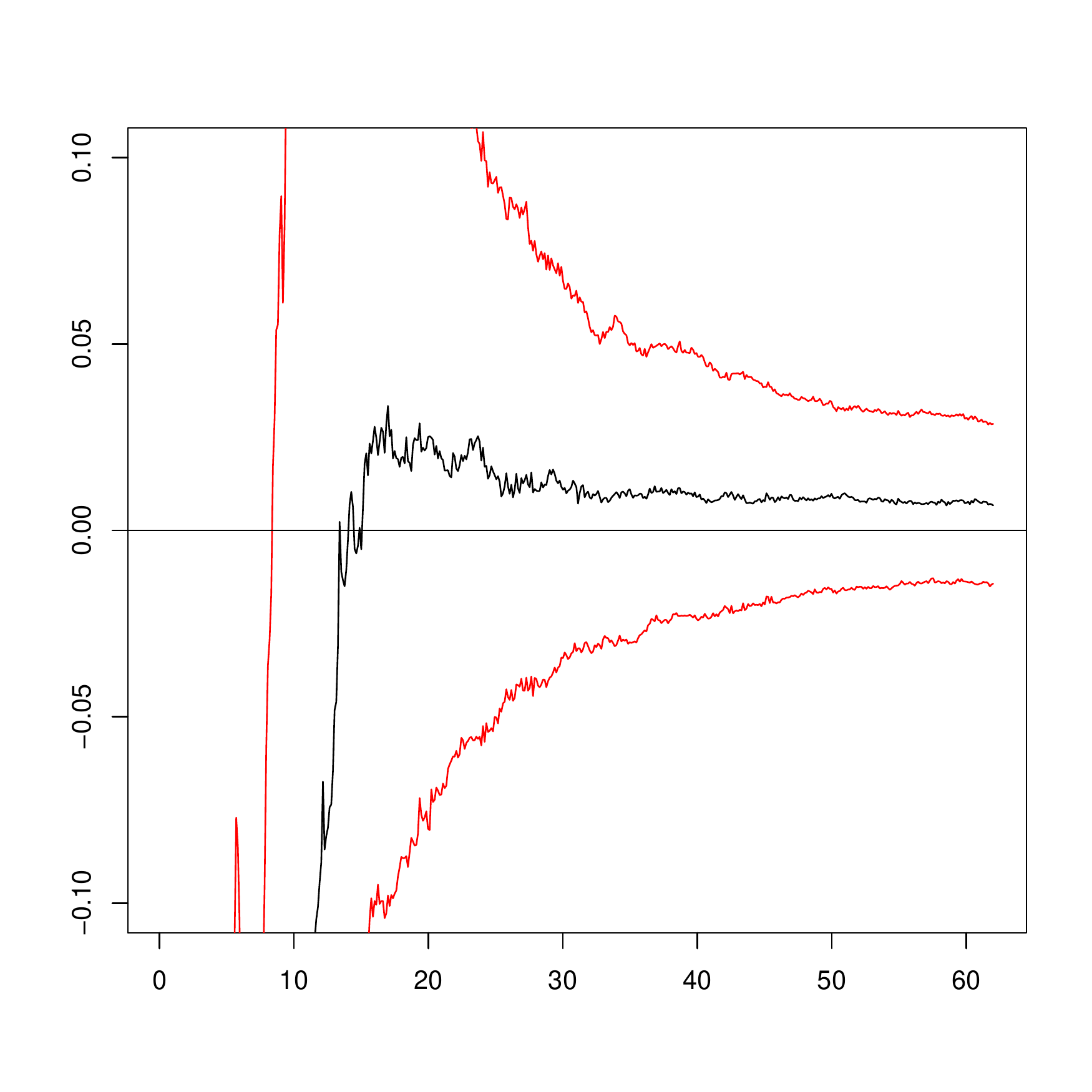}
\includegraphics[width=4.9cm]{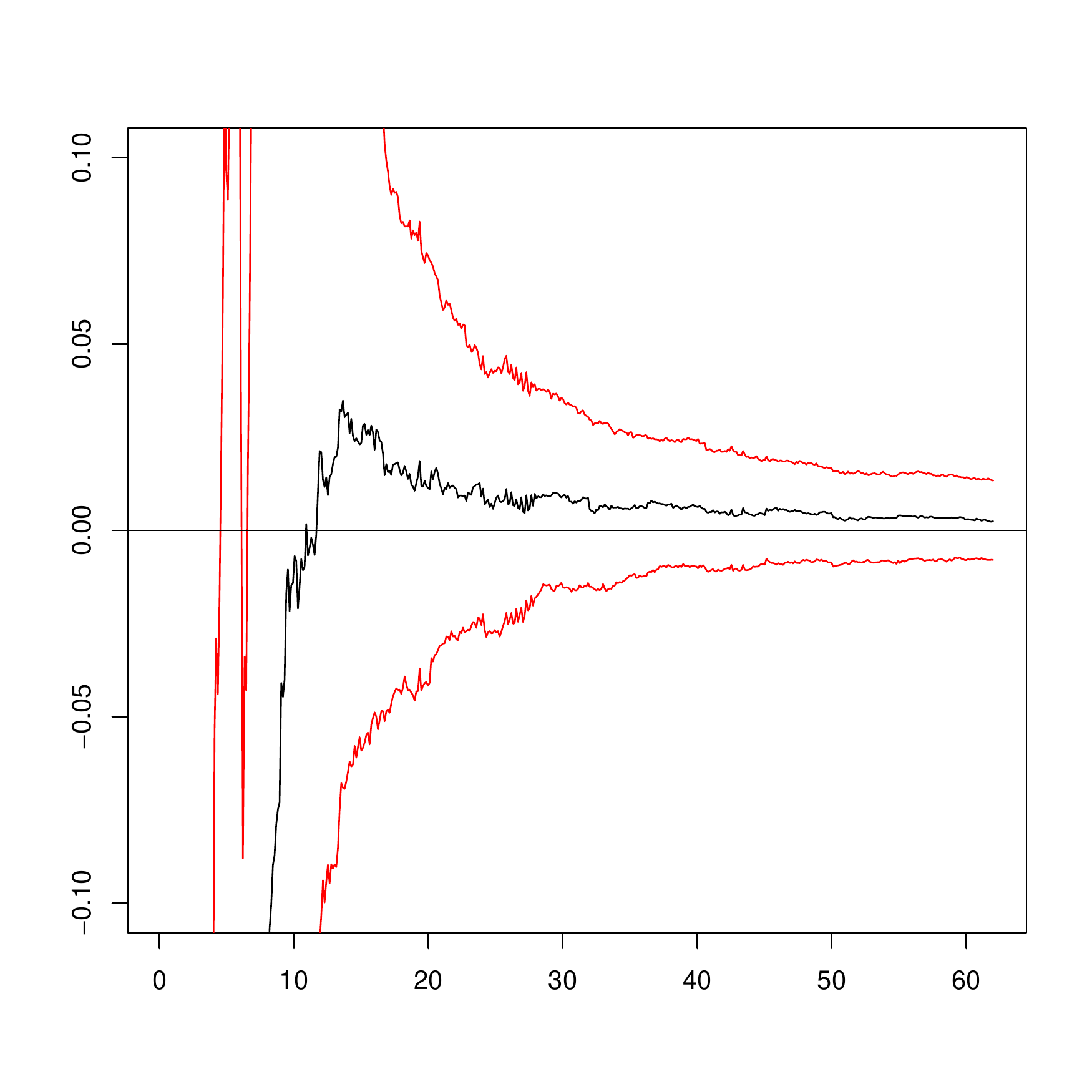}
\\ \vskip-2.9cm  \hskip3.5cm 
{\tiny $N=250$} \hskip3.8cm {\tiny $N=1000$} \hskip3.8cm {\tiny $N=1000$}
\\ \vskip-.6cm \hskip3.5cm 
{\tiny $K=250$} \hskip3.8cm {\tiny $K=250$} \hskip3.9cm {\tiny $K=1000$}
\\\vskip0.9cm $\left.\right.$
\end{minipage}}
\vip

\noindent\fbox{\begin{minipage}{\textwidth}
{\small Symmetric case, $p=0.48$, $\mu=20$ (slightly subcritical but large $\mu$). The choice is always bad 
for $t\in [1,15]$ and always good for $t\in [17,20]$.}
\\
\includegraphics[width=4.9cm]{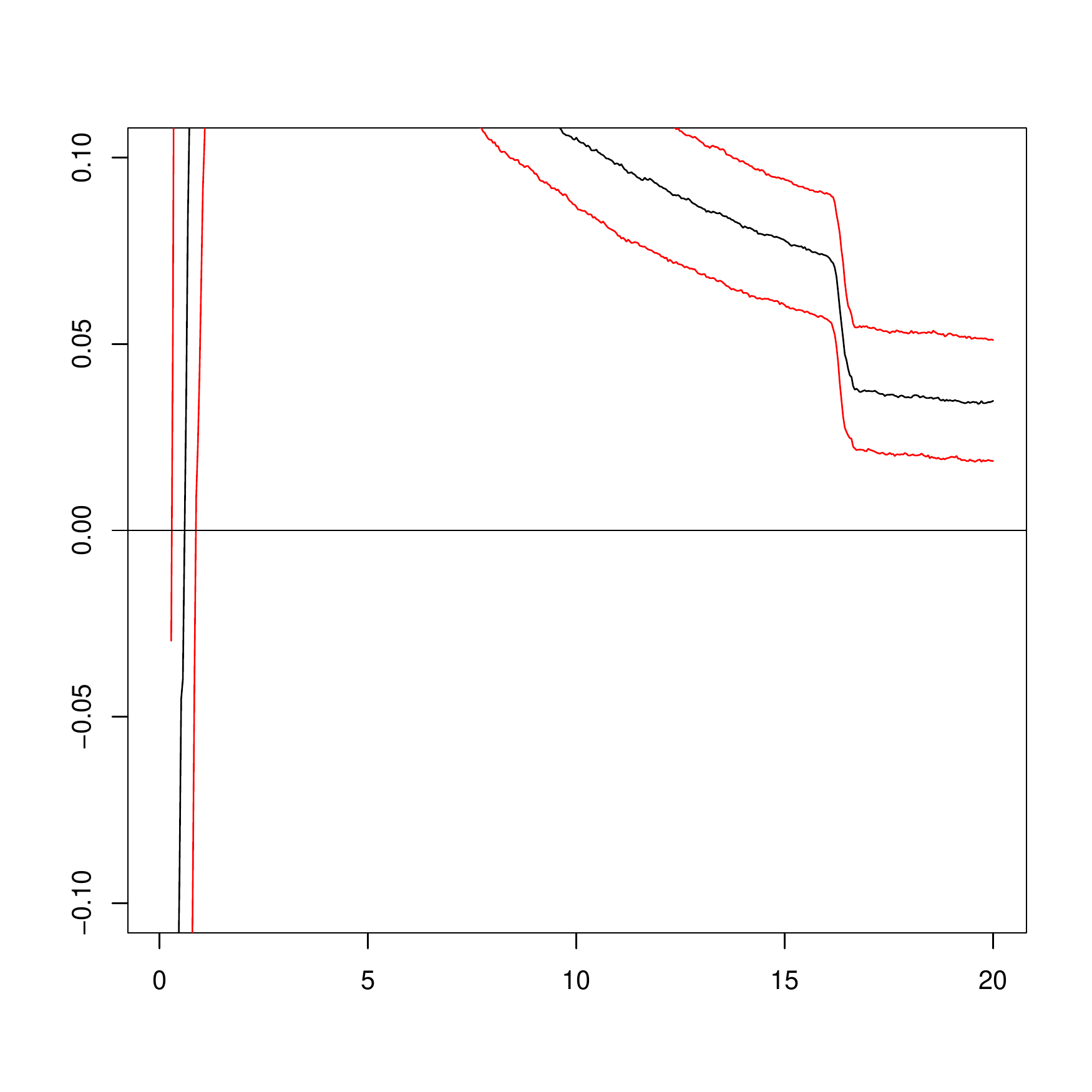}
\includegraphics[width=4.9cm]{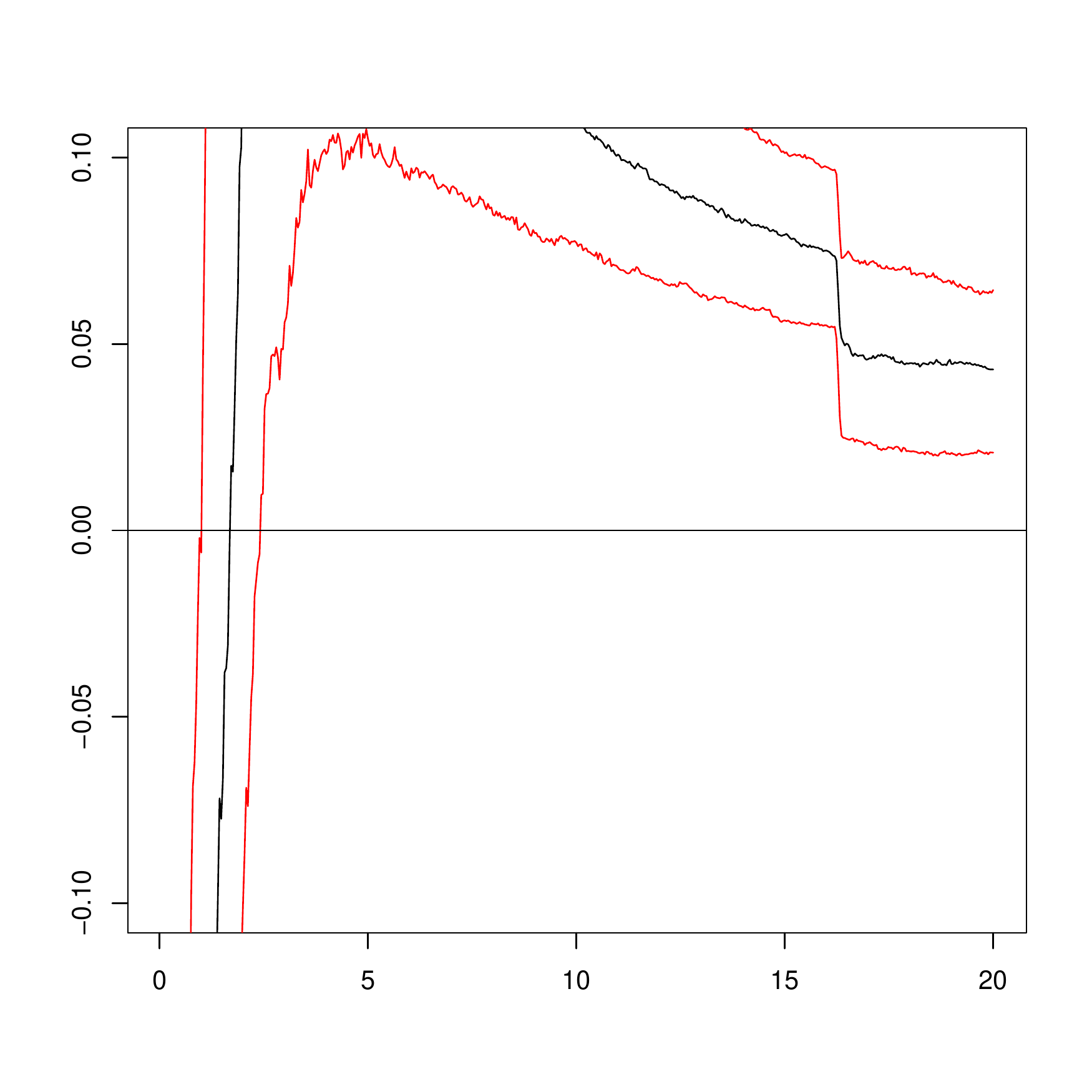}
\includegraphics[width=4.9cm]{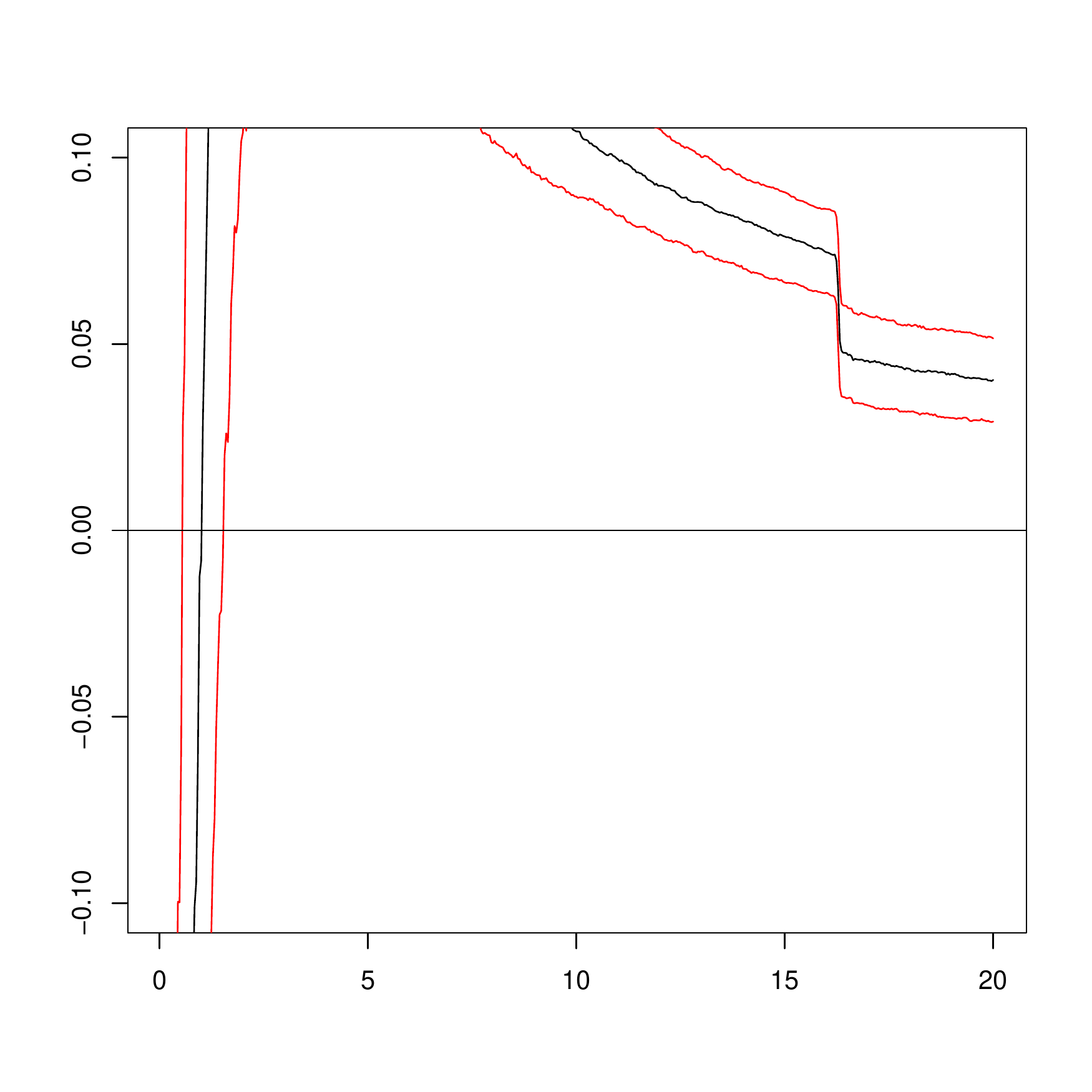}
\\ \vskip-2.9cm  \hskip3.5cm 
{\tiny $N=250$} \hskip3.8cm {\tiny $N=1000$} \hskip3.8cm {\tiny $N=1000$}
\\ \vskip-.6cm \hskip3.5cm 
{\tiny $K=250$} \hskip3.8cm {\tiny $K=250$} \hskip3.9cm {\tiny $K=1000$}
\\\vskip0.9cm
\hskip1cm {\small $-0.014,0.00095,0.018$} \hskip1.8cm{\small $-0.013,0.0020,0.017$}\hskip2cm
{\small $-0.0074,0.00053,0.0078$}
\end{minipage}}

\vip

\noindent\fbox{\begin{minipage}{\textwidth}
{\small Symmetric case, $p=0.35$, $\mu=1$ (fairly subcritical). 
The choice is always good for $t \in (0,900]$.}
\\
\includegraphics[width=4.9cm]{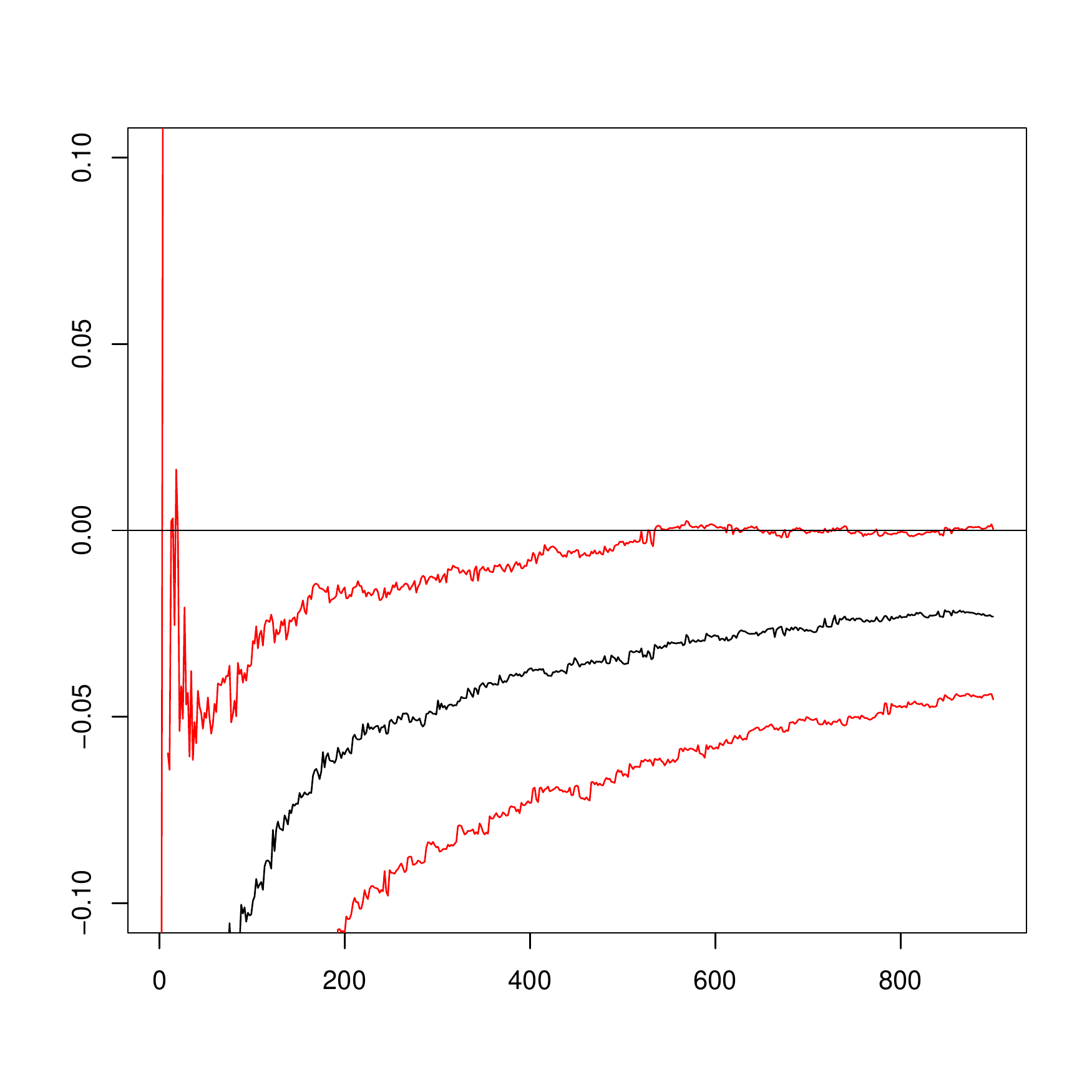}
\includegraphics[width=4.9cm]{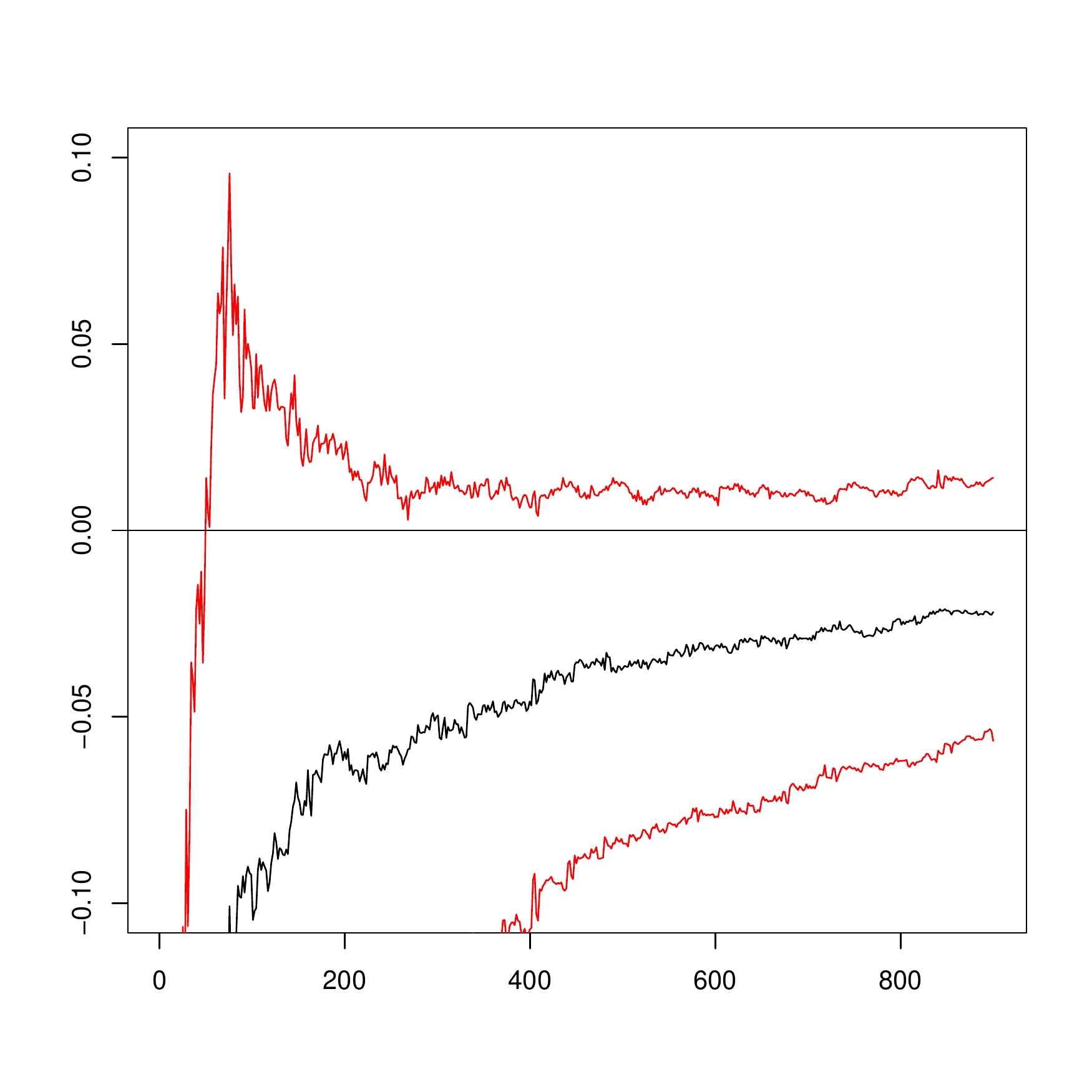}
\includegraphics[width=4.9cm]{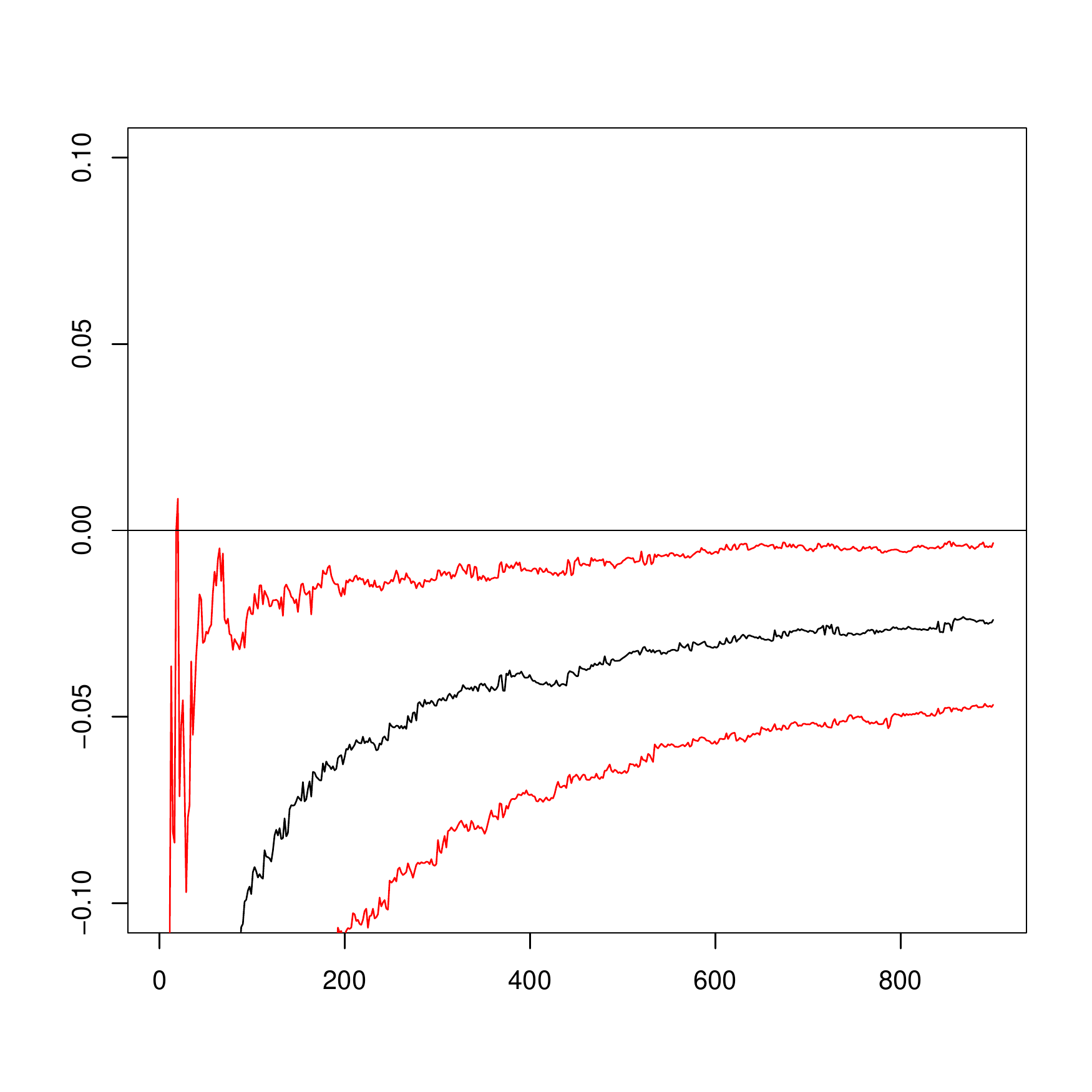}
\\ \vskip-4.9cm  \hskip3.5cm 
{\tiny $N=250$} \hskip3.8cm {\tiny $N=1000$} \hskip3.8cm {\tiny $N=1000$}
\\ \vskip-.6cm \hskip3.5cm 
{\tiny $K=250$} \hskip3.8cm {\tiny $K=250$} \hskip3.9cm {\tiny $K=1000$}
\\ \vskip3cm
\hskip0.9cm {\small $-0.010,0.0021,0.0165$} \hskip2cm{\small $-0.012,0.0021,0.018$}\hskip2cm
{\small $-0.0067,0.0011,0.0075$}
\end{minipage}}

\vip
\noindent\fbox{\begin{minipage}{\textwidth}
{\small Symmetric case, $p=0.1$, $\mu=1$ (fairly subcritical). 
The choice is always good.}
\\
\includegraphics[width=4.9cm]{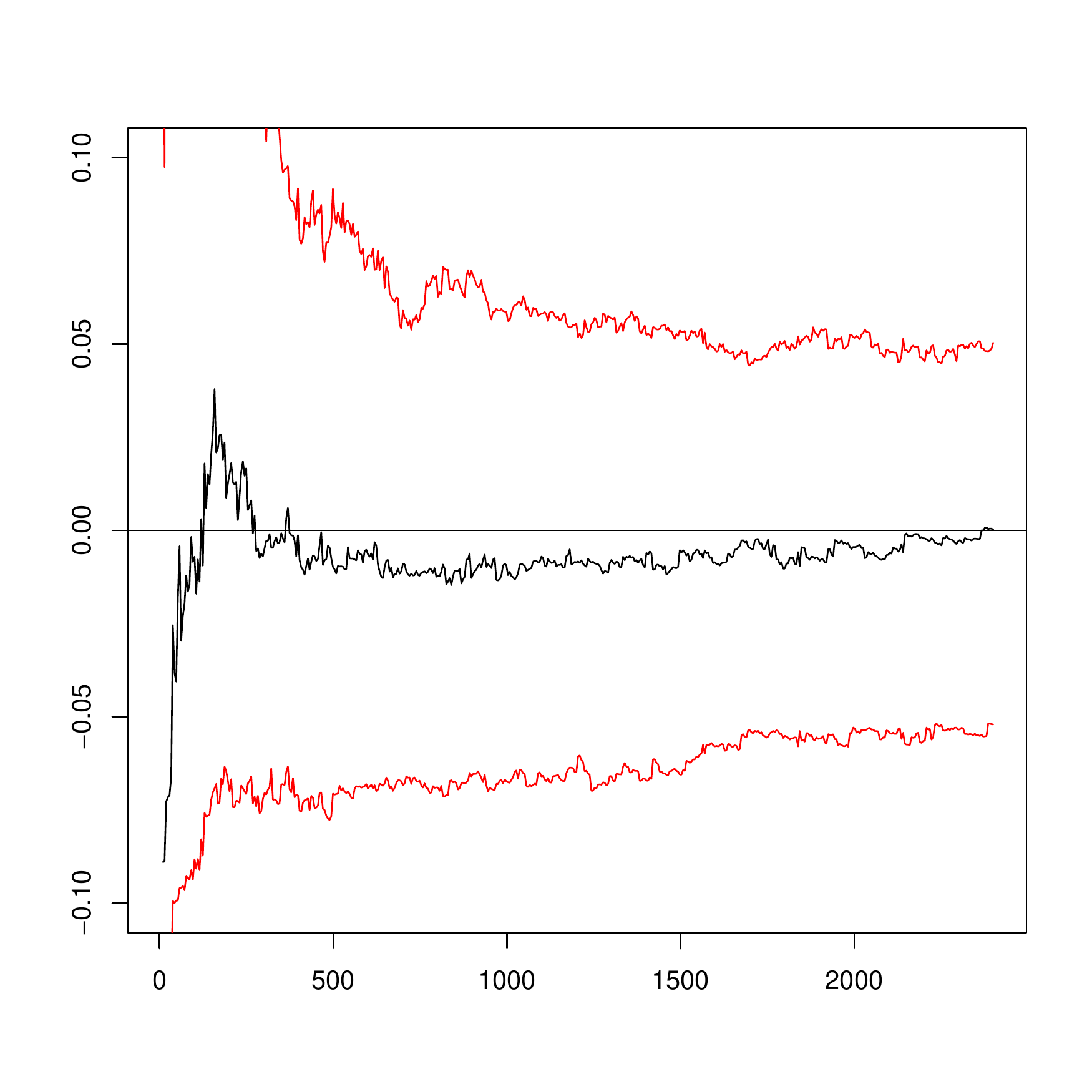}
\includegraphics[width=4.9cm]{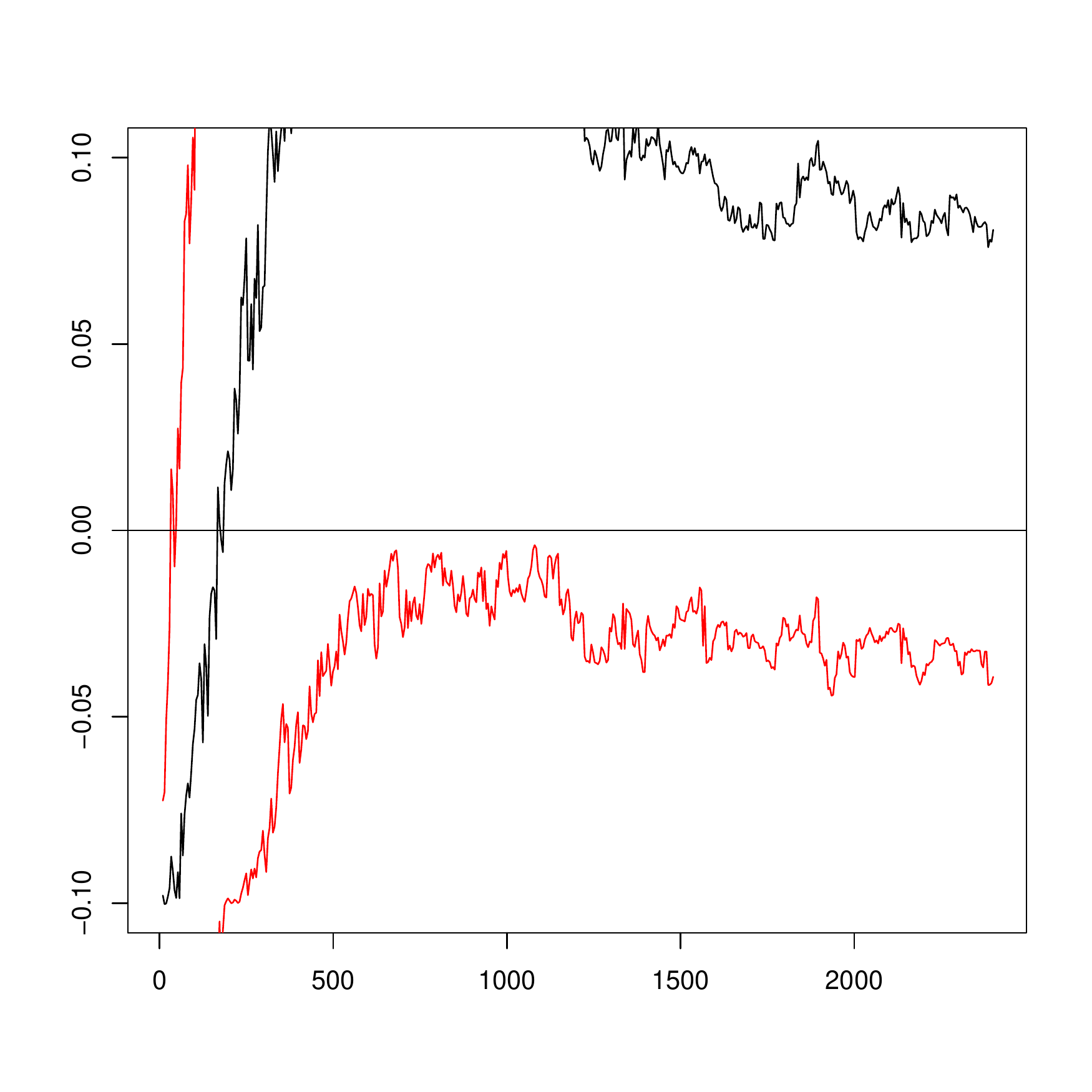}
\includegraphics[width=4.9cm]{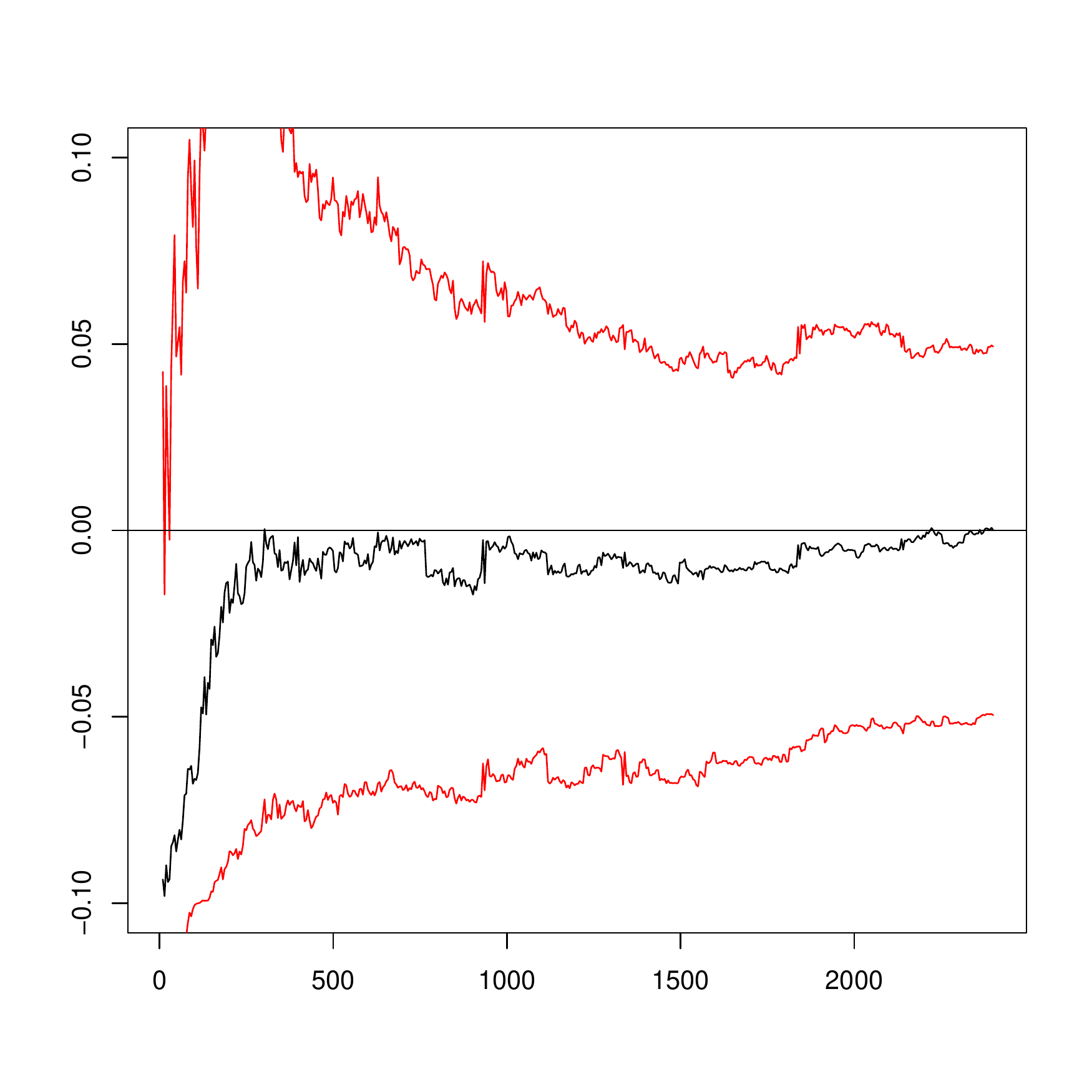}
\\ \vskip-4.9cm  \hskip2.5cm 
{\tiny $N=250$} \hskip2.8cm {\tiny $N=1000$} \hskip4.8cm {\tiny $N=1000$}
\\ \vskip-.6cm \hskip2.5cm 
{\tiny $K=250$} \hskip2.8cm {\tiny $K=250$} \hskip4.9cm {\tiny $K=1000$}
\\ \vskip3cm
\hskip0.9cm {\small $-0.0030,0.0029,0.0087$} \hskip1.6cm{\small $-0.0043,0.00091,0.0029$}\hskip1.6cm
{\small $-0.0022,0.00056,0.0029$}
\end{minipage}}

\vip

Finally, we discuss the practical choice of $\Delta$.

\vip

\noindent\fbox{\begin{minipage}{\textwidth}
{\small Independent case, $\mu=1$, $p=0.35$ (fairly subcritical), with $T=900$ and $N=K=1000$.
On the left, we have plotted $\cP^{sub,N,K}_{\Delta,T}-p$ as a function of $\Delta\in [1,15]$ 
obtained with one simulation.
On the right, we have plotted the quartiles of the same quantity using $1000$ simulations.}
\\
\centerline{\includegraphics[width=4.9cm]{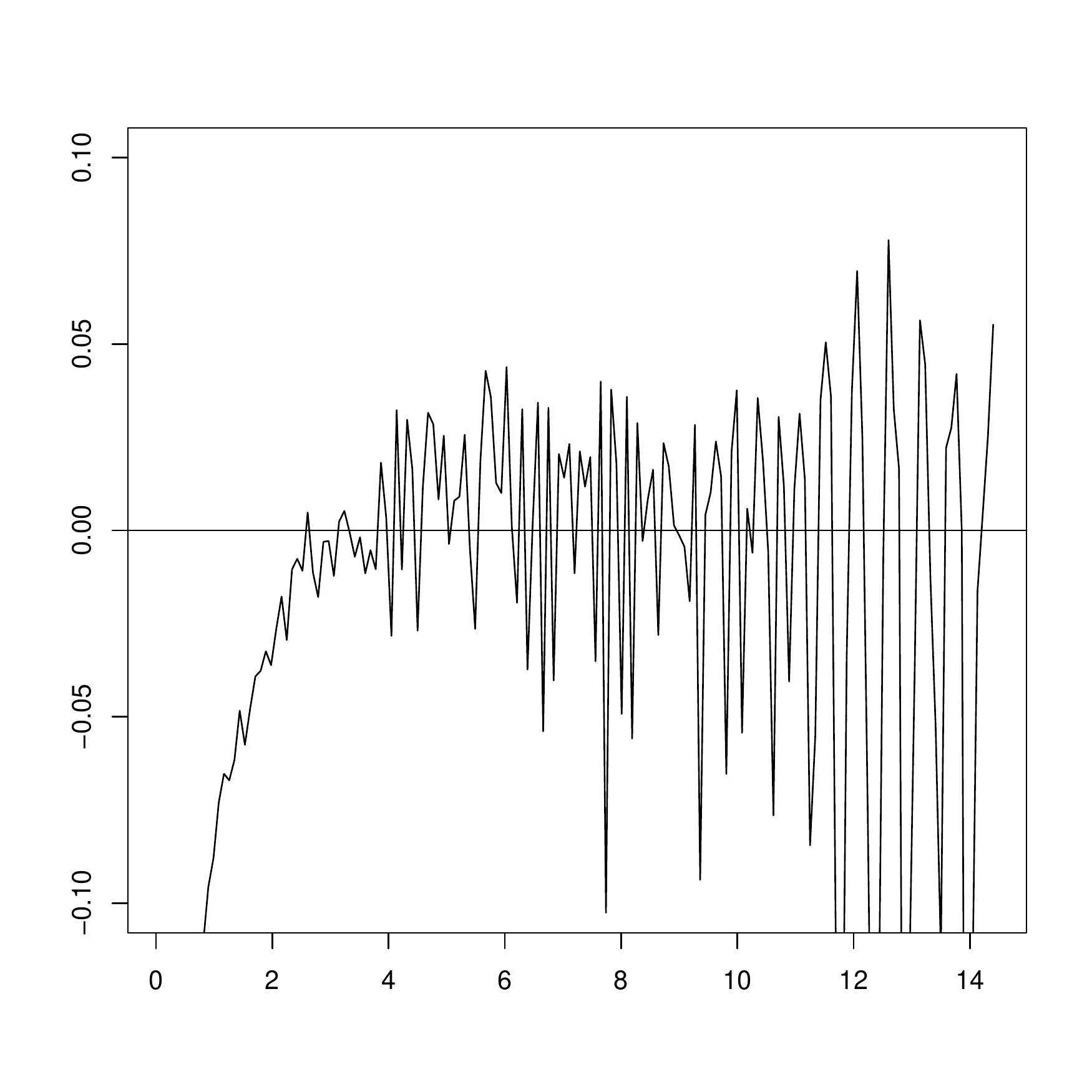}
\hskip1.75cm
\includegraphics[width=4.9cm]{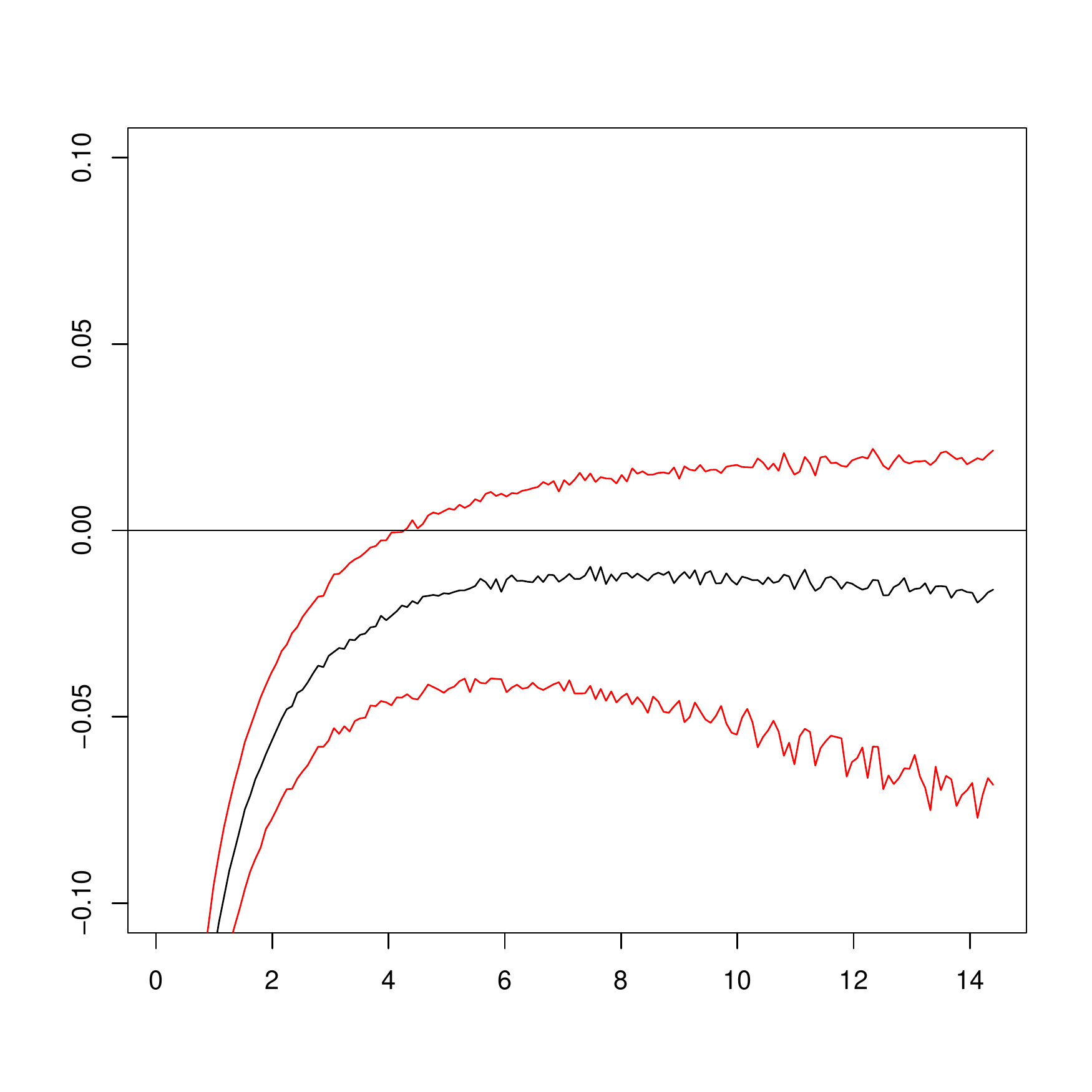}}
\\ 
{\small Our (arbitrary) choice $\Delta_T=T/(2\lfloor T^{9/13}\rfloor)\simeq 4.1$
seems rather suitable: we see on the right picture that the ``optimal'' $\Delta$ lies between $4$ and $6$.
This is mainly due to chance and probably depends strongly on the parameters of the model. 
We see on the left picture that given one set of data, 
$\cP^{sub,N,K}_{\Delta,T}$ varies a lot.}
\end{minipage}}


\begin{thebibliography}{99}


\bibitem{acl}{Y. Ait-Sahalia, J. Cacho-Diaz, R.J.A.Laeven, {\it Modeling financial 
contagion using mutually exciting jump processes,} to appear in the Journal of Financial Economics.}

\bibitem{bdhm1}{E. Bacry, S. Delattre, M. Hoffmann, J.F. Muzy, {\it Modeling microstructure 
noise with mutually exciting point processes,} Quantitative Finance 13 (2013), 65--77.}

\bibitem{bdhm2}{E. Bacry, S. Delattre, M. Hoffmann, J.F. Muzy, {\it Some limit theorems for Hawkes 
processes and applications to financial statistics},  Stoch. Processes Appl. 123 (2013), 2475--2499.}

\bibitem{bmu}{E. Bacry and J.F. Muzy, {\it Hawkes model for price and trades high-frequency dynamics}, 
Quantitative Finance 14 (2014), 1147--1166.}

\bibitem{bmu2}{E. Bacry and J.F. Muzy, {\it Second order statistics characterization of Hawkes processes 
and non-parametric estimation}, arXiv:1401.0903.}

\bibitem{bh}{L. Bauwens, N. Hautsch, {\it Modeling financial high frequency data using point processes}, 
ser. In T. Mikosch, J-P. Kreiss, R. A. Davis, and T. G. Andersen, editors, Handbook of Financial Time Series. 
Springer, 2009.}

\bibitem{birk}{G. Birkhoff, {\it Extensions of Jentzsch’s theorem}, Trans. Am. Math. Soc. 85 (1957),
219--227.}

\bibitem{bhb}{C. Blundell, K.A. Heller, J.F. Beck, {\it Modeling reciprocating relationships with Hawkes 
processes}, Neural Information Processing Systems 2012.}

\bibitem{bm}{P. Br\'emaud, L. Massouli\'e, {\it Stability of nonlinear Hawkes processes,} Ann. Probab.
24 (1996), 1563--1588.}

\bibitem{bnt}{P. Br\'emaud, G. Nappo, G.L. Torrisi, {\it Rate of convergence to equilibrium of 
marked Hawkes processes}, J. Appl. Probab. 39 (2002), 123--136.}

\bibitem{cc}{R. Cavazos-Cadena, {\it An alternative derivation of Birkhoff's formula for the contraction 
coefficient of a positive matrix}, Linear Algebra Appl. 375 (2003), 291--297.}

\bibitem{dvj}{D.J. Daley, D. Vere-Jones, {\it An introduction to the theory of point processes},
Vol. I. Probability and its Applications. Springer-Verlag, second edition, 2003.} 

\bibitem{dfh}{S. Delattre, N. Fournier, M. Hoffmann, {\it Hawkes processes on large networks}, 
to appear in Ann. Appl. Probab.}

\bibitem{f}{W. Feller, {\it On the integral equation of renewal theory},
Ann. Math. Statistics 12 (1941), 243--267.}

\bibitem{gda}{S. Gr\"un, M. Diedsmann, A.M. Aertsen, {\it Unitary events analysis}, in {\it Analysis of 
parallel spike trains}, S. Gr\"un and S. Rotter, Springer series in computational neurosciences, 2010.}

\bibitem{hrr}{N.R. Hansen, P. Reynaud-Bouret, V. Rivoirard, {\it Lasso and probabilistic inequalities for 
multivariate point processes}, to appear in Bernoulli.}

\bibitem{h}{A. Hawkes, {\it Spectra of some self-exciting and mutually exciting point processes}, 
Biometrika 58 (1971), 83--90.}

\bibitem{ho}{A. Hawkes, D. Oakes, {\it A cluster process representation of a self-exciting process}, 
J. Appl. Probability 11 (1974), 493--503.}

\bibitem{hs}{A. Helmstetter and D. Sornette, {\it Subcritical and supercritical regimes in epidemic 
models of earthquake aftershocks,} Journal of geophysical research, 107 (2002), 2237.}

\bibitem{he}{P. Hewlett {\it Clustering of order arrivals, price impact and trade path optimisation},
In Workshop on Financial Modeling with Jump processes. Ecole Polytechnique, 2006.}

\bibitem{hoe}{W. Hoeffding, {\it Probability inequalities for sums of bounded random variables}, 
J. Amer. Statist. Assoc. 58 (1963), 13--30.}

\bibitem{js}{J. Jacod, A.N. Shiryaev, {\it Limit theorems for stochastic processes,} Second edition. 
Springer-Verlag, 2003.}

\bibitem{k}{Y.Y. Kagan, {\it Statistical distributions of earthquake numers: consequence of branching process}, 
Geophysical Journal International 180 (2010), 1313--1328.}

\bibitem{m}{L. Massouli\'e, {\it Stability results for a general class of interacting point processes dynamics, 
and applications}, Stochastic Process. Appl. 75 (1998), 1--30.}

\bibitem{o1}{Y. Ogata, {\it The asymptotic behaviour of maximum likelihood estimators for stationary 
point processes}, Ann. Instit. Math. Statist. 30 (1978), 243--261.}

\bibitem{o2}{Y. Ogata, {\it Seismicity analysis through point-process modeling: A review}, Pure and Applied 
Geophysics 155 (1999), 471--507.}

\bibitem{owb}{M. Okatan, M.A. Wilson, E.N. Brown, {\it Analyzing functional connectivity using a network 
likelihood model of ensemble neural spiking activity}, Neural Computation 17 (2005), 1927-1961.}

\bibitem{psp}{J.W. Pillow, J. Shlens, L. Paninski, A. Scher, A.M. Litke, E.J. Chichilnisky, E.P. Simoncelli, 
{\it Spatio-temporal correlations and visual signalling in a complete neuronal population}, Nature 454 
(2008), 995--999.}

\bibitem{r}{J.G. Rasmussen, {\it Bayesian inference for Hawkes processes}, 
Methodol. Comput. Appl. Probab. 15 (2013), 623--642.}

\bibitem{rs}{P. Reynaud-Bouret and S. Schbath, {\it Adaptive estimation for Hawkes processes: application to 
genome analysis,} Ann. Statist. 38 (2010), 2781--2822.}

\bibitem{rrgt}{P. Reynaud-Bouret, V. Rivoirard, F. Grammont, C. Tuleau-Malot, {\it Goodness-of-fit tests and 
nonparametric adaptive estimation for spike train analysis},  Journal of Math. Neuroscience 4:3 (2014).}

\bibitem{rrt}{P. Reynaud-Bouret, V. Rivoirard, C. Tuleau-Malot, {\it Inference of functional connectivity in 
Neurosciences via Hawkes processes}, 1st IEEE Global Conference on Signal and Information Processing, 2013.}

\bibitem{s}{A.S. Sznitman, {\it Topics in propagation of chaos,} Ecole d'\'Et\'e de Probabilit\'es 
de Saint-Flour XIX-1989, Vol. 1464 of Lecture Notes in Math. Springer, 1991, 165--251.}

\bibitem{zzs}{K. Zhou, H. Zha, L. Song, {\it Learning triggering kernels for multi-dimensional Hawkes 
processes}, Proceedings of the 30th International Conference on Machine Learning, 2013.}

\bibitem{z1}{L. Zhu, {\it Central limit theorem for nonlinear Hawkes processes}, J. App. Probab. 50 (2013), 
760--771.}

\bibitem{z2}{L. Zhu, {\it Large deviations for Markovian nonlinear Hawkes processes}, 
Ann. App. Probab. 25 (2014), 548--581.}

\end{thebibliography}
\end{document}